\documentclass[11pt]{amsart}
\usepackage{enumitem}
\usepackage[margin=3cm]{geometry}

\usepackage{amsmath,amsthm,amsfonts}
\usepackage{mathtools}
\usepackage{mathrsfs}

\usepackage{tikz}
\usepackage{hyperref}

\usepackage{comment}
\usepackage{xcolor}

\usepackage[english]{babel}

\numberwithin{equation}{section} 

\theoremstyle{plain}

\newtheorem{thm}{Theorem}[section]
\newtheorem{cor}[thm]{Corollary}
\newtheorem{prop}[thm]{{\bf Proposition}}
\newtheorem{lem}[thm]{{\bf Lemma}}
\newtheorem{claim}[thm]{Claim}
\newtheorem{rem}[thm]{Remark}

\newcounter{hyp-counter}
\theoremstyle{definition}
\newtheorem{defn}[thm]{Definition}
\newtheorem{conjecture}[thm]{Conjecture}

\theoremstyle{remark}

\newtheorem{example}[thm]{Example}

\newcommand{\diag}{\operatorname{diag}}

\newcommand{\diam}{\operatorname{diam}}

\newcommand{\Diff}{\operatorname{Diff}}

\newcommand{\Jac}{\operatorname{Jac}}
\newcommand{\len}{\operatorname{len}}

\newcommand{\N}{\mathbb{N}}

\newcommand{\R}{\mathbb{R}}
\newcommand{\RP}{\mathbb{R}\operatorname{P}}
\newcommand{\SL}{\operatorname{SL}}
\newcommand{\SO}{\operatorname{SO}}

\newcommand{\vol}{\operatorname{vol}}

\newcommand{\Z}{\mathbb{Z}}

\newcommand{\wt}[1]{\widetilde{#1}}

\newcommand{\abs}[1]{\left| #1\right|}

\newcommand{\mc}[1]{\mathcal{#1}}

\newcommand{\pez}[1]{\left( #1\right)}
\newcommand{\E}[1]{\mathbb{E}\left[{#1}\right]}

\makeatletter
\def\blfootnote{\xdef\@thefnmark{}\@footnotetext}
\makeatother

\title[Expanding on Average Diffeomorphisms]{Expanding on Average Diffeomorphisms of Surfaces: Exponential Mixing}
\author{Jonathan DeWitt and Dmitry Dolgopyat}
\address{Department of Mathematics, The University of Maryland, College Park, MD 20742, USA}
\email{dewitt@umd.edu, dolgop@umd.edu}
\date{\today}

\begin{document}
\begin{abstract}
We show that the Bernoulli random dynamical system associated to a expanding on average tuple of volume preserving diffeomorphisms of a closed surface is exponentially mixing.
\end{abstract}

\maketitle

\tableofcontents

\allowdisplaybreaks

\section{Introduction}

\subsection{The main result}
In this paper, we prove exponential equidistribution and mixing results for expanding on average random dynamical systems. 
Suppose that $M$ is a closed Riemannian surface with a smooth area, and
$(f_1,\ldots,f_m)$ is a tuple of diffeomorphisms in $\Diff^2_{\vol}(M)$. We then define a random dynamical system, where at each time step we choose uniformly at random an index $i\in \{1,\ldots , m\}$ and apply $f_i$ to $M$. 
We call this the (uniform Bernoulli) random dynamical system on $M$ associated to the tuple 
$(f_1,\ldots,f_m)$. A  realization of the randomness is then given by a word from $\Sigma=\{1,\ldots, m\}^{\N}$.
As usual, we equip $\Sigma$ with the distance $d(\omega', \omega'')=2^{-k}$
where $k=\max\{N: \omega'_n=\omega''_n\text{ for }n<N\}$.
We let $\sigma\colon \Sigma\to \Sigma$ denote the left shift and let $\mu$ the uniform Bernoulli product measure on $\Sigma$.

For such random dynamical systems, mixing does not hold for all tuples $(f_1,\ldots,f_m)$. We will introduce an additional hypothesis. We say that a tuple $(f_1,\ldots,f_m)$ is \emph{expanding on average} if there exists $\lambda>0$ and $n_0\in \N$ such that for all $v\in T^1M$,  the unit tangent bundle of $M$,
\begin{equation}
\label{EqExpAv}
\frac{1}{n_0}\E{\ln \| Df^{n_0}_{\omega} v\|}\ge \lambda>0.
\end{equation}
Note that \eqref{EqExpAv} is  a $C^1$-open condition on the tuple $(f_1,\ldots,f_m)$, so in principle it could be checked on a 
computer (cf.~\cite{chung2020stationary}).

The main result of our paper is that the systems satisfying \eqref{EqExpAv} enjoy exponential mixing. 

\begin{thm}(Quenched Exponential Mixing) 
\label{ThQEM}
Suppose that $M$ is a closed surface and that $(f_1,\ldots,f_m)$ is an expanding on average 
 tuple of diffeomorphisms in $\Diff^2_{\vol}(M)$. Let $\beta\in (0,1)$ be a H\"older regularity. There exists $\eta>0$ such that for 
 a.e.~$\omega\in\Sigma$, there exists
 $C_{\omega}$ such that for any $\phi,\psi\in C^{\beta}(M)$,
\begin{equation}
 \label{EqQEM}   
\abs{\int \phi \psi\circ f^n_{\omega}\,d\vol-\int \phi\,d\vol\int \psi\,d\vol}\le C_{\omega}e^{-\eta n}\|\phi\|_{C^{\beta}}\|\psi\|_{C^{\beta}}
\end{equation}
 where $f_{\sigma^j(\omega)}^i=f_{\omega_{j+i}}\cdots f_{\omega_{j+1}}$.
Further, there exists $D_1>0$ such that
\begin{equation}
 \label{EqEMTail}  
\mu(\omega:C_{\omega}\ge C)\le D_1C^{-1}.
\end{equation}
\end{thm}

 In fact, the tail bound \eqref{EqEMTail} implies a related result, annealed exponential mixing for the associated skew product.
  We give the proof of the following in 
\S \ref{sec:limit_theorems}. 

\begin{cor}
\label{CrAnEM}
(Annealed Exponential Mixing) Let $M$ be a closed surface, let $(f_1,\ldots,f_m)$ be an expanding on average tuple in $\Diff^2_{\vol}(M)$, and $\beta\in(0,1)$ be a H\"older regularity. 
Let $F\colon \Sigma\times M\to \Sigma\times M$ be the skew product defined by
$$ F(\omega, x)=(\sigma(\omega), f_{\omega_0} (x)).$$
Then $F$ is exponentially mixing, that is,
there exist $\bar\eta>0$, $D$ such that for any $\Phi,\Psi\in C^{\beta}(\Sigma\times M)$, 
$$
\left|\iint \Phi (\Psi\circ F^n)\, d\mu \,d\vol
-\iint \Phi\,d\mu\, d\vol \iint\Psi\,d\mu\, d\vol\right|
\le De^{-\bar\eta n}\|\Phi\|_{C^{\beta}}\|\Psi\|_{C^{\beta}}.
$$
\end{cor}
 Before we proceed to discussing the relationship of this work with the existing literature, we will look at some examples of systems satisfying 
 \eqref{EqExpAv}.

\begin{rem}
 Although we have written this paper for a finite tuple $(f_1,\ldots,f_m)$ of diffeomorphisms to emphasize the discreteness of the noise, one can consider random dynamics generated by any probability measure $\mu$ on $\Diff^2_{\vol}(M)$. 
Similar arguments to the ones we present here imply the analogous conclusions hold for random dynamics generated by a measure $\mu$ with compact support on $\Diff^2_{\vol}(M)$, where $M$ is a closed surface. 
\end{rem}

\subsection{Examples}\label{subsec:examples}
There are a number of sources of tuples $(f_1,\ldots,f_m)$ that are  
expanding on average.
The random dynamics arising from such tuples may exhibit uniform or non-uniform hyperbolicity. One of the simplest and archetypal examples is the following.

\begin{example}
Suppose that $(A_1,\ldots,A_m)$ is a tuple of matrices in $\SL(2,\Z)$ satisfying the hypotheses of Furstenberg's theorem, namely the tuple is strongly irreducible and contracting. Then the Bernoulli random product of these matrices has a positive top Lyapunov exponent. It follows from the proof of Furstenberg's theorem, see, e.g.~\cite[Thm.~III.4.3]{BougerolLacroix},
that there exists $N$ and $\lambda>0$ such that for all unit vectors $v\in \R^2$,
\[
N^{-1}\E{\ln \|A^N_{\omega}v\|}\ge \lambda>0.
\]
Each $A_i\in \SL(2,\Z)$ acts on $\mathbb{T}^2=\R^2/\Z^2$, and the associated random dynamics on $\mathbb{T}^2$ is uniformly expanding on average. Because this is an open condition, we see that any volume preserving perturbation of the $A_i$ is also uniformly expanding. Thus, our theorem applies to a class of non-linear systems that do not exhibit any uniform hyperbolicity. 

 In addition, the expanding on average property generalizes to many other random walks on homogeneous spaces,  see for example \cite[Def.~1.4]{eskin2018random}, which uses this property to study stiffness of stationary measures of random walks on homogeneous spaces.
\end{example}

Expanding on average systems also arise as perturbations of isometric systems.

\begin{example}
Perhaps the first example where this condition was considered for nonlinear diffeomorphisms was  the paper of Dolgopyat and Krikorian \cite{dolgopyat2007simultaneous}.
Suppose that $(R_1,\ldots,R_m)$ is a tuple of isometries of $S^2$ that generates a dense subgroup of $\SO(3)$. Then 
\cite{dolgopyat2007simultaneous} shows that there exists $k_0$ such that if $(f_1,\ldots,f_m)$ is a sufficiently $C^{k_0}$ small volume preserving perturbation of $(R_1,\ldots,R_m)$, and the tuple $(f_1,\ldots,f_m)$ has a stationary measure with non-zero Lyapunov exponents, then $(f_1,\ldots,f_m)$ is expanding on average. See also DeWitt \cite{dewitt2024simultaneous}.
\end{example}

 Other work has explored how ubiquitous expanding on average systems are, in some cases studying whether expanding on average systems can be realized by perturbing a known system of interest.
\begin{example}
Chung \cite{chung2020stationary} gives a proof that certain random perturbations of the standard map are expanding on average
(see also \cite{blumenthal2017lyapunov, blumenthal2018lyapunov} 
which studies the size of Lyapunov exponents for perturbations of the standard map with a large coupling constant).
\cite{chung2020stationary} also 
presents convincing numerical simulations showing 
that certain actions on character varieties are expanding on average as well. 
\end{example}

There are also some results that construct expanding on average systems densely in a weak* sense.

\begin{example}
The paper \cite{potrie2022remark} says that for every open set $\mc{U}\subseteq \Diff^{\infty}_{\vol}(M)$, where $M$ is a surface, there exists a finitely supported measure on $\mc{U}$ that is expanding on average. This result was generalized to higher dimensions in \cite{elliott2023uniformly}. 
\end{example}

\subsection{Relationship with other works}
Exponential mixing plays the central role in the study of statistical properties of dynamical systems. 
In particular, multiple exponential mixing implies several probabilistic results including
the Central Limit Theorem \cite{Chernov06, BG20}, Poisson Limit Theorem \cite{DFL22},and the dynamical Borel Cantelli Lemma
\cite{Galatolo10} among others.
Further, exponential mixing was recently shown to imply Bernoullicity \cite{dolgopyat2024exponential}.

 For deterministic systems, however, robust
exponential mixing has been only established for a limited class of systems:
uniformly hyperbolic systems in both smooth and piecewise smooth settings \cite{chernov2006chaotic, VianaLima, Young98},
or for partially hyperbolic systems where all Lyapunov exponents in the central direction have the same sign
\cite{deCastro, DeCastroVarandas, dolgopyat00}. 
Here we say that a certain property holds robustly if it holds for a given
system as well as for its small perturbations. In contrast, if additional symmetries are present then there are many other
cases where exponential mixing is known, see \cite{GorodnikSpatzier14, KleinbockMargulis96, Liverani04, TsujiiZhang23}.
There are also checkable conditions for exponential mixing in the nonuniformly hyperbolic setting,
see \cite{Young98, Young99}. However, except for the aforementioned examples, 
these conditions hold for individual systems rather than open sets. 
On the other hand KAM theory tells us that away from (partially) hyperbolic systems
one has open sets of non-ergodic systems, so one cannot expect chaotic behavior to be generic.

The situation is different for random systems. In fact, if the supply of random maps is rich enough then 
one show that exponential mixing and other statistical properties hold generically. 
Such results are known for stochastic 
flows of diffeomorphisms \cite{DKK04} as well as for random deterministic shear flows \cite{BCZG}.
It is therefore natural to ask how large should the set of random diffeomorphisms must be so that the
corresponding random dynamical system exhibits random behavior. The following conjecture is
formulated in \cite{dolgopyat2007simultaneous}.

\begin{conjecture}
\label{ConExpGen}
For each closed manifold $M$ with volume and regularity $k\ge 1$, there exists $m$, such that the space of tuples $(f_1,\ldots,f_m)$ that are stably ergodic is open and dense in $\displaystyle\left(\Diff^k_{\vol}(M)\right)^m$.
\end{conjecture}

The point of this conjecture is that only a tiny bit of randomness, perhaps even the minimum amount, should be sufficient to ensure robust ergodic and statistical properties for dynamical system. Consequently, the situation where the driving measure has uniformly small, finite support on $\Diff^{2}_{\vol}(M)$ is the most interesting, and hardest case to consider this question. 
The obvious approach to this conjecture is to first to show that an open and dense set of tuples is expanding on average. 

Other papers have significantly extended the properties of expanding on average systems.
One of the first is 
\cite{BrownRodriguezHertz}, which shows a strong stiffness property of these systems: any stationary measure for the Markov process that is not finitely supported is volume \cite[Thm.~3.4]{BrownRodriguezHertz}. Thus, in some sense, volume is the only measure whose statistical properties are interesting to study. 
 The only statistical property beyond ergodicity studied before for expanding on average systems
is large deviations for ergodic sums established in \cite[Thm.~4.1.1]{liu2016lyapunov}.
Our paper provides an additional contribution to this topic by showing that expanding on average systems
enjoy exponential mixing.
In fact, Conjecture \ref{ConExpGen} provides an additional motivation for this work, because it shows that should the conjecture be true, then exponential mixing is a generic property for random dynamical systems.

Some work has been done towards showing that uniform expansion is a generic property.
In particular, \cite{obata2022positive} shows that one may obtain positive integrated Lyapunov exponent for conservative random systems on surfaces. This work differs from the papers \cite{potrie2022remark} and \cite{elliott2023uniformly} as \cite{obata2022positive} does not require an arbitrarily large number of diffeomorphisms to obtain its result.

Returning to deterministic systems, it is natural to ask for conditions for strong statistical properties to hold in a robust way. 
Optimal conditions are not yet well understood. 
While there are strong indications that at least a dominated splitting 
is necessary \cite{PalisConj}, the best available results pertain to partially hyperbolic systems.
A well known conjecture of Pugh and Shub 
 \cite{shub2006all} states that stably ergodic systems contain an open and dense subset of partially
hyperbolic systems. Currently the best results on this problem are due to \cite{burns2010ergodicity} which can be consulted 
for a detailed discussion on this subject. In fact, the methods of Pugh and Shub also give the $K$-property \cite{burns2010ergodicity}. Going beyond the $K$-property remains an outstanding challenge even in the partially hyperbolic 
setting. In view of the strong consequences of exponential mixing it is natural to  conjecture the following.

\begin{conjecture}
    Exponential mixing holds for an open and dense set of volume preserving partially hyperbolic systems.
\end{conjecture}

Currently there are two possible ways to attack this conjecture. The first one is based on the theory of weighted Banach
spaces, \cite{AGT01, CastorriniLiverani, GouezelLiverani06, Tsujii01, TsujiiZhang23}. To describe the second approach recall that
the  papers \cite{Viana08, AvilaViana10} show that partially hyperbolic systems often have non-zero exponents.
It is therefore natural to see if one could try to extend the methods  used in proving exponential mixing in 
non-uniformly hyperbolic systems to handle partially hyperbolic setting. As mentioned above, this approach 
was successful in handling the case there the central exponents have the same sign. In the present paper we consider
a skew product with a shift in the base and where the Lyapunov exponents in the central direction have different signs. 
We hope that a similar approach could be useful for studying more general skew products,
and hopefully could provide a blueprint for studying mixing in partially hyperbolic systems. 

In summary, the present work is the first step in extending mixing to a large class of smooth systems both
random and deterministic, and we hope that various extensions will be addressed in future works.
\medskip

\noindent\textbf{Acknowledgments:} The first author was supported by the National Science Foundation under Award No.~DMS-2202967. The
second author was supported by the National Science Foundation under award No.~DMS-2246983.
The authors are grateful to 
Matheus Manzatto de Castro for comments on an earlier version of the manuscript.

\section{Setting and basic definitions}
\subsection{Random dynamics and skew products}
In this section, we will state some basic definitions that will be used throughout the paper. Although we introduce many of these definitions and notations here, we will recall and reintroduce them when they are used; this section is just an overview.

We begin by recalling the main definition of our setup.

\begin{defn}\label{defn:expanding_on_average}
    We say that a tuple $(f_1,\ldots,f_m)\in \Diff^1(M)$ is \emph{expanding on average} if there exists some $n_0\in \N$ and $\lambda_0>0$ such that for all $v\in T^1M$,
\begin{equation}
\E{n_0^{-1}\ln \|Df^{n_0}_{\omega} v\|}\ge \lambda_0>0.
\end{equation}
\end{defn}

Throughout the paper, $(f_1,\ldots,f_m)$ typically denotes an uniformly expanding on average tuple of volume preserving diffeomorphisms of a closed surface $M$. However, in some cases, we merely are referring to a tuple and do not make use of any further assumptions. 

We write $(\Sigma,\sigma)$ for the one sided shift on $m$ symbols, i.e. $\Sigma=\{1,\ldots,m\}^{\mathbb{N}}$ with $\sigma$ being the left shift. We endow this space with the measure $\mu$, which is the uniform Bernoulli measure on $\Sigma$. Write $\hat{\Sigma}$ and $\hat{\mu}$ for the two-sided shift and the invariant Bernoulli measure over $\mu$.

We may view the random dynamics in two ways. First, as a Markov process on $M$. The second way, as mentioned in the statement of Corollary \ref{CrAnEM}, is as the skew product $F\colon \Sigma\times M \to \Sigma\times M$. This skew product preserves the product measures $\mu\otimes\vol$. When we say that the tuple $(f_1,\ldots,f_m)$ is \emph{ergodic}, we mean that the skew product $F$ is ergodic for the measure $\mu\otimes \vol$. This is equivalent to the absence of almost surely invariant Borel subsets of $M$ of intermediate measure. See \cite{kifer1986ergodic} for more discussion of the relationship between the skew product and the random dynamics on $M$.

For a word $\omega\in \Sigma$, we write $f^n_{\omega}\colon M\to M$ for the composition $f_{\omega_{n}}\cdots f_{\omega_1}$. We use the same notation for finite words $\omega$. For a sequence of linear maps $(A_i)_{1\le i \le n}$, we write $A^i=A_i\cdots A_1$. We do not always start this product with the first matrix, so we also have the notation
\[
A_i^k=A_{i+k}\cdots A_{i+1}.
\]
Note that this is compatible with the notation
$f_{\sigma^j(\omega)}^i=f_{\omega_{j+i}}\cdots f_{\omega_{j+1}}$
from above.

\subsection{Stable subspaces} For a sequence of linear maps, we will frequently use the singular value decomposition when it is defined. If we have a sequence of matrices $A_1,A_2,\ldots$  then, when it is defined, we write $E^s_n$ for the most contracted singular direction of $A^n$. We usually apply this to the sequence of linear maps $D_xf^n_{\omega}$. We write $E^s_i(\omega,x)$ for most contracted singular direction of $D_xf^i_{\omega}$, and we write $E^u_{i}(\omega,x)$ for the most expanded singular direction of $D_xf^i_{\omega}$, should these directions be well defined. Often we will suppress the $x$ and $\omega$ and just write $E^s_i$, other times we will write $E^s_{\omega}(x)$.

Throughout the paper we will consider sets
$\Lambda^{\omega}_n$ which are the sets of points $x\in M$ that  are $(C,\lambda,\epsilon)$-tempered for the word $\omega$ up until time $n$, where temperedness is defined in \S\ref {subsec:tempered_sequences_of_linear_maps}. These points are essentially the finite time analogue of a Pesin block, c.f.~\cite{barreira2007nonuniform}.

\subsection{Stable manifolds}
\label{SSStable}

The most important dynamical objects we will consider are the stable manifolds and fake stable manifolds.
Given a point $x\in M$, we define its stable manifold to be the set of points
\[
W^s(\omega,x)=\{y\in M: d(f^n_{\omega}(x),f^n_{\omega}(x))\text{ exponentially fast}\}.
\]
Note that the stable manifold depends on $\omega$. We denote a segment of length $2\delta$ centered at $x$ in $W^s(\omega,x)$ by $W^s_{\delta}(\omega,x)$. The properties of these ``true" stable manifolds are discussed in Section~\ref{ScStMan}. For general information about stable manifolds in random dynamical systems, see \cite{Liu1995smooth}.

As alluded to above, we will not only work with the stable manifolds, but also with finite time versions of stable manifolds. We will denote by $W^{s}_{n,\delta_0}(\omega,x)$ the time $n$ fake stable manifold of $x$ for the word $\omega$ restricted to segment of radius $\delta_0$ centered at $x$. The point of the fake stable manifolds is that up to time $n$, they have similar contraction properties to an actual stable manifold. In the limit, they converge to the true stable manifold. Their definition is somewhat technical, but a detailed treatment of the fake stable manifolds is given in Appendix \ref{sec:finite_time_pesin_theory} which essentially concerns itself with a quantified, finite time version of Pesin theory.

An important application of stable manifolds, fake or otherwise, is their holonomy.
Suppose that we have two curves $\gamma_1,\gamma_2$ and a locally defined lamination $\mc{W}$ such that each leaf of $\mc{W}$ intersects $\gamma_1$ and $\gamma_2$ at a unique point. Let $I_1$ and $I_2$ be the points of intersection of $\mc{W}$ with $\gamma_1$ and $\gamma_2$. Then $\mc{W}$ defines a \emph{holonomy} map $H^{\mc{W}}\colon I_1\to I_2$ by carrying the unique point of intersection with a particular plaque of the lamination to the corresponding point in the other curve. 

An important property that such a holonomy may satisfy  is \emph{absolutely continuity} with respect to volume, which means that it carries Riemannian volume of $\gamma_1$ restricted to $I_1$ to a measure equivalent to the restriction to $I_2$ of Riemannian volume on $\gamma_2$. These properties will be discussed in more detail in Appendix \ref{sec:finite_time_pesin_theory}. 

\subsection{Norms}\label{subsubsec:norms}

In this paper, we will use many estimates from calculus.

First we consider the norms of curves. 
An unparametrized curve in a manifold does not come equipped with any $C^1$ norm, as the $C^1$ norm of a curve is dependent on parametrization. Consequently, we will always view such a curve with its arclength parametrization. 
For $x\in \gamma$, we may consider the norm of the second derivative of $\gamma$ at the origin
when we view $\gamma$ as a graph over its tangent in an exponential chart. We then define $\|\gamma\|_{C^2}$ as the supremum of this norm over all $x\in \gamma$. Note that this is essentially the same thing as the supremum of the extrinsic curvature of $\gamma$ at $x$ over all points $x\in \gamma$.

Throughout the proof, we will be interested in studying the log H\"older norms of some densities along curves. We will be slightly unconventional and write $\|\ln \rho\|_{C^{\alpha}}$ for the H\"older constant of $\ln \rho$, where $\rho$ is a density.  Note that this doesn't include an estimate on $\|\ln\rho\|_{\infty}$, as such a norm usually contains. This is because the magnitude of the density is infrequently the important things in our arguments.

When we work in coordinates, we will write $\|\phi\|_i$ as the supremum of all the $i$th partial derivatives of the function $\phi$. For example, if $\phi\colon \R^2\to \R$, then we define 
\[
\|\phi\|_2=\sup_{x\in \R^2} \max\left\{\abs{\frac{d^2\phi}{dxdy}},\abs{\frac{d^2\phi}{dx^2}},\abs{\frac{d^2\phi}{d y^2}}\right\}.
\]

\subsection{Probability facts}

In the course of the paper we will some facts from probability, which we state here for the convenience of readers who are familiar with dynamics but not as much with probability.
Sometimes we will write something like $\mathbb{P}_{\omega}(A)$ for the measure $\mu(A)$ when we are thinking probabilistically. Also, we will often write $\E{\ldots}$ when we are taking expectations with respect to $\mu$, as $\mu$ is the measure driving the random dynamics. 

The following concentration in equality is very useful for us. 

\begin{thm}\label{thm:azumas_inequality}\cite[Thm.~1.3.1]{steele1997probability}
(Azuma-Hoeffding inequality) Suppose that $X_1,X_2,\ldots$ is a martingale difference sequence. Then
\begin{equation}
\mathbb{P}\left(\abs{\sum_{i=1}^n X_i}\ge \lambda\right)\le 2\exp\left(\frac{-\lambda^2}{2\sum_{i=1}^n \|X_i\|^2_{L^{\infty}}}\right).
\end{equation}
\end{thm}

\section{Outline of the paper}
\label{ScOutline}
\subsection{Quenched and annealed properties}
The main technical result of this paper is a type of ``annealed" coupling theorem, Proposition \ref{prop:main_coupling_proposition}. From this theorem 
 we deduce after
a small amount of additional work, quenched exponential equidistribution  (Proposition \ref{prop:quenched_exp_equidistribution_on_subfamilies})
as well as quenched exponential mixing, which, in turn, implies the 
annealed exponential mixing (see Corollary \ref{CrAnEM}).

Before proceeding, let us recall what is meant, 
in the probabilistic sense, by an \emph{annealed} as opposed to a \emph{quenched} limit theorem for a random dynamical system defined by Bernoulli random application of maps $(f_1,\ldots,f_m)$. In an \emph{annealed} limit theorem, we average over the entire ensemble whereas in a \emph{quenched}  limit theorem one obtains a limit theorem for almost every realization of the random dynamics. For example, in the case of equidistribution consider $\phi\colon M\to \R$ a H\"older observable and $\nu$ a probability measure on $M$, such as a curve with density. Then annealed equidistribution says:
\[
\frac{1}{m^n}\sum_{\omega^n\in \{1,\ldots,m\}^n}\int  \phi\circ f^n_{\omega}\,d\nu\to \int \phi\,d\vol,
\]
whereas quenched equidistribution says
that for almost every $\omega\in \Sigma^{\mathbb{N}}$ with respect to the Bernoulli measure $\mu$ on $\Sigma$,
\[
\int \phi\circ f^n_{\omega}\,d\mu\to \int \phi\,d\vol.
\]
 Note that the annealed result follows from the mixing of the skew product studied in 
\S \ref{SS-K}.

While the two notions are not always equivalent, our annealed coupling theorem comes 
with such fast rates that 
by the Fubini theorem, we can deduce quenched limit theorems. This reduction happens in Section \ref{sec:claims_for_use_in_limi_theorem}.

\subsection{Description of the key step}
The main results of this paper 
follow from our annealed exponentially fast coupling proposition, Proposition \ref{prop:main_coupling_proposition},
which says the following.
Suppose we have two standard pairs $\hat{\gamma}_1$ and $\hat{\gamma}_2$. Each standard pair is a $C^2$ curve $\gamma_i$ along with a density $\rho_i$ defined along $\gamma$. Suppose that $\omega\in \Sigma$ is a random word. We say that two points $x\in \gamma_1$ and $y\in \gamma_2$ are ``coupled" at time $k$ if:
\begin{enumerate}
\item 
$f^k_{\omega}(x)\in W^s_{loc}(\sigma^k(\omega),f^k_{\omega}(y))$,
\item 
The stable manifold $W^s_{loc}(\sigma^k(\omega),f^k_{\omega}(y))$ contracts uniformly exponentially quickly, so that $f^k_{\omega}(x)$ and $f^k_{\omega}(y)$ attract uniformly exponentially fast, independent of $x,y,\omega$.
\end{enumerate}
In other words, after two points couple at time $k$ they attract uniformly quickly.
 In fact, in our coupling procedure
if $x$ and $y$ couple at time $k$ then
$f^k_{\omega}(x)$ and $f^k_{\omega}(y)$ both lie in a uniformly $(C,\lambda,\epsilon)$-tempered stable manifold (see Definition \ref{defn:uniformly_tempered_stable_manifold}).
Proposition \ref{prop:main_coupling_proposition} constructs a coupling which occur exponentially quickly in the sense that the set of points where the coupling 
time is greater than $k$ has exponentially small measure.

The first step towards constructing the coupling is to show that for two ``nice" standard pairs $\hat{\gamma}_1$ and $\hat{\gamma}_2$ that are quite close, there exist uniform $\epsilon_0,\epsilon_1>0$ such that with $\epsilon_0$ probability at least $\epsilon_1$ proportion of the mass of $\hat{\gamma}_1$ couples at time $0$.
Namely, with $\epsilon_0$ probability, the stable manifolds $W^s_{\omega}$ intersect $\hat{\gamma}_1$ and $\hat{\gamma}_2$ in sets of uniformly large measure, thus those points can be coupled. 
This fact implies that a positive proportion of the mass on $\hat{\gamma}_1$ can be coupled at the first attempt.

The complement of the pairs that couple is the disjoint union of a potentially large number of very small curves. For these ``leftover" curves we will wait a potentially long time for them to grow and smoothen and then equidistribute at small scale so that we can try coupling them again. We refer to this growth and smoothening as ``recovery" and the equdistribution as ``precoupling." As a positive proportion of the remaining mass gets coupled during each attempt at coupling,
we expect only an exponentially small amount of mass to remain uncoupled after $n$ attempts.

The actual argument is much more complicated for a fairly simple reason: we cannot determine if two points $x$ and $y$ lie in the same stable manifold until we have seen the entire word $\omega$. 
 However, we do not want to look into the future at the entire word $\omega$ since then we would loose the Markov character of dynamics and
would not be able to use many estimates that rely on the Markov property.
Consequently, we define a ``stopping" time for each pair $(x,\omega)$ which tells us when to ``give up" on trying to couple during the current attempt and switch to recovery. For the moment, we regard the coupling argument as having three main steps:
\begin{enumerate}
    \item 
    (Local Coupling) Attempt to couple two uniformly smooth nearby curves $\hat{\gamma}_1$ and $\hat{\gamma}_2$.
    \item 
    (Recovery) Show that pieces of curve that fail to couple recover quickly so that their image become long and smooth.
    \item (Precoupling) There is a time $N_0$ such that given two long smooth curves we can divide them into subcurves such that for most
    of the subcurves their images $N_0$ units of time later are close to each other, 
so  we can 
    then try to locally couple them again.
\end{enumerate}

We now describe the outline of the rest of the paper and how its different sections relate to the three main steps described above.

The first goal of the paper is show that for any point $x\in M$ that for 
most words $\omega\in \Sigma$  the stable manifolds $W^s(\omega,x)$ have good properties including
good distribution of their tangent vector, controlled $C^2$ norm, and that they contract quickly. To do this, we will need to obtain good estimates on $Df^n_{\omega}$. We  show that for typical words $\omega$, $Df^n_{\omega}$ has a putative stable direction that has all of the properties that the stable direction of a Pesin regular point would have. We formalize these properties with our notion of $(C,\lambda,\epsilon)$-temperedness, which is described in detail in \S \ref{subsec:tempered_sequences_of_linear_maps}. We remark, however, that this notion is weaker than the usual notion of $\epsilon$-temperedness used in Pesin theory. We show that there exist $\lambda,\epsilon>0$ such that for almost every word $\omega$ that the trajectory will exhibit $(C(\omega),\lambda,\epsilon)$-temperedness for some $C(\omega)>0$. Further, we obtain estimates for the tail of $C(\omega)$. We then also study the distribution of $E^s_{\omega}(x)$, the stable direction for the word $\omega$ at the point $x$ and obtain estimates on the regularity of this measure, which show that the distribution of $E^s(\omega)$ and hence the stable manifolds is not concentrated in any particular direction, see 
 Proposition~\ref {prop:hausdorff_est_on_stable_dirs}.
This discussion occupies Section \ref{sec:temperedness}.
Through the application of Azuma's inequality, we are able to show that a typical trajectory exhibits temperedness. 

In Section \ref{sec:k_property},
 we study the mixing properties of the skew product map $F$. 
 The proofs rely on the properties of stable manifolds that are recalled in
 Section \ref{ScStMan}. Mixing plays a crucial role in the Finite Time Mixing Proposition
given in Section \ref{sec:finite_time_mixing_prop}. This plays an important role at the precoupling stage.

 Section \ref{sec:coupling} contains the precise statement of the main coupling Proposition \ref{prop:main_coupling_proposition}. 
 We then divide the proof into three
 main parts: 
the Local Coupling Lemma \ref {ref:small_scale_coupling_lemma}, 
the Coupled Recovery Lemma \ref{prop:coupled_recovery_lemma}, and the Finite Time Mixing Proposition \ref{prop:finite_time_mixing}
which corresponds to steps (1)--(3) in the outline above.
Lemma \ref{prop:coupled_recovery_lemma} is proven in Section~\ref{sec:finite_time_smoothing_estimates}, Proposition \ref{prop:finite_time_mixing} is proven in Section~\ref{sec:finite_time_mixing_prop}, and
Lemma \ref{ref:small_scale_coupling_lemma} is proven in Section \ref{ScLocalCoupling}.

Finally, in Section \ref{sec:claims_for_use_in_limi_theorem} we derive our main results from the main coupling proposition: we derive Theorem \ref{ThQEM}
and Corollary \ref{CrAnEM} from Proposition \ref{prop:main_coupling_proposition}. 

The paper contains two appendices. Appendix \ref{AppSmoothing} describes how the smoothness of a curve which is transversal to the stable direction
improves under the dynamics, while Appendix \ref{sec:finite_time_pesin_theory} discusses fake stable manifolds and their holonomy. 
In particular, we show that these objects converge exponentially fast to true stable manifolds and holonomies respectively.
While the estimates in the appendices are similar to several results in Pesin theory, we provide the proofs in our paper
since we could not find exact references in the existing literature. This is partially due to the fact that we put a greater
emphasis to the finite time estimates because we want to preserve the Markov 
property of the dynamics and hence cannot base our coupling algorithm on the knowledge of the future behavior of orbits.
 
\subsection{Mixing in hyperbolic dynamics} We now compare our work with strategies used in other works.
Historically the first mixing results for hyperbolic systems relied on symbolic dynamics, see
\cite{BowenLNM, RuelleTF, SinaiGibbs, ParryPollicott}. Currently the most flexible realization of this approach is via
symbolic dynamics given by Young towers (\cite{Young98}). 
Later, several methods working directly with the hyperbolic systems were developed.
In particular, we would like to mention weighted Banach spaces developed in \cite{GouezelLiverani06} (see
\cite{BaladiBook} for a review)
as well as the coupling approach developed in \cite{Young99}. We note that most hyperbolic systems could be analyzed by
each of these methods but a different amount of work is required in different cases. For example, a recent paper
\cite{DemersLiverani} constructs weighted Banach spaces suitable for the billiard dynamics. However, these spaces are necessarily 
complicated reflecting the complexity of billiards systems. 

In our work, we use the coupling approach. This method was originally used in \cite{Young99} to handle symbolic systems,
while the modifications which allow working directly on the phase space are due to \cite{dolgopyat00, chernov2006chaotic}.
The two papers mentioned above implemented the coupling methods for systems with dominated splitting. In our case, we have to deal with
the general non-uniformly hyperbolic situation and this significantly expands the potential applications of the coupling method.

An attractive feature of our result is that we make only one assumption \eqref{EqExpAv} which is, in fact, open.
Our result is an example of a successful implementation of the line of research asking which dynamical properties 
follow just from existence of a hyperbolic set with controlled geometry. This direction is exemplified by a conjecture
of Viana \cite{VianaICM}, which asks if the existence of positive measure hyperbolic set implies existence of a physical measure.
While several important recent results obtained progress on this question (see \cite{BenOvadia21, Burguet24, BCS23, CLP22}
as well as \cite{BCS22} which deals with a measure of maximal entropy), much less is known about qualitative properties.
In the present (and a follow up) paper we are able to get a full package of statistical properties starting from a simple
assumption \eqref{EqExpAv}.

Below we list key ingredients of our approach since similar ideas could be useful in studying other hyperbolic systems.

\begin{enumerate}
    \item Using martingale large deviation bounds, we demonstrate an  abundance of times where the orbit of a given vector
    is backward tempered.
    \item Using two dimensionality and volume preservation, we promote exponential growth of the norm to existence of a hyperbolic
    splitting.
        \item Using Pesin theory we show that hyperbolic set cannot have gaps of too small a size since these gaps
        would be filled with orbits of slightly weaker hyperbolicity.
      \item We use fake stable manifolds and quantitative estimates on their convergence to construct a finite time ``fake" coupling.
      \item Using a Ma\~ne type argument we show that a fake coupling converges quickly to a real coupling for most trajectories.
\end{enumerate}

Finally, we would like to mention that recently a different approach to quenched mixing based on random Young towers
has been developed, see \cite{ABR22, ABRV23}. So far, the authors have proved the existence of random towers for relatively 
simple systems where hyperbolicity is uniform at least in one direction. It might be possible to obtain exponential mixing
in our case by verifying the conditions of \cite{ABRV23}, however, this would not simplify our analysis.
Indeed the main ingredients of the Young towers is the following: the existence of a positive measure horseshoe,
an exponential tail on the return time, and a finite time mixing estimate. The last ingredient is already established in
our paper. To construct a large horseshoe would require estimates similar to our local coupling lemma of Section \ref{ScLocalCoupling},
while having an exponential tail on return times  would be similar to our recovery lemma of Section \ref{sec:finite_time_smoothing_estimates}. 
In addition there several technical properties of Young tower whose verification would require additional space and effort.
For this reason we prefer to give a direct proof of exponential mixing in our setting rather than deducing our result  
by a lengthy verification of the conditions of 
the deep recent work of \cite{ABRV23}.

\section{Estimates on the growth of vectors and temperedness}\label{sec:temperedness}

In this section, we study infinitesimal properties of uniformly expanding random dynamical systems. The main results of this section are a proof that the sequence of linear maps $Df_{\omega_0}, Df_{\omega_1},\ldots,Df_{\omega_n}$ applied along the trajectory of a point $x$ typically has a splitting with most of the same properties as a point in a Pesin block has. Moreover, we give quantitative estimates on the angle between the vectors in the splitting, as well as the probability that the splitting experiences a renewal.

\subsection{Tempered vectors and sequences of linear maps}\label{subsec:tempered_sequences_of_linear_maps}
In this subsection we discuss some notions of tempering for sequences of linear maps. We remark that typical notions of tempering used in Pesin theory involve both lower and upper bounds, i.e.~they involve a statement like $e^{\lambda-\epsilon}\le \|A\vert_{E^u}\|\le e^{\lambda+\epsilon}$. We will only take one of these two bounds to avoid having to do more estimates than necessary. Further, the version of tempering used in Pesin theory is often adapted so that the value of $\lambda$ is a particular Lyapunov exponent for a particular measure. In such a context, a tempered splitting will have expansion at rate $e^{\lambda-\epsilon}$ rather than at rate $e^{\lambda}$, as we have below. Compare for example, with the definition of $(\lambda,\mu,\epsilon)$-tempered in \cite[Def.~1.2.]{barreira2007nonuniform}. In the language of this section, points that are $(\lambda,\mu,\epsilon)$-tempered in the sense of \cite{barreira2007nonuniform}, have a splitting that is $(C,\lambda-\epsilon,\epsilon)$-tempered in our sense.

Before we get to our ultimate notion of a tempered splitting, Definition~\ref{defn:tempered_splitting}, we first record several estimates and introduce intermediate notions.

\begin{defn}\label{defn:sub_super_tempered}
Consider a finite or infinite sequence of linear maps $(A_n)_{n\in I}$ between a sequence of normed $2$-dimensional vector spaces $V_i$, where $I$ is either $\mathbb{N}$ or a set of the form $\{1,\ldots, n\}$, and $A_i\colon V_{i}\to V_{i+1}$.
\begin{enumerate}
    \item\label{item:subtempered_norms}
We say that $(A_n)$ has $(C,\lambda,\epsilon)$-\emph{subtempered} norms when
\[
\|A^{i+j}\|\ge e^{C}e^{\lambda i}e^{-\epsilon j}\|A^j\|,
\]
for all $i\ge 1$, $j\ge 0$, with $i+j\in I$.
\item 
We say that a vector $v$ is $(C,\lambda,\epsilon)$-\emph{subtempered} for the sequence of linear transformations $A_i$ if 
\begin{equation}
\|A^{m}_kv^k\|\ge e^{C}e^{\lambda m}e^{-\epsilon k},
\end{equation}
where  $A^{m}_k=A_{k+m}\cdots A_{k+1}$ and $v^k=A^kv/\|A^kv\|$, for all $k,m\in \N$ with $k+m\in I$. 
\item 
We say that the vector $v$ is $(C,\lambda,\epsilon)$-\emph{supertempered} if 
\begin{equation}
\|A^{m}_kv^k\|\le e^{C}e^{\lambda m}e^{\epsilon k},
\end{equation}
for all $m,k$ and $v^k$ as above.
\item 
Similarly, we may speak of a vector $v\in T_xM$ being sub or super tempered for a sequence of diffeomorphisms $(f_n)_{n\in I}$ if it sub or super tempered for the sequence of differentials $D_xf_1,D_{f_1(x)}f_2,\ldots$, etc.
\end{enumerate}
\end{defn}

Finally, we say that a sequence of maps has an $(C,\lambda,\epsilon)$-tempered splitting if there exists a pair of directions $e^u$ and $e^s$ such that the action of the maps is $(-C,\lambda,\epsilon)$-subtempered on $e^u$ and $(C,-\lambda,\epsilon)$-supertempered on $e^s$. In addition, we impose a lower bound on the angle between these two directions. Note that we do not require the angle itself to be tempered in the sense that it locally decays slowly: we just require that it stay bounded below by a slowly decaying function.

\begin{defn}\label{defn:tempered_splitting}
We say that a finite or infinite sequence $A_1,\ldots,A_n$ of linear maps $A_i\colon V_{i}\to V_{i+1}$ of $2$-dimensional inner product spaces has a $(C,\lambda,\epsilon)$-\emph{tempered} splitting if there exists a pair of unit vectors $e^s,e^u\in V_1$ such that 
\begin{align}
\|A^m_k (A^{k}e^u)\|/\|A^{k}e^u\|&\ge e^{-C}e^{\lambda m}e^{-\epsilon k},\\
\|A^{m}_k (A^{k}e^u)\|/\|A^{k}e^s\|&\le e^{C}e^{-\lambda m}e^{+\epsilon k},\\
\angle (A^ke^s,A^ke^u)\ge e^{-C}e^{-\epsilon k}.
\end{align}
Similarly, we say that this sequence of maps has a \emph{reverse tempered} splitting, if the sequence of maps $A_n^{-1},\ldots,A_1^{-1}$ has a tempered splitting.
\end{defn}

In the rest of this section we will show that typically the sequence of differentials 
along a random orbit has a tempered splitting. 

\subsection{Temperedness of sums of real valued random variables}

In order to study the temperedness of vectors, we will first study additive sequences of real random variables. This will be sufficient for our purposes because one may think of the norm of a vector acted upon by matrices as the sum of random variables of the form $\ln \|Av\|\|v\|^{-1}$.

In what follows, we will be studying tempered sequences of sums of real valued random variables.  The  results of this subsection 
will be used in the proof of Proposition \ref{prop:exponentially_return_to_tempered_additive}, which says that tempered times occur exponentially fast. 

\begin{defn}    
If $X_1,\ldots, X_n$ is a finite or infinite sequence of real numbers then we say that this sequence is $(C,\lambda,\epsilon)$\emph{-tempered} if for each $0\le j< k\le n$, we have that 
\begin{equation}
\sum_{i=j+1}^k X_i -\lambda(k-j)+j\epsilon\ge C.
\end{equation}
We also say that a finite sequence $X_1,\ldots X_n$ is $(C,\lambda,\epsilon)$-\emph{reverse tempered} if the sequence $X_n,\ldots,X_1$ is $(C,\lambda,\epsilon)$-tempered.
\end{defn}
\noindent Note that for fixed $\lambda,\epsilon>0$  every finite sequence is $(C,\lambda,\epsilon)$-tempered for a sufficiently negative choice of $C$. Further, note that this condition is harder to satisfy for large positive $C$, and easier to satisfy for very negative $C$.

We are interested in finding tempered times for sequences of random variables.

\begin{prop}\label{prop:temperedness_tail_bound}
Fix constants $c>\lambda_0>\lambda_1>0$ and $\epsilon>0$. Then there exist $D_1,D_2>0$ such that the following hold. Suppose that $X_1,X_2,\ldots$ is a submartingale difference sequence with respect to a filtration $(\mc{F}_n)_{n\in \N}$ such that
\begin{enumerate}
    \item $\abs{X_i}\le c$;
    \item $\E{X_i\vert \mc{F}_{i-1}}\ge \lambda_0$.
\end{enumerate}
Then the temperedness constant of the random sequence has an exponential tail. Namely, for $C\ge 0$,
    \begin{equation}\label{eqn:prob_not_subtempered_additive_seq1}
    \mathbb{P}(X_1,X_2,\ldots,\text{ is not } (-C,\lambda_1,\epsilon)\text{-tempered})\le D_1\exp(-D_2C).
    \end{equation}
\noindent Under the same assumptions on a finite sequence,  \eqref{eqn:prob_not_subtempered_additive_seq1} holds with the same constants. 
\end{prop}

\begin{proof}
For a fixed $C$, for the sequence to be $(-C,\lambda_1,\epsilon)$-tempered, for each pair of indices $0\le j< k$ the following inequality must be satisfied:
\begin{equation}\label{eqn:temperedness_defC}
X_k+\cdots+X_{j+1}-(k-j)\lambda_1+j \epsilon \ge -C.
\end{equation}
To estimate the probability of this event 
consider $\chi_{k+1}=\E{X_{k+1}\vert \mc{F}_k}$, and let
$\hat{X}_k\!\!=\!\!X_{k+1}\!-\!\chi_{k+1}$. Then the sequence $\hat{X}_k$ is a martingale difference sequence. Then,
\begin{align*}
&\mathbb{P}(X_k+\cdots+X_{j+1}-(k-j)\lambda_1+j\epsilon \le\!\! -C)\!\!=\!\!\mathbb{P}(\hat{X}_k+\cdots+\hat{X}_{j+1}
\!\!+\!\!\sum_{i=j+1}^k \chi_i-(k-j)\lambda_1+j\epsilon \le\!\! -C)\\
&\le \!\mathbb{P}\!\left(\abs{\sum_{i=j+1}^k \hat{X}_i}\ge \abs{-\!\!\!\sum_{i=j+1}^k \chi_i+(k-j)\lambda_1-j\epsilon \!\!-\!\! C}\right)\!
\le \!\mathbb{P}\!\left(\abs{\sum_{i=j+1}^k \hat{X}_i}\ge \abs{-\!(k-j)(\lambda_0-\lambda_1)\!\!-\!\!j\epsilon \!\!-\!\! C}\right)
\end{align*}
because we know that the term in the right hand absolute value is negative and $\chi_i\ge \lambda_0>\lambda_1$. Then by Azuma's inequality (Thm.~\ref{thm:azumas_inequality}),
\begin{equation}
    \mathbb{P}\left(X_k+\cdots+X_{j+1}-(k-j)\lambda_1+j\epsilon 
\le -C\right)\le 2\exp\left(-\frac{(m(\lambda_0-\lambda_1) +j\epsilon +C)^2}{2mc^2}\right)\label{eqn:good_azuma_est}
\end{equation}
$$
   \le 2\exp\left(-\frac{m(\lambda_0-\lambda_1)^2+2(j\epsilon+C)(\lambda_0-\lambda_1) }{2c^2}\right), $$
 where $m=k-j$.
Summing over  $j$ and $m$ we obtain that there exist $D_1,D_2>0$ independent of $n$ such that:
\begin{equation}
\sum_{k\ge j+1}^n\mathbb{P}(X_k+\cdots+X_{j+1}-(k-j)\lambda_1+j \epsilon \le -C)\le D_1\exp(-D_2C),
\end{equation}
which gives the needed conclusion.
\end{proof}

We now  estimate the probability that a sequence of random variables as above first fails to be tempered at a time $n$. This will be used to ensure that failure times in the local coupling lemma have an exponential tail.

\begin{prop}\label{prop:temperedness_failure_time_tail}
Fix constants $c>\lambda_0>\lambda_1>0$ and $\epsilon>0$. Then there exists $\eta>0$ such that the following holds. For each $C$ there exists $D_1$ such that if $X_1,X_2,\ldots$ is a submartingale difference sequence with respect to a filtration $(\mc{F}_n)_{n\in \N}$ and
\begin{enumerate}
    \item $\abs{X_i}\le c$;
    \item $\E{X_i\vert \mc{F}_{i-1}}\ge \lambda_0$,
\end{enumerate}
then if $\mc{S}$ is the first $n$ such that $X_1,X_2,\ldots,X_n$ is not $(C,\lambda_1,\epsilon)$-tempered then:
\[
\mathbb{P}(\mc{S}\ge n)\le D_1e^{-\eta n}.
\]
\end{prop}

\begin{proof}
To obtain a proof of the proposition we show that except on a set of exponentially small probability, the sequence $X_1,\ldots,X_n$ satisfies better estimates than $(C,\lambda_1,\epsilon)$-temperedness requires for the constraints related on $X_{n+1}$. In fact, these estimates are so much better than what is needed, that regardless of what $X_{n+1}$ is the sequence will remain $(C,\lambda_1,\epsilon)$-tempered as long as $X_1,\ldots,X_n$ is $(C,\lambda_1,\epsilon$)-tempered. Hence the sequence fails to be tempered for the first time at time $n+1$ with exponentially small probability.

We claim that there exist $\eta,D_1>0$ such that with probability at least $1-D_1e^{-n\eta}$, for all $0\le j< n$, 
\begin{equation}\label{eqn:strengthened_temperedness_formula}
\sum_{i=j+1}^n X_i -\lambda_1(n-j)+j\epsilon\ge C+(n-j)(\lambda_0-\lambda_1)/2.
\end{equation}

We now estimate the probability that \eqref{eqn:strengthened_temperedness_formula} holds for each $0\le j< n$. 
This is the same as estimating the probability that
\[
\sum_{i=j+1}^n X_i <\lambda_1(n-j)-j\epsilon+ C+(n-j)(\lambda_0-\lambda_1)/2.
\]
Note that this is the same inequality as 
\eqref{eqn:temperedness_defC}, with $(n-j)(\lambda_0-\lambda_1)/2$ added to the constant $C$ appearing there. 
Thus
\eqref{eqn:good_azuma_est} gives
\[
\!\!\!\mathbb{P}\left(\!\!\sum_{\;\;i=j+1}^n \! \! \!X_i <\lambda_1(n-j)\!-\!j\epsilon+\!C\!+\frac{(n-j)(\lambda_0-\lambda_1)}2\right)
\!\!\le\!\! 2\exp\left(\!\!-\frac{((n-j)(\lambda_0-\lambda_1)/2 \!+\!j\epsilon +C)^2}{2(n-j)c^2}\right)
\]
As at least one of $j$ and $n-j$ exceeds $n/2$ in size, we see that there exists $a>0$ such that 
\[
\mathbb{P}\left(\sum_{i=j+1}^n X_i <\lambda_1(n-j)-j\epsilon+ C+(n-j)(\lambda_0-\lambda_1)/2\right)\le e^{-an}.
\]
Hence there exists $D_1>0$ such that 
\[
\sum_{j=0}^{n-1} 
\mathbb{P}\left(\sum_{i=j+1}^n X_i <\lambda_1(n-j)-j\epsilon+ C+(n-j)(\lambda_0-\lambda_1)/2\right)\le D_1e^{-(a/2)n}.
\]
Thus we see that there is a set of probability $1-D_1e^{-(a/2)n}$ such that the inequalities \eqref{eqn:strengthened_temperedness_formula} all hold. In particular as long as $n$ is sufficiently large, for a realization $X_1,\ldots,X_n$ in this set, it follows that $X_1,\ldots,X_n,X_{n+1}$ is necessarily also $(C,\lambda_1,\epsilon)$-tempered if $X_1,\ldots,X_n$ is.

This implies that the probability of $X_1,X_2\ldots$ failing to be $(C,\lambda_1,\epsilon)$-tempered for the first time at time $n$ is at most $D_1e^{-(a/2)n}$, and the proposition follows.
\end{proof}

\subsection{Tempered splittings from tempered norms}

In this subsection, we show that one may obtain a tempered splitting for a sequence of matrices in $\SL(2,\R)$ when the norms of the matrix products are themselves tempered. 
Namely, we show that if the norms of a product of matrices has subtempered norm in the sense of 
Definition~\ref{defn:sub_super_tempered}, then the product has a hyperbolic splitting. 
The proof consists of several steps. The first step is to show that there is a stable subspace on which the product's action is super-tempered.

As before, we write $A^n=A_n\cdots A_1$. We denote by $s_n$ the most contracted singular direction of $A^n$
and by $u_n$ the most expanded singular direction. Recall that for $A\in \SL(2,\R)$ we have $\|As\|=\|A\|^{-1}$
where $s$ is a unit vector in the most contracted singular direction.

Before proceeding to the next proof, we see how the most contracted singular direction changes as we compose more matrices. Note that the following computation does not use any temperedness assumptions. Define $\alpha_n$ as follows:
\begin{equation}\label{eqn:defn_of_alpha_n}
s_n=\cos\alpha_n s_{n+1}+\sin \alpha_n u_{n+1}.
\end{equation}
Then we can compute that
$$
\|A^{n+1} s_n\|=\sqrt{\|A^{n+1}\|^{-2}\cos^2\alpha_n +\|A^{n+1}\|^2\sin^2\alpha_n}
\ge \|A^{n+1}\|\sin \alpha_n.
$$
But we also have the estimate:
\[
\|A^{n+1}s_n\|\le \|A_{n+1}\|\|A^n(s_n)\|=\|A_{n+1}\|\|A^n\|^{-1}.
\]
Thus 
\begin{equation}\label{eqn:sine_angle_estimate}
\sin \alpha_n \le \frac{\|A_{n+1}\|}{\|A^{n+1}\|\|A^n\|}.
\end{equation}

We now observe that if the sequence $(A_n)_{n\in \N}$ has a well defined stable direction $E^s$,
then $s_n\to E^s$ and we can estimate their distance by
\begin{equation}\label{eqn:angle_est1}
\angle (E^s,s_n)\le D\sum_{m\ge n} \alpha_m.
\end{equation}
This is good because we expect this sum to be dominated by its first term in the presence of non-trivial Lyapunov exponents.

Now consider a sequence of matrices $A_1,A_2,\ldots$ whose norm is $(C,\lambda,\epsilon)$-tempered and such that each matrix has norm bounded above  by $\Lambda>0$.  If we have $\|A^nv\|\ge e^{C}e^{n\lambda}$
 for some unit vector $v$,
then 
\begin{equation}\label{eqn:angle_est2}
\angle(E^s,s_n)\le D\sum_{m\ge n} e^{-2C}\Lambda e^{-2m\lambda}\le e^{D'-2C}\Lambda e^{-2n\lambda},
\end{equation}
for some $D'$ depending only on $\lambda$.
\begin{prop}\label{prop:tempered_norm_implies_splitting}
Suppose that $C_0,\lambda,\epsilon,\Lambda>0$ are fixed. Then there exist $D$ and $N\in \N$ such that if $A_1,\ldots,A_n$, $n\ge N$ is a sequence of matrices in $\SL(2,\R)$ with $(C,\lambda,\epsilon)$-subtempered norms. Then:
\begin{enumerate}[leftmargin=*]
    \item 
     There exist perpendicular vectors $s$ and $u$ so that $(A_i)_{1\le i\le n}$ has a $(\max\{0,-3C\}+D, \lambda-2\epsilon,3\epsilon)$ tempered splitting in the sense of Definition \ref{defn:tempered_splitting}.
    In the case that $\|A^n\|>1$, we may take $s$ and $u$ to be the most contracted and expanded singular directions of $A^n$, respectively.
    \item 
    In the case of an infinite sequence 
    $(A_i)_{i\in \N}$ with subtempered norms
    there exists an orthogonal pair of unit vectors $s$ and $u$ that defines such a splitting. Further, there exists a unique one dimensional subspace $E^s$ such that any non-zero $v\in E^s$ that satisfies
    $\displaystyle
    \limsup_{n\to \infty} n^{-1}\ln \|A^nv\|<0
    $
    is in $E^s$.
    \item 
    Finally, there exists $N_0(C)=\lceil (C+\ln(2))/\lambda \rceil$ and $D'$ such that for $n\ge N_0$ and  
    $m_2 \! \ge\! m_1\!\ge\! N_0$, and any $(C,\lambda,\epsilon)$-tempered sequence of matrices $(A_i)_{1\le i\le n}$ as above, $A^{m_1}$ and $A^{m_2}$ have unique contracted singular directions $E^s_{m_1}$ and $E^s_{m_2}$ and moreover, 
    $$
        \angle(s_{m_1},s_{m_2})\le e^{-4C+D'}e^{-2(\lambda-\epsilon)m_1}.
    $$
    The analogous statement also holds for $n=\infty$.
\end{enumerate}
\end{prop}
\begin{proof}
If $\|A^n\|=1$, choose arbitrarily a vector $s_n$. Otherwise, let $s_n$ be a unit vector most contracted by $A^n$.
Let $s_m$ be the most contracted vector for $A^m$. If $s_m$ does not exist because $\|A^m\|=1$, then there is no most contracted direction, and we instead set $s_m=s_n$. Let $u_n$ be a unit vector in the orthogonal complement of $s_n$.
We show that $u_n$ and $s_n$ define a tempered splitting. This requires estimating three things: the contraction of $s_n$, the growth of $u_n$, and the decay of the angle between them.

We now proceed with the proof of (1).
First, we will show that the action on the vector $s_n$ is super-tempered. Define $\alpha_m$ as in \eqref{eqn:defn_of_alpha_n}. Then there exists some $D_1$ such that 
\begin{equation}\label{eqn:proof_sin_alpha_m}
\sin\alpha_m \le D_1\frac{\|A_m\|}{\|A^m\|\|A^{m+1}\|}.
\end{equation}
Indeed for indices $m$ where $s_m$ and $s_{m+1}$ are both defined by the actual most contracting directions, this follows as in \eqref{eqn:sine_angle_estimate}. Otherwise, note that one of $A^m$ or $A^{m+1}$ has norm $1$, hence the right hand side is uniformly bounded below by $e^{-2\Lambda}$, and thus there exists such a~$D_1$.

From \eqref{eqn:proof_sin_alpha_m}, it is immediate that there exists $D_2>0$ such that
\begin{equation}\label{eqn:angle_crude_bound_12}
\angle(s_m,s_n)\le D_2 \sum_{m\le j<n} \frac{\|A_j\|}{\|A^j\|\|A^{j+1}\|}.
\end{equation}
From $(C,\lambda,\epsilon)$-subtempered norms we have for all $m+l\le n$,
\begin{equation}\label{eqn:applied_temperedness_12}
\|A^{m+l}\|\ge e^Ce^{\lambda l}\|A^m\|e^{-\epsilon m}. 
\end{equation}
Combining \eqref{eqn:angle_crude_bound_12} and \eqref{eqn:applied_temperedness_12}, and the uniform bound $\|A\|\le e^{\Lambda}$, we get
\begin{equation}
\angle(s_m,s_n)\le D_2e^{-2C+2\Lambda}\|A^m\|^{-2}e^{2\epsilon m}\sum_{0\le l<n-m} e^{-2\lambda l}
\le \label{eqn:est_on_angle}D_2D_{\lambda}e^{-2C+2\Lambda}\|A^m\|^{-2}e^{2\epsilon m}
\end{equation}
Hence there exists $D_3>0$ such that
 for all $0\le m\le n$,
\begin{align}
\|A^ms_n\|&\le \|A^m\|^{-1}+\sin\angle (s_n,s_m)\|A^m\|
\le \|A^m\|^{-1}+D_3D_{\lambda}e^{-2C+2\Lambda}\|A^m\|^{-1}e^{2\epsilon m}\notag \\
&\le (1+D_3D_{\lambda}e^{-2C+2\Lambda}e^{2\epsilon m})\|A^m\|^{-1}\label{eqn:prelim_temperedness_bound}.
\end{align}

We now check that $s_n$ is supertempered. This is more complicated. Write $\hat{s}_n^k$ for $A^ks_n/\|A^ks_n\|$.
For all $j+k\le n$, we have 
\[
\|A^j_k\hat{s}^k_n\|\|A^ks_n\|=\|A^{j+k}s_n\|.
\]
Thus
\[
\|A^j_k\hat{s}^k_n\|\le \|A^{j+k}s_n\|\|A^ks_n\|^{-1}.
\]
Applying \eqref{eqn:prelim_temperedness_bound}
with $m=j+k$ we get
\begin{align}
\|A^j_k\hat{s}^k_n\|&\le (1+D_3D_{\lambda}e^{-2C+2\Lambda}e^{2\epsilon(j+k)})\|A^{j+k}\|^{-1}\|A^ks_n\|^{-1}
\end{align}
By subtemperedness, $\|A^{j+k}\|\ge e^{C}e^{j\lambda}e^{-k\epsilon}\|A^k\|$, thus
\[
\|A^j_k\hat{s}_n^k\|\le e^{-C}e^{-j\lambda}e^{k\epsilon}(1+D_3D_{\lambda}e^{-2C+2\Lambda}e^{2\epsilon (j+k)}).
\]
Hence 
there exists $D_4$ such that 
\begin{equation}\label{eqn:s_n_supertemperedness_bound}
\|A^j_k\hat{s}_n^k\|\le e^{-\min\{-C,-3C\}+D_4}e^{-j(\lambda-2\epsilon)}e^{3k\epsilon}.
\end{equation}
Thus $s_n$ is $(\max\{0,-3C\}+D_4,\lambda-2\epsilon,3\epsilon$)-supertempered.

Next we  estimate how fast the angle between $s_n$ and $u_n=(s_n)^{\perp}$ decays. This will lead to a growth estimate on $u_n$. 
Consider the angle $\theta_m$ between $A^ms_n$ and $A^mu_n$. Because the maps are in $\SL(2,\R)$, 
\begin{equation}\label{eqn:area_preserved}
1=\|A^ms_n\|\|A^mu_n\|\sin \theta_m.
\end{equation}
Hence by  \eqref{eqn:prelim_temperedness_bound},
\begin{equation}
\sin\theta_m\ge\frac{1}{\|A^ms_n\|\|A^m\|}
\ge (1+D_3D_{\lambda}e^{-2C+2\Lambda}e^{2\epsilon m})^{-1}.
\label{eqn:angle_est_eqn_1} 
\end{equation}
For $0\le D_3D_{\lambda}e^{-2C+2\Lambda}e^{2\epsilon m}\le 1$,
\begin{equation}
    \sin\theta_m\ge 1/2.
\end{equation}
Otherwise,  as $1/(1+x)\ge 1/(2x)$ for $x\ge 1$,
\begin{equation}
 \sin\theta_m\ge  (2D_3)^{-1}D_{\lambda}^{-1}e^{2C-2\Lambda}e^{-2\epsilon m}.
\end{equation}
In both cases, we see that there exists $ D_5$ such that 
\begin{equation}\label{eqn:final_angle_temp_bound}
\sin \theta_m\ge e^{\min\{2C,0\}- D_5}e^{-2\epsilon m}.
\end{equation}

Finally, we estimate the rate of growth of $u_n$. First, note that because $s_n$ and $u_n$ are orthogonal, applying 
\eqref{eqn:angle_est_eqn_1} and \eqref{eqn:prelim_temperedness_bound} to \eqref{eqn:area_preserved} gives
$$
\|A^m u_n\| = (\sin\theta_m)^{-1}\|A^m s_n\|^{-1}
\ge 1\cdot (1+D_3D_{\lambda}e^{-2C+2\Lambda}e^{2\epsilon m})^{-1}\|A^m\|.
$$
Then letting $\hat{u}_n^k=A^ku_n/\|A^ku_n\|$, we can estimate $\|A^j_k\hat{u}_n^k\|$ as before:
\begin{align}
\|A^j_k\hat{u}_n^k\|&=\|A^{j+k} u_n\|\|A^k u_n\|^{-1}\\
&\ge  (1+D_3D_{\lambda}e^{-2C+2\Lambda}e^{2\epsilon (j+k)})^{-1}\|A^{j+k}\|\|A^k\|^{-1}\\
&\ge(1+D_3D_{\lambda}e^{-2C+2\Lambda}e^{2\epsilon (j+k)})^{-1}e^{C}e^{-\epsilon k}e^{\lambda j}\|A^{k}\|\|A^k\|^{-1} \\
& = (1+D_3D_{\lambda}e^{-2C+2\Lambda}e^{2\epsilon (j+k)})^{-1}e^{C}e^{-\epsilon k}e^{\lambda j}.
\end{align}
If $D_3D_{\lambda}e^{-2C+2\Lambda}e^{2\epsilon (j+k)}<1$, then
\begin{equation}
\|A^j_k\hat{u}_n^k\|\ge \frac{1}{2}e^Ce^{-\epsilon k}e^{\lambda j}. 
\end{equation}
Otherwise, as $1/(1+x)\ge 1/ (2x)$ for $x\ge 1$, we see that there exists $D_5>0$ such that:
\begin{equation}
\|A^j_k\hat{u}_n^k\|\ge { (2 D_3)}^{-1}D_{\lambda}^{-1}e^{2C-2\Lambda}e^{-2\epsilon (j+k)}e^{C}e^{-\epsilon k}e^{\lambda j}\\
\ge e^{D_5}e^{3C-2\Lambda}e^{-3\epsilon k} e^{(\lambda-2\epsilon )j}.
\end{equation}
So, we see that there exists $D_6$ such that 
\begin{equation}\label{eqn:u_n_super_tempered}
\|A^j_k\hat{u}_n^k\|\ge e^{\min\{C,3C\}+D_6}e^{(\lambda-2\epsilon)j}e^{-3\epsilon k},
\end{equation}
which shows that $u_n$ is $(\max\{0,-3C\}+D_6,\lambda-2\epsilon,3\epsilon)$-subtempered.

We can now conclude by reading off the constants for the splitting we just obtained from equations \eqref{eqn:s_n_supertemperedness_bound}, \eqref{eqn:final_angle_temp_bound}, and \eqref{eqn:u_n_super_tempered} and comparing with Definition \ref{defn:tempered_splitting}. Thus there is $D_7$ depending only on $\lambda,\Lambda,\epsilon$, such that $s_n$ and $u_n$ define a subtempered splitting with constants:
\begin{equation}
   D_7=(\max\{0,-3C\}+D_7,\lambda-2\epsilon,3\epsilon).
\end{equation}
This finishes the proof of the first conclusion of the proposition. 

The proof of (2) is straightforward, similar to part (1), and very similar to a usual proof of Osceledec theorem \cite[Ch.~4]{viana2014lectures}, so we omit it. 

Item (3) also follows from the above proof  
once we know that $N$ is large enough that the stable subspace is well defined.  This certainly holds if $n\ge \lceil (C+\ln(2))/\lambda\rceil$ since then $\|A^n\|\ge 2$. Then from equation \eqref{eqn:est_on_angle} and temperedness of the norm, if $m_1\le m_2$, we have that 
\begin{align*}
\angle(s_{m_1},s_{m_2})&\le D_2D_{\lambda}e^{-2C+2\Lambda}\|A^{m_1}\|^{-2}e^{2\epsilon m_1}\\
&\le D_2D_{\lambda}e^{-2C+2\Lambda}e^{2\epsilon m_1}(e^{-2C}e^{-2m_1\lambda})
\le e^{-4C+D_8}e^{-2(\lambda-\epsilon)m_1},
\end{align*}
for some $D_8$, which gives item (3).
\end{proof}

\subsection{Tempered splittings for expanding on average diffeomorphisms}
\label{SSTempDiffeo}

In this subsection, we apply the above developments to describe hyperbolicity of
expanding on average random dynamical systems. There are two main results, the first is Proposition \ref{prop:splitting_with_high_probability}, which is a quantitative estimate on the probability that $D_xf^n_{\omega}$ has a $(C,\lambda,\epsilon)$-tempered splitting. The second estimate is Proposition \ref{prop:finite_time_distribution_of_stable}, which controls the stable direction for this splitting.

To begin, we estimate the probability that the sequence $\|D_xf^n\|$ is tempered.
\begin{prop}\label{prop:diffeos_are_subtempered}
For a closed surface $M$, suppose that $(f_1,\ldots,f_m)$ is a uniformly expanding on average tuple in $\Diff_{\vol}(M)$ with constants $n_0$ and $\lambda_0$.
Then for all $0<\lambda_1<\lambda_0$ and all sufficiently small $\epsilon>0$, there exists $D,\alpha>0$ such that for all $x\in M$,

\begin{equation}\label{eqn:norm_of_matrix_subtempered}
\mu(\{\omega: \|D_xf^n_{\omega}\| {\ \rm is\ not\ } (-C,\lambda_1,\epsilon){\rm-subtempered}\})\le De^{-\alpha C}.
\end{equation}
\end{prop}

\begin{proof}
This follows from the estimates on temperedness obtained for submartingales. Essentially, for a fixed $v\in T^1_xM$, $X_n=\|D_xf^{nn_0}_{\omega}v\|$ is a submartingale with respect to a filtration $\mc{F}_n$ generated  by the coordinates of $\omega$, and $\E{X_n\vert \mc{F}_{n-1}}\ge \lambda_0$. Thus Proposition \ref{prop:temperedness_tail_bound}
gives that for all sufficiently small $\epsilon>0$, and $0<\lambda_1<\lambda_0$, there exist $D_1,D_2>0$ such that:
\[
\mathbb{P}( \|D_xf^{nn_0}\|\text{ is not } (-C,\lambda_1,\epsilon)\text{-tempered})\le D_1e^{-D_2C}.
\]
Then to obtain temperedness along the entire sequence, not just times of the form $n n_0$, note that we have a uniform bound on the norm and conorm of all $\|D_xf_{\omega_i}\|$, $1\le i\le m$. 
\end{proof}

Since a tempered sequence of norms implies the existence of a tempered splitting by Proposition \ref{prop:tempered_norm_implies_splitting}, 
the following is immediate.

\begin{prop}\label{prop:splitting_with_high_probability}
Suppose that $M$ is a closed surface and $(f_1,\ldots,f_m)$ is uniformly expanding on average tuple of diffeomorphisms in $\Diff_{\vol}^{2}(M)$ with expansion constant $\lambda_0$. Then for all $0<\lambda_1<\lambda_0$, and sufficiently small $\epsilon>0$, there exists $D,\alpha>0$ such that for all $x\in T^1M$,
\begin{equation}
\mu(\{\omega: D_xf^n_{\omega} {\rm\ does\ not \ have\ a\ }(C,\lambda,\epsilon)-{\rm tempered\ splitting}\})\le De^{-\alpha C}.
\end{equation}
In particular,
for all $x\in M$ and almost every $\omega$, $D_xf^n_{\omega}$ has a well defined one-dimensional stable subspace $E^s_{\omega}(x)$.
\end{prop}

Below, it will be important to consider the probability that a trajectory that is $(C,\lambda,\epsilon)$-tempered suddenly fails to be tempered. In order to quantify this we will introduce an auxiliary quantity for $(C,\lambda,\epsilon)$-tempered orbits of length $n$. We call this the \emph{cushion} of the orbit and it measures how far the inequalities from Definition \ref{defn:sub_super_tempered}\eqref{item:subtempered_norms} are from failing.

\begin{defn}\label{defn:cushion}
If the sequence of matrices $A_1,\ldots,A_n$ is $(C_0,\lambda,\epsilon)$-tempered, then we define its \emph{cushion} $U$ to be 
\[
U= \min_{0\le k<n} \left[\ln \|A^{n}\|-\ln \|A^k\|-C_0 -(n-k)\lambda+\epsilon k\right]
\]
\end{defn}
\noindent Note that a trajectory can have such a large cushion that whatever happens at the next iterate, the trajectory will not fail to be tempered. The cushion reflects the only inequalities relevant to tempering that the term $A_{n+1}$ would affect, should it be added to the sequence.

The following proposition is a large deviations estimate that says that typically the cushion is quite large. 

\begin{prop}\label{prop:large_deviations_for_cushion}
For a closed surface $M$, suppose that $(f_1,\ldots,f_m)$ is an expanding on average tuple in $\Diff^2_{\vol}(M)$ with expansion constant  $\lambda_0>0$. For fixed $C_0$, let $U(n,\omega,x)$ be the cushion of $D_xf^n$ when viewed as a $(C_0,\lambda,\epsilon)$-tempered trajectory.
 
Then for any $C_0$, $\lambda<\lambda_0$, and $\epsilon>0$, there exist $\delta,\eta,D>0$ such that
\[
\mathbb{P}(U(n,\omega,x)<n\delta\vert D_xf^n_{\omega}\text{ is } (C_0,\lambda,\epsilon)\text{-tempered})\le D e^{-\eta n}.
\]
\end{prop}

\begin{proof}
The proof is straightforward: we are just estimating the difference between $\ln \|D_xf^n_{\omega}\|$ and $\ln \|D_xf^i_{\omega}\|$. 

Note that in order for a given trajectory to fail to have a cushion of size $\bar{\epsilon}n$, it needs to be the case that for each $0\le k\le n$, that
\begin{equation}\label{eqn:defn_of_cushion1}
\bar{\epsilon} n>\ln \|D_xf^n_{\omega}\|-\ln \|D_xf^k_{\omega}\|-C_0-\lambda(n-k)+\epsilon k.
\end{equation}
Call this event $\Omega_{n,k}$. Note that this event is a subset of the event that 
\[
\overline{\epsilon}n+C_0\ge \ln \|D_xf^n_{\omega}\|-\ln \|D_xf^k_{\omega}\|-\lambda(n-k)
\]

As before,
$\ln \|D_xf^n_{\omega}\|-\ln \|D_xf^k_{\omega}\|-\lambda(n-k)$ is a submartingale with differences bounded by some $\Lambda>0$. Hence as $\bar{\epsilon}n+C_0$ is positive for $n$ sufficiently large, it is less than the expectation of $\ln \|D_xf^n_{\omega}\|-\ln \|D_xf^k_{\omega}\|-\lambda(n-k)$. Thus Azuma's inequality gives
\begin{align*}
\mathbb{P}(\Omega_{n,k})&\le \mathbb{P}\left(\abs{\ln \|A^n\|-\ln \|A^k\|-\E{\ln \|A^n\|-\ln \|A^k\|}}>\bar{\epsilon}n+C_0\right)\\
&\le 2\exp\left(-\frac{(\bar{\epsilon}n +C_0)^2}{2\Lambda n}\right)
\le C_1\exp\left(-\frac{\overline{\epsilon}}{2\Lambda }n\right).
\end{align*}
Summing over $k$, we find that the probability that  at least one of the inequalities \eqref{eqn:defn_of_cushion1} fails for $1\le k\le n$ is exponentially small, which gives the result.
\end{proof}

Next, we study the distribution of the stable subspaces in an expanding on average system. We obtain two estimates. First, we obtain an estimate on the distribution of all stable subspaces through a point, Proposition \ref{prop:hausdorff_est_on_stable_dirs}. 
 Second, in Proposition \ref{prop:finite_time_distribution_of_stable}, we show that the empirical distribution of stable subspaces converges quickly to the actual distribution of the true stable subspaces.

\begin{prop}
\label{prop:hausdorff_est_on_stable_dirs}
Suppose that $M$ is a closed surface and that $(f_1,\ldots,f_m)$ is an expanding on average tuple of diffeomorphisms in $\Diff^2_{\vol}(M)$. Then there exist constants $C,\alpha>0$ such that if $\nu^s_x$ denotes the distribution of stable subspaces through the point $x$, then for each $v\in \mathbb{P}T_xM$,
\[
\nu^s_x(\{z\mid d(z,v)\le \epsilon\})\le C\epsilon^\alpha,
\]
where $d$ is the angle between those points and $\mathbb{P}(T_xM)$ denotes the projectivization of $T_xM$.
\end{prop}

Naturally, before proceeding with the proof, we must show for $v\in T^1M$ that the norm of $D_xf^n_{\omega}v$ along a typical trajectory does grow exponentially. In fact, we show 
that even slow exponential growth is quite unlikely.
\begin{lem}\label{lem:growth_with_high_prob}
In the setting of Proposition \ref{prop:hausdorff_est_on_stable_dirs}, suppose that \eqref{EqExpAv} holds with constants 
$n_0\in \N$ and $\lambda_0>0$. Then there exist $\gamma,C>0$ such that if $v\in T^1M$, then
\begin{equation}
\mathbb{P}_{\omega}(\|Df^n_{\omega} v\|\le e^{\lambda_0 n/3})\le Ce^{-\gamma n}.
\end{equation}
\end{lem}

\begin{proof}
First, note that by considering the Taylor expansion of $e^{-t}$, that for sufficiently small $t$ and all $v\in T^1M$,
\[
\E{e^{-t\ln \|Df^{n_0}_{\omega}v\|}}\le (1-(n_0\lambda_0/2)t).
\]
Next, observe that writing $\overline{v}$ for $v/\|v\|$,
\begin{align*}
\E{e^{-t\ln \|Df^{2n_0}_{\omega}v \|}}&=\E{e^{-t\ln \|Df^{n_0}_{\omega}v\|}e^{-t\ln \|Df^{n_0}_{\sigma^{n_0}(\omega)}(\overline{Df^{n_0}_{\omega}v})\|}}\\
&\le \E{e^{-t\ln \|Df^{n_0}_{\omega}v\|}(1-(n_0\lambda_0/2)t)}
\le \left(1-(n_0\lambda_0/2)t\right)^2,
\end{align*}
where we have used the independence of $\sigma^{n_0}\omega$ from $\omega_i$ for $i<n_0$. Similarly, by boundedness of the $C^1$ norm of the $f_i$, we see inductively that there exists $D>0$ such that for all $n$,
\[
\E{e^{-t\ln \|Df^{n}_{\omega}v\|}}\le D\left(1-(n_0\lambda_0/2)t\right)^{n/n_0}
\le e^{-n \lambda_0/2},
\]
since $1-t/2<e^{-t}$ for small $t$.
By Markov's inequality 
$$
 \mathbb{P}(\|Df^n_{\omega} v\|\le e^{\lambda_0n/3})\le \mathbb{P}(e^{-t\ln \|Df^n_{\omega}v\|}\ge e^{-t\lambda_0n/3})
$$
\hskip1cm $\displaystyle \le \frac{\E{e^{-t\ln \|Df^n_{\omega}v\|}}}{e^{-t\lambda_0 n/3}}
\le D\frac{\left(1-(n_0\lambda_0/2)t\right)^{ n/n_0}}{e^{-t\lambda_0n/3}}
\le De^{-n\lambda_0t/2+\lambda_0nt/3}\le De^{-n\lambda_0t/6}.$
\end{proof}

For $v\in T^1M$, let $B_{\epsilon}(v)$ be the set of directions $w$ with $\sin( \angle (v,w))\le \epsilon$ and $\Lambda$ be the maximum of the norm of $\|D_xf_i\|$ over the set of all $1\le i\le m$ and $x\in M$.

\begin{lem}\label{lem:close_vectors_get_closer}
For all $\sigma>0$ sufficiently small there exist $0<\theta<1$ such that for any $v\in \mathbb{P}(T_xM)$ and sufficiently small $\epsilon>0$, if
$ -\frac{\lambda_0}{6\Lambda}\ln(\epsilon)\le  n\le -\frac{\lambda_0}{3\Lambda}\ln(\epsilon)$, and
\[
\delta = \max_{u\in B_{\epsilon}(v)} \sin \angle(Df^n_{\omega}u,Df^n_{\omega}v),
\]
then
\[
\mathbb{P}(\delta\le \epsilon^{1+\sigma} \text{ and  for all } u\in B_{\epsilon}(v),\,\, \|Df^n_{\omega} u\|\ge 2^{-1}e^{n\lambda_0/3}\|u\|)\ge 1- \epsilon^{\theta}.
\]
\end{lem}

\begin{proof}
By Lemma \ref{lem:growth_with_high_prob}, for each $n$ we have $\|Df^n_{\omega}v\|\ge e^{\lambda_0n/3}$
 on a set of measure $1-Ce^{-\gamma n}$.  Then for any unit vector $u$ with $\sin(\angle (v,u))\le \epsilon$,
$$
\|Df^n_{\omega} u\|\ge \|Df^n_{\omega}v\|-\|Df^n_{\omega}(u-v)\|
\ge  e^{\lambda_0 n/3}-\epsilon e^{\Lambda n}
\ge e^{\lambda_0 n/3}/2,
$$
as long as $\epsilon$ is sufficiently small and $n$ satisfies 
$
 n\le -\frac{\lambda_0}{3\Lambda}\ln(\epsilon).
$

 Since the $f_i$ are volume preserving, the areas of the triangles between vectors are preserved. Since all vectors in $B_{\epsilon}(v)$ are stretched, we see that
\[
\sin \angle(Df^n_{\omega}v,Df^n_{\omega} u)= \epsilon \|Df^n_{\omega}v\|^{-1}\|Df^n_{\omega}u\|^{-1}\le 2\epsilon e^{-(2/3)\lambda_0 n}.
\]
But if $ n\ge -\frac{\lambda_0}{6\Lambda}\ln(\epsilon)$ 
and $\epsilon$ is sufficiently small, then
$\displaystyle\sin \angle(Df^n_{\omega}v,Df^n_{\omega} u)\le 2\epsilon e^{-\frac{2}{3}\lambda_0\frac{\lambda_0}{6\Lambda}(-\ln(\epsilon))}.
$
Thus we see that for sufficiently small $\epsilon$ and $\sigma>0$ that for $n$ satisfying
\begin{equation*}
 -\frac{\lambda_0}{6\Lambda}\ln(\epsilon)\le  n\le -\frac{\lambda_0}{3\Lambda}\ln(\epsilon)
\end{equation*}
it holds that
$\displaystyle
\sin \angle(Df^n_{\omega}v,Df^n_{\omega} u)\le \epsilon^{1+\sigma}
$
for all $\omega$ in a set of size $1-C\epsilon^{-\gamma n}$.
\end{proof}

\begin{proof}[Proof of Proposition \ref{prop:hausdorff_est_on_stable_dirs}.]

Using Lemma \ref{lem:close_vectors_get_closer} we may now conclude. Fix some $\sigma>0$ as in the lemma, $\Lambda/(6\lambda_0)<\alpha<\Lambda/(3\lambda_0)$ and let $\epsilon>0$ be small enough that the lemma applies. Let $\epsilon_1=\epsilon$ and then define $\epsilon_k=\epsilon^{(1+\sigma)^k}$. Let $b_k=\lfloor -\alpha (1+\sigma)^k\ln(\epsilon)\rfloor$ and  
$\displaystyle n_k=\sum_{k=0}^{k-1} b_k$ be an increasing sequence of times. By our choice of $\alpha$ we may apply the lemma to each additional block of iterations of $f_\omega$ of length $b_k$ with $\epsilon=\epsilon_k$. We then define:
\begin{align*}
\eta_k^{\omega}(\epsilon, v) &= \max_{w\in B_{{\epsilon_k}}(Df_{\omega}^{n_{k-1}}v)} \sin \angle(Df^{b_k}_{\omega}w,Df^{b_k}_{\omega}v),\\
\tau_k^{\omega}(\epsilon,v)&= \inf_{w\in B_{\epsilon_k}(Df^{n_{k-1}}_\omega v)} \|Df^{b_k}_{\sigma^{n_{k-1}}\omega}w\|.
\end{align*}
Lemma \ref{lem:close_vectors_get_closer} asserts that for every $v$ and $k$ that
\[
\mathbb{P}(\eta_k^{\omega}(\epsilon_k,v)\le \epsilon_k^{1+\sigma}\text{ and }\tau_k^{\omega}(\epsilon_k,v)\ge 2^{-1}e^{\lambda_0(n_k-n_{k-1})/3})\ge 1-\epsilon_k^{\theta}.
\]
As the dynamics is IID and the above estimate is independent of the vector $v\in \mathbb{P}(TM)$, we see that there exists $C>0$ such that:
\begin{equation}
\label{eqn:intermediate_distribution_estimate_a}
\mathbb{P}\left(\text{for all  } k\,\, \eta_k^{\omega}(\epsilon,v)\le \epsilon_k \text{ and } \tau_k^{\omega}(\epsilon,v)\ge \frac{e^{\lambda_0n_k/3}}{2}\right)\ge \prod_{i=1}^{\infty} \left(1-\epsilon_k^{\theta}\right)
\geq 1-C\epsilon^\theta.
\end{equation}
By Proposition \ref{prop:splitting_with_high_probability}, at the point $x$ almost every word $\omega$ has a well defined stable subspace $E^s_{\omega}(x)$. If a vector $v\in T^1_xM$ satisfies \eqref{eqn:intermediate_distribution_estimate_a}, then for any $w\in B_{\epsilon}(v)$, $\|Df^{n_k}_\omega w\|\ge e^{\lambda_0n_k/3}2^{-k}$, which grows rapidly in $k$ as long as $\epsilon$ was chosen sufficiently small. Thus this vector cannot be in $E^s_{\omega}(x)$. Thus
$\mathbb{P}(E^s_{\omega}(x)\in B_{\epsilon}(v))\le C\epsilon^{\theta},
$
and we are done.
\end{proof}

Next we check that if we consider the distribution of stable subspaces for finite time realizations of the dynamics that the distribution of the finite time stable subspaces converges quickly to the stationary stable distribution. Essentially this should be true for the same reason that it is true for IID matrix products. The proof is a slight extension of the argument that appears above.

\begin{prop}\label{prop:finite_time_distribution_of_stable}
Suppose that $M$ is a closed surface and $(f_1,\ldots,f_m)$ is an expanding on average tuple in $\Diff^2_{\vol}(M)$. There exist $c_0,C,\theta$ such that for any $x\in M$ and $v\in T^1_xM$, if $N_0\ge c_0\abs{\ln(\epsilon)}$ the 
following 
holds. Let $E^s_{n}(\omega)$ be the maximally contracted subspace of the product $D_xf^n_{\omega}$. Then:
\begin{equation}
\mathbb{P}(\text{for some } n>N_0, E^s_{n}(\omega)\in B_{\epsilon}(v)\text{ or } E^s_{n}(\omega) \text{ does not exist})\le C\epsilon^{\theta}.
\end{equation}

\end{prop}
\begin{proof}
The proof of the above fact is essentially a corollary of the estimates obtained in the proof of Lemma \ref{lem:close_vectors_get_closer}.

We apply that same proof and choose sufficiently small $0<\sigma<\lambda_0/(3\Lambda)$ where $\lambda_0$ and $\Lambda$ are as in that proposition, as are $b_k$ and $n_k$. 
Then we find that there exists $C,\theta$ such that for all sufficiently small $\epsilon>0$, we have equation \eqref{eqn:intermediate_distribution_estimate_a}, so for $\epsilon_k=\epsilon^{(1+\sigma)^k}$,
\begin{equation}
\mathbb{P}(\text{for all  } k\,\, \delta_k^{\omega}(\epsilon,v)\le \epsilon_k \text{ and } \tau_k^{\omega}(\epsilon,v)\ge 2^{-1}e^{\lambda_0n_k/3})\ge 1-C\epsilon^{\theta}.
\end{equation}
This shows as before that at the times $n_k$, that we have the estimate
\[
\|Df^{n_k}_{\omega}w\|\ge e^{\lambda_0n_k/3}2^{-k}
\]
for all $w\in B_{\epsilon}(v)$ on a set of measure $1-C\epsilon^{\theta}$. In particular, as we chose $\sigma$ quite small, for $k\ge 2$, we see that for any time $n$ from $n_{k-1}$ to $n_k$, that 
\[
\|Df^n_{\omega}w\|\ge \|Df^{n_{k-1}}_{\omega}w\|e^{-{(n-n_{k-1})}\Lambda}\ge e^{n_{k-1}\lambda_0/3-(n-n_{k-1})\Lambda}.
\]
But by choice of $\sigma$, that exponent is at least
\begin{align*}
((1+\sigma)^{k-1}\lambda_0/3-((1+\sigma)^k-(1+\sigma)^{k-1})\Lambda)\ln(\epsilon)=(1+\sigma)^{k-1}(\lambda_0/3-\sigma\Lambda)\ln(\epsilon)>0.
\end{align*}
Thus from the definition of the $n_k$ in Lemma  \ref{lem:close_vectors_get_closer}, we see that on a set of probability $1-C\epsilon^{\theta}$  
for any $n>n_1=\abs{\alpha(1+\sigma)\ln(\epsilon)}$, that $E^s_{n}(\omega)$ does not lie in $B_{\epsilon}(v)$ and the result follows.
\end{proof}

\subsection{Reverse tempered sequences}
We are interested in reverse tempered times since they are key for proving smoothing lemmas.
The main result of this subsection is Proposition~\ref{prop:exponential_tail_tempered_times}, which shows that the waiting time until a reverse tempered time occurs has an exponential tail. 

The following lemma estimates how much the temperedness of a sequence improves when we prepend entries on it. Note that by reversing the order of the sequence, this gives the corresponding estimate for reverse temperedness.

\begin{lem}\label{lem:concat_tempered_additive_est}
Suppose that $a_1,\ldots,a_n$ is a $(C,\lambda_0,\epsilon)$ tempered sequence and $b_1,\ldots,b_m$ is a $(D,\lambda_1,\epsilon/2)$ tempered sequence where $\lambda_1-\lambda_0>\epsilon$, then $b_1,\ldots,b_m,a_1,\ldots,a_n$ is 
\[
(\min\{D,m\epsilon/2 +C+D,m\epsilon+C\},\lambda_0,\epsilon)
\]
tempered sequence.
\end{lem}
\begin{proof}
Let $c_1,\ldots,c_{m+n}$ denote the new joined sequence and
let $C'$ be the $(\lambda_0,\epsilon)$ temperedness constant for this sequence. Each pair of indices $0\le j< k\le n+m$ gives a constraint on the constant of temperedness:
\begin{equation}
C'=\min_{0\le j< k\le n+m} j\epsilon+\sum_{i=j+1}^k(c_i-\lambda_0).
\end{equation}
Note that the only pairs of indices that offer a non-trivial constraint are those with at least one of $j+1,k\ge m+1$. The constraint arising from a pair of indices with $j,k\le m$, is certainly satisfied as long as the temperedness constant is at most $D$. This leaves two cases.

For a pair of indices $j< m<k$, we obtain the constraint that
\begin{equation}
C'\le j\epsilon+ \sum_{i=j+1}^m (b_i-\lambda_0)+\sum_{i=m+1}^k (a_i-\lambda_0).
\end{equation}
But by temperedness, we can bound the right hand side below: 
$$
j\epsilon+ \sum_{i=j+1}^m (b_i-\lambda_0)+\sum_{i=m+1}^k (a_i-\lambda_0)\ge D+\frac{j\epsilon}{2}+(m-j)(\lambda_1-\lambda_0)+C
\ge m\epsilon/2+D+C.
$$
If both $j+1,k\ge m+1$, then as the sequence $a_1,\ldots,a_m$ is already $(C,\lambda_0,\epsilon)$-tempered, the constraint on these entries of the sequence improves by $m\epsilon$ as they are now additionally offset by $m$ from $0$. So, they give the constraint $C'\le C+m\epsilon$. 

Taking the minimum over the three bounds above  gives the result.
\end{proof}

Using the above, we will now prove  that 
 for submartingale difference sequences
the renewals of backward temperedness have exponential tails. 

\begin{prop}\label{prop:exponentially_return_to_tempered_additive}
(Exponential return times to the tempered set) 
Fix $c>\lambda_0>\lambda>0$ and pick $0<\epsilon<(\lambda_0-\lambda)/3$. There exist $C_0,D_1,D_2>0$ such that the following holds. Let $X_1,X_2,\ldots$ be a submartingale difference sequence with respect to a filtration $(\mc{F}_n)_{n\in \N}$ such that for all $n\in \N$,
\begin{enumerate}
    \item $\abs{X_n}<c$;
    \item $\E{X_n\vert \mc{F}_{n-1}}\ge \lambda_0$.
\end{enumerate}
Fix $N\in \N$ and let $T$ denote the first time $k$ 
after $N$ such that $X_1,\ldots,X_{N+k}$ is  $(C_0,\lambda,\epsilon)$-reverse tempered. Then
\begin{equation}
\mathbb{P}(T>N+k)\le D_1e^{-D_2k}.
\end{equation}
\end{prop}

\begin{proof}
The proof has essentially two steps. First, in the following claim, we study how long it takes for a sequence with bad temperedness constant to recover. This happens with linear speed because we are studying a submartingale sequence with $\E{X_n\vert\mc{F}_{n-1}}$ uniformly bounded away from zero. We estimate how fast the reverse-temperedness constant improves as we append blocks of a fixed size $\Delta_0$. As a sequence of length $N$ might have a bad temperedness constant, to obtain the result we then apply the tail estimate on the temperedness constant for sequences of length $N$. As each of these things has an exponential tail, we  obtain the result.

The main claim is the following.
\begin{claim}
There exist $C_0$ and $A,B>0$ independent of $N$, such that if $X_1,\ldots,X_N$ 
is $(R,\lambda,\epsilon)$-tempered and $T$ is the first time greater than $N$ that is $(C_0,\lambda,\epsilon)$-reverse tempered, then
\[
\mathbb{P}(T>N+k\vert X_1,\ldots,X_N\text{ is }(R,\lambda,\epsilon)\text{-tempered})\le Ae^{R-Bk}.
\]
\end{claim}
\begin{proof} Let $\lambda_1=(\lambda+\lambda_0)/2$ and denote by $B_{i,\Delta}$ the backwards $(\lambda_1,\epsilon/2)$-temperedness constant of the sequence $X_{i+1},\ldots,X_{i+\Delta}$. By Proposition~ \ref{prop:temperedness_tail_bound}, 
there exist $A_2,B_2$  (independent of $i$ and $\Delta$)
such that for $C\ge 0$,
\[
\mathbb{P}(X_{i+1},\ldots,X_{i+\Delta}\text{ is not }(-C,\lambda,\epsilon)\text{-tempered})\le A_2e^{-B_2C}.
\]
As this tail on the temperedness constant is independent of $i$ and $\Delta$, we see that there exists $\Delta_0$ sufficiently large and $\delta>0$ such that for any $i\in\N$,
\begin{equation}
\E{\Delta_0 \epsilon/2 + B_{i,\Delta_0}\vert \mc{F}_i}>\delta>0.
\label{DeltaDrift}
\end{equation}

We now check how much appending a block of length $\Delta_0$ improves temperedness.
Let $C_i'$ denote the backwards $(\lambda,\epsilon)$-temperedness constant of the sequence 
\[
X_1,\ldots,X_{N},X_{N+1},\ldots,X_{N+i\Delta_0}.
\]
and let $D_{i}$ denote the $(\lambda_1,\epsilon/2)$ backwards tempered constant of the sequence 
\[
X_{N+(i-1)\Delta_0+1},\ldots,X_{N+i\Delta_0}.
\]
Then by Lemma \ref{lem:concat_tempered_additive_est},
\[
C'_{i+1}=\min\{D_{i+1},\epsilon \Delta_0/2+D_{i+1}+C'_i,\epsilon \Delta_0+C'_i\}.
\]
We also define $\hat{C_0}=C'_0$ and 
\[
\hat{C}_{ i+1}=\min\{\epsilon\Delta_0/2+D_{i+1}+\hat{C}_{i},\epsilon \Delta_0+\hat{C}_i\}.
\]
Note that by \eqref{DeltaDrift}
there exists $\delta>0$ depending only on $c,\lambda,\lambda_1,\epsilon$, such that 
\begin{equation}
\E{\hat{C}_{i+1}\vert \mc{F}_{N+i\Delta_0}}-\hat{C}_i\ge \delta>0.
\end{equation}

Suppose that we define $T$ so that we decide to stop when $C'_{i}\ge -\epsilon \Delta_0/2$. Observe that if $i+1$ is the first index such that  $\hat{C}_{i+1}\ge 0$ then because 
\[
\hat{C}_{i+1}\ge \epsilon \Delta_0/2+D_{i+1}+\hat{C}_i,
\]
and $\hat{C}_i<0$ we must have that $D_{i+1}\ge -\epsilon\Delta_0/2$. Thus
\begin{equation}
C'_{i+1}\ge \min\{D_{i+1},\epsilon \Delta_0/2+D_{i+1}+C'_i,\epsilon \Delta_0+C'_i\}\ge -\epsilon \Delta_0/2.
\end{equation}
Let $C_0=-\epsilon\Delta_0/2$. Thus if $k$ is the first index such that $\hat{C}_k\ge 0$, then $T<n+\Delta_0k$. Thus we need to obtain a bound for the first time $\hat{C}_i\ge 0$. 

We now bound the tail on the first time $\hat{C}_i\ge 0$. Note that $\hat{C}_i$ is a submartingale. Further let $M$ be an upper bound on $\abs{C'_{i+1}-C'_i}$ over all $i$ (an upper bound exists because $\abs{X_i}<c$). 
Let
$\displaystyle
\chi_i=\E{\hat{C}_i\vert\mc{F}_{n+(i-1)\Delta_0}}\ge \delta>0.
$
Then $\beta_i=\hat{C}_{i+1}-\chi_i$ is a martingale difference sequence. 
We now estimate:
$$
\mathbb{P}(\hat{C}_k\le 0)
\le \mathbb{P}\left(-R+\sum_{i=1}^k\beta_k\le -\sum_{i=0}^{k-1}\chi_i\right)
\le \mathbb{P}\left(\sum_{i=1}^k\beta_k\le -k\delta+R \right)
$$
Thus for $k\ge R/\delta$, by Azuma's inequality (Theorem \ref{thm:azumas_inequality}),
$$
\mathbb{P}(\hat{C}_k\le 0)\!\!\le 2\exp\!\!\left(\!\!-\frac{(k\delta-R)^2}{2kM^2}\right)
\!\!\le \!\!2\exp\!\!\left(\!\!-\frac{k\delta^2}{2M^2}+\frac{R\delta}{M^2}-\frac{R^2}{2kM^2}\right)
\!\!\le \!\!2\exp\!\!\left(\!\!-k\frac{\delta^2}{2M^2}\!+\!R\left(\frac{\delta}{M^2}\right)\!\right).
$$
If $\delta/M^2\le 1$, then we are already done with $B=\delta^2/(2M^2\Delta_0)$. Otherwise,  if $\delta/M^2>1$, then for $k\ge 2R/\delta$, which is the only range where the bound is less than $1$, the right hand side is bounded above by
\[
2\exp\left(-k\frac{\delta^2}{2M^2}+R\left(\frac{\delta}{M^2}\right)\right)\le 2\exp\left(R-k\frac{\delta^2}{2M^2}\frac{M^2}{\delta}\right),
\]
and thus the estimate holds with $B=\delta/(2\Delta_0)$ in this case as well. This finishes the proof of the claim.
\end{proof}

Let $A,B$ and $C_0$ be as in the claim. From Proposition \ref{prop:temperedness_tail_bound}, there exists $D_1,D_2$ such that for all $C\ge 0$,
\[
\mathbb{P}(X_1,\ldots,X_N\text{ is } (-C,\lambda,\epsilon)\text{-tempered})\ge 1-D_1\exp(-D_2C).
\]
From the claim we know that if $X_1,\ldots,X_N$ is $(-C,\lambda,\epsilon)$-tempered and $T$ is the waiting time for a future $(C_0,\lambda,\epsilon)$-tempered time, then
\[
\mathbb{P}(T>N+k)\le Ae^{C-Bk}.
\]
Combining these two estimates we see that

\begin{align*}
\mathbb{P}(T>N+k)\le& \mathbb{P}(X_1,\ldots,X_N\text{ is } (-Bk/2,\lambda,\epsilon)\text{-tempered and } T > N+k)\\
&+\mathbb{P}(X_1,\ldots,X_N\text{ is not } (-Bk/2,\lambda,\epsilon)\text{-tempered}) \\
&\le A\exp(Bk/2-Bk) +D_1\exp(-D_2Bk/2) \\
&\le A\exp(-Bk/2)+ D_1\exp(-D_2Bk/2).
\end{align*}
The conclusion is now immediate.
\end{proof}

The above results imply that expanding on average diffeomorphisms  have frequent reverse tempered times. 

\begin{prop}\label{prop:exponential_tail_tempered_times}
Suppose that $(f_1,\ldots,f_m)$ is an expanding on average tuple of diffeomorphisms in $\Diff^2_{\vol}(M)$. There exist $\lambda>0$ such that for all sufficiently small $\epsilon>0$, there exists $C_0,C,\alpha$ such that for all $x\in M$ and $N\in \N$, if we let $T(x)$ be the first $(C_0,\lambda,\epsilon)$-reverse tempered time for $\|D_xf^n_{\omega}\|$ that is greater than or equal to $N$, then
\[
\mathbb{P}(T(x)\le  N+k)\ge 1-Ce^{-\alpha k},
\]
and $D_xf^{T(x)}_{\omega}$ has a well defined splitting into maximally expanded and contracted singular directions.
\end{prop}

\begin{proof}
$X_n=\|D_xf^{nn_0}\|$ is a submartingale satisfying the hypotheses of 
Proposition~\ref{prop:exponentially_return_to_tempered_additive}, hence $X_n$ satisfies the required estimate on reverse tempered times. 
The last claim follows from Proposition \ref{prop:tempered_norm_implies_splitting}.
\end{proof}

Proposition \ref{prop:exponential_tail_tempered_times}
 shows that
there is a uniformly large density subset of points such that $D_xf^n_{\omega}$ is reverse tempered.
  We now show that the stable direction of the resulting tempered splitting does not lie too close to any particular vector $v$. 

\begin{lem}\label{lem:prob_good_stable_when_stopped}
Suppose that $(f_1,\ldots,f_m)$ is an expanding on average tuple  in $\Diff^2_{\vol}(M)$, for $M$ a closed surface.
There exist $D,\alpha,c_0$ and $C,\lambda>0$ such that for all sufficiently small $\epsilon>0,x\in M$ and interval $I\subset T^1_xM$, if $n\ge c_0\ln\abs{I}$, where $\abs{I}$ is the length of $I$,
if $T(x)$ is the first time greater than $n$ that the sequence $D_xf^n_{\omega}$ has a $(C,\lambda,\epsilon)$ reverse tempered splitting, 
denoting the most contracted direction of $D_xf^n_{\omega}$ by $E^s_T$,
\[
\mathbb{P}(E^s_T\in I\vert T(x)\le n+k)\le C\abs{I}^{\alpha}.
\]
\end{lem}

\begin{proof}
This probability equals
$\displaystyle
\frac{\mathbb{P}(E^s_T\in I\text{ and } T(x)\le n+k)}{\mathbb{P}(T(x)\le n+k)}.
$
By Proposition~\ref{prop:exponential_tail_tempered_times}, the denominator is at least $1-C_1e^{-k C_2}$, for some 
$ C_1,C_2$. If $c_0$ is as in Proposition \ref{prop:finite_time_distribution_of_stable}, then for $n\ge c_0\ln\abs{I}$, then the numerator is bounded above by 
$\displaystyle \mathbb{P}(E^s_T\in I)\le C_3\abs{I}^{\alpha}.$
\end{proof}

\section{Stable manifolds of expanding on average systems}
\label{ScStMan}

In this section we show Proposition \ref{prop:stable_manifolds_exist_with_high_prob}, which says that with probability $1-C^{-\alpha}$  
a point has a stable manifold of length at least $C$. The proof has two parts. First we state a abstract proposition that gives  the existence of 
a stable manifold with good properties through a point $x$
provided that
there exists a tempered hyperbolic splitting along the orbit of $x$. We then estimate the probability that this criterion holds. 

In \S \ref{SSStable} 
we introduced the stable manifolds for the random dynamics. We now introduce a quantitative property of them that will be of use later.

\begin{defn}\label{defn:uniformly_tempered_stable_manifold}
We say that a stable manifold $W^s(\omega,z)$ is $(C,\lambda,\epsilon)$-tempered if the length of $W^s(\omega,z)$ is at least  $C^{-1}$ and the points in the stable manifold attract uniformly quickly: for 
$\displaystyle
x,y\in f^{n}_\omega(W^s_{C^{-1}}(\omega,z))
$,
\[
d_{f^{n+m}_{\omega}(W^s_{C^{-1}}(\omega,z))}({ f^{m}_{\sigma^n(\omega)}(x),f^{m}_{\sigma^n(\omega)}(y)})\le Ce^{-\lambda m}e^{\epsilon n}.
\]
\end{defn}

Now we give a quantitative estimate on the number of stable curves of a given $C^2$ norm and length. This result follows from a careful reading of the construction of stable manifolds in the book of Liu and Qian \cite{Liu1995smooth},
 in particular, Theorem III.3.1, which constructs stable manifolds of random dynamical systems lying in a certain type of Pesin block that the authors denote by $\Lambda^{l,r}_{a,b,k,\epsilon}$. 
 In the case that the random dynamics only arises from a finite collection of diffeomorphisms (i.e. has bounded $C^2$ norm), the constraint from the $r$ parameter does not matter---$r$ essentially measures how small a neighborhood of $x$ one must look at for the map in an exponential chart to be uniformly close to its derivative. In our setting, once we pick sufficiently large $r_0>0$ there is no constraint. The number $k$ is our case also does not matter---it specifies the dimension of the splitting we are considering.

In the $2$-dimensional setting a point $x\in M$ lies in $\Lambda^{l,r}_{a,b,k,\epsilon}$ for the sequence of diffeomorphisms $f_1,f_2,\ldots$ if, writing $f^{n+k}_{n}=f_{n+k}\cdots f_{n+1}$, we have an invariant splitting along the trajectory $E^s_{f^n(x)}\oplus E^u_{f^n(x)}$ such that for the reference metric on the manifold we have that:
\begin{align*}
\abs{Df^{n+k}_n(f^{n}(x))\vert_{E^s}}&\le l e^{\epsilon n}e^{(a+\epsilon)k}\\
\abs{Df^{n+k}_n(f^{n}(x))\vert_{E^u}}&\ge l^{-1} e^{-\epsilon n}e^{(b-\epsilon)k}\\
\angle(E^s_{f_1^{n}(x)},E^u_{f_1^{n}(x)})&\ge l^{-1}e^{-\epsilon n}.
\end{align*}
This is defined at the beginning of \cite[Sec.~3]{Liu1995smooth}.
In the language we have been using above, a   $(-C,\lambda,\epsilon)$-tempered trajectory
belongs to the set $\Lambda^{e^C,r_0}_{\lambda,-\lambda,1,\epsilon}$. From 
\cite[Thm.~III.3.1]{Liu1995smooth}, we may now deduce the following proposition.
\begin{prop}\label{prop:norm_of_stable}
Suppose that $(f_1,\ldots,f_m)$ is a tuple in $\Diff^2_{\vol}(M)$, where $M$ is a closed surface. Fix $\lambda,\epsilon>0$. Then there exist constants $D_1,D_2$ such that if $(\omega,x)$ is a $(-C,\lambda,\epsilon)$-tempered trajectory, then $W^s_{\omega}(x)$ exists and is at least 
$D_1e^{-2C}$ long. Further, on this interval, its $C^2$ norm is at most $D_2e^{6C}
$ (when viewed as a graph over its tangent space at $x$). Moreover these estimates are $e^{7\epsilon}$-tempered along the trajectory.

\end{prop}
\begin{proof}
From the above discussion,
a $(C,\lambda,\epsilon)$-tempered point lies in $\Lambda^{e^{C},r_0}_{\lambda,-\lambda,1,\epsilon}$. So, we just need to recover the estimates from the proof of \cite[Thm.~III.3.1]{Liu1995smooth}. In fact these estimates are stated there. As we are keeping $\lambda,\epsilon$ fixed, the conclusion will follow once we compute the quantities $\alpha_n$ and $\beta_n$ appearing in that theorem given our particular choices. Although \cite{Liu1995smooth} only shows the stable manifolds are $C^{1,1}$, the estimates provided there on the Lipschitz constant of the derivative is enough for controlling the $C^2$ norm because we know that the stable manifolds are in fact as smooth as the dynamics, which is $C^2$ \cite[Rem.~7.3.20]{arnold1998random}.

First we explain how to estimate $\beta_n$, which controls the norm. The first quantity that gets defined in the proof is
$\displaystyle
c_0=4Ar'e^{2\epsilon}.
$
Here, $A$ is the quantity appearing in the proof of \cite[Lem.~1.3]{Liu1995smooth}, which is equal to 
$4(l^2)(1-\epsilon^{-2\epsilon})^{-1/2}$. Thus $c_0\leq C_1e^{2C}$. Therefore the quantity $D=(1-e^{-2\epsilon})^{-3}(1+e^{-2\epsilon})^2c_0e^{-a}$ on p.~66 of \cite{Liu1995smooth} is at most $C_2e^{2C}$. Hence $\beta_n$, which is defined on p.~68 of \cite{Liu1995smooth} as $2DA^2e^{7\epsilon n}$ and controls the norm of the stable curve, is at most $C_3e^{6C}e^{7\epsilon n}$. 

The length of the curve given by the quantity $\alpha_n$ defined on p.~68 of \cite{Liu1995smooth} where it is defined to be $A^{-1}r_0e^{-5\epsilon n}$. From the definition of $A$ given above, this is bounded below by  $C_4e^{-2C}e^{-5\epsilon n}$.  We are done.
\end{proof}

We then estimate the probability that a stable manifold is $(C,\lambda,\epsilon)$-tempered. 

\begin{prop}\label{prop:stable_manifolds_exist_with_high_prob}
Suppose that $(f_1,\ldots,f_m)\in \Diff_{\vol}(M)$ is a uniformly expanding on average tuple, where $M$ is a closed surface. Then there exists $\lambda,\epsilon,\alpha>0$ such that for all $C>0$
\[
\mu(\{\omega:W^s_{\omega}(x){\rm\ is\ not\ }(C,\lambda,\epsilon)\text{-tempered}\})\le C^{-\alpha}.
\]
\end{prop}

\begin{proof}

As the maps $f_1,\ldots,f_m$ are uniformly $C^{1+\text{H\"older}}$ and uniformly expanding, the trajectory is $(-C,\lambda,\epsilon)$-tempered with probability $1-De^{-\alpha C}$ by Proposition \ref{prop:splitting_with_high_probability}. This stable curve is at least $D_1e^{-2C}$ long from Proposition \ref{prop:norm_of_stable}.
The contracting of the stable manifold required by Definition \ref{defn:uniformly_tempered_stable_manifold}
then follows from a standard graph transform argument, appearing in Chapter 7 of \cite{barreira2007nonuniform} or \cite[Lem.~3.2]{Liu1995smooth}, or from keeping track of the contraction in the graph transform arguments in \S \ref{subsec:graph_transform}. 
\end{proof}

\section{Exactness of the skew product}\label{sec:k_property}

We now consider measure theoretic properties of the skew product $F\colon \Sigma\times M\to \Sigma\times M$. We begin with the most basic property, ergodicity, in Proposition~\ref{prop:expanding_on_average_ergodic}. Then we show that this system is exact in Proposition~\ref{prop:skew_product_is_mixing}.
 As exactness implies mixing, this proposition plays a key role in the proof of finite time mixing in Section \ref{sec:finite_time_mixing_prop} where it is used  in the proof of fiberwise mixing in  Proposition~\ref{prop:fibrewise_mixing}.

\subsection{Ergodicity}

The ergodicity of expanding on average systems has been known since \cite[Section~10]{dolgopyat2007simultaneous}. 
 We need an extension of this result.
Consider the diagonal skew product 
\begin{equation}
\label{KDiag}
F_k\colon \Sigma\times M^k\to \Sigma\times M^k\quad\text{given by}\quad (\omega, x_1,\ldots,x_k)\mapsto (\sigma(\omega),f_{\omega_0}(x_1),\ldots,f_{\omega_0}(x_k)).
\end{equation}
Note that $F_k$ preserves the measure $\mu\otimes \vol^k$.

\begin{prop}\label{prop:expanding_on_average_ergodic}
Suppose that $(f_1,\ldots,f_m)$ is an expanding on average tuple in $\Diff^2_{\vol}(M)$ for $M$ a closed surface. Then for each $k\in \N$, $F_k$ is ergodic with respect to $\mu\otimes \vol^k$. 
\end{prop}

 We will not include a full proof of the above proposition as the result for $F=F_1$ is explained quite clearly in \cite[\S 3.2]{chung2020stationary} as well as \cite[Lem.~4.41]{liu2016lyapunov}. For $k>1$, the result can be deduced along similar lines. 
No higher dimensional dynamics is needed because the dynamics is a product and hence all dynamical constructs, like stable manifolds, are just products of the constructs for the system $F_1$.

The proof of Proposition \ref{prop:expanding_on_average_ergodic} relies implicitly on the following lemma
 which will be important in Section \ref{SS-K} as well.
For $x\in M$, we let $B_{\delta}(x)$ denote the ball of radius $\delta$ centered at $x$.

\begin{lem}\label{lem:soft_local_configuration}
Suppose that $(f_1,\ldots,f_m)$ is an expanding on average tuple in $\Diff^2_{\vol}(M)$. Then there exist $0<\delta_1<\delta_2$ and  $\lambda,\epsilon,C_0,\epsilon_0>0$ such that for all $x\in M$ there exist two positive measure subsets $V_1,V_2\subseteq \Sigma$ and a pair of transverse cones $\mc{C}_1,\mc{C}_2$ defined on $B_{\delta_2}(x)$ by parallel transport of cones based at $x$
 such that the following holds.
Let $\Lambda_{\omega}$ denote the set of $(C_0,\lambda,\epsilon)$-tempered points in $B_{\delta_1}(x)$  under the dynamics
defined by $\omega$, and set
\[
Q^{\omega}(x)=\bigcup_{y\in \Lambda_{\omega}\cap B_{\delta_1}(x)} W^s_{\delta_2}(\omega,y).
\]
Then 
\begin{enumerate}[leftmargin=*]
    \item 
For $i\in \{1,2\}$, $\omega_i\in V_i$, and $y\in \Lambda_{\omega_i}$ the stable manifold $W^s_{\delta_2}(\omega, y)$ is uniformly contracting and tangent to $\mc{C}_i$.
\item 
 For $i\in \{1,2\}$ and $\omega_i\in V_i$, the laminations by stable manifolds satisfy the usual absolute continuity properties: 

\noindent \textbf{(AC 1)} If $K\subseteq M$ is a Borel set, and for almost every $y\in \Lambda_{\omega_i}$ the Riemannian leaf measure of $K\cap W^s_{\delta_2}(\omega_i,y)$ is zero, then $\vol(Q^{\omega_i}\cap K)=0$.

\noindent \textbf{(AC 2)} If $T$ is a transversal to $\mc{C}_i$  and 
$K\subseteq M$ is a Borel set, and for a positive measure subset of $z\in T$, $W^s_{\delta_2}(\omega_i,z) \cap K$ has positive leaf measure, then $\vol(K)>0$. 

\item 
For $i\in \{1,2\}$ and $\omega_i\in V_i$, $\vol(Q^{\omega_i}\cap B_{\delta_1}(x))> .99\vol(B_{\delta_1}(x))$.
\end{enumerate}
\end{lem}

This lemma is implicit in Chung \cite{chung2020stationary} and Liu \cite{liu2016lyapunov}, and further can be deduced from the propositions we prove below. In particular, our Propositions \ref{prop:set_up_scale_prop} and \ref{prop:holonomies_converge_exponentially_fast} contain the needed claims. Lemma \ref{lem:soft_local_configuration} allows a random version of the Hopf argument where the stable manifolds for different words $\omega\in \Sigma$ play the role of the stable and unstable manifolds in the usual Hopf argument. This can be used to prove 
 Proposition \ref{prop:expanding_on_average_ergodic}. We will not repeat this argument here as it is adequately explained in the sources mentioned.

\subsection{Strong mixing}
\label{SS-K}

Here we show that for $k\ge 1$  the skew product $F_k\colon \Sigma\times M^k\to \Sigma\times M^k$ defined in \eqref{KDiag} is strong mixing for the measure $\mu\otimes \vol^k$. We will use this property later. A good reference for many of the properties discussed in this section is \cite{ rohlin1967lectures}.

\begin{defn}
An endomorphism $T$ of a Lebesgue space $(M,\mc{B},\mu)$  is \emph{exact} if $\displaystyle\bigcap_{n=0}^{\infty} T^{-n}\mc{B}=\mc{N}$, the trivial sub-sigma algebra of $M$. 

An invertible map, i.e.~an automorphism, $T$ of a Lebesgue space $(M,\mc{B},\mu)$, is called a \emph{$K$-automorphism} if there exists a sub-sigma algebra $\mc{K}\subset \mc{B}$ such that:

(1) $\mc{K}\subset T\mc{K}$;\hskip2mm
(2) $\bigvee_{n=0}^{\infty} T^n\mc{K}=\mc{B}$; \hskip2mm
(3)  $\displaystyle\bigcap_{n=0}^{\infty} T^{-n}\mc{K}=\{\emptyset, M\}$.
\end{defn}
Both exact systems and  $K$-automorphisms are strong multiple mixing \cite[p.~17, 27]{rohlin1964exact}, \cite[15.2]{rohlin1967lectures}. 
Further, an endomorphism is exact if and only if its natural extension is a $K$-automorphism \cite[p.~27]{rohlin1964exact}.

We now describe how one may show that
 an automorphism $T\colon (M,\mu)\to (M,\mu)$ is exact. The \emph{Pinsker partition} of $M$ is the finest measurable partition $\pi(T)$ of $M$ that has zero entropy. This means that any other measurable partition with zero entropy is coarser, mod $0$, than $ \pi(T)$. It turns out that $T$ is a $K$-automorphism if the Pinsker partition of $T$ trivial, i.e.~$\pi(T)=\{\emptyset, M\}$, see \cite[13.1,13.10]{rohlin1967lectures}.
In fact, the conditions enumerated in the definition of $K$-automorphism above essentially say that the Pinsker partition is trivial.

A useful fact for studying the Pinsker partition is the following.

\begin{lem}\label{rem:detect_pinsker}
(see \cite[p.~288]{barreira2007nonuniform}, \cite[12.1]{rohlin1967lectures})
If a measurable partition $\eta$ satisfies $T\eta \ge \eta$ and $\bigvee_{n=0}^{\infty} T^n\eta=\epsilon$, the partition into points, then $\bigwedge_{n=0}^{\infty} T^{-n}\eta\ge \pi(T)$ .
\end{lem}
Here we use the standard notation for partitions where we write $\mc{A}\le \mc{B}$ if $\mc{A}$ is coarser than $\mc{B}$. 
An example of a partition satisfying the hypotheses of Lemma \ref{rem:detect_pinsker} is the partition of a shift space $\Sigma$ into local stable sets, $W^s_{loc}(\omega)=\{\eta: \omega_i=\eta_i\text{ for } i\ge 0\}$.

We now show for $k\ge 1$ that the map $F_k$ defined above is mixing.

\begin{prop}\label{prop:skew_product_is_mixing}
Let $(f_1,\ldots,f_m)$ be an expanding on average tuple in $\Diff^2_{\vol}(M)$ for $M$ a closed surface. Then the associated skew product $F\colon \Sigma\times M\to \Sigma\times M$ is exact, and hence strong mixing of all orders, for the measure $\mu\otimes \vol$. The same holds for 
$F_k\colon \Sigma\times M^k \to \Sigma\times M^k$. 
\end{prop}
\begin{proof} 
To show exactness and hence strong mixing of $F$, we will show that the natural extension of the skew product $F\colon \Sigma\times M\to \Sigma\times M$ has the $K$-property. As before, we denote by $\hat{\Sigma}$ the two sided shift, so that the natural extension 
of $F$ is $\hat{F}\colon (\hat{\Sigma}\times M,\hat{\mu}\otimes \vol)\to (\hat{\Sigma},\hat{\mu}\otimes \vol)$, where $\hat{\mu}$ is the Bernoulli measure on $\hat{\Sigma}$. Note that the measure on the natural extension has this simple description because each $f_i$ preserves volume.

We begin by showing that modulo $0$, any element of the Pinsker partition is of the form $\hat{\Sigma}\times U$ where $U\subseteq M$.  
The local stable sets of the words $\omega\in \hat{\Sigma}$, form a measurable partition of $\hat{\Sigma}$ indexed by the elements of $\Sigma$. 
Further, the sets $\{W^s_{loc}(\omega)\times \{x\}\}_{x\in M}$ form a measurable partition of $\hat{\Sigma}\times M$.
If we let $\eta$ denote this partition, then  $\bigwedge_{n=0}^{\infty} {F}^{-n}\eta$ is the partition into sets of the form $\hat{\Sigma}\times \{x\}$, where $x\in M$. By Lemma \ref{rem:detect_pinsker}, we see that $\pi(\hat{F})\le \{\hat{\Sigma}\times \{x\}:x\in M\}$. Note that this shows that the atoms of the 
Pinsker partition of $\hat{F}$ are of the form $\Sigma\times A$ where $A$ are the atoms of a partition of $M.$
We denote this partition by $\mc{P}$ and the atom containing a point $x\in M$ by $\mc{P}(x)$.

 We now show that 
the Pinsker partition is even coarser by using the dynamics in the fiber; in fact our goal is to show that $\pi(\hat{F})$ has an atom with positive mass.  From Liu and Qian, there is a measurable partition of $\hat{\Sigma}\times M$ subordinate to the partition into full stable leaves \cite[Proposition~VI.5.2]{Liu1995smooth} where each atom is a non-trivial curve in a stable leaf. 
This shows that for almost every $x\in M$ and almost every $\omega$, that Lebesgue almost every $y\in W^s(\omega,x)$ is in $\mc{P}(x)$. (This uses AC1 for the stable lamination.)
Let $G^{\omega_i}$ be the subset of $Q^{\omega_i}$ of points $y$ such that $W^s_{\delta_2}(\omega_i,y)$ 
satisfies that almost every $z\in W^s_{\delta_2}(\omega_i,y)$ is in $\mc{P}(y)$.
Note that there there is a subset $\bar V_i$ of full measure in $V_i$ 
such that for $\omega_i\in \bar V_i$,\; $G^{\omega_i}$ has full measure in $Q^{\omega_i}$. 
Now for $\omega_2\in \bar V_2$ and $z\in G^{\omega_2}$, 
 consider the intersection of a leaf $W^s_{\delta_2}(\omega_2,z)$ with $G^{\omega_1}$, where $\omega_1\in \bar V_1$.
Suppose that for some such $z$ the set  $G^{\omega_1}\cap W^s_{\delta_2}(\omega_2,z)$ has positive measure. Then by definition of $G^{\omega_1}$, almost every $y\in G^{\omega_1}$ has $W^s_{\delta_2}(\omega_1,y)$ saturated with points in $\mc{P}(z)$, and hence by AC2, $\mc{P}(z)$ has positive measure. Thus the Pinsker partition has a positive measure atom.
If there were no such point $z$, then for almost every $z\in G^{\omega_2}$, the intersection $G^{\omega_1}\cap W^s(z,\omega_2)$ has zero leaf measure. Thus by AC1,  $Q^{\omega_2}\cap Q^{\omega_1}\cap B_{\delta_1}(x)$ has measure zero. But as $Q^{\omega_1}$ and $Q^{\omega_2}$ each take up $.99$ proportion of the volume of $B_{\delta_1}(x)$, this is impossible. Thus we see that there is a positive volume atom of $\mc{P}$.  Let $\Sigma\times A$  be this  atom of $\pi(\hat{F})$ of positive measure.

As $\hat{F}$ is ergodic, it must cyclically permute a finite number of these positive measure sets.
Because $\hat{F}$ is expanding on average, every power of $\hat{F}$ is also  expanding on average.
 Hence, by Proposition \ref{prop:expanding_on_average_ergodic},
 every power of $F$ is ergodic.
 Thus the Pinsker partition has only a single non-trivial element, hence $\pi(\hat{F})$ is trivial. Hence $\hat{F}$ is a $K$-automorphism and so $F$ is exact.
 
For the higher ``diagonal" skew products $F_k$, the proof proceeds along very similar lines. As before, one has stable and unstable manifolds in each of the factors of $M^k$ and hence through any particular point $(x_1,\ldots,x_k)\in M^k$, one has the stable/unstable manifold that is the product of the stable manifolds $W^{s/u}_{loc}(\omega,x_i)$. Hence in the extended system the stable an unstable foliations are transverse as before. By using these, one can similarly deduce that the Pinsker partition is finite. Further, from Proposition \ref{prop:expanding_on_average_ergodic} every power of $F_k$ is ergodic, which, as before implies that the Pinsker partition is trivial and thus the $K$-property holds for $\hat{F}_k\colon \hat{\Sigma}\times M^k\to \hat{\Sigma}\times M^k$. 
\end{proof}

\section{Coupling}\label{sec:coupling}
 In this section we present our main technical tool: the coupling lemma. We divide its proof into several steps according to the plan from
Section \ref{ScOutline}. Accordingly, this section contains the outline of the rest of the paper.

\subsection{Standard pairs and standard families}\label{sec:standard_pairs}

The proof of exponential mixing in this paper proceeds by showing that if $\mu_1$ and $\mu_2$ are two measures with smooth densities
and $\psi$ is a H\"older function then 
$\mu_1(\psi\circ f_\omega^n x)-\mu_2(\psi\circ f_\omega^n x) $ 
is exponentially small. Taking $\mu_2$ to be $\vol$ and $\mu_1$ to be the measure with density $\phi$ we obtain Theorem \ref{ThQEM}.
Unfortunately, the set of measures whose densities satisfy a certain bound on their H\"older norm is not invariant by the dynamics,
since compositions worsen H\"older regularity. So we need to consider a larger class of measures: the
measures that are convex combinations of measures on (unstable) curves. This leads to notions of standard pairs and 
standard families that we now recall. We refer to \cite[Chapter 7]{chernov2006chaotic}
for a detailed discussion of these notions.

\begin{defn}\label{defn:standard_pair}
A \emph{standard pair} in a Riemannian manifold $M$ is an arclength parametrized $C^2$ curve $\gamma\colon [a,b]\to M$ of bounded length along with a log-H\"older density 
$\rho$ defined along $\gamma$ (or equivalently $[a,b]$). We denote the pair of the curve and density by $\hat{\gamma}$ for emphasis. 
\end{defn}

There are two different ways of thinking about standard pairs. The first is that a standard pair is literally a pair of a curve and a density as in Definition \ref{defn:standard_pair}. The second way is that we think of  $\hat{\gamma}=(\gamma,\rho)$ as a ``thickened" version of the underlying curve $\gamma$ where the ``thickness" is given at a point $x$ by $\rho(x)$. More precisely, we may think of $\hat{\gamma}$ as a 
subset of $[a,b]\times [0,\max{\rho}]$ comprising the points $(c,y)$ where $y\le \rho(c)$. We will often write $x\in \hat{\gamma}$ when referring to a point in this set associated to $\hat{\gamma}$. By thinking of the standard pair in this manner, we can imagine geometrically subdividing the pair into pieces. This type of subdivision is frequently used below.

 Each standard pair defines a measure on $M$ given for continuous $\psi\colon M\to \R$ by the formula
\begin{equation}\label{DefRhoHat}
    \hat\rho_\gamma(\psi)=\int_\gamma \psi(x) \rho(x) dx
    \end{equation}
        where $dx$ denotes the arclength parametrization of $\gamma$.

A standard curve comes with a notion of regularity. The regularity of $\hat{\gamma}$ is determined by the $C^2$ norm of $\gamma$ as well as the $C^2$ norm of the density along $\gamma$. We recall now some notions from \S \ref{subsubsec:norms}.
 Recall that we define the $C^2$ norm, $\|\gamma\|_{C^2}$, of the curve $\gamma$ as the supremum of its second derivative as a graph over its tangent space in exponential charts.

\begin{defn}\label{defn:R_good_standard_pair}
Suppose that $\hat{\gamma}$ is a $C^2$ standard pair consisting of a curve $\gamma$ and a density $\rho$. 
We say that $\hat{\gamma}$ is \emph{$R$-good} if

(1) The length of $\gamma$ is at least $e^{-R}$. 
 
   (2) The $C^2$ norm of $\gamma$ is at most $e^{R}$.
    
  (3)  The density of $\rho$ satisfies $\|\ln \rho\|_{C^{\alpha}}\le e^R$, where we measure distance with respect to the arclength parameter of $\gamma$.
    Recall that  $C^{\alpha}$ only means the H\"older constant of the function. \vskip1mm

We say that a standard pair $\hat{\gamma}$ is $R$\emph{-regular} when at least (2) and (3) are satisfied.

\end{defn}
Note that a larger $R$ corresponds to a less regular curve.

\begin{defn}
For a standard pair $\hat{\gamma}=(\gamma,\rho)$, we say that $x\in \gamma$ has an \emph{$R$-good neighborhood}, if there is a subcurve $\gamma'\subseteq \gamma$ containing $x$ such that $(\gamma',\rho\vert_{\gamma'})$ is $R$-good. 
\end{defn}

Note that if $x$ is in an $R$-good neighborhood of $\hat\gamma$, this does not imply that $x$ is centered in long neighborhood. The point $x$ might still be quite close to the edge. Later we will also deal with points $x$ that are \emph{centered} in an $R$-good neighborhood, meaning that the segments on either side of $x$ form $R$-good neighborhoods.

\begin{defn}
A standard family is a collection  of standard pairs $\{\hat\gamma_\theta\}_{\theta\in\Lambda}$ indexed by points from a probability space
$(\Lambda, \lambda).$
\end{defn}

Thus in the case that $\lambda$ is atomic we just have a finite collection of standard pairs (counted with weights).

 We say that a standard family is \emph{$R$-good} if each standard pair that comprises it is $R$-good.
 We will only consider standard families where the goodness is bounded below.

 Given a standard family $\{\gamma_\theta\}_{\theta\in\Lambda}$ 
we can associate a measure 
 by integrating the measures corresponding to individual standard pairs with respect to the factor measure $\lambda$.
For a function $\psi\colon M\to \R$, we set
\begin{equation}
\label{SFMes}    
 \hat\rho_\Lambda(\psi)=\int_\Lambda \hat\rho_{\gamma_\theta}(\psi) d\lambda(\theta) 
 \end{equation}
where $\hat{\rho}_{\gamma_\theta}$ is defined by \eqref{DefRhoHat}.

A particularly useful property of standard families is that they can represent volume.  It is straightforward to check that a standard pair representing volume exists by using charts.

\begin{prop}
\label{PrVolStandard}
Given a closed smooth manifold $M$ endowed with a volume, there exists some $C>0$ and a $C$-good standard family $P_{\vol}$ such that the associated measure represents volume on $M$, i.e. for any continuous function 
\[
\int \phi\,dP_{\vol}=\int\phi\,d\vol.
\]
\end{prop}

 Below we will use a na\"ive estimate saying that the goodness of a standard pair  can deteriorate at most exponentially quickly.

\begin{prop}\label{prop:decay_of_goodness_in_general}
Suppose that $(f_1,\ldots,f_m)$ are $C^2$ diffeomorphisms of a closed manifold. Then there exists $C,\eta>0$ such that for any standard pair $\hat{\gamma}$ that is $R$-good and any $\omega\in \Sigma$, $f^n_{\omega}(\hat{\gamma})$ is $\max\{C+R+n\eta,C+n\eta\}$-good.
\end{prop}
\begin{proof}
The condition that the length of the curve can shrink at most exponentially fast is clear from the uniform bound on the derivative. The fact about the $C^2$ norm of curve follows immediately from Lemma \ref{lem:growth_of_C_2_norms}. 
This leaves the estimate on
 the density, which follows from Lemma \ref{lem:pushforward_density_est_diffeo} because the $C^2$ norm of $f^n_{\omega}$ grows at most exponentially.
 \end{proof}

 Note that the representation \eqref{SFMes} (including the representation of the volume from Proposition \ref{PrVolStandard})
is highly non-unique.
One type of non-uniqueness that we shall often exploit in our proof is the possibility to divide a standard pair into pieces.
 To do so we
 partition the underlying curve $\gamma$ into multiple disjoint subcurves $\gamma_1,\ldots,\gamma_n$. We then obtain a subdivision of $(\gamma,\rho)$ from the restrictions $(\gamma_1,\rho\vert_{\gamma_1}),\ldots,(\gamma_n,\rho\vert_{\gamma_n})$. We give each piece unit mass for the indexing measure $\lambda$.
Note that $(\gamma,\rho)$ as well as the standard family $\{(\gamma_i,\rho\vert_{\gamma_i})\}_{1\le i\le n}$   both represent the same measure on $M$.

A more subtle type of subdivision occurs when we view a standard pair as a subset of $\gamma\times [0,\max \rho]$ and partition this subset in the vertical direction. Similarly, we will obtain a new standard family. But now the underlying curves of the family may not be disjoint. For a simple example, something we do multiple places in the local coupling argument is take a standard pair $(\gamma,\rho)$, a number $\alpha\in (0,1)$, and subdivide this standard pair into $\{(\gamma,\alpha \rho),(\gamma,(1-\alpha)\rho)\}$ and give each piece mass $1$ for the indexing measure $\lambda$.  Alternatively, we could take 
$\hat\gamma_1=\hat\gamma_2=(\gamma,\rho)$ and allow the indexing measure assign them mass $\alpha$ and $1-\alpha$, 
which gives the same measure on $M$ independent of $\alpha$.
Below, we will often think of this geometrically: we take the region associated to the standard pair in $\gamma\times [0,\max\rho)$ and slice it into regions. Projecting the Lebesgue measure on each region down to $\gamma$ naturally defines a standard pair.

Next, if we have a standard family $\hat{\gamma}$ and a subfamily $\hat{\gamma}'$ of $\hat{\gamma}$ defined by some subdivision of $\gamma\times [0,\max\rho)$ as mentioned above, then we define $\hat{\gamma}\setminus \hat{\gamma}'$ to be the standard family defined by the complement of $\hat{\gamma}'$ in the subdivision.

\subsection{Main coupling proposition}

We now state the main technical result of the paper, from which the main mixing results of this paper are a consequence.

\begin{prop}\label{prop:main_coupling_proposition}
Suppose that $(f_1,\ldots,f_m)$ is an expanding on average tuple in $\Diff^2_{\vol}(M)$, where $M$ is a closed surface. There exists $\lambda>0$ such that for all sufficiently small $\epsilon>0$, there exist $C,\alpha>0$, such that for any $R$, a goodness of standard pairs, the following holds.

Let $\hat{\gamma}_1$ and $\hat{\gamma}_2$ be two standard pairs with associated measures $\rho_1$ and $\rho_2$ of equal mass that are $R$-good. Then we have the measures $\mu \otimes \rho_i$ on $\Sigma\times \hat{\gamma}_i$, where $\mu$ is the Bernoulli measure on the one sided shift. There exists a coupling function $\Upsilon\colon \Sigma\times \hat{\gamma}_1\to 
\hat{\gamma}_2$, where for each $\omega$ the map $\Upsilon(\omega,\cdot)\colon \hat{\gamma}_1\to \hat{\gamma}_2$ is measure preserving,
and a time $\hat{T}(\omega,x)$ such that 
\[
f^{\hat{T}(\omega,x)}_{\omega}(x)\in W^s_{\sigma^{\hat{T}(\omega,x)}\omega,C^{-1}}({f^{\hat{T}(\omega,x)}_{\omega}\Upsilon(\omega,x)}),
\]
and this stable manifold is uniformly $(C,\lambda,\epsilon)$-tempered in the sense of Definition \ref{defn:uniformly_tempered_stable_manifold}. Further
\[
\mathbb{P}_{\omega,x}(\hat{T}(\omega,x)\ge n)\le e^{\max\{R,0\}}e^{-\alpha n}.
\]
\end{prop}

The proof of this proposition is a combination of a local coupling lemma 
(Lemma~\ref{ref:small_scale_coupling_lemma}) along with a recovery procedure.

When we attempt to couple two curves, we will insist that they are in a configuration that allows us to try and apply the Local Coupling Lemma (Lemma~\ref{ref:small_scale_coupling_lemma}). What we mean by this is that the curves have controlled regularity and are sufficiently near to each other. 

\begin{defn}\label{defn:C_0_delta_configuration}
Let $\hat\gamma$ be a standard pair and $x\in \gamma$. We say that $x$ is {\em $(C,\delta)$-well positioned in $\hat{\gamma}$}
if $\hat\gamma$ is $C$-regular and $x$ is $\delta$ distance away from the endpoints of $\gamma$, with distance measured along $\gamma$.

We say that two standard pairs $\hat{\gamma}_1$ and $\hat{\gamma}_2$ are in a \emph{$(C,\delta,\upsilon)$-configuration} 
if there exist $x$ which is $(C, \delta)$-well positioned in $\hat\gamma_1$, and $y$ which is  $(C, \delta)$-well positioned in
$\hat \gamma_2$ such that $d(x,y)\!<\!\upsilon$.
\end{defn}

The proof of Proposition \ref{prop:main_coupling_proposition} proceeds along the following steps. 
We start with two $C_0$-good standard pairs, $\hat{\gamma}_1$ and $\hat{\gamma}_2$. Here $C_0$ is some uniform regularity appearing in Proposition \ref{prop:coupled_recovery_lemma} that we may obtain starting from an arbitrarily bad curve by waiting long enough.

\begin{enumerate}[leftmargin=*]
\item 
We prove that for a large proportion of words $\omega\in \Sigma$, the images $f^n_{\omega}(\hat{\gamma}_1)$ and $f^n_{\omega}(\hat{\gamma}_2)$ are mostly quite regular, and moreover, there is a large measure subset of the images that can be paired to form 
$ (C_1,\delta,\upsilon)$-configurations for some $C_1$ that is worse that $C_0$. 
 This relies on the mixing properties of our system studied in 
 Section \ref{sec:k_property}, and the needed conclusions are made precise in Proposition \ref{prop:finite_time_mixing}.
    \item 
    We then run a ``local" coupling argument on each tiny $(C_1,\delta,\upsilon)$-configuration. At each time step, we attempt to couple the remaining well tempered points using ``fake" stable manifolds.  This local coupling argument, Lemma \ref{ref:small_scale_coupling_lemma}, has a number of steps and draws on several intermediate estimates.
    
        (a)  There are $C,\lambda,\epsilon>0$ and a cone field $\mc{C}_{\theta}$  
        that is uniformly transverse to both $\gamma_1$ and $\gamma_2$
        such that the probability that any point is $(C,\lambda,\epsilon)$-tempered and has $E^s$ tangent to 
          $\mc{C}_{\theta}$ 
        is positive.
         Further, 
        the probability that the tempering fails at time $n$ is exponentially small.
       
       (b) For a $(C,\lambda,\epsilon)$-tempered point at time $n$, we see that there is a ``fake" stable manifold $W^s_n$ given by taking a curve nearly tangent to $Df^n_{\omega}(E^s_n)$ and pushing this curve backwards by $(Df^n_{\omega})^{-1}$. 
       (This construction is the subject of \S \ref{subsec:construction_of_fake_stable})

       (c)
        There exist worse $(C',\lambda',\epsilon')$ such that for every $(C,\lambda,\epsilon)$-tempered point $x$ in $\gamma_1$, all points within distance$\|D_xf^n_{\omega}\|^{-(1+\sigma)}$ of 
        are       
        $(C',\lambda',\epsilon')$-tempered points at time $n$.
        (This is the content of Proposition \ref{prop:nearby_points_inherit_temperedness}). These $(C',\lambda',\epsilon')$-tempered points also have fake stable manifolds. We will try to couple these thickened neighborhoods of the $(C,\lambda,\epsilon)$-tempered points with some neighborhoods in $\gamma_2$ determined by the fake stable holonomies.  At the time when $D_xf^n_{\omega}$ fails to be $(C,\lambda,\epsilon)$-tempered with $E^s$ tangent to $\mc{C}_{\theta}$ we discard the point $x$ and stop trying to couple it. 
        
       (d)
        For $(C',\lambda',\epsilon')$-tempered points, the holonomies of the fake stable manifolds $W^s_n$ between $\gamma_1$ and $\gamma_2$ converge exponentially fast to the true, limiting stable holonomy. Moreover, the image of a point $ x\in \gamma_1$ under $H^s_n$ has fluctuations, as $n$ changes, of size $\|D_xf^n_{\omega}\|^{-1.99}$, i.e.~the distance between $H^s_n(x)$ and $H^s_{n+1}(x)$ in $\gamma_2$ is at most $\|D_xf^n_{\omega}\|^{-1.99}$. (This is proved in Proposition \ref{prop:fluctuations_in_fake_stable_leaves}.)

        (e)
        The points we try to couple with on $\gamma_2$ are the image of the points on $\gamma_1$ under the fake stable holonomy $H^s_n$. 

       (f)
        By carefully choosing subdivisions of the standard pairs $\hat{\gamma}_1$ and $\hat{\gamma}_2$ we may discard mass from the standard pairs so that at the end of the procedure  a positive proportion of the mass above each $(C,\lambda,\epsilon$)-tempered point remains. The control on the size of the fluctuations of $H^s_n$ relative to the lengths of the intervals of $(C',\lambda',\epsilon')$-tempered points containing the $(C,\lambda,\epsilon)$-tempered points $\|D_xf^n_{\omega}\|^{-1.99}\ll \|D_xf^n_{\omega}\|^{-(1+\sigma)}$ allows us to ensure that we always have enough points on $\gamma_2$ to try to couple with.

  \item 
  We prove that we may find simultaneous recovery times for a pair of $R$-good standard pairs (Proposition \ref{prop:coupled_recovery_lemma}), so that if we have failed to couple and are left with a 
    short standard subcurve of $\hat{\gamma}_1$ we can have this subcurve recover at the same time as a subcurve of $\hat{\gamma}_2$.

        \item  Once we recover we will try to couple again using steps (1)--(3) above.
        Each time we try to couple, a positive amount of mass couples, and as the tail 
 on the recovery time is exponential we do not spend too much time recovering.
 \end{enumerate}

\subsection{Statements of the lemmas for use during coupling}
We now state the main propositions and lemmas that are used in the proof of Proposition \ref{prop:main_coupling_proposition}.

\begin{lem}\label{prop:coupled_recovery_lemma}
(Coupled Recovery Lemma) Let $M$ be a closed surface and let $(f_1,\ldots,f_m)$ be an expanding on average tuple with entries in $\Diff^2_{\vol}(M)$. There exist $C_0,D_1,\alpha>0$ such that if $\hat{\gamma_1}=(\gamma_1,\rho_1)$ and $\hat{\gamma_2}=(\gamma_2,\rho_2)$ are $R$-good standard families of equal mass then there is a pair of stopping times $\hat{T}_1$ and $\hat{T}_2$ defined on $\hat{\gamma_1}$ and $\hat{\gamma}_2$ with the following properties:
\vskip1mm

\noindent
(1)    There is an exponential tail on the stopping time. Namely,
    \[
    (\mu\otimes \rho_1)((\omega,x)\mid \hat{T}_1(\omega,x)>n)\le D_1e^{\max\{R,0\}-\alpha n}.
    \]
    \vskip1mm
\noindent
(2)    If $z\in \hat{\gamma}_i$ is a point that stops at time $n$, and $B_i(z)$ is the connected component of $z$ in the set $\{x\in \hat{\gamma}_i:\hat{T}_i(\omega,x)=n\}$, i.e~the set of points $z\in \hat{\gamma_i}$ stopped at time $n$, then
$\displaystyle
    f^{\hat{T}_i(z)}_{\omega}(B_i(z))
    $
    is a $C_0$-good standard pair.
    \vskip1mm
  
  \noindent
  (3)  For each $\omega\in \Sigma$, we always stop on the same amount of mass of $\hat{\gamma}_1$ and $\hat{\gamma}_2$ at each time $n$. Specifically,
    for each $\omega$ and $n$,
 denote $S_i(\omega,n)=\{x\in \hat{\gamma}_i: \hat{T}_i(\omega,x)=n\}$. For each pair $(\omega,n)$ there is a  measure preserving
 map $\Phi_n^{\omega}\colon S_1(\omega,n)\to S_2(\omega,n)$ carrying $C_0$-good connected components of $S_1(\omega,n)$ to $C_0$-good connected components of $S_2(\omega,n)$.
 \end{lem}

The following lemma is the most technical part of the coupling argument.

\begin{lem}\label{ref:small_scale_coupling_lemma}
(Local Coupling Lemma)
Suppose that $(f_1,\ldots,f_m)$ is an expanding on average tuple.
There exists $0<\tau<1$ such that for any $C_1>0$ there exists $\delta_0,L, D_1,D_2,\beta,C,\lambda,\epsilon>0$ such that for any $0<\delta'<\delta_0$ there exists $\delta_1$ and $\epsilon_0,a_0>0$ such that 
for any two standard pairs $\hat{\gamma}_1$ and $\hat{\gamma}_2$ that are in a $ (C_1,\delta',\upsilon)$-configuration with $\upsilon\le \tau\delta'$, we may couple a uniform proportion of the points on the two curves with an exponential tail on the points that do not couple.

Specifically, for two $C_1$-good standard pairs $\hat{\gamma}_1,\hat{\gamma}_2$ of the same mass in a $(C_1,\delta',\upsilon)$-configuration with $\upsilon\le \tau\delta'$, there is a point $x\in M$, a ball $B_{\delta_0}(x)\subset M$ and connected components $\Gamma_1$ and $\Gamma_2$ of $\hat{\gamma}_1\cap B_{\delta_1}(x)$ and $\hat{\gamma}_2\cap B_{\delta_1}(x)$ such that $\Gamma_1$ and $\Gamma_2$ each contain $a_0$ proportion of the mass of $\hat\gamma_1$ and $\hat\gamma_2$ respectively.

Further, there exist a pair of stopping times $\hat{T}_1(\omega,x)$ and $\hat{T}_2(\omega,x)$ defined on $\hat{\gamma}_1$ and $\hat{\gamma}_2$ such that if  $B^{\hat{T}_i}(\omega,x)\subseteq \hat{\gamma}_i$ denotes the block of points stopped at the same time as $x$, then
\begin{enumerate}[leftmargin=*]
\item \label{item:equal_mass_stops_lcl}
For all $\omega,n$ there exists $\Psi^{\omega}_n\colon \{x\in \hat{\gamma}_1\colon \hat{T}_1(\omega,x)=n\}\to \{x\in \hat{\gamma}_2: \hat{T}_2(\omega,x)=n\}$ such that if $\hat{T}_i(\omega,x)=n$, then $B(\omega, x)$ is an $nL$-good standard pair and $\Phi^{\omega}_n$ carries $B(x)$ to an $nL$-good standard pair $B(\Phi^{\omega}_n(x))\subseteq \hat{\gamma}_2$ of equal mass that is also stopped at time $n$. 
\item \label{item:tail_points_intertwined}
For each $\omega$, the set of points in $\hat\gamma_1$ and $\hat\gamma_2$ where $\hat{T}_i=\infty$ are of equal measure and moreover these sets are intertwined by a measure preserving stable holonomy along uniformly $(C,\lambda,\epsilon)$-tempered stable manifolds. 
\item\label{item:good_tail} 
There exists $D_1>0$ such that
$\displaystyle
(\mu\otimes \hat{\rho}^1)(\{(\omega,\hat{x}): \hat{T}_1(\omega,\hat{x})=n\})\le D_1e^{-\beta n}.
$
For $\hat{\gamma}_2$, we have a similar estimate, 
$\displaystyle
(\mu\otimes \hat{\rho}^2)(\{(\omega,\hat{x}): \hat{T}_2(\hat{x})=n\})\le D_1e^{-\beta n}.
$
\item\label{item:positive_prob_of_coupling}
For all $x\in \Gamma_1 $, the measure of words $\omega$ such that $\hat{T}_i(\omega, x)=\infty$ is at least $\epsilon_0$. 
\end{enumerate}
\end{lem}
\vskip2mm

In the lemma above, part \eqref{item:tail_points_intertwined} says that the points where $\hat{T}_i=\infty$
are coupled and such points attract exponentially fast. Part \eqref {item:positive_prob_of_coupling} says
that the probability that the next coupling attempt is successful is at least $\epsilon_0.$
Part \eqref{item:good_tail} says that the probability that ``a point" stops and fails to couple at time $n$ is exponentially small,
while part \eqref{item:equal_mass_stops_lcl} controls he regularity of the set of such points.
\vskip1mm

The following proposition says that there is a fixed time $N_0$ required for the $C_0$-good pairs produced by the coupled recovery lemma 
to 
get into position for the application of the local coupling lemma. The proof relies on the mixing properties from Section \ref{sec:k_property}.

\begin{prop}\label{prop:finite_time_mixing} 
(Finite Time Mixing)
Suppose $(f_1,\ldots,f_m)$ is an expanding on average tuple as in Proposition \ref{prop:main_coupling_proposition}. For any fixed $C_0>0$, there exist $C_1,C_2,\delta,\upsilon >0$ such that the following holds. 
\begin{enumerate}[leftmargin=*]
    \item 
    $C_1,\delta,\upsilon >0$ are such that a $(C_1,\delta,\upsilon)$-configuration
    satisfies the hypotheses of the Local Coupling Lemma \ref{ref:small_scale_coupling_lemma} with $C_1=C_1$, $\delta'=\delta$, and $\upsilon=\upsilon$.
    
    \item 
    There exists $N_0\in \N$ and $b_0>0$ such that for any $C_0$ regular standard pairs $\hat{\gamma}_1$ and $\hat{\gamma}_2$ of equal mass, for $.99\%$ of the words $\omega\in \{1,\ldots,m\}^{N_0}$, there is a subdivision $P^1_{\omega},P^2_{\omega}$ of the standard families $f^{N_0}_{\omega}(\hat{\gamma}_1)$ and $f^{N_0}_{\omega}(\hat{\gamma}_2)$ and subfamilies $Q^1_{\omega},Q^2_{\omega}$ of $P^1_{\omega}$ and $P^2_{\omega}$, and a map $\Psi\colon Q^1_{\omega}\to Q^2_{\omega}$ preserving measure such that the following hold.
    \begin{enumerate}
        \item 
        Each pair $\hat{\gamma}\in Q^1_{\omega}$ is associated by $\Psi$ with a pair $\Psi(\hat{\gamma})$ such that 
       these pairs have equal mass and satisfy (1) above. 
        \item 
        The set $ Q^1=\bigcup_{\omega\in \hat{\Sigma}} \{\sigma^{N_0}(\omega)\}\times Q^1_{\omega}$ has measure $b_0\rho_1(\hat{\gamma})$ with respect to $\hat{\mu}\otimes \rho_1$. The same holds for $Q^2$.
    \end{enumerate}
    \item 
    The complement of $Q^1_{\omega}$ in $f^n_{\omega}(\hat{\gamma}_1)$ is a standard family of $C_2$-good standard pairs. The same holds for $Q^2_{\omega}$.
\end{enumerate}
\end{prop}

As mentioned before, the proofs of these lemmas appear later in the paper.
Lemma \ref{prop:coupled_recovery_lemma} is proven in Section \ref{sec:finite_time_smoothing_estimates},
Proposition \ref{prop:finite_time_mixing} is proven in Section \ref{sec:finite_time_mixing_prop}, and 
Lemma \ref{ref:small_scale_coupling_lemma} is proven in Section \ref{ScLocalCoupling}.

\subsection{Proof of the main coupling proposition}

We now show how to deduce the main coupling proposition, Proposition \ref{prop:main_coupling_proposition}, from the various results stated in this section.
 We need a preliminary estimate showing that if we fail to couple then the whole failed attempt does not take too long.
In the lemma below the recovery time is the sum of three terms:

(1) The time when we stop trying to locally couple as in Lemma \ref{ref:small_scale_coupling_lemma} item \eqref{item:good_tail}; 

(2) The time it takes for a point to recover so that it belongs to a $C_0$-good pair as in the Coupled Recovery Lemma \ref{prop:coupled_recovery_lemma};

(3) The fixed time $N_0$ where the point has a chance to enter a $(C_1,\delta,\upsilon)$-configuration according to Proposition \ref{prop:finite_time_mixing}.

The following lemma verifies that each trip through the coupling procedure has an exponential tail on its duration.

\begin{lem}\label{lem:tail_onrecovery_time}
In the setting of Proposition  \ref{prop:main_coupling_proposition}, for each $C$ there exist $\hat C$ and $\bar r$ such that if $\hat\gamma_1$ and $\hat\gamma_2$ are $C$-good standard pairs of equal mass, then  
  $$ (\mu\otimes \rho_1)((\omega,\hat{x}): (\omega,\hat{x})\text{ fails to couple and the recovery time is greater than } n)\leq 
  \hat C e^{-\bar r n}. $$
  \end{lem}

\begin{proof}

Take a small $\kappa>0$ that will be specified below. First we try to locally couple, and then we recover. 
Let $T$ be the recovery time and $S$ be the time when we stop our attempt at coupling $(\omega,x).$
Then if $T\ge n$ then either:

(i) $S\ge \kappa n$ or (ii) $S\le \kappa n$ and the time it takes the corresponding part of the curve to recover is
at least $(1-\kappa)n .$

The probability of the first event is exponentially small due to Proposition~\ref{ref:small_scale_coupling_lemma}(3). 
In the second case since $S\leq \kappa n$, it follows that $(\omega,x)$ belongs to $\kappa L n$-good component.
Thus by Proposition \ref{prop:coupled_recovery_lemma} the probability that the recovery takes more than $(1-\kappa)n$ time is less than
$D_1 e^{(\kappa L-\alpha(1-\kappa)) n}$ which is exponentially small if $\kappa<\alpha/(L+\alpha).$  
\end{proof}

The main coupling proposition is now easy to deduce because each coupling attempt couples a positive proportion of the remaining mass and, from Lemma \ref{lem:tail_onrecovery_time}, there is an exponential tail bound on how long a coupling attempt takes.

\begin{proof}[Proof of Proposition ~\ref{prop:main_coupling_proposition}.]
Let $N(\omega,x)+1$ be the number of total attempts at local coupling before $(\omega,x)$ couples. Let $\hat{T}(\omega,x)$ be the time when $(\omega,x)$ couples, and let $T_k(\omega,x)$ be its $k$th recovery time, i.e.~the  $k+1$st time we attempt to locally couple. As a positive amount of mass couples each time we apply the local coupling lemma, we see that there exists $\delta>0$ such that 
\begin{equation}\label{eqn:number_of_trials}
(\mu\otimes \rho_1)((\omega,x): N(\omega,x)>k)\le e^{-k\delta}.
\end{equation}

Next we show that for points that take $k$-attempts at local coupling to couple, that these attempts occur linearly fast. This will follow once we have a tail bound on $T_k$.
By Lemma~\ref{lem:tail_onrecovery_time}, ${T}_1$ has an exponential moment. In particular, 
$\sup \E{e^{tT_1}}=M(t)$ is finite for $t\le r$ where $r< \bar r$ and $\bar r$ is the constant from Lemma \ref{lem:tail_onrecovery_time}
and the supremum is taken over all pairs $\hat\gamma_1, \hat\gamma_2$ of
$C_1$-good standard pairs which are in $(C_1,\delta,\upsilon)$-configurations
as required by Lemma \ref{ref:small_scale_coupling_lemma} and produced by Proposition~\ref{prop:finite_time_mixing}. 

Extend $T_k=T_N(\omega)$ if $k>n(\omega).$
A straightforward induction shows that 
$\displaystyle
\E{e^{tT_k}}\le M(t)^k.
$
Thus by the Chernoff bound
$ \displaystyle
(\mu\otimes\rho_1)(T_k\ge n)
\le M(t)^{k}e^{-tn}.
$
In particular taking $t=r$, there is some $\beta>0$ such that 
$\displaystyle
(\mu\otimes\rho_1)(T_k\ge n\vert N=  k)\le e^{\beta k}e^{-rn}.
$
 Fix some small number $\alpha$ such that $0<\beta\alpha <r/2$.
Then
$$
(\mu\otimes\rho_1)(T_N> n \text{ and } N \le \alpha n)\le
(\mu\otimes\rho_1)(T_{\alpha n}> n) \le
D_1e^{-r/2 n}.
$$
By \eqref{eqn:number_of_trials}, with probability $1-e^{-\delta \alpha n}$, a point $(\omega,x)$ couples after at most $\alpha n$ trials, and the result  follows.
\end{proof}

\section{Proof of the Coupled Recovery Lemma}\label{sec:finite_time_smoothing_estimates}

\subsection{Recovery times}
In this subsection, we use the preceding lemmas to 
 describe a recovery algorithm for the $C^2$ norm of an irregular curve and estimate the tail of the recovery time.

The next definition 
describes an iterate of $f^n_{\omega}$ that has a good enough splitting that $f^n_{\omega}(\gamma)$ will have a good neighborhood of a particular point. Note that a ``good enough" splitting requires both a condition on the hyperbolicity as well as a condition on the angle between the curve $\gamma$ and and the stable subspace. This definition will be used in the proof of the recovery lemma.

\begin{defn}\label{defn:backwards_good_time}
Fix a tuple of non-negative numbers $(C,\lambda,\epsilon,A,\epsilon',R)$. For a standard pair $\hat{\gamma}$, a point $x\in \gamma$ and a word $\omega\in \Sigma$, we say that $n$ is a $(C,\lambda,\epsilon,A,\epsilon',R)$\emph{-backwards good time} for $x,\gamma,\omega$ if $n=A\max\{R,1\}+i$, for some $i\ge 0$ and
\begin{enumerate}
    \item 
    $Df^n_{\omega}$ has a $(C,\lambda,\epsilon)$-reverse tempered splitting, for which we write $E_m^s,E_m^u$ for the stable and unstable subspaces of this splitting in $T_{f^m_{\omega}(x)}M$.
    \item 
    $\angle(E^s_0,\dot{\gamma}(x))\ge e^{-\epsilon'i}$.

\end{enumerate}
\end{defn}

The following lemma asserts that this type of backwards good time is sufficient to conclude that an $R$-good curve $\gamma$ has its neighborhood of $x$ smoothed by the random dynamics $f^n_{\omega}$.

Note that the second condition in the lemma considers the situation where $\gamma$ ``recovers" in a neighborhood of $x$ prior to time $n$. It is important in this case to know that from that point on, we can just restrict to the portion of the curve that has already recovered. This is useful because it helps us deal with situations where we wish to ``stop" on certain parts of the curve and know that the parts we have stopped on will not be needed later when a different part of the curve recovers. Recall from Definition \ref{defn:R_good_standard_pair} that an $R$-regular curve has all the characteristics of $R$-good curves except that it is not required to be $e^{-R}$ long.

\begin{lem}\label{lem:deterministic_recovery_lemma_3}
Suppose $M$ is a closed surface and that $(f_1,\ldots,f_m)$ is a tuple in $\Diff^2_{\vol}(M)$. Then for any $\lambda>0$, sufficiently small $\epsilon,\epsilon'>0$, and any $C>0$, there exists $A,C_0,C_1>0$ such that for any $R$-regular standard pair $\hat{\gamma}=(\gamma,\rho)$ and any $(C,\lambda,\epsilon, A, \epsilon',R)$-backwards good time $n$ for $\omega\in \Sigma$ and $x\in \gamma$ if:
\begin{enumerate}
    \item 
    $\hat{\gamma}$ is $R$-good, or
    \item 
    there exists a time $0\le m< n$ and a subinterval $I\subseteq \gamma$ such that $f^m_{\omega}(I)$ contains a neighborhood of $f^m_{\omega}(x)$ that is $e^{-C_1}e^{-.8\lambda(n-m)}$-long;
    \end{enumerate}
then $f^n_{\omega}(\hat{\gamma})$ contains a $C_0$-good neighborhood of $f^n_{\omega}(x)$.  Moreover, if (2) holds, this neighborhood is contained in $f^n_{\omega}(I)$.
\end{lem}
The above lemma follows immediately from the  result below.  The second paragraph of the statement of the lemma essentially says: if there is another point in $\gamma$ that also experiences a recovery time, then we can stop on that recovering segment while still leaving enough of the curve $\gamma$ so that $x$ can still recover.

\begin{lem}\label{lem:deterministic_recovery_lemma2}
(Deterministic Recovery Lemma) 
Given a closed surface $M$ and a tuple $(f_1,\ldots,f_m)$ in $\Diff^2_{\vol}(M)$, for any $\alpha,\lambda>0$ and all sufficiently small $\epsilon,\epsilon'>0$ and any $C>0$, there exist $C_0,A>0$ such that for any $R$-good standard pair $\hat{\gamma}=(\gamma,\rho)$,
and any word $\omega$ such that time $n$ is a $(C,\lambda,\epsilon,A,\epsilon',R)$-backwards good time for $x\in \gamma$,
then there exists a neighborhood $B(x)\subseteq \gamma$ of size at most $e^{-.9\lambda n}$ such that $f^n_{\omega}(\hat{B}(x))$ is $C_0$-good, i.e.~the pushforward of the standard pair $\hat{\gamma}$ restricted to $B(x)$ is $C_0$-good. 

Further, there exists $C_1$ such that for $\omega,x,\gamma$ as in the first part of the lemma, if $I\subseteq \gamma$ is an interval containing $x$ and for some $1\le m< n$, $f^m_{\omega}(I)$ has length at least $e^{-C_1}e^{-.8\lambda(n-i)}$, then $f^n_{\omega}(I)$ contains a $C_0$-good neighborhood of $f^n_{\omega}(x)$. 
\end{lem}

\begin{proof}
We divide the proof into several steps. We begin by fixing some preliminaries. For the given $(C,\lambda,\epsilon)$, we apply 
Proposition \ref{prop:finite_time_smoothing_estimate} with $e^{-i\epsilon'}=\theta$,
which gives us the constants $\epsilon_0,\ell_{\max},D_2,\ldots,D_8$ appearing in that proposition.

\noindent\textbf{Step 1.} (Length of $f^n_{\omega}\gamma$) 
By Proposition \ref{prop:finite_time_smoothing_estimate}\eqref{item:preimages_of_recovered_curves_shrink}, if 
\begin{equation}
\label{N-D5}
n\ge D_{5}+\frac{\max\{R,0\}-2\ln(e^{-i\epsilon'})}{.99\lambda},
\end{equation}
then $f^n_{\omega}\gamma$ contains a neighborhood $\gamma_n$ of $f^n_{\omega}(x)$ of length $\ell_{\max}$. For $\epsilon'$ sufficiently small relative to $\lambda$, it follows that  \eqref{N-D5} holds
as long as $n\ge A_1\max\{R,1\}+i$ for some $A_1$ depending only on $D_5, \lambda,\epsilon'$.

\noindent\textbf{Step 2.} ($C^2$ estimate) 
By Proposition \ref{prop:finite_time_smoothing_estimate}\eqref{item:C_2_estimate_smoothing_seq} 
\begin{equation}
\|\gamma_n\|_{C^2}<D_6e^{-2.9\lambda n}e^{D_7\ln\theta}\max\{\|\gamma\|_{C^2},1\}+D_8.
\end{equation}
Thus 
there exists $A_2, C_2$ such that as long as $n\ge A_2\max\{R,1\}+i$, that $\|\gamma_n\|_{C^2}\le C_2$.

\noindent\textbf{Step 3.} (Smoothing the density)
From Proposition \ref{prop:finite_time_smoothing_estimate}\eqref{item:regularity_of_density_sequence_lem} applied to $D_9=C_2$ from the previous step, we see that there exists $D_{10},D_{11}$ such that the following holds. If $\|\gamma_n\|_2<D_8$, then the pushforward of $\rho$ along $\gamma_n$ is given by:
\begin{equation}
    \|\ln \rho_n\vert_{\gamma_n}\|_{C^{\alpha}}\le D_{10}e^{-.9\alpha \lambda n}e^{D_7\ln \theta}(1+\|\ln \rho\|_{C^{\alpha}}+\|\gamma\|_{C^2})+D_{11}.
\end{equation}
In particular as long as $N\ge A_2\max\{R,1\}+i$, the above estimate holds. In the case that this estimate holds, then as $\|\ln\rho\|_{C^{\alpha}}$ and $\|\gamma\|_{C^2}$ are both at most  $e^R$, we similarly see that there exists $C_3$ and $A_3$ such that if $n\ge A_3\max\{R,1\}+i$ then $\|\ln \rho_n\vert_{\gamma_n}\|_{C^{\alpha}}\le C_3$. Thus we see that there exists $A$ such that the conclusion of the first paragraph holds.

For the claim in the second paragraph of the Lemma, we can apply 
Proposition \ref{prop:finite_time_smoothing_estimate}\eqref{item:preimages_of_recovered_curves_shrink}. 
The choice of $A,C_0$ in the first part of the proof imply that for such $n$, $\ell_{\max}$ is realized and thus by the final part of item \eqref{item:preimages_of_recovered_curves_shrink} then 
the preimage of $\gamma_n$ in $f^i_{\omega}\gamma$ has length at most $D_4e^{-.9\lambda (n-i)}$, thus if $f^i_{\omega}(I)$ has length at least $D_4e^{-.8\lambda(n-i)}$, then the image of $f^i_{\omega}(I)$ will have image that is a $C_0$ good neighborhood of $f^n_{\omega}(x)$.
\end{proof}

Next we show that  the recovery times from the above lemma occur frequently. 
\begin{prop}\label{prop:exponential_recovery_time_pointwise_annealed_over_curve} Let $M$ be a closed surface and suppose that $(f_1,\ldots,f_m)$ is an expanding on average tuple in $\Diff^2_{\vol}(M)$. There exists $\lambda>0$ such that for any $A>0$ and sufficiently small $\epsilon,\epsilon'>0$, there exist $C>0$ and $\alpha_3>0$ such that for any $R$-good standard pair $\hat{\gamma}$, if for $x\in \gamma$ we let $\hat{T}(\omega,x)$  be the first 
$(C,\lambda,\epsilon,A,\epsilon',R)$-backwards good time.
Then
\begin{equation}\label{eqn:exp_tail_on_local_recovery_time0}
(\mu\otimes\rho)((\omega,x): \hat{T}(\omega,x)>A\max\{R,1\}+i)\le Ce^{-\alpha_3 i}.
\end{equation}
The same holds for the analogous stopping time defined on an $R$-good standard family.
\end{prop}

\begin{proof}
It suffices to prove this estimate at a single point $x$ as we may then integrate the resulting estimate over all of $\hat{\gamma}$. 
From Proposition  \ref{prop:exponential_tail_tempered_times} there exist $C_1,\alpha_1$ and $C,\lambda>0$ such that for all sufficiently small $\epsilon>0$ there exists $N\in \N$ such that if we let $S(\omega)$ be the stopping time that stops at the first $(C,\lambda,\epsilon)$-reverse tempered time of $D_xf_{\omega}^n$ greater than any fixed $n\ge N$, then at that time there is a well defined splitting $T_xM=E^s_S\oplus E^u_S$ into maximally expanded and contracted singular directions, and
\begin{equation}\label{eqn:speed_of_stopping_89}
\mathbb{P}(S(\omega)>n+k)\le C_1e^{-\alpha_1k}.
\end{equation}
By Lemma \ref{lem:prob_good_stable_when_stopped} there exist $C_2,\alpha_2>0$ such that as long as $n\ge c_0\abs{\ln \theta}$,
\begin{equation}
\mathbb{P}(\angle( E^s_S,\dot{\gamma}(x))<\theta\vert S\le n+k)<C_2\theta^{\alpha_2}.
\end{equation}
Hence there exists $\alpha_3>0$ such that if $S$ is the first time greater than $n=c_0\epsilon' i$ that has a reverse tempered splitting, then 
\begin{equation}\label{eqn:stopped_angle_est_close_E_s}
\mathbb{P}(\angle(E^s_S,\dot{\gamma}(x))<e^{-\epsilon' i}\vert S\le n+k)<C_2e^{-\alpha_2\epsilon' i}.
\end{equation}
In particular, as long as $\epsilon'$ is sufficiently small relative to $c_0$, then $c_0\epsilon' i<i/2$. Let $S$ be the first $(C,\lambda,\epsilon)$-reverse tempered time greater than $A\max\{R,1\}+i/2$. 
Multiplying equations \eqref{eqn:speed_of_stopping_89} and \eqref{eqn:stopped_angle_est_close_E_s}, we find that there exist $C_3,\alpha_3>0$ such that:

\hskip2cm $\displaystyle
\mathbb{P}(S\le A\max\{R,1\}+i\text{ and } \angle(E^s_S,\dot\gamma)\ge e^{-\epsilon'i})\ge 1-C_3e^{-\alpha_3 i}.
$
\end{proof}

We now state without proof a more technical variant of the preceding lemma. It will be used in the proof of the coupled recovery lemma to allow ``recovery times" for the hyperbolicity. We will divide the iterates of the system into blocks of size $\Delta q+\Delta$, where $\Delta,q\in \N$. Each block will be divided into two pieces one of length $\Delta q$ and one of length $\Delta$. We will only be interested in backwards good tempered times that occur in the second part of the block, which has length $\Delta$. This is to ensure that there are large (temporal) gaps between possible recovery times. The following lemma shows that given this extra restriction on the backwards good times, we still have an exponential tail. 

\begin{prop}\label{prop:blocked_backwards_good_times_1}
Let $M$ be a closed surface and suppose that $(f_1,\ldots,f_m)$ is an expanding on average tuple in $\Diff^2_{\vol}(M)$. There exists $\lambda>0$ such that for any $A>0$ and sufficiently small $\epsilon,\epsilon'>0$, there exist $C>0$ and 
$\alpha_4>0$ such that for all $\Delta,q\in \N$ and any $R$-good standard pair $\hat{\gamma}$, for any $N\ge A\max\{R,1\}$, if for $x\in \gamma$ we let $\hat{T}(\omega,x)$ be the first time greater than equal to $N$ such that 
\[
\lceil A\max\{R,1\}\rceil +j(q+1)\Delta+q\Delta<\hat{T}(\omega,x)\le \lceil A\max\{R,1\}\rceil +(j+1)(q+1)\Delta,
\] 
for some $j>0$ and $\hat{T}$ is a
$(C,\lambda,\epsilon,A,\epsilon',R)$ backwards good time, 
then
\begin{equation}\label{eqn:exp_tail_on_local_recovery_time1}
(\mu\otimes\rho)((\omega,x): \hat{T}(\omega,x)>N+i(q+1)\Delta)\le Ce^{-\alpha_4 i\Delta}.
\end{equation}
\end{prop}

\subsection{Coupled Recovery Lemma}
In this subsection, we prove the coupled recovery lemma, Lemma \ref{prop:coupled_recovery_lemma}. In the statement we view the standard pair as the uniform distribution on the
subset of $\gamma\times [0,\infty)$ of pairs $(x,t)$ where $t\le \rho(x)$. We do this so that we may define stopping times for $\hat{\gamma}$ that stop on only part of the fiber over each point in $\gamma$. Additionally, in an abuse of notation, we will identify the density $\rho$ with a measure that we also call $\rho$.

\begin{proof}[Proof of Lemma \ref{prop:coupled_recovery_lemma}]
After initial preliminaries, the proof divides into two parts. The first part is a coupled stopping procedure, which takes a word $\omega\in \Sigma$ and two standard pairs $\hat{\gamma}_1$ and $\hat{\gamma}_2$, and shows which parts of each curve get stopped as we follow the dynamics specified by $\omega$ so that we always stop on the same amount of mass of each pair. 
In the second part we show that with high probability the procedure from the first part actually stops on all but an exponentially small amount of $\hat{\gamma}_1,\hat{\gamma}_2$ in a linear amount of time. In the proof, we consider the case that $R>1$ as otherwise we can stop immediately and conclude.

We now fix some constants. By Proposition \ref{prop:blocked_backwards_good_times_1}
there exists $\lambda>0$ such that for any $A>0$ and sufficiently small $\epsilon,\epsilon'>0$, there exists $C>0$ and $\alpha>0$ such that $(C,\lambda,\epsilon,A,\epsilon',R)$-backwards good times at the end of blocks of length $(q+1)\Delta$ occur exponentially fast after any time $N$ greater than $A\max\{R,1\}$ for an $R$-good standard pair $\hat{\gamma}$, i.e.  \eqref{eqn:exp_tail_on_local_recovery_time1} holds.

We then apply Lemma \ref{lem:deterministic_recovery_lemma_3}, which shows that for this choice of $\lambda,C,\epsilon,\epsilon',A$, that any $R$-good standard pair $\hat{\gamma}$ and any $(C,\lambda,\epsilon,A,\epsilon',R)$-backwards good time to $x\in \hat{\gamma}$, 
$f^n_{\omega}(x)$ has a $C_0$-good neighborhood in $f^n_{\omega}(\hat{\gamma})$, i.e.~the dynamics smoothens a neighborhood of $x$ and makes it $C_0$ regular. Lemma \ref{lem:deterministic_recovery_lemma_3} also gives the constant $C_1$ so that as long as $f^i_{\omega}(I)$ contains a neighborhood of $f^i_{\omega}(x)$ of size at least $e^{-C_1}e^{-.8\lambda(n-i)}$, then $f^{n-i}_{\sigma^i(\omega)}(f^i_{\omega}(I))$ contains a $C_0$-good neighborhood of $f^n_{\omega}(x)$.

For the rest of the proof we will not repeat $(C,\lambda,\epsilon,A,\epsilon',R)$-backwards good but just refer to such times as \emph{tempered times} with this particular choice of constants being understood. 

In the proof that follows, we divide the iterates of the system into blocks of size $(q+1)\Delta$. We will attempt to stop on a neighborhood of a point $x$ when $D_xf^n_{\omega}$ has a tempered time in the interval $(\lceil AR\rceil+i(q+1)\Delta+q\Delta, \lceil AR\rceil +(i+1)(q+1)\Delta]$. This is the $i$th block, if there is such a tempered time, then we say that this is a \emph{tempered block}. In the following, there will be points $x$ that experience a tempered block ending at $\lceil AR\rceil +iq\Delta$ but that we do not stop because 
there was not enough mass stopping on the other curve to couple them. For these curves, we then wait for their next tempered time relative to the original curve. 
That we only allow stopping on the last $\Delta$ iterates of a block of length $(q+1)\Delta$ is to ensure that the hyperbolicity has enough time to stretch what remains of the recovered neighborhood of $f^{\lceil AR\rceil +i\Delta}_\omega(\gamma)$ so that it can recover to be a $C_0$-good curve at the tempered time.

In the proof we only try to couple recovered curves at the very last time in 
each block, whereas a curve may have a tempered time up to $\Delta$ iterates before then. If we have a $C_0$-good curve, $\hat{\gamma}$, and we apply the dynamics from $(f_1,\ldots,f_m)$ at most $\Delta$ additional times, then there is some $C_0'\ge C_0$, so that the image of the curve will still be $C_0'$ good even after those extra iterates. Consequently, for any $\alpha>0$, there exists $\delta(\alpha)>0$, such that if $\hat\gamma$ is a $C_0'$ good curve, and we trim off the end segments of the curve of length $e^{-\delta}$, then we have lost at most $e^{-\alpha}$ proportion of the curve, where $\alpha$ is some number we will choose below. Further, note that as long as $\delta$ is sufficiently large, 
the trimmed off curves will be $e^{-\delta}$-good and that when we trim a $C_0'$-good curve, what remains will also still be $\delta$-good.

The proof involves four additional parameters
some of which were alluded to above, and which we choose to be sufficiently large that the following hold:

(1)
There is an exponential tail on the wait for the first tempered block. For any $N\ge \lceil AR\rceil$, if $T(\omega,x)$ is the next tempered block after $N$, then 
     \begin{equation}\label{eqn:tempered_time_ptwise_coupled_recovery}
    \mathbb{P}_{\omega}(T(\omega,x)\ge N + i(q+1)\Delta)\le e^{-i\alpha}. 
     \end{equation}
   
   (2) We also fix a small constant $\beta>0$. Then by possibly increasing $\Delta$ even further we can arrange that $\beta<\alpha/7$ and in addition have that $\alpha$ is greater than the cutoffs in Claims~\ref{claim:simultaneous_stopping_claim_1} and \ref{claim:simultaneous_stopping_claim2} below.

    (3)
    We then choose $q$ sufficiently large that
    $e^{-\delta}>e^{-C_1}e^{-.8\lambda q\Delta}$, where $\delta$ is the goodness of the recovered curve from above and depends on $\alpha$ and $\Delta$.

Note that when picking the constants above, from the statement of Proposition \ref{prop:blocked_backwards_good_times_1} we first choose $\Delta$ to make $e^{-\alpha}$ arbitrarily small and both (1) and (2) hold. Then we increase $q$ to ensure that (3) holds as well, which does not affect (1) or (2).

\noindent\textbf{Part 1: Coupled Stopping Procedure.} Fix a word $\omega\in \Sigma$.
We begin with two standard pairs $\hat{\gamma}_1$ and $\hat{\gamma}_2$. We will let $P_n^i$ be the subset of $\hat{\gamma}_i$ that has not been coupled after $n$ attempts at coupled stopping, i.e.~it consists of points that are not permanently stopped at time $\lceil AR\rceil +i(q+1)\Delta$. Note that $P_n^i$ is naturally viewed as a standard family. We let $I^i_j$ be the set of points in $P^i_j$ whose $(j+1)$st block is a tempered block. 
For every point $x\in P^i_j$ its next stopping time  $T(x,\omega)$ is defined to be the end of the next tempered block for that point. To simplify the notation, we write $N_0=\lceil AR\rceil$. 

An inductive assumption of the following procedure is the following:
\begin{align}\label{claim:enough_curve_remains_inductive}
&\text{ For any $\hat\gamma\in P^i_j$, and $x\in \hat\gamma$, $\hat\gamma$ is sufficiently long that if for some $k>j$,} \\
\notag
&\text{ the $k$th block is tempered, then $f^{(q+1)(k-j)\Delta}_{\sigma^{N_0+(q+1)j\Delta}(\omega)}(\hat\gamma)$ is $C_0'$-good.}
\end{align}

For $i\in \{1,2\}$, let $\wt{U}_j^i$ be the union of the $C_0'$ good intervals of the points $x\in I^i_j$ at the end of the $(j+1)$st block; if two intervals within a single standard pair in $P^i_j$ overlap, we take their union, so some intervals may be longer than $e^{-C_0'}$. 
Note that $\wt{U}_j^i$ is a $C_0'$-good standard family. Then for each standard pair $I\in \wt{U}_j^i$, we discard the interval of size $e^{-\delta}$ from the end of the interval. This gives us a new standard family $U_j^i\subseteq \wt{U}_j^i$. By choice of $\delta(\alpha)$ from above,
\[
\rho_i(U_j^i)\ge (1-e^{-\alpha})\rho_i(\wt{U}_j^i). 
\]

We now choose which of the subpairs in $\wt{U}_j^1$ and $\wt{U}_{j}^2$ to stop on for our fixed word $\omega$.
Suppose without loss of generality that $U^{1}_j$ has less mass than $U^{2}_j$. We now stop on all points in $U^{1}_j$. We would like to stop on all the points in $U^{2}_j$, however $U^{2}_j$ has too much mass compared with $U^{1}_j$. To compensate, we subdivide the standard family to create pieces with the appropriate height so that we can stop on a set of equal mass to $U^{1}_j$. 
First we subdivide $\hat{\gamma}_2$ vertically at height $\rho_1(U^{1}_j)(\rho_2(U^{2}_j))^{-1}\rho_2$ so that we keep over each point the same proportion of the mass. Call the two pieces of $\hat{\gamma}_2$ by $A$ and $B$, where $A$ is the piece with mass $\rho_1(U^{1}_j)(\rho_2(U^{2}_j))^{-1}\rho_2(\hat{\gamma}_2)$. Then if we take $A'$ to be the restriction of the standard pair $A$ to the points over $U^{2}_j$, this subpair satisfies that $\rho_2(A')=\rho_1(U^{1}_j)$. We stop on all points in $A'$. The map $\Phi$ in the statement of the proposition associates $A'$ and $U^{1}_j$. The complement of these stopped sets $A'$ and $U^1_j$ then defines a pair of new standard families $
P^i_{j+1}$.

In order for us to be able to proceed with this argument inductively, we must verify that the inductive assumption
\eqref{claim:enough_curve_remains_inductive}  still holds. 
From the second part of Lemma \ref{lem:deterministic_recovery_lemma_3}, as long as $x\in f^{N_0+(j+1)(q+1)\Delta}_{\omega}(\gamma)$ has length at least $e^{-\delta}$, and a point $x$ experiences another tempered time $q\Delta$ iterates later, then by choice of $q$, 
\[
e^{-\delta}>e^{-C_1}e^{-.8\lambda q\Delta},
\]
so by that lemma if there is a future tempered time $n>N_0+(j+1)(q+1)\Delta+q\Delta$, then at that time the image of $x$ will lie in a $C_0$-good pair.
Note that as we only consider future tempered times that are at least $q\Delta$ past the point where the curve is $e^{-\delta}$ long that by our choice of constants and the last part of Lemma \ref{lem:deterministic_recovery_lemma_3} 
 the assumption \eqref{claim:enough_curve_remains_inductive} holds inductively. 

This completes the description of the stopping procedure. We now turn to estimating the tail of the stopping time. \\

\noindent\textbf{Part 2: Rate of Stopping.} Let $A_n^1$ and $A_n^2$ be the pairs $(\omega,x)\subset \Sigma\times \hat{\gamma}_1$ and $\Sigma\times \hat{\gamma}_2$ that have not permanently stopped at time $n(q+1)\Delta$, i.e.~after $n$ attempts at coupled stopping they are still not stopped. Our goal now is to show that $(\mu\otimes \rho_1)(A_n^1)$ has an exponential tail. We begin with several claims. The idea is that if the amount of mass that has not stopped at time $n$ is large, then this implies that a large proportion of points will have a tempered time very quickly. If a large proportion of each curve has a tempered time, then we can stop on these points and obtain the result.  

In this part of the proof, we will write all stopping times as if we had reindexed things so that $N_0=\lceil AR\rceil $ is time $0$, $\lceil AR\rceil +(q+1)\Delta$ is time $1$, etc, to avoid a mess of notation. Keep in mind from our choice of constants earlier that we can pick $\Delta$ as large as we like at the beginning of the proof to ensure that $\alpha$ is as large as we like below.

\begin{claim}\label{claim:simultaneous_stopping_claim_1}
For any $\beta>0$, there exists $\alpha_0\ge 2\beta$ such that for all $\alpha\ge \alpha_0$, if we have chosen the block size $\Delta$ as above to ensure an $e^{-n\alpha}$ tail on tempered times pointwise \eqref{eqn:tempered_time_ptwise_coupled_recovery}, then if for some $n\in \N$ and all $i<n$, $ (\mu\otimes \rho_1)(A_i^1)\le e^{-i\beta}e^{\beta}$ and $e^{-n\beta} \le(\mu\otimes \rho_1)(A_n^1)\le e^{2\beta}e^{-n\beta}$, then at the end of 
the next block,  $1-e^{-\frac{99}{100}\alpha}$ proportion of the points $(\omega,x)$ in $A_n^1$ experience a tempered time. 
\end{claim}

\begin{proof}
Let $T(\omega, x)$ denote the next tempered time for $(\omega, x)\in A_n^1$ then we wish to study a conditional probability 
$\displaystyle
\mathbb{P}(T(\omega,x)> n+1\vert (\omega,x)\in A_n^1),
$
as this gives a bound on the probability that we stop at the next attempt. Then
\begin{align}
\mathbb{P}(T(\omega,x)>n+1\vert (\omega,x)\in A_n^1)=\frac{\mathbb{P}(T(\omega,x)>n+1\text{ and } (\omega,x)\in A_n^1)}{\mathbb{P}(A_n^1)}
\end{align}
Let $B_j^n\subseteq A_n^1$ be the set of trajectories that have not had a tempered time since iterate $j$ and hence are in $A_n^1$ for this reason. Thus $\displaystyle A_n^1=\sqcup_{j=0}^n B_j^n$. 
Note that $B_j^n\subseteq A_j^1$ as these points certainly weren't stopped at time $j$. Hence
\begin{align*}
&\mathbb{P}(T(\omega,x)>n+1\vert (\omega,x)\in A_n^1)=\frac{\sum_{j=0}^n \mathbb{P}(T(\omega,x)>n+1\text{ and } (\omega,x) \in B_j^n)}{\mathbb{P}(A_n^1)}
\\&
\le \frac{\sum_{j=0}^n \mathbb{P}(T(\omega,x)>n+1\text{ and } (\omega,x) \in A_j^1)}{\mathbb{P}(A_n^1)}
\le (\mathbb{P}(A_n^1))^{-1}\sum_{j=0}^n e^{-(n-j+1)\alpha}e^{-\beta j+2\beta}
\quad \text{by } \eqref{eqn:exp_tail_on_local_recovery_time1}\\
&\le e^{2\beta}e^{n(\beta-\alpha)}e^{-\alpha}\sum_{j=0}^n e^{j(\alpha-\beta)}
=e^{2\beta}e^{-\alpha}\sum_{j=0}^n e^{(n-j)(\beta-\alpha)}
=e^{2\beta}e^{-\alpha}\sum_{j=0}^n e^{j(\beta-\alpha)}\\
&\le e^{2\beta} e^{-\alpha}(1+2e^{(\beta-\alpha)})
\le e^{-\frac{99}{100}\alpha}, 
\end{align*}
for $\alpha$ sufficiently large relative to $\beta$. This is the needed claim, so we are done.
\end{proof}

The following claim shows that if most of the remaining pairs $(\omega,x)$ are experiencing a tempered time  at time $n$ then we stop on a relatively large amount of mass at that step.

\begin{claim}\label{claim:simultaneous_stopping_claim2} There exists $\alpha_0$ such that for all $\alpha\ge \alpha_0$, if $B_n^1$ and $B_n^2$ are the subsets of $A_n^1$ and $A_n^2$ having tempered times at time $n+1$ and if for $i\in \{1,2\}$,
\begin{equation}\label{eqn:assump1}
(\mu\otimes \rho_i)(B_n^i)\ge (1-e^{-\alpha})(\mu\otimes \rho_i)(A_n^i),
\end{equation}
then 
\begin{equation}
(\mu\otimes \rho_i)(A_{n+1}^i)\le e^{-\alpha/3}(\mu\otimes \rho_i)(A_n^i).
\end{equation}
\end{claim}
\begin{proof}
Let $\pi\colon \Sigma\times \hat{\gamma}_1\to \Sigma$ denote the projection. Associated to $A_n^1$ and $A_n^2$ we have a measure $\wt{\mu}_n$ on $\Sigma$, given by 
$$\wt{\mu}_n(X)=  (\mu\otimes \rho_1) (\pi^{-1}(X)\cap A_n^1).$$ 
Note that if we had used $A_n^2$ to define $\wt{\mu}_n$, we would have obtained the same result.

Let $A_n^i(\omega)$ denote $\pi^{-1}(\{\omega\})\cap A_n^i$.  
We claim that there is a set $X\subseteq \Sigma$ such that $\wt{\mu}_n(X)\ge (1-e^{-\alpha/2})(\mu\otimes\rho_1)(A^i_n)$ and for $\omega\in X$, we have that \begin{equation}\label{eqn:good_saturated_sets}
\rho_1(A_n^1(\omega)\cap B^1_n)\ge (1-e^{-\alpha/2})\rho_1(A_n^1(\omega)).
\end{equation}
Otherwise there would exist a set $Y$ such that $\wt{\mu}_n(Y)> e^{-\alpha/2}(\mu\otimes \rho_1)(A^1_n)$ such that for $\omega\in Y$, equation \eqref{eqn:good_saturated_sets} fails. Then by Fubini, we would find
\[
(\mu\otimes \rho_1)(B_n^1)\le ((1-\wt{\mu}_n(Y))+\wt{\mu}_n(Y)(1-e^{-\alpha/2}))(\mu\otimes \rho_1)(A_n^1)<(1-e^{-\alpha/2})(\mu\otimes \rho_1)(A_n^1),
\]
which is impossible from our assumption \eqref{eqn:assump1}. 

Thus we may find a set $X_1\subseteq X$ such that $\wt{\mu}_n(X_1)\ge (1-e^{-\alpha/2})(\mu\otimes \rho_1)(A_n^1)$ and for $\omega\in X_1$, \eqref{eqn:good_saturated_sets} holds. Similarly we may find a set $X_2$ such that the same holds for $A_n^2$. 
Then $\wt{\mu}_n(X_1\cap X_2)\ge (1-2e^{-\alpha/2})\wt{\mu}_n(A_n^1)$ and  for every point $\omega\in X_1\cap X_2$, each curve in $A^i_{n}(\omega)$ has at least $1-e^{-\alpha/2}$ proportion of its remaining mass recovering. 
As described in the first part of the proof, we then trim segments of length $e^{-\delta}$ off these subcurves, which by the choice of $\delta$, leaves us with $(1-e^{-\alpha})$ proportion of the remaining mass. 
Thus on each curve there is at least
\[
(1-e^{-\alpha/2})(1-e^{-\alpha})(\mu\otimes\rho_1)(A_n^1(\omega))
\]
mass to stop on.
Hence by the estimate on the measure of such $\omega$, we can stop on 
\[
(1-2e^{-\alpha/2})(1-e^{-\alpha/2})(1-e^{-\alpha})(\mu\otimes\rho_1)(A_n^1)
\]
of the remaining mass.
In particular, this implies that for sufficiently large $\alpha$, that the unstopped mass remaining at the $(n+1)$th step satisfies:
\begin{equation}
(\mu\otimes \rho_1)(A_{n+1}^1)\le e^{-\alpha/3}(\mu\otimes \rho_1)(A_{n}^1),
\end{equation}
as desired.
\end{proof}

We can now conclude the desired rate of stopping. From our choice of constants, we have $\beta>0$ sufficiently small and $\alpha>0$ sufficiently large that $\beta<\alpha/7$ and both Claims \ref{claim:simultaneous_stopping_claim_1} and \ref{claim:simultaneous_stopping_claim2} of the proof hold. As mentioned previously, from the choice of $\Delta$ at the beginning, we may take $\alpha$ as large as we like. Then we will show that for $n\in \N$,
\begin{equation}\label{eqn:simultaneous_stopping_rate}
(\mu\otimes \rho_1)(A_n^1)\le e^{-n\beta}e^{\beta}. 
\end{equation}
We consider two cases depending on how much mass is left at time $n$. \smallskip

(1)
First, suppose that
\begin{equation}
(\mu\otimes \rho_1)(A_n^1)\le e^{-n\beta}
\end{equation}
Then certainly, 
$\displaystyle
(\mu\otimes \rho_1)(A_{n+1}^1)\le e^{\beta}e^{-(n+1)\beta}.
$

(2) If at time $n$, 
\begin{equation}
e^{-n\beta}\le  (\mu\otimes \rho_1)(A_n^1)\le e^{2\beta}e^{-\beta n},
\end{equation}
and at all previous times $(\mu\otimes \rho_1)(A_n^1)\le e^{\beta}e^{-n\beta}$, then Claim \ref{claim:simultaneous_stopping_claim_1} applies to $A_n^1$ and $A_n^2$, which gives that at time $n+1$, that $1-e^{-99/100\alpha }$ proportion of the points in $A^1_n$ and $A^2_n$ will recover at time $n+1$. Thus by Claim \ref{claim:simultaneous_stopping_claim2} and our choice of $\alpha>7\beta$, we see that 
\begin{equation}
(\mu\otimes \rho_i)(A_{n+1}^i)\le e^{-\frac{99}{300}\alpha}(\mu\otimes \rho_i)(A_n^i)<e^{-2\beta}(\mu\otimes \rho_i)(A_n^i),
\end{equation}
and for the next iterate we are back in the first case, $(\mu\otimes \rho_1)(A_{n+1}^1)\le e^{-(n+1)\beta}$.

In order to conclude, we apply the two options above inductively to obtain equation \eqref{eqn:simultaneous_stopping_rate} for all $n$.  In fact, we will show something slightly stronger: there are never two consecutive indices $n,n+1$ such that
\[
e^{-n\beta}<(\mu\otimes \rho_1)(A_n^1)\le e^{-n\beta}e^{\beta}
\]
holds for both $n$ and $n+1$.

Throughout the induction either we have 
\begin{equation}\label{eqn:two_cases}
(\mu\otimes \rho_i)(A_n^i)<e^{-\beta n}\text{ or }e^{-n\beta}\le (\mu\otimes \rho_i)(A_n^i).
\end{equation}
In the former case, we may apply item (1) in the list just mentioned.

Suppose we are in the latter case, that at time $n-1$ that $(\mu\otimes \rho_i)(A_n^i)<e^{-\beta (n-1)}$ and at time $n$ that $e^{-\beta n}\le (\mu\otimes \rho_i)(A_n^i)\le e^{-\beta (n-1)}$, and that for all prior iterates equation  \eqref{eqn:simultaneous_stopping_rate} holds.
Then we may apply (2) above to find that
\begin{equation}
(\mu\otimes \rho_i)(A_{n+1}^i)
<e^{-2\beta}(\mu\otimes \rho_i)(A_n^i)
\le e^{-2\beta}e^{-(n-1)\beta}
=e^{-(n+1)\beta}.
\end{equation}
Thus for the iteration $n+1$ we have $(\mu\otimes \rho_i)(A_{n+1}^i)<e^{-(n+1)\beta}$. Note that this means that the second case in \eqref{eqn:two_cases} cannot occur twice in a row.
Hence we may proceed inductively to verify that \eqref{eqn:simultaneous_stopping_rate} holds for every $n$. This concludes the proof of the lemma. 
\end{proof}

\section{Precoupling}
\label{sec:finite_time_mixing_prop}
In this section, we prove the finite time mixing proposition, Proposition~\ref{prop:finite_time_mixing}, which prepares curves for the application of the local coupling lemma.
\subsection{Fibrewise mixing} 
\label{SSFiberwise}
In this subsection we study fiber-wise mixing properties of the skew product $F\colon \Sigma\times M\to \Sigma \times M$. 
A skew product being mixing does not imply that it has any mixing properties fiberwise. For example, the system could be isometric on the fibers. For this reason we will leverage 
 the mixing of $F_k\colon \Sigma\times M^k\to \Sigma\times M^k$. We will obtain a sort of coarse fiberwise mixing by using a concentration of measure argument. The basic idea of the argument is that if $A$ is a subset of $M$, and $B\subset \Sigma\times M$ is a set giving equal measure to each fiber, then if $B$ does not mix with $A$ fiberwise, then it implies that on many fibers $A\cap F^n(B)$ is quite concentrated. 
 As a consequence of this concentration we show that $F_k$ cannot be mixing as there are too many points that stay in the set $A^k\subset M^k$. 

\begin{prop}\label{prop:fibrewise_mixing}
Suppose that the skew product $F_k\colon \Sigma\times M^k\to  \Sigma\times M^k$ from \eqref{KDiag} is mixing for $\mu\otimes \vol^k$ for all $k\in \N$. Let $A\subseteq M$ be a positive measure set. Then for all $\epsilon_1,\epsilon_2>0$ if $U\subseteq \hat{\Sigma}\times M$ is a set giving exactly mass $\alpha_0>0$ to $(1-\epsilon_2)$ of the fibers of $\hat{\Sigma}$ and $0$ to the rest, then there exists $N\in \N$, such that for all $n\ge N$, there exist $(1-2\epsilon_2)$ proportion of words $\omega\in \hat{\Sigma}$, such that 
\begin{equation}\label{eqn:fibre_mixing_bound}
\vol(A)\alpha_0(1-\epsilon_1)\le \vol(f^n_{\omega}(U_{\omega})\cap A)\le \vol(A)\alpha_0(1+\epsilon_1),
\end{equation}
where we write $U_{\omega}\subseteq M$ for the portion of $U$ in the fibre over $\omega$.
\end{prop}

\begin{proof}
We will prove the lower bound; the upper bound then follows by taking the complement of $A$. For the sake of contradiction, suppose that the lower bound in \eqref{eqn:fibre_mixing_bound} is false. Then there exist $\epsilon_1,\epsilon_2>0$ such that for arbitrarily large $n$, there exist measure $2\epsilon_2$ words $\omega$ such that 
\begin{equation}\label{eqn:full_fibre_bound}
 \vol(U_\omega)=\alpha_0\text{ and }
\vol(f^{n}_{\omega}(U_{\omega})\cap A)<  \vol(A)\alpha_0(1-\epsilon_1).  
\end{equation}
For these words $\omega$
\begin{equation}
    \vol(f^{n}_{\omega}(U_{\omega})\cap (M\setminus A))\ge \alpha_0(\vol(M\setminus A)+\epsilon_1\vol(A)).
    \end{equation}

We now consider what this implies on $\hat{\Sigma}\times M^k$. 
Write $U^k$ for the union of the sets $\{\omega\}\times U_{\omega}^k$. Then for the words $\omega$ satisfying \eqref{eqn:full_fibre_bound}, we obtain
\begin{equation}
(\vol^k)(F^{n}_{k,\omega}(U^k_{\omega})\cap \{\sigma^k(\omega)\}\times (M\setminus A)^k)\ge \alpha_0^k(\vol(M\setminus A)+\epsilon_1\vol(A))^k,
\end{equation}
because fiberwise this intersection is equal to the product $(f^n_{\omega}(U_{\omega})\cap  (M\setminus A))^k$. Thus integrating over this set of  $\omega$ of measure $2\epsilon_2$, we find that 
\begin{equation}\label{eqn:lower_bound_fibre_mixing_k}
(\hat{\mu}\otimes \vol^k)(F^n_k(U^k)\cap \hat{\Sigma}\times (M\setminus A)^k)\ge 2\epsilon_2 \alpha_0^k(\vol(M\setminus A)+\epsilon_1\vol(A))^k. 
\end{equation}
Note that 
$(\hat{\mu}\otimes \vol^k){ (U^k)\le} (1-\epsilon_2)\alpha_0^k$ by the definition of $U$. Since
$(\hat{\mu}\otimes \vol^k)(\hat{\Sigma}\times (M\setminus A)^k)=\vol(M\setminus A)^k$,
 mixing of $F_k$ implies that for sufficiently large $n$,
\begin{equation}\label{eqn:upper_bound_fibre_mixing_k}
(\hat{\mu}\otimes \vol^k)(F^{n}_k(U^k)\cap \hat{\Sigma}\times (M\setminus A)^k)\le (1-\epsilon_2/2)\vol(M\setminus A)^k\alpha_0^k.
\end{equation}
For large $k$ the bounds \eqref{eqn:lower_bound_fibre_mixing_k} and \eqref{eqn:upper_bound_fibre_mixing_k} are incompatible, so we obtain a contradiction. 
\end{proof}

\subsection{Proof of the finite time mixing proposition}

In this subsection we prove the finite time mixing Proposition  \ref{prop:finite_time_mixing}. The idea is straightforward. We can saturate the curve $\hat{\gamma}$ with stable manifolds to embed $\hat{\gamma}$ in a positive measure set that will contract onto the image of $\hat{\gamma}$ forward in time. As the skew product $F\colon \hat{\Sigma}\times M\to \hat{\Sigma}\times M$ is
 fibrewise mixing (Proposition \ref{prop:fibrewise_mixing}), this positive measure thickening of $\hat{\gamma}$ must equidistribute for most words. Simultaneously, we know that most images of $\hat{\gamma}$ will be relatively smooth. This allows us to conclude. 

In the proof we will need some intermediate claims. 

\begin{defn}\label{defn:thickening}
An $\epsilon$\emph{-thickening} of a curve $\gamma$ for a word $\omega\in \Sigma$ consists of two pieces of information. The first piece is a subset $\gamma_0\subset\gamma$ that will be thickened. The second piece is a set of the form
\[
\bigcup_{x\in \gamma_0}W^s_{\epsilon(x)}(\omega,x),
\]
and $W^s_{\epsilon(x)}(\omega,x)$ is the local stable leaf of radius $0<\epsilon(x)<\epsilon$ through $x$.
We will often denote such sets by $\kappa_{\omega}(\gamma)$. 
\end{defn}
Note that although the thickening can in principle be defined over all of $\gamma$, we will usually only use it on a special subset $\gamma_0$ that has better properties.

The following lemma shows that we may choose thickenings of $\gamma$ so that the pushforward of the volume along the thickening to $\gamma$ by the stable holonomy is proportional to $\rho$ on $\gamma_0$. 

\begin{lem}\label{lem:local_thickening}(Local Thickening Lemma)
Fix $\epsilon_1>0$ and $C_0>0$, a level of goodness of standard pairs. For any $\epsilon_2>0$, there exist $\epsilon_3,c_1,C_2, \varrho>0$ such that for $1-\epsilon_2$ of words $\omega\in \Sigma$, and any $C_0$-good standard pair $\hat{\gamma}=(\gamma,\rho)$ of unit mass, we can form an $\epsilon_1$-thickening of $\gamma$, $\kappa_{\omega}(\gamma)$, in the sense of Definition \ref{defn:thickening}, such that:
\begin{enumerate}[leftmargin=*]
\item
Let $\pi^s$ be the projection to $\gamma$ along the stable leaves. Then
    $\displaystyle
    \pi^s_*(\vol\vert_{\kappa_{\omega}(\gamma)})=c_1\rho\vert_{\pi^s(\kappa_{\omega}(\gamma))}
    $
    and 
    $\displaystyle
    \rho(\pi^s(\kappa_{\omega}(\gamma)))>\varrho.
    $

\item
    Every  stable leaf in $\kappa_{\omega}(\gamma)$ is uniformly $(C_2,\lambda,\epsilon_3)$-tempered under forward iterations.
    
\item 
 The choice of thickening $\kappa_{\omega}(\gamma)$ depends measurably on $\omega$. 

 \end{enumerate}
\end{lem}
\begin{proof}
We know that for every point $x$ and almost every word $\omega$, that $x$ is in the Pesin block $\Lambda^{\omega}_{\infty}(C)$ for some sufficiently large $C$, and on a measure one subset, 
$E^s$ is not tangent to $\gamma$. Thus we can saturate a positive measure subset of $\gamma$ with stable manifolds with uniformly controlled geometry by increasing $C$. By taking a shorter subset of the saturating stable curves in such a Pesin block, we can ensure that the volume measure of the saturation projected along the stable leaves to $\gamma$ gives a measure that is proportional to $\rho$ restricted to the images of $\pi^s$.
\end{proof}

The following lemma says that if we start with $C$-good curve, then we can ensure that a large proportion of the images of the curve are $C_0$-good at any time in the future.

\begin{lem}\label{lem:bulk_stays_smoothed}
For any $\epsilon>0$, there exists $C_0$, such that for any $C>0$, a level of goodness, there exists $N_0\ge 0$, such that for any $C$-good standard pair $\hat{\gamma}$ and all $n\ge N_0$, there exists a set $\Sigma_0^n\subseteq\Sigma$ of measure at least $1-\epsilon$, such that for $\omega\in \Sigma_0^n$,
\begin{equation}
\rho(x:f^n_{\omega}(\hat{\gamma})\text{ has a } C_0\text{-good neighborhood of }f^n_{\omega}(x))\ge (1-\epsilon)\rho(\hat{\gamma}).
\end{equation}
The same holds for a $C$-good standard family.

\end{lem}
\begin{proof}
This is immediate from Proposition \ref{prop:exponential_recovery_time_pointwise_annealed_over_curve}, which says that for large enough $\Delta$, we may ensure that $1-\epsilon$ of the pairs $(\omega,x)$ will have a tempered time between times $n+\Delta$ and $n+2\Delta$ for any $n$. We choose $N_0$ large enough that such a tempered time recovers a $C$-good curve to being $D$-good for some uniform $D$. Then we wait until to the end of the block, which gives a further, bounded loss of goodness. As in other places in the paper, a Fubini argument gives the fiberwise estimate stated here. Finally, note that this argument is independent of $n\ge N_0$. 
\end{proof}

We are now ready to prove the finite time mixing proposition. 

\begin{proof}[Proof of Proposition \ref{prop:finite_time_mixing}]
The outline of the proof is as follows. We first find a collection of balls in $M$ that a thickened version of $\gamma_1$ and $\gamma_2$ will mix onto due to the fibered mixing lemma. Then once mixing is accomplished most subcurves of $f^n_{\omega}(\gamma_1)$ and $f^n_{\omega}(\gamma_2)$ will still be long. Consequently, if there are subcurves intersecting a small ball $B_{\upsilon}(x)$ then those subcurves will form a $(C_1,\delta,\upsilon)$ configuration for some $C_1$. To achieve this setup, we will construct subsets $\Sigma_0,\ldots,\Sigma_4$ of $\Sigma$. Each of these sets will consist of words $\omega$ that have some particular finite time mixing properties, so that their intersection has all the properties we need to conclude along the lines just described. We will also have some additional parameters $m_i$ that are chosen below.

The input to this proposition requires some constants. First, let $0<\tau<1$ be the constant from the Local Coupling Lemma, Lemma \ref{ref:small_scale_coupling_lemma}, which says that the conclusions of that lemma hold for $(C,\delta,\tau\delta)$-configurations for any $C$ as long as $\delta$ is sufficiently small relative to $C$. We then obtain the following claim---note that this holds for all  sufficiently small $\delta$ with a uniform lower bound in the last term. 

\begin{claim}\label{claim:delta_covering}
There exists $m_0$ such that for all sufficiently small $\delta>0$, we can find a family of disjoint balls $B_i=B_{\delta_i}(x_i)$ in $M$ such that:

\noindent
   (1)  Each $B_i$ has equal volume between $(1/10)\delta^2$ and $10\delta^2$;
   
\noindent   (2)  Each $B_i$ contains a ball $B_i'$ of diameter at most $\tau \delta/2$ so that $d(\partial B_i',\partial B_i)>\delta/2$;
  
   \noindent (3)
    Each $B_i'$ contains a ball $B_i''$ with the same center and radius between $\tau \delta/9$ and $\tau \delta/10$, and the balls $B_i''$ all have equal volume;
  
    \noindent (4)
    $\displaystyle\vol(\bigcup_i B_i'')\ge 10^{-m_0}$.
\end{claim}

We now pick $\Sigma_0$, which are words where $\gamma_1$ and $\gamma_2$ have good thickenings. Both $\gamma_1$ and $\gamma_2$ are $C_0$-good by assumption. 
Then for any $m_1\in \N$, which we will pick later,
we see that there exists $c_1$, which is distinct from $C_1$, and $\varrho$ such that there is a set $\Sigma_0\subseteq \Sigma$, such that $\mu(\Sigma_0)>1-10^{-m_1}$ such that for $\omega\in \Sigma_0$,    there exists a thickening $\kappa_{\omega}(\gamma_i)$, $i\in \{1,2\}$ satisfying the properties of Lemma \ref{lem:local_thickening}. By possibly shrinking the thickening we may make the thickenings each have the same identical mass $\varrho$. For 
the words in $\Sigma_0$, we form a set $U^1\subseteq \Sigma\times M$ by taking the union of the sets $\{\omega\}\times \kappa_{\omega}(\gamma_1)$, similarly we define $U^2$. We denote by
$U^1_{\omega}$ the part of $U^1$ above $\omega$ and use a similar notation for
$U^2_{\omega}$. 

We now choose $C_1$, the regularity of the pairs that will be in $(C_1,\delta,\tau\delta)$-configurations in the conclusion of the proposition, as well as $\Sigma^n_1$ and $\Sigma^n_2$, words where most images of $\hat{\gamma}_1$ and $\hat{\gamma}_2$ are $C_1$-good curves.
Choose $C_1>0$ such that the conclusion of Lemma \ref{lem:bulk_stays_smoothed} holds for a set $\Sigma_1^n$ of words of measure $(1-10^{-m_2})$, for some $m_2$ that we will choose later, so that for $ \omega \in \Sigma_1^n$,
\begin{equation}\label{eqn:most_pts_good_nbd}
\rho_1(x:f^n_{\omega}(\hat{\gamma}_1)\text{ has a } C_1\text{-good neighborhood of }f^n_{\omega}(x))\ge (1-10^{-m_2})\rho_1(\hat{\gamma}_1).
\end{equation}
For all $\epsilon>0$, there exists $D(\epsilon)>0$ such that  for a $C_0$-good standard pair $\hat{\gamma}=(\gamma,\rho)$, the measure of the points $x\in \hat{\gamma}$ such that 
\begin{equation}
\rho(x\in \gamma:d(x,\partial \hat{\gamma})<D)<\epsilon\rho(\gamma).    
\end{equation}
Recalling Definition \ref{defn:C_0_delta_configuration}, the previous equation implies that there exists $D_1>0$ such that we may strengthen the conclusion in equation \eqref{eqn:most_pts_good_nbd} above:
\begin{equation}\label{eqn:well_positioned_est1}
\rho_1(x:f^n_{\omega}(x)\text{ is } (C_1,D_1)\text{-well positioned in }f^n_{\omega}(\hat{\gamma}_1))\ge (1-10^{-(m_2-1)})\rho_1(\hat{\gamma}_1).
\end{equation}
Call this set of $(C_1,D_1)$-well positioned points $G_{n,\omega}^1$. Similarly, for $\hat{\gamma}_2$ there exists a set $\Sigma_2^n$ and a set $G_{n,\omega}^2$ with this same property. 

Take a covering $B_i$ as in  Claim \ref{claim:delta_covering}
applied with the parameter $\delta$ small enough that the local coupling lemma holds for $(C_1,\delta,\tau\delta)$-configurations. Let $m_1=m_2=m_3=m_4=20$.

Next we choose $\Sigma_3^n$ and $\Sigma_4^n$, which are sets that mix the thickenings of $\hat{\gamma}_1$ and $\hat{\gamma_2}$ onto the balls $B_i''$. Let $\epsilon_2=10^{-m_3}$ from above,
and let $0<\epsilon_1<10^{-m_3}$. Then by the fibrewise mixing proposition (Proposition~\ref{prop:fibrewise_mixing}), there exists $N_1$ such that for $n\ge N_1$, there is a set $\Sigma_3^n$ of $\omega\subseteq \Sigma$ of $\mu$-measure $1-2\cdot 10^{-m_3}$ such that for $\omega\in \Sigma_3^n$, $U_{\omega}^1$ mixes onto the $B_i''$ for each $B_i''$ in the covering, i.e. for $\omega\in \Sigma_3^n$, 
\begin{equation}\label{eqn:fiber_wise_mixing_est_3}
(1-10^{-m_3})\varrho\vol(B_i'') \le \vol(f^n_{\omega}(U_{\omega}^1)\cap \{\omega\}\times B_i'')\le \vol(B_i'')\varrho(1+10^{-m_3}).
\end{equation}
Similarly we have a cutoff $N_2$, and sets $\Sigma_4^n$ such that the same holds for $U_2$. We will strengthen this estimate even further,
we will let $B_i'''$ be a ball with the same center as $B_i''$ but with slightly larger radius so that the ratio of the volumes $\vol(B_i''')/\vol(B_i'')=1+10^{-m_4}$. Then by possibly enlarging the numbers $N_1$ and $N_2$, we can arrange that the same estimate holds simultaneously for the sets $B_i'''$ as well. 

Now consider what happens for $\omega\in \Sigma_0\cap \Sigma_1^n\cap \Sigma_3^n$. These are words $\omega$ where $\gamma_1$ has a good thickening by stable manifolds, and many of the points in the image of $\hat{\gamma}_1$ are good standard pairs and there is equidistribution. For any $m_4$ as long as $n$ is sufficiently large, the diameter of the image of any $W^s_{\epsilon(x)}(\omega,x)$ leaf in the thickening of $\hat{\gamma}_1$ is at most $\tau\delta/10^{2m_4}$. Thus from the measure preservation of the projection $\pi^s$ of $f^n_{\omega}(U_{\omega})$ onto $f^n_{\omega}(\hat{\gamma}_1)$, we see that if some point $x\in f^n_{\omega}(U_{\omega}^1)$ is in $B_i''$, then as $B_i'''$ contains a neighborhood of $B_i''$ of radius $\tau\delta/10^{m_4}$,  all points on $f^n_{\omega}(W^s_{\epsilon(x)}(\omega,x))$ and, in particular, the points of $\hat{\gamma}$ lie in this set. Hence, writing  $\rho^{n,\omega}_1$ for the density on $f^n_{\omega}(\hat{\gamma}_1)$, 
\begin{equation}
\rho_1^{n,\omega}(f^n_{\omega}(\hat{\gamma}_1)\cap B_i''')\ge \vol(B_i'')(1-10^{-m_3})\rho_1(\hat{\gamma}_1).
\end{equation}

We claim that for such $\omega\in \Sigma_0\cap \Sigma_1^n\cap \Sigma_3^n$ that there exists a subfamily of the $B_i''$ containing at least $90\%$ of the $B_i''$, and such that for each of these $B_i''$,
\begin{equation}
\rho_1^{n,\omega}(G_{n,\omega}^1\cap B_i''')\ge \vol(B_i''')\rho_1(\hat{\gamma})/2.
\end{equation}
Suppose that this were not the case, then for such an $\omega$ there is a set of $10\%$ of the balls $B_i'''$ such that for these balls we have $\rho_1^{n,\omega}(G_{n,\omega}^1\cap B_i''')<\vol(B_i''')\rho_1(\hat{\gamma})/2$. 
Then, from \eqref{eqn:well_positioned_est1} and the fibrewise mixing estimate \eqref{eqn:fiber_wise_mixing_est_3},
$$
\vol(f^n_{\omega}(U_{\omega})\cap \bigcup_i B_i'')\le \frac{\rho_1^{n,\omega}(G_{n,\omega}^{1}\cap \bigcup_i B_i''')\varrho}{\rho_1(\hat{\gamma})}+10^{-(m_2-1)}\varrho $$
$$ \le .1\sum_{i } \vol(B_i''')\varrho\frac{1}{2}+.9\sum_{i}\vol(B_i''')\varrho(1+10^{-m_3})+10^{-(m_2-1)}\varrho $$
$$\le .96\sum_i \vol(B_i''')\varrho 
\le .96(1+\frac{1}{10^{m_4-1}})\varrho\sum_i\vol(B_i'')
$$
which contradicts fiberwise mixing of $U^1$ with the set $\bigcup_i B_i'''$. 

Now consider $\omega\in \Sigma_0\cap \Sigma_1^n\cap \Sigma_2^n\cap \Sigma_3^n\cap \Sigma_4^n$. We have that for $90\%$ of the balls $B_i'''$, that $B_i'''$ has radius at most $\tau\delta/8$, and this ball contains points of $f^n_{\omega}(\hat{\gamma}_1)$ that are $(C_1,D_1)$-well centered of measure at least $\rho_1(\hat{\gamma}_1)\tau^2\delta^2/200$. The same holds for $\hat{\gamma}_2$ for a possibly different $90\%$ of balls. Thus for $80\%$ of the balls $B_i'''$ each of $f^n_{\omega}(\hat{\gamma}_1)$ and $f^n_{\omega}(\hat{\gamma}_1)$ contains measure $\rho_1(\hat{\gamma}_1)\tau^2\delta^2/200$ points that are $(C_1,D_1)$-well centered. As these points are in a ball of radius $\tau\delta/8$. From our choice of $\delta$, it follows that any pair of such images is $(C_1,\delta,\tau\delta)$-configured. Thus the needed conclusion follows by possibly subdividing the standard pairs we have identified so that they may be coupled in a measure preserving way.
We may now conclude because

$\displaystyle \mu(\Sigma_0\cap \Sigma_1^n\cap \Sigma_2^n\cap \Sigma_3^n\cap \Sigma_4^n)\ge 1-10^{-m_1}-2\cdot 10^{-m_2}-4\cdot 10^{-m_3}\ge 1-10^{-19}.
$
\end{proof}

\section{Proof of the Local Coupling Lemma}
\label{ScLocalCoupling}

\subsection{Inductive local coupling procedure}\label{subsec:local_coupling_procedure}
To prove the Local Coupling Lemma~\ref{ref:small_scale_coupling_lemma}, we would like two positive measure sets to be intertwined under the true stable holonomy. 
However, at any finite time, we do not yet know what the true limiting stable manifold is.
To compensate, at finite times we approximate the limiting holonomy by using the fake stable manifolds. 
In the proof of the local coupling lemma, we will consider the differences between different standard families as discussed in  \S \ref{sec:standard_pairs}.

To begin this section, we introduce a notion of a ``fake coupling" of two standard pairs $\hat{\gamma}_1$ and $\hat{\gamma}_2$. 
We use fake couplings because in our setting we cannot use the stable manifold as is done in the deterministic setting. 
In the deterministic setting, if $\hat{\gamma}_1$ and $\hat{\gamma}_2$ are near each other, then we can immediately determine which points in $\hat{\gamma}_1$ attract to which in $\hat{\gamma}_2$ by using the stable holonomy.
We work in an opposite manner: at each time $n$ we discard points that cannot couple yet. 
For example, if $y\in \hat{\gamma}_2$ and none of the time $n$ fake stable manifolds come near $y$, then $y$ can't couple because the true stable manifold is near the fake one.
Consequently, we stop trying to couple $y$ at time $n$.
After we see the dynamics for all time, the points that remain in $\hat{\gamma}_1$ and $\hat{\gamma}_2$ are those that can be coupled with each other using the stable manifold. 
Hence after the fact, we see that they were coupled. 
The fake coupling is not a coupling. A time $n$ fake-coupling is a pair of subfamilies $P_n^1\subseteq \hat{\gamma}_1$ and $P_n^2\subseteq \hat{\gamma}_2$ that could \emph{potentially} be coupled by the true stable manifolds. For a time $n$ fake coupling, we insist that the holonomies of the time $n$ fake stable manifolds carry $P_n^1$ to $P_n^2$. Another way to describe  this is that $P_n^1$ and $P_n^2$ seem  coupled until time $n$.

The definition of a fake coupling that follows that is adapted to the neighborhood $B_{\delta_0}(x)$ from Proposition \ref{prop:set_up_scale_prop} and relies on the constants obtained in that proposition. Fake stable manifolds $W^s_n$ and their properties are discussed in detail in Appendix \ref{sec:finite_time_pesin_theory}.

\begin{defn}\label{defn:fake_coupling}
Suppose that $\hat{\gamma}_1$ and $\hat{\gamma}_2$ are two standard pairs that we are attempting to couple that are $(C,\delta',\upsilon)$-configured where $C,\delta,\upsilon$ are parameters as in Proposition \ref{prop:set_up_scale_prop}. 
Fix some $x$ and neighborhood $B_{\delta_0}(x)$, $\mc{C}_{\theta}$ as in part \ref{item:there_exists_well_cfd_nbd} of that Proposition. We will use the other constants from that proposition as well without reintroducing them. 

For $n\ge N$, we say that $P^1_n\subseteq \hat{\gamma}_1$ and  $P^2_n\subseteq\hat{\gamma}_2$ are a 
$(b_0,\hat{\eta})$-\emph{fake coupled pair at time $n\ge N$} for some word $\omega$ on $B_{\delta_0}(x)$ if the following statements hold. Write $\rho_n^1$ and $\rho_n^2$ for the densities of $P^1_n$ and $P^2_n$ on $\gamma_1$ and $\gamma_2$. Let $\mc{I}^1_n$ and $\mc{I}^2_n$ be the underlying curves of $P^1_n$ and $P^2_n$.
\begin{enumerate}
    \item $P^1_n$ and $P^2_n$ have equal mass and $(H^s_{n-1})_*(\rho^1_n)=\rho^2_n$. 

    \item 
    $H^s_{n-1}$ carries $\mc{I}^1_n$ to $\mc{I}^2_n$.
   \item\label{item:P_1_well_tempered_points_padded}
    If $x\in \gamma_1$ is  $(C,\lambda,\epsilon,\mc{C}_{\theta})$-tempered for times $N\le i\le n$, 
    then $x\in \mc{I}_n^1$. 
    \item 
    At each point $x$ in the curve underlying $P^1_n$, we have that 
    \[
    \rho^1_n(x)\ge b_0\prod_{N\le i\le n}(1-e^{-i\hat{\eta}})\rho^1(x).
    \]
\end{enumerate}
\end{defn}
We will see below that
if for a given word $\omega$ we are able to arrange that the statements above hold for each $n$, then in the limit, for each point $x\in \gamma_1$ that is $(C,\lambda,\epsilon)$-tempered and in each $P^1_n$ that at least $\epsilon_0$ of the mass above $x$ in $\hat{\gamma}_1$ couples. 
Thus as typically a positive measure set of $x$ have this property, a positive proportion of the mass of $P^1_n$ couples. 

The structure of the rest of this section is as follows.  In \S\ref{subsec:nearby_points_tempered_splitting} and
\S\ref{SSCushion} we show that if a trajectory has a tempered splitting then nearby trajectories also have tempered splittings. In 
\S \ref{subsec:scale_selection} we prove Proposition \ref{prop:set_up_scale_prop} which shows how small a scale we need to work at in order to run a coupling procedure. Then in 
\S \ref{subsec:inductive_local_coupling_procedure} we prove the local coupling lemma in two steps. First, we prove Lemma \ref{lem:local_coupling_lemma}, which describes a deterministic local coupling procedure that can be applied to a fixed word $\omega$ under the choice of constants provided by Proposition \ref{prop:set_up_scale_prop}. We then finish the proof of Lemma \ref{ref:small_scale_coupling_lemma} by using that the hypotheses of this deterministic local coupling procedure are satisfied with high probability.

\subsection{Nearby points inherit tempered splitting}\label{subsec:nearby_points_tempered_splitting} In this subsubsection we prove Proposition \ref{prop:nearby_points_inherit_temperedness}, which says that nearby trajectories inherit splittings from each other. This will be used later to show that the set of points on a curve that have a tempered splitting after $n$ iterations is quite fat.
The idea that points close to hyperbolic orbits inherit hyperbolicity is useful in many problems in dynamics. For example, a
classical Collet--Eckmann condition is used in one dimensional dynamics to show that near critical orbits recover hyperbolicity
if the critical orbit is hyperbolic (see \cite{ColletEckmann}). Analogous results for two dimensional strongly dissipative maps
appear in \cite{BenedicksCarleson, WangYoung01}. In this paper we present a  version for general two dimensional maps based 
on Pesin theory.

We begin with a fact showing how far attracting and repelling directions of a linear map of $\RP^1$ move under perturbation.

\begin{lem}\label{lem:linearized_perturbation_RP1}
Fix some $\lambda>1$,  then there exists $C,\epsilon_0,\delta_0,N_0>0$, such that if $L\colon \R^2\to \R^2$ is a linear map of the form 
\begin{equation}
\begin{bmatrix}
\sigma_1 & 0 \\
0 & \sigma_2
\end{bmatrix}
\end{equation}
 with $\abs{\sigma_1},\abs{\sigma}_2^{-1}\ge \abs{\lambda}>1$,
$g_0\colon \RP^1\to \RP^1$ is the induced map, and $g_{\epsilon}$ is a perturbation with $d_{C^1}(g_0,g_{\epsilon})=\epsilon<\epsilon_0$, then:

    (1)
    $g_{\epsilon}$ has a unique repelling fixed point $r_{\epsilon}$ and a unique attracting fixed point $a_{\epsilon}$, and these satisfy
    $\displaystyle d(r_{\epsilon},(0,1))\le C\epsilon,$  
    $\displaystyle d(a_{\epsilon},(1,0))\le C\epsilon.$
   
    (2)
    On the neighborhood $B_{\delta_0}((0,1))$, $\|Dg_{\epsilon}\|\ge \lambda-C\epsilon$ and on the neighborhood $B_{\delta_0}((1,0))$, $\|Dg_{\epsilon}\|\le \lambda^{-1}+C\epsilon$. These neighborhoods are overflowing and under-flowing, respectively.
     
    (3) If $y\notin B_{\delta_0}((0,1))$, then $g^{N_0}_{\epsilon}(y)\in B_{\delta_0}((1,0))$. 

\end{lem}
We omit the proof of the above lemma as these are standard facts about the dynamics in a neighborhood of a hyperbolic 
fixed point. 
The proof of the next result is long and relies on a number of intermediate lemmas.

\begin{prop}\label{prop:nearby_points_inherit_temperedness}
(Nearby points inherit temperedness)
Fix $C_0,C_1,\lambda,\alpha,\epsilon_0,D_0,\sigma>0$ and $0<\lambda'<\lambda$. Then for sufficiently small $\epsilon>0$ there exist $\nu,k,D_1,N>0$ such that $k\epsilon<\epsilon_0$ and if we have a sequence of matrices of length $n\ge N$ $(A_i)_{1\le i\le n}\in \SL(2,\R)$ that are uniformly bounded in norm by $D_0$
and are $(C_0,\lambda,\epsilon)$-tempered,  and $(B_i)_{1\le i\le n}$ is another sequence of matrices such that $\|A_i-B_i\|\le C_1e^{-\alpha(n-i)}$ then:
\begin{enumerate}
    \item $B_i$ has a $( D_1C_0,\lambda',k\epsilon)$-subtempered splitting with the stable direction equal to the contracting singular direction of $B^n$, and
    \item
    The angle between $(B_i)_{i=1}^n$ and $(A_i)_{i=1}^n$'s stable directions is at most $e^{-\nu n}$.
    \item 
    $\displaystyle
    \|B^n\|\ge \|A^n\|^{(1-\sigma)}.
 $
\end{enumerate}

\end{prop}
\begin{proof}
Before we begin, observe that due to the presence of the factor $D_1$ in the conclusion, it suffices to show that the needed claim holds for $n$ sufficiently large as we may always deal with small $n$ by adjusting $D_1$. 
Let $\hat{\lambda}=\frac{\lambda+\lambda'}{2}$. 

As long as $\epsilon_0<(\lambda-\lambda')/2$, we may view the sequence of matrices $A_i$ in the finite time Lyapunov charts from Lemma \ref{lem:lyapunov_metric},
where we view the sequence as being $(C_0,\hat{\lambda},\epsilon+\frac{\lambda-\lambda'}{2})$-tempered. In these charts, we have: 
$\displaystyle
A_i=\begin{bmatrix}
\sigma_{1,i} & 0 \\
0 & \sigma_{2,i}
\end{bmatrix},
$
where $\min\{\sigma_{1,i},\sigma_{2,i}^{-1}\}\ge e^{\hat{\lambda}}$. From Lemma \ref{lem:lyapunov_metric}, the ratio of the reference norm and the Lyapunov norm at step $i$ is $O_{C_1,\alpha,\lambda,\lambda'}(e^{4\epsilon i})$. 

As $B_i$ is a perturbation of size $e^{-\alpha(n-i)}$ by viewing $B_i$ in the same Lyapunov coordinates as $A_i$, we have that 
\begin{equation}\label{eqn:lyapunov_chart_B_i_expression}
B_i=\begin{bmatrix}
\sigma_{1,i}&  0\\
0 & \sigma_{2,i}
\end{bmatrix}+O_{C_1,\alpha,\lambda,\lambda'}(e^{-\alpha (n-i)}e^{4\epsilon i}),
\end{equation}
where $\min\{\sigma_{1,i},\sigma_{2,i}^{-1}\}\ge e^{\hat{\lambda}}$. 

Using this representation, we will now study $B_i$ as a perturbation of the matrix product involving the $A_i$.  For most $i$, the two are quite close and consequently $B_i$ will inherit temperedness of its norm. The remaining $i$ will be negligible. To show this, we first identify where the stable direction of $B^n$ lies. Then using this we show that the norm of $B^i$ is subtempered up to a particular time. Then we do a little bookkeeping to show that if we relax the subtemperedness condition, then norm will remain subtempered up to time $n$.

First we study how temperedness changes as we continue appending matrices to a sequence.

\begin{lem}\label{lem:bad_tail_still_tempered}
Fix some bound $e^{\Delta}>1$. Suppose that $A_1,\ldots,A_n$ is a sequence of matrices whose splitting into singular directions is $(C,\lambda,\epsilon)$-tempered. Then for any $k$ and sequence $B_1,\ldots,B_m$ with $\|B_i\|\le e^{\Delta}$ and $m<\Delta^{-1}(nk\epsilon-C)$, the sequence $A_1,\ldots,A_n,B_1,\ldots,B_m$ is $(C,\lambda-k\epsilon,k\epsilon)$-tempered.
\end{lem}
\begin{proof}
A straightforward generalization of Lemma \ref{lem:concat_tempered_additive_est} gives that if we have a sequence of matrices $A_1,\ldots,A_n$ with $(C,\lambda,\epsilon)$-tempered norm and we append  a sequence $B_1,\ldots,B_m$ that is $(-\Delta m,\lambda-\epsilon,\epsilon)$-tempered, then the concatenation is a $(\tilde C,\lambda-k\epsilon,k\epsilon)$ tempered sequence with
\begin{equation}
 \tilde C= \min\{C,C-m\Delta+n k\epsilon/2,-m\Delta+nk\epsilon\}. 
\end{equation}
Thus the needed conclusion holds as long as 
$\displaystyle
m\le \frac{nk\epsilon-C}{\Delta}.
$
\end{proof}

The following lemma gives tight control on where $v^s$, the most contracted vector for the sequence $(B_i)_{i=1}^n$ lies. Below we will write $g_{i,\epsilon}$ for the map on $\RP^1$ induced by $B_i$, viewed in the Lyapunov coordinates above. We write $g^i_{\epsilon}$ for the composition $g_{i,\epsilon}\circ \cdots\circ g_{1,\epsilon}$. 

\begin{lem}\label{lem:nearness_of_unstable_directions_lemma}
For all $C_0,C_1,\alpha,\lambda,\lambda',D_0>0$ as above and all sufficiently small $\epsilon>0$, there exists $\nu>0$ and $N_s\in \N$ such that if $n\ge N_s$ and $(B_i)_{i=1}^n$ is a sequence of matrices as above, a perturbation of $(A_i)_{i=1}^n$, a sequence of matrices with a $(C_0,\lambda,\epsilon)$-subtempered splitting, then the most contracted direction of $B^n$, $v^s_B$, lies within a neighborhood of size $e^{-n\nu}$ of the most contracted direction of $A^n$.
\end{lem}
\begin{proof}
We will use the perturbed dynamics $g_{i,\epsilon}$ on $\RP^1$ from above and  prove this result by studying how fast a vector near the vector $(0,1)$ escapes and goes to $(1,0)$. We will use the estimates of Lemma \ref{lem:linearized_perturbation_RP1} freely and not restate them here. Given $\delta_0>0$ in the conclusion of that lemma, we see that as long as the size of the perturbation is at most some $\epsilon_{\delta_0}$, then  
on the neighborhoods of size $\delta_0$ of $(0,1)$ the expansion is by a factor of at least $e^{.9 \hat{\lambda}}$ and similarly in the $\delta_0$-neighborhood of $a_{\epsilon}$, the contraction of distance is by a factor of $e^{-.9 \hat{\lambda}}$. As long as $\epsilon$ is sufficiently small relative to $\alpha$ and $i\le \frac{99}{100}n$, then $g_{i,\epsilon}$ is a perturbation of size less than $\epsilon_{\delta_0}$ and the estimate for the norm of $g_{i,\epsilon}$ on $B_{\delta_0}((1,0))$ and $B_{\delta_0}((0,1))$ holds.

Next, we study the norm growth of $v$ over its entire trajectory. Define
$\Phi^i_{\epsilon}\colon \RP^1\to \R^+$ by
\begin{equation}
\Phi^1_{\epsilon}(v)=\ln \frac{\|B_iv\|}{\|v\|}.
\end{equation}
Then  $\|B^iv\|$ is the sum of $\Phi^i_{\epsilon}$ along the trajectory of $v$.  We divide the trajectory of $v$ into three segments. The first segment is when $v$ is does not yet lie in $B_{\delta_0}((1,0))$. The middle segment is when it lies in $B_{\delta_0}((1,0))$ and $B_i$ remains a small enough perturbation of $A_i$ that we may use the approximations of Lemma \ref{lem:linearized_perturbation_RP1}. Finally, during the last part of the trajectory $i$ is so big that these estimates no longer hold. We will let $1< n_1<n_2<n$ denote the indices where $g^i_{\epsilon}(v)$ first enters $B_{\delta_0}((1,0))$ and $n_2$ the index where the approximations of Lemma \ref{lem:linearized_perturbation_RP1} first cease to hold. We now proceed to estimate how large $n_1$ and $n_2$ are. Then using this information we will calculate $\|B^nv\|$. 

By estimating in this manner, we will see that any vector that starts at distance more than $e^{-n\nu}$ from $(0,1)$ cannot be a stable vector as its norm grows. Below, we will track the estimates for $B_i$, the same apply to $A_i$. Consequently, we see that the stable vector for both $A_i$ and $B_i$ must lie within distance $e^{-n\nu}$ of $(1,0)$ for some sufficiently large $\nu$.

We now estimate $n_1$, i.e.~we study how long it takes a vector $v$ near $(0,1)$ to leave $B_{\delta_0}((0,1))$. 
We claim that if $\nu$ is sufficiently small then for sufficiently large $n$,
any vector $v$ that starts $e^{-n\nu}$ away from $(0,1)$  will exit $B_{\delta_0}((0,1))$
 after at most $(2\nu/\lambda) n$ iterates. 
To this end consider 
$$
d(g_{\epsilon}^i(v),(0,1))\ge d(g^i_{\epsilon}(v),g_{i,\epsilon}(0,1))-d(g_{i,\epsilon}((0,1)),(0,1))
$$$$\ge e^{.9\hat{\lambda}}d(g^{i-1}_{\epsilon}(v),(0,1))\!\!\pez{1\!-\!\frac{C_1e^{-n\alpha} e^{i(\alpha+4\epsilon)}}{e^{.9\hat{\lambda}}d(g^{i-1}_{\epsilon}(v),(0,1))}}
$$
As long as $i\le (1/3)n$, $\epsilon<\alpha/100$, and $n$ is sufficiently large, 
\begin{equation}\label{eqn:size_of_error_B_i1}
C_1e^{-n\alpha} e^{i(\alpha+4\epsilon)}\le e^{-\frac{\alpha}{2}n}.
\end{equation}
Thus if $\nu<\alpha/2$, then for sufficiently large $n$, if $d(g^{i-1}_{\epsilon}(v),(0,1))\ge e^{-n\nu}$, then 
\begin{equation}
\pez{1-\frac{C_1e^{-n\alpha} e^{i(\alpha+4\epsilon)}}{e^{.9\hat{\lambda}}d(g^{i-1}_{\epsilon}(v),(0,1))}}\ge e^{-.1\hat{\lambda}}. 
\end{equation}

From the above, we see that as long as $n$ is sufficiently large, $i\le n/3$, and the trajectory of $v$ has not left the $B_{\delta_0}((0,1))$ after $i$, iterates, then 
\begin{equation}
d(g_{\epsilon}^i(v),(0,1))\ge e^{.8\hat{\lambda}}d(g^{i-1}_{\epsilon}((0,1)),(0,1)).
\end{equation}
Proceeding iteratively, we see that after $i$ iterations, assuming $i\le n/3$ and that the trajectory of $v$ has not left $B_{\delta_0}((0,1))$,  
\begin{equation}
    d(g_{\epsilon}^i(v),(0,1))\ge e^{.8\hat{\lambda}i}d(v,(0,1)).
\end{equation}

In particular, if $g_{\epsilon}^i(v)$ has not left $B_{\delta_0}((0,1))$ after 
$ (2\nu/.8\hat{\lambda}) n$, iterates then we would have that 
$\displaystyle
d(g_{\epsilon}^i(v),(0,1))\ge e^{\nu n},
$
which is absurd.

Thus as long as $\epsilon<\alpha/10$ it follows  for sufficiently large $n$ that $g^i_{\epsilon}(v)$ exits $B_{\delta_0}((0,1))$ after at most $\frac{2\nu}{.8\hat{\lambda}}n$ steps. Moreover, by Lemma \ref{lem:linearized_perturbation_RP1},
 it enters the neighborhood $B_{\delta_0}((1,0))$ after an additional $N_0$ iterates. Thus for sufficiently large $n$, $n_1\le \frac{2\nu}{.79\hat{\lambda}}n$.

We now estimate $n_2$. 
In the Lyapunov charts, $B_i$ is a perturbation of $A_i$ of size $e^{-\alpha(n-i)}e^{4\epsilon i}$.
Lemma \ref{lem:linearized_perturbation_RP1} ceases to hold when the size of the perturbation is size $O_{\epsilon_0}(1)$. 
This will occur when $e^{-\alpha(n-i)}e^{4\epsilon i}$ is order $1$, which happens when $i\approx \frac{\alpha}{\alpha+4\epsilon}n$. 
If $\epsilon$ is sufficiently small relative to $\alpha$, then $\alpha/(\alpha+4\epsilon)\ge 1-8\epsilon/\alpha$. 
Hence by picking some $N_2'$ depending only on $\epsilon_0,\delta_0$ and $C_1$, we see that $n_2$ may be chosen to be the smallest number satisfying
$n_2\ge (1-8\frac{\epsilon}{\alpha})n-N_2'$. Hence for
sufficiently large $n$ we can take the bound $n_2\ge  (1-9\frac{\epsilon}{\alpha})n$.

Thus between times $n_1$ and $n_2$ there are at least $(1-9\frac{\epsilon}{\alpha}-\frac{2\nu}{.79\hat{\lambda}})n$ iterates. As long as $n$ is sufficiently large and
\begin{equation}
\left(1-9\frac{\epsilon}{\alpha}-\frac{2\nu}{.79\hat{\lambda}}\right)>\frac{1}{2}
\end{equation}
which we can certainly arrange if we take $\epsilon,\nu$ sufficiently small,  we see that there are at least $n/2$ iterates between $n_1$ and $n_2$.

We now estimate $\|B^nv\|$. Let us first consider the norm $\|B^{n_2}v\|$ by estimating in the Lyapunov metric. Let $v^i$ equal $g^i_{\epsilon}(v)$. 
Then, for $i\le n_1$ and $n$ sufficiently large, 
using \eqref{eqn:lyapunov_chart_B_i_expression} and \eqref{eqn:size_of_error_B_i1} 
and the inequality $e^X+Y\leq e^{X+Y}$, valid for $X, Y\geq 0$, we obtain\\
$\displaystyle \|B_i\|'\le e^{\Delta}+e^{-\frac{\alpha}{2} n}\le e^{\Delta+e^{-\frac{\alpha}{2}n}}$. 
Taking logarithms we get  $\displaystyle\ln \|B_i\|'\le \Delta+e^{-\frac{\alpha}{2}n}$.
Thus,
\begin{align*}
\ln \|B^{n_2}v\|'&\ge \sum_{i=n_1}^{n_2}\Phi^i_{\epsilon}(v^i)+\sum_{i=0}^{n_1} \Phi_{\epsilon}^i (v^i)
\ge (n_2-n_1).8\hat{\lambda}-\pez{n\frac{2\nu}{.79\hat{\lambda}}}(\Delta+e^{-\frac{\alpha}{2}n}).
\end{align*}
This is the amount of growth in the Lyapunov coordinates. For the original metric, by 
Lemma~\ref{lem:lyapunov_metric}(3) this implies from our bounds on $n_1$ and $n_2$, that 
\begin{equation}
    \ln\|B^{n_2}v\|\ge (n_2-n_1).8\hat{\lambda}-n\frac{2\nu}{.79\hat{\lambda}}(\Delta+e^{-\frac{\alpha}{2}n})
    -4 n_2\epsilon.
\end{equation}
Since  $\ln \|B_i\| \leq \Delta$, and because $n-n_2\le  (9\epsilon/\alpha)n$, and $n_2-n_1>n/2$, we see that
\begin{equation}
\ln \|B^n\|\ge .4\hat{\lambda}n-n\frac{2\nu}{.79\hat{\lambda}}(\Delta+e^{-\frac{\alpha}{2}n})-n_24\epsilon-
\frac{9\Delta \epsilon n}\alpha.
\end{equation}
So, we may conclude if 
\begin{equation}
.4\hat{\lambda}-\frac{2\nu}{.79\hat{\lambda}}-9\Delta \frac{\epsilon}{\alpha}>0,
\end{equation}
which is certainly true as long as $\epsilon$ and $\nu$ are sufficiently small relative to $\alpha$, $\lambda'$, and $\Delta$.
\end{proof}

\begin{rem}\label{rem:location_of_v_A_v_B}
Note that the proof of the previous claim shows something more precise: letting $v_A^s,v_A^u$ be the most contracted and expanded direction of $A^n$, in the Lyapunov charts both $v_A^s$ and $v_B^s$ lie within the neighborhood $B_{\delta_0}((0,1))$ and $v_A^u$ and $v_B^u$ both lie within the neighborhood $B_{\delta_0}((1,0))$ of where the conclusions of Lemma \ref{lem:linearized_perturbation_RP1}  hold. 
\end{rem}

Now that we have located where $v^s$, and hence $v^u$ lies, we check that the norm of $B^i$ is subtempered.

\begin{lem}
For any $\epsilon_0>0$, suppose that we have a sequence of matrices as above. Then there exists $k(C,\lambda,\epsilon,\alpha,\Delta)$ such that $k\epsilon<\epsilon_0$ and the norm $\|B^i\|$ is $(C,\lambda-k\epsilon,4k\epsilon)$ sub-tempered.
\end{lem}

\begin{proof}
From  Lemma \ref{lem:linearized_perturbation_RP1}, we see that if $v^u\in (E^s_0)^{\perp}$, then $v^u$ lies in $B_{\delta_0}((1,0))$. Given any $\beta_0>0$ and $n$ sufficiently large, any vector $v$ in this neighborhood satisfies that for $i<n_2$,
\begin{equation}
\Phi^i_{\epsilon}(v)\ge (1-\beta_0)\lambda.
\end{equation}
Thus we see that along the trajectory from time $1$ to $n_2$ that every vector that begins in $B_{\delta_0}(0,1)$ is $(C, (1-\beta_0)\lambda,0)$-subtempered for the sequence of matrices $B_i$ viewed in Lyapunov charts. Take $\beta_0$ such that $(1-\beta_0)\lambda>(\lambda+\lambda')/2$.

With respect to the reference metric, such a sequence is $(C,(1-\beta_0)\lambda,4\epsilon)$-tempered due to 
Lemma~\ref{lem:lyapunov_metric}(3). 
 This gives temperedness up to time $n_2$. 

 Recall that Lemma \ref{lem:bad_tail_still_tempered}  says that if we extend the sequence by $m$ matrices where 
\[
m<\Delta^{-1}(nk\epsilon-C),
\]
then the result will be $(C,(1-\beta_0)\lambda-k\epsilon,4k\epsilon)$-tempered. In our case
because $n_2\ge (1-\frac{9\epsilon}{\alpha})n$, we would like to append $\frac{9\epsilon}{\alpha}$ matrices of norm at most $e^{\Delta}$ and have the resulting sequence still be tempered. So, we need that
\begin{equation}
\frac{9\epsilon}{\alpha}n<\Delta^{-1}(nk\epsilon -C)
\end{equation}
For sufficiently large $n$, this holds as long as $k\epsilon\Delta^{-1}>9\epsilon/\alpha$,
that is, $k>9{\Delta}/{\alpha}$. 
Taking $\epsilon$ sufficiently small we can arrange that $k\epsilon<\epsilon_0$. 
In particular choosing $\beta_0$ sufficiently small, we can have that $(1-\beta_0)\lambda-k\epsilon\ge \lambda'$, so the needed conclusion holds.
\end{proof}

The first and second conclusions of Proposition \ref{prop:nearby_points_inherit_temperedness} for sufficiently large $n$ are now immediate from the two lemmas once we apply 
Proposition \ref{prop:tempered_norm_implies_splitting}, 
which constructs a splitting for a norm subtempered sequence.

We now turn to the proof of the third conclusion of the proposition. 
We need additional estimates.

We let $n_2$ be as above; it is the point past which the estimate in Lemma \ref{lem:linearized_perturbation_RP1} ceases to hold. Note that there exists $\beta_1$ such that $\|A_i-B_i\|'\le e^{-\beta_1(n_2-i)}$ where $\|\cdot\|'$, denotes the Lyapunov metric. Also, recall that from our choice of $n_2$, that on a neighborhood of $(1,0)$ of size $\delta_0$ that $B_i$ contracts distances by a factor of $e^{-.9\hat{\lambda}}$. 

\begin{claim}\label{claim:nearness_of_vu}
There exists $\beta_2>0$ such that if $v^u$ is the unstable vector for the $A_i$, then 
\begin{equation*}
d'(A^iv^u,B^iv^u)\le Ke^{- \beta_2(n_2-i)},
\end{equation*}
where 
$\displaystyle d'(u_1, u_2)={\left\|\frac{u_1}{\|u_1\|'}-\frac{u_2}{\|u_2\|'}\right\|'}$
is  the metric on $\RP^1$ with corresponding to the Lyapunov metric.
\end{claim}
\begin{proof}
Recall 
that in the Lyapunov coordinates, we have  $A^i((1,0))=(1,0)$. Further, from the previous Lemma, $v_A$ is within distance $e^{-n\nu}$ of $(1,0)$. Consequently, we begin by suppose that $v$ is a vector with $d'(v,(1,0))<e^{-n\nu}$ and then seeing how this vector shadows the trajectory of $(1,0)$. Then as both $v_A$ and $v_B$ are vectors satisfying this property, the needed conclusion follows by the triangle inequality. 

This can be seen inductively because, by that lemma\footnote
{ \label{FtRegLyap} Note that Lemma \ref{lem:linearized_perturbation_RP1}
applies to the Lyapunov metric since the eigenvalues of the matrices $A_i$ are 
 uniformly bounded in both in both original and Lyapunov coordinates.},

\begin{align*}
 d'(B^{i}(v),(1,0))&\le d'(B^{i}(v),B_{i}(1,0)) +d'(B_i(1,0),(1,0)) \\
 &\le e^{-.9\hat{\lambda}}d'(B^{i-1}(v),(0,1))+C \|A_{i}-B_{i}\|'.
\end{align*}
We may continue inductively as long as $B^iv$ still lies in the neighborhood $B_{\delta_0}$. For such $i$ before this point, the form of the estimate that we obtain is:
\[
d'(B^{i}(v),(1,0))\le e^{-n\nu}e^{-.9i\hat{\lambda}}+C\sum_{j=1}^i e^{-.9\hat{\lambda}(i-j)}
 \|A_j-B_j\|' \le C'e^{-\beta_2 (n_2-i)}.
\]
Note that as this estimate is growing exponentially quickly that the difference between the index $i$ where it first exceeds $\delta_0$ and $n_2$ is of size at most $\ln(C')/\beta_2$, which is constant. Hence by possibly adjusting the constant, the needed result follows. 

To conclude we apply apply the triangle inequality to the corresponding estimates on $d'(B^i(v),(1,0)$ and $d'(A^i(v),(1,0))$ 
\end{proof}

Before proceeding further, we record an additional quantitative estimate about the norms of the maps considered in Lemma \ref{lem:linearized_perturbation_RP1}. 

\begin{claim}\label{claim:perturb_size_norm}
For a matrix $A$ as in Lemma \ref{lem:linearized_perturbation_RP1}, for all $\sigma>0$, there exists $\epsilon_1>0$ such that if $E\colon \R^2\to \R^2$ is a matrix of norm $\epsilon\le \epsilon_1$, then if $v\in \RP^1$ with $d((1,0),v)\le \epsilon_1$:
\begin{enumerate}
    \item 
    $\abs{\Phi_A(v)-\Phi_{A+E}(v)}\le \|E\|$.
    \item 
    $\abs{\Phi_A(v)-\Phi_A((1,0))}\le (\sigma/2)\ln\|A\|$.
\end{enumerate}
\end{claim}
\begin{proof}
This claims follows easily because we are restricting to a neighborhood in $\RP^1$ where $A$ has large norm.  
Note that if $v$ is a unit vector $v$ and $\epsilon_1$ is sufficiently small then $\|Av\|$ and $\|(A+E)v\|$ are both greater than $1$, hence as $\ln$ is $1$-Lipschitz on $[1,\infty)$, so the first claim follows. The second claim is straightforward because by assumption $A=\diag(\sigma_1,\sigma_2)$.
\end{proof}

 Similar to before
we have the map $\Phi_A'(v)=\ln(\|Av\|'/\|v\|')$ on $\RP^1$; note that this measures the expansion of vectors with respect to the Lyapunov metric. By possibly decreasing the constants in the statement of Lemma \ref{lem:linearized_perturbation_RP1}, we can arrange that the conclusions of Claim \ref{claim:perturb_size_norm} hold as well for all $i\le n_2$.  (Both
statements hold  with respect to the Lyapunov metric, see footnote \ref{FtRegLyap}).
We record two facts that follow from Claim \ref{claim:nearness_of_vu} along with the estimate  $\|A_i-B_i\|'\le e^{-\beta_1(n_2-i)}$:
\[
\abs{\Phi_{A_i}'(A^{i-1}v^u)-\Phi_{A_i}'(B^{i-1}v^u)}\le (\sigma/2)\ln \|A_i\|'
\]
\[
\abs{\Phi_{A_i}'(B^{i-1}v^u)-\Phi_{B_i}'(B^{i-1}v^u)}\le \|B_i\|'\le e^{-\beta_1(n_2-i)}.
\]

Using these claims, we now estimate $\|B^{n_2}v^u\|'$:
$$
\abs{\|B^{n_2}v^u\|'- \|A^{n_2}v^u\|'}=\abs{\sum_{i=1}^{n_2} \Phi_{A_i}'(A^{i-1}v^u)-\Phi_{B_i}'(B^{i-1}v^u)}
$$
$$
\le \sum_{i=1}^{n_2}\abs{\Phi_{A_i}'(A^{i-1}v^u)-\Phi_{A_i}'(B^{i-1}v^u)}+\abs{\Phi_{A_i}'(B^{i-1}v^u)-\Phi_{B_i}'(B^{i-1}v^u)}
$$
$$
\le \sum_{i=1}^{n_2} \left((\sigma/2)\ln \|A_i\|'+e^{-\beta_1(n_2-i)}\right)
$$
Thus we see that $\ln \|B^{n_2}\|'\ge (1-\sigma/2)\ln\|A^{n_2}\|'-C_1$. 
Using this we now estimate the norm of $\|B^n\|$. As the norm of all $B_i$ and $A_i$ are uniformly bounded by $e^{\Delta}$ by assumption, it follows that 
$\abs{\ln \|A^{n}\|-\ln\|A^{n_2}\|}\le (n-n_2)\Delta$. Thus,
\begin{align*}
\ln \|B^n\|&\ge \ln\|B^{n_2}\|-(n-n_2)\Delta\\
&\ge (1-\sigma/2)\ln \|A^{n_2}\|-4n_2\epsilon-C_1-(n-n_2)\Delta\\
&\ge (1-\sigma/2)\ln \|A^n\| -(\sigma/2)(n-n_2)\Delta-4n_2\epsilon-C_1-(n-n_2)\Delta.\label{eqn:last_equation_est_A_n}
\end{align*}
By subtemperedness $\ln \|A^n\|\ge \lambda n-C_2$ for some $C_2$, and hence  
 if $\varepsilon$ is small enough compared to $\lambda$ and $\sigma$, then as $n-n_2=O(\epsilon)$ the estimate 
of part (3) of Lemma \ref{lem:linearized_perturbation_RP1} 
holds.
\end{proof}

Proposition \ref{prop:nearby_points_inherit_temperedness} implies that nearby points have close splittings so that the blocks where a tempered splitting fails to exist are
 not too small.

\subsection{Cushion of nearby points}
\label{SSCushion}
In this subsection, we prove a refinement of the estimate from the previous subsection. 
Recall Definition \ref{defn:cushion}.
We show that points with very close trajectories have cushion that differs by $O(1)$. This will be used later because it shows that if a short curve has a single point with bad cushion, then all of these points have bad cushion.

\begin{prop}\label{prop:nearby_points_close_cushion}
Fix $(C_0,\lambda)$, $\Lambda>0$, $\sigma>0,$ $\varpi>0$, then for all sufficiently small $\epsilon>0$ there exists $N$ and $D$ such that the following holds. 
Suppose that $(A_i)_{1\le i\le n}$ and $(B_i)_{1\le i\le n}$ are sequences of matrices in $\SL(2,\R)$ with norm at most $e^{\Lambda}$ that are $(C_0,\lambda,\epsilon)$-tempered such that $\|A_i-B_i\|\le C_1e^{-\sigma(n-i)}e^{-n\varpi}$. Let $U(A)$ and $U(B)$ denote the cushion of $A$ and $B$. Then
\[
\abs{U(A)-U(B)}\le D.
\]
\end{prop}
\begin{proof}

 In view of the definition of the cushion, it suffices to prove 
 that there exists $D>0$ such that for two such sequences and $1\le k\le n$, $\abs{\ln \|A^k\|-\ln \|B^k\|}\le D$. This will follow from the claim below, which gives an exponential shadowing for the most expanded directions of $A^k$ and $B^k$

\begin{claim}\label{claim:angles_nearby}
There exists $N_1$, $D_1$ such that as long as $n\ge N$, 
There exists $\beta_2(\lambda,\epsilon,\sigma)$ and $K(C_0,\lambda,\Lambda,\sigma,\varpi,\epsilon)$ such that 
for any $N\le k\le n$ the following holds.
If $v_k$ is the unstable vector for the $A^k$, then for $i\le k$,
\[
d(A^iv_k,B^iv_k)\le K(e^{-\beta_2(k+i)}+e^{-n\varpi/2}).
\]
\end{claim}

\begin{proof}
This essentially follows due to an enhancement of the argument surrounding 
Claim~\ref{claim:nearness_of_vu}, which we can improve due to the stronger assumptions of the present claim.

As before, we work in Lyapunov charts, and estimate the distance that a vector near $(1,0)$ can drift away from it. 
Comparing with \eqref{eqn:lyapunov_chart_B_i_expression}, when we look in the Lyapunov charts adapted to the sequence $A_1,\ldots,A_k$, we now have that
\[
A_i=\begin{bmatrix}
\sigma_{i,1} & 0 \\
0 & \sigma_{2,i}
\end{bmatrix} =B_i+O_{C,\lambda,\Lambda}(e^{-\sigma (n-i)}e^{-\varpi n}e^{4\epsilon i}). 
\]
Hence Lemma \ref{lem:linearized_perturbation_RP1} holds for all $1\le i\le n$, i.e.~for the entire sequence, as long as $\epsilon$ is sufficiently small relative to $\varpi$. Note that this implies that there exists some $C'$ such that $\|A_j-B_j\|'\le C'e^{-\sigma(n-i)}e^{-\varpi n}e^{4\epsilon i}$.

We now do an induction similar to that in Claim \ref{claim:nearness_of_vu}.  
Denote
$$ d_{n, k, i}=e^{-k\nu}e^{-.9 \hat{\lambda} i}+ \bar{C} e^{-n(2/3)\varpi}e^{-\sigma(n-k)},$$
where $\bar C$ is a large constant that will be chosen below.
From Lemma \ref{lem:linearized_perturbation_RP1}, we can take  $\delta_0$ so small that any vector making angle less than $\delta_0$ with $(1, 0)$ is contracted by at least
$e^{-.9\hat\lambda}$. Take $N$ so large that for all $N\leq k\leq n$ we have that $d_{n, k, 0}\leq \delta_0$
and hence also for all $i\leq N$, $d_{n, k, i}\leq \delta_0.$
We now verify by induction on $i$ that if we start with a vector $v$ such that $d'(v,(1,0))\le e^{-k\nu}$, 
then for all $i\leq k$,\; $d'(B^{i}(v),(1,0))\le d_{n,k, i}.$ Indeed
\begin{align*}
d'(B^{i}(v),(1,0))&\le e^{-.9 \hat \lambda} d_{n,k, i-1}+\|A_i-B_i\|'\\
&\le e^{-k\nu}e^{-.9i\hat{\lambda}}+e^{-.9 \hat \lambda}\bar{C}e^{-n(2/3)\varpi}e^{-\sigma(n-k)}+C'e^{-\sigma(n-i)}e^{-\varpi n}e^{4\epsilon i}.
\end{align*}
As long as $\bar C$ is sufficiently large and $\epsilon$ is sufficiently small relative to $\varpi$, it then follows that:
$$d'(B^{i}(v),(1,0))  \le e^{-k\nu}e^{-.9 \hat{\lambda} i}+ \bar{C} e^{-n(2/3)\varpi}e^{-\sigma(n-k)}
\le d_{n,k,i}.
$$

Thus for $1\le i\le k$,
\[
d'(B^i v,A^i(1,0))\le C_1(e^{-k\nu}e^{-.9 \hat{\lambda} i}+  \bar Ce^{-n(2/3)\varpi}e^{-\sigma(n-k)}).
\]
Lemma \ref{lem:change_of_metric_change_of_angle}, which compares distance on $S^1$ for different metrics, implies that as long as $\epsilon$ is sufficiently small relative to $\lambda$ and $\sigma$, then respect to the reference metric on $\RP^1$ that there exists $C_2$ such that
\[
d(B^i v,A^i((1,0)))\le C_2(e^{-k\nu}e^{-.45 \hat{\lambda} i}+\bar{C} e^{-n\varpi/2}e^{-\sigma(n-k)}).
\]
The above estimate holds for any vector $v$ at distance $e^{-k\nu}$ from $(1,0)$.

In particular, from Lemma \ref{lem:nearness_of_unstable_directions_lemma} whose weaker hypotheses $(A_i)_{1\le i\le k}$ and $(B_i)_{1\le i\le k}$ satisfy, we see that $v^k_A$ and $v^k_B$ are both within $e^{-k\nu}$ distance of $(1,0)$ in the Lyapunov charts as long as $k\ge N_2$ for some $N_2$. Thus by specializing to these vectors and applying the triangle inequality, we find that
$\displaystyle
d(B^i v_A^k,A^iv_A^k)\le C_3(e^{-k\nu}e^{-.45\hat{\lambda i}}+e^{-n\varpi/2}),
$
which is the desired claim.
\end{proof}

Because the norm of all the matrices we are considering is uniformly bounded by $e^{\Lambda}$, the estimate in Claim \ref{claim:angles_nearby} gives that for $k>N$, 
\[
\abs{\ln \|A^k\|-\ln \|B^k\|}\le \sum_{i=1}^k Ke^{-\beta_2(k+i)}+e^{-n\varpi/2}\le K'
\]
for some fixed $K'$. Note that this gives the conclusion of the lemma about cushioning for all indices greater than $N$. For those less than $N$,  since
there are only finitely many such words and the norms of matrices are bounded, 
we can accommodate them by increasing the constant in the conclusion of the theorem.
\end{proof}

\subsection{Scale selection proposition}\label{subsec:scale_selection}
Given two nearby standard pairs, we can attempt to ``couple" them using the fake stable manifolds. For this we need more quantitative estimates on how close and smooth standard pairs need to be so that we can couple 
a significant proportion of them. For example, if they are too far apart then a fake stable leaf may not reach from one to the next. 
Proposition \ref{prop:set_up_scale_prop} below are mostly a summary of results appearing elsewhere in the paper. Note that the first parts of the proposition are statements about temperedness and splittings on uniformly large balls in $M$. Part \eqref{item:there_exists_well_cfd_nbd} shows that for fixed $C_0$ if we consider sufficiently  small $(C_0,\delta,\upsilon)$-configurations that on balls of radius $O(\delta)$ that transversality to the contracting direction and temperedness of the splitting imply that the holonomies between the curves in a configuration exist and converge exponentially fast.

Below we say that a curve and a cone field are $\theta_0$-transverse if the smallest angle they make is at least $\theta_0$. Also, see Definition \ref{defn:cone_tempered} in the appendix for the definition of $(C,\lambda,\epsilon,\mc{C})$-tempered, which means $(C,\lambda,\epsilon)$-tempered plus the additional condition that the stable direction lies in the cone $\mc{C}$.

\begin{prop}\label{prop:set_up_scale_prop}
Suppose that $(f_1,\ldots,f_m)$ is an expanding on average tuple in $\Diff^2_{\vol}(M)$ with $M$ a closed surface.
There exists $\lambda>0$ such that for any $0<\lambda'<\lambda$, $0<\sigma$ there exists $0<\epsilon_0,\tau<1$ 
such that for any $0<\epsilon<\epsilon'<\epsilon_0$ 
there exist $\delta_0,\delta_1,\theta,b_0,C,C',C'',\theta_0,\eta >0$ and $N\in \N$ such that:
for any $x\in M$, $i\in \{1,2,3\}$, there are three nested cone fields $\mc{C}^i_{\theta}\subset \mc{C}^i_{2\theta}\subset \mc{C}^i_{3\theta}$ of angles $\theta$, $2\theta$, $3\theta$, respectively defined on $B_{\delta_0}(x)$ by parallel transport from a cone at $x$. 
Further, the $\mc{C}^i_{3\theta}$ are uniformly transverse on $B_{\delta_0}(x)$. These conefields satisfy the following properties for words $\omega$, where probabilities below are with respect to the Bernoulli measure $\mu$ on $\Sigma$.

\begin{enumerate}[leftmargin=*]
\item\label{item:positive_probability_of_tangency}
(Positive probability of tangency to $\mc{C}^i_{\theta}$)
For any point $y\in B_{\delta_0}(x)$  and any $i\in \{1,2,3\}$, the probability that
$D_xf^n_{\omega}$ is $(C,\lambda,\epsilon,\mc{C}^i_{\theta})$-tempered for all $n\ge N$ is at least $b_0>0$.

\item\label{item:nearby_points_are_also_tempered}
(Nearby points are also tempered)
For any curve $\gamma$, if $x\in \gamma$ is $(C,\lambda,\epsilon,\mc{C}^i_{\theta})$-tempered at time $n$ and $y\in \gamma$ is a point with $d_{\gamma}(x,y)\le \|D_xf^n_{\omega}\|^{-(1+\sigma)}$, then $y$ is $(C',\lambda',\epsilon',\mc{C}^i_{2\theta})$-tempered at time $n$ and 
\begin{equation}
\|D_yf^n_{\omega}\|\ge \|D_xf^n_{\omega}\|^{1-\sigma}.
\end{equation}

\item\label{item:existence_of_fake_manifolds1}
(Existence of fake stable manifolds)
For any $(C',\lambda',\epsilon',\mc{C}^i_{2\theta})$-tempered point $y\in B_{\delta_0}(x)$ at time $n\ge N$, 
the fake stable curve $W^s_{n,\delta_1}(\omega, y)$ exists, has length at least $\delta_1$, has $C^2$ norm at most $C''$, and is tangent to $\mc{C}^i_{3\theta}$ on $B_{\delta_0}(x)$.

\item\label{item:there_exists_well_cfd_nbd} (There exists a well configured neighborhood) For any $C_0$, there exists 
$\delta \in (0,1)$, 
$a_0,a_1,D_1,D_2>0$ and 
$N_1\in \N$
such that for all $0<\delta'<\delta$, and any $\upsilon\le \delta'\tau$, the following holds for any $(C_0,\delta',\upsilon)$-configuration $(\hat{\gamma}_1,\hat{\gamma}_2)$.  
There exists $x\in M$ and $i\in \{1,2,3\}$ such that $\gamma_1$ and $\gamma_2$ are uniformly $\theta_0$-transverse to $\mc{C}^i_{3\theta}$ on $B_{\delta_0}(x)$. We let $B_{2\nu}(x)$ be a ball that demonstrates that $\hat{\gamma}_1,$ and $\hat{\gamma}_2$ are in a $(C_0, \delta',  \upsilon)$-configuration, i.e.~it contains points of $\gamma_1$ and $\gamma_2$ that are distance at least $\upsilon$ from the boundary of those curves.
We maintain this choice of $x$ and $i$ in the following lettered items:

\begin{enumerate}
\item\label{item:fake_stable_manifolds_are_transverse_to_pairs}
(Fake stable manifolds tangent to $\mc{C}^i_{3\theta}$ are transverse to pairs) 
If $y\in B_{2\upsilon}(x)$  is as in item \eqref{item:existence_of_fake_manifolds1} above, then
 $W^s_{n,\delta_1}(\omega, y)$ intersects both $\gamma_1$ and $\gamma_2$ and the points of intersection are both $\theta_0$-transverse, i.e.~both $\gamma_1$ and $\gamma_2$ make an angle at least $\theta_0$ with $W^s_{n,\delta_1}(\omega, y)$.

\item\label{item:lower_bound_on_derivative}
(Lower bound on derivative of the holonomy)
For $n\ge N_1$, if $B\subseteq \gamma_1\cap B_{2\nu}(x)$ is a subset of $\gamma_1$
 consisting of
$(C',\lambda',\epsilon',\mc{C}^i_{2\theta})$-tempered points at time $n$, 
then $H^s_n(B)\subseteq \gamma_2$ has length at least $D_1\len(B)$.
Further, as long as $\hat{\gamma}_1$ and $\hat{\gamma}_2$ have equal mass and are at most
$4\delta'$-long, there are a pair of connected components of $\hat{\gamma}_1\cap B_{2\upsilon}(x)$ and $\hat{\gamma}_2\cap B_{2\upsilon}(x)$ each containing at least $a_1$ proportion of the mass of $\hat{\gamma}_1$ and $\hat{\gamma}_2$ such that if $B$ is as above and lies in this set, then
\[
a_0(H^s_n)_*\rho^1\vert_{B}\le \rho^2\vert_{H^s_n(B)}.
\]

\item\label{item:fluctuations_in_holonomies}
(Fluctuations in the holonomies) 
For any $(C_0,\delta',\upsilon)$-configured pair $(\gamma_1,\gamma_2)$, if $z\in B_{2\upsilon}(x)$ is a $(C',\lambda',\epsilon',\mc{C}_{2\theta}^i)$ tempered point at times $n,n-1\ge N_1$ and $y$ is any point with $d_{\gamma_1}(x,y)\le \|D_xf^n_{\omega}\|^{-(1+\sigma)}$, 
then  
\begin{equation}\label{eqn:fluctuation_of_endpoints_eqn}
d_{\gamma_2}(H_{n}^s(y),H^s_{n-1}(y))\le e^{-1.99\ln \|D_xf^n_{\omega}\|}.
\end{equation}
Further, for $n\ge N_1$ the rate of convergence of the Jacobians is exponentially fast
\begin{equation}
    \label{JacExp}
\abs{\Jac H^s_{n}-\Jac H^s_{n-1}}\le e^{-\eta n}.
\end{equation}

\item\label{item:hold_cont_jac}
(Log-$\alpha$-H\"older control of Jacobian)
If $B\subseteq \gamma_1\cap B_{2\upsilon}(x)$ is an open set comprised of $(C',\lambda',\epsilon',\mc{C}_{\theta}^i)$-tempered points at time $n$, then 
\begin{equation}
\abs{\log \Jac H^s_n(x)-\log \Jac H^s_{n}(y)}\le D_2d_{\gamma_1}(x,y)^{\alpha}. 
\end{equation}
\end{enumerate}
\end{enumerate}

\end{prop}
\begin{proof}
The main non-trivial input to this proposition is the definition of the cones. After they are chosen correctly, the remaining statements follow in a straightforward manner from facts about the fake stable manifolds proven elsewhere.

For any point $x\in M$, we let $\nu_x$ denote the distribution of the true stable directions $E^s$ at the point $x$, which is a measure on $\RP^1_x$, the projectivization of $T_xM$. 
As $\nu_x$ is non-atomic, we can find three disjoint intervals $I_1,I_2,I_3$ of width $\theta$ that are each separated by angle at least $4\theta$ for some angle $\theta>0$ and such that $\nu_x(I_1),\nu_x(I_2),\nu_x(I_3)$ are each positive. 
We then use these intervals to define nested cones $\mc{C}^i_{\theta/2}(x)\subset \mc{C}^i_{\theta}(x)\subset \mc{C}^i_{2\theta}(x)\subset \mc{C}^i_{3\theta}(x)$ at $x$ for $i\in \{1,2,3\}$.
Due to the continuity of $\nu_x$ from Proposition~\ref{prop:stationary_measure_is_continuous}, we see that if we parallel translate $I_1,I_2,I_3$ to form cone fields $\mc{C}_1,\mc{C}_2,\mc{C}_3$ over a ball $B(x)$ around $x$, then we similarly have that $\nu_y(\mc{C}^i_{\theta/2})$ is uniformly positive for all $y\in B(x)$. 
All these properties are uniform, so we can do this for any $x\in M$ and obtain a neighborhood of uniform size, with uniform lower bound on $\nu_y(\mc{C}^i_{\theta/2})$ over all these neighborhoods.

We now verify item \eqref{item:positive_probability_of_tangency}.
There exist $\lambda,\epsilon>0$ such that for any $y\in M$ and almost every $\omega$, $D_yf^n_{\omega}$ is $(C(\omega),\lambda,\epsilon)$-tempered for some $C(\omega)$. Further, by Proposition \ref{prop:diffeos_are_subtempered} we have a uniform estimate on the tail on $C(\omega)$ independent of the point $y$. Thus by choosing $C_1$ sufficiently large for any $y\in B(x)$ and $1\le i\le 3$, with probability at least $b_0$, $D_yf^n_{\omega}$ is $(C_1,\lambda,\epsilon)$-subtempered and $E^s_n(\omega,y)\in \mc{C}^i_{\theta/2}(y)$ for all $n\ge N$. 
By Proposition \ref{prop:tempered_norm_implies_splitting}, there exists $N_0\in N$ such that for any $(C_1,\lambda,\epsilon)$-subtempered trajectory of length $n\ge N_0$, then for all $n\ge N_0$, $\angle (E^s_n(\omega,y),E^s)<\theta/4$ and so $E^s_n\in \mc{C}^i_{\theta}$. This gives us the  uniformly positive probability of at least $b_0>0$.

Item (2)  is immediate from Proposition \ref{prop:nearby_points_inherit_temperedness}.

 Item \eqref{item:existence_of_fake_manifolds1}, which states the existence of the fake stable manifolds for $(C',\lambda',\epsilon',\mc{C}^i_{2\theta})$-tempered points, 
follows from Proposition \ref{prop:fake_stable_C_2_control} (possibly after decreasing $\delta$).

We now verify item \eqref{item:there_exists_well_cfd_nbd}, which has many subparts. 
The statement in the initial part follows by making a judicious choice of $x$ as well as the particular cone $\mc{C}^i_{3\theta}$ on $B_{\delta_0}(x)$ that the fake stable manifolds will be tangent to. Because $\hat{\gamma}_1$ and $\hat{\gamma}_2$ are a $(C_0,\delta',\upsilon)$-configuration 
then there exists a pair of points $x\in \gamma_1$ and $y\in \gamma_2$ with $d(x,y)<\upsilon$. 
We choose to work on the neighborhood $B_{\delta_0}(x)$. We then must show that we can pick one of the cones $\mc{C}^i_{3\theta}$ that is uniformly transverse to $\gamma_1$ and $\gamma_2$ on $B_{\delta_0}(x)$.
 Let $\mathcal K_1$ be a small cone around $\gamma_1'(x)$ and $\mathcal K_2$ be a small cone around $\gamma_2'(y)$.
We can extend both cones to the whole of $B_{\delta_0}(x)$ by parallel transport.
Since there are  three cones, we let $i\in \{1,2,3\}$ be an index such that the cone $\mc{C}^i_{3\theta}$ is transverse 
to  both $\mathcal K_1$ and $\mathcal K_2$. We let $\theta_0$ be a lower bound on the angle that $\gamma_1,\gamma_2$ make with $\mc{C}^i_{3\theta}$ and note that, as before, that $\theta_0$ is uniform as it only relies on knowing $C_0,\delta'$.
We now proceed to checking the lettered items that follow.

Item \eqref{item:fake_stable_manifolds_are_transverse_to_pairs} says that the fake stable manifolds of $(C',\lambda',\epsilon',\mc{C}^i_{2\theta})$-tempered points are $\theta_0$-transverse to $\gamma_1,\gamma_2$ and intersect them. This follows from Proposition \ref{prop:fake_stable_C_2_control} because by choice of our constants, for such a tempered point $y$, it follows that $W^s_{n}(\omega, y)$ is tangent to $\mc{C}^i_{3\theta}$, and the uniform transversality follows from our control on the $C^2$ norm of $W^s_n(\omega, y)$ and the H\"older continuity of the most contracting subspace $E^s_n$. Further, the fact that we only need the curves to be at most $\upsilon=\tau\delta'$ apart from each other, with $\tau$ depending only on $\theta_0,\lambda',\lambda$ 
is clear from the uniform $C^2$ bound on the norm of the fake stable manifolds $W^s_n(\omega,y)$  from item \eqref{item:existence_of_fake_manifolds1}. Item \eqref{item:fake_stable_manifolds_are_transverse_to_pairs} follows because as long as $\delta$ is sufficiently small compared with $C_0$, 
the tangent direction to $\gamma_i$ is close to constant on a segment of length $\delta$.

The first part of item \eqref{item:lower_bound_on_derivative} saying that there is a lower bound on the derivative of the holonomies follows from Proposition \ref{prop:holonomies_converge_exponentially_fast}.

The next claim is that restricted to a segment in $B_{\delta_0}$, $\gamma_1$ and $\gamma_2$ have a positive proportion of their mass there. This follows due to the log-H\"older regularity of $\rho^1$ and $\rho^2$ as long as $\delta$ is sufficiently small. Due to the boundedness of the Jacobian, the log-H\"olderness of the densities and them both having a positive amount of their mass on $B_{\delta_0}(x)$, it additionally follows that there exists such a uniform constant $a_0$ as stated in item \eqref{item:lower_bound_on_derivative}.

Item \eqref{item:fluctuations_in_holonomies} 
is immediate from the statement of Proposition \ref{prop:fluctuations_in_fake_stable_leaves}.

Finally, item \eqref{item:hold_cont_jac}, which concerns the fluctuations in the Jacobian of $H^s_n$, follows from Proposition \ref{prop:holonomies_converge_exponentially_fast}.
\end{proof}

\subsection{Proof of Inductive Local Coupling Lemma.}\label{subsec:inductive_local_coupling_procedure}

We are now ready to prove the inductive local coupling lemma.

First we prove a result that
does not make any assertions about the quantity of points on the curve $\gamma_1$ that have a tempered splitting. It just shows that given an infinite trajectory $\omega\in \Sigma$, we may use this trajectory to define a fake coupling in the sense of 
Definition~\ref{defn:fake_coupling} at all future times.

\begin{lem}\label{lem:local_coupling_lemma}
(Inductive Coupling Lemma.) Let $(f_1,\ldots,f_m)$ be an expanding on average tuple in $\Diff^2_{\vol}(M)$ for $M$ a closed surface. 

For any $C_0>0$ let $\lambda,\lambda',\epsilon_0,\tau,\epsilon,\epsilon'$, etc.~be a valid choice of constants in the first paragraph of Proposition \ref{prop:set_up_scale_prop} and  $\delta,\delta',\upsilon,$ etc., be a valid choice of constants in part \eqref{item:there_exists_well_cfd_nbd} of that proposition.
Then there exist $b_1,\hat{\eta},\Lambda>0$ such that for any  $(C_0,\delta',\upsilon)$-configuration $(\hat{\gamma}_1,\hat{\gamma}_2)$ the conclusions of Proposition \ref{prop:set_up_scale_prop} apply
and
the following holds.
If $x\in M$ and  $B_{\delta_0}(x)$ is the  neighborhood  where the statements from Proposition 
\ref{prop:set_up_scale_prop}\eqref{item:there_exists_well_cfd_nbd} hold, then we can construct a $(b_1,\hat{\eta})$-fake couplings out of $(\hat{\gamma}_1,\hat{\gamma}_2)$: For each $\omega\in \Sigma$ there exists a decreasing sequence of pairs of standard subfamilies $P^1_n\subseteq \hat{\gamma}_1$ and $P^2_n\subseteq \hat{\gamma}_2$ that are $(b_1,\hat{\eta})$-fake coupled at each time $n\ge N_1$.
Further, for $n\ge N_1$ and $i\in \{1,2\}$ 
$P^i_n\setminus P^i_{n+1}$ are $n\Lambda$-good standard families. 

These sequences of standard families are decreasing and converge to measures $P^1_{\infty}$ and $P^2_{\infty}$. Further,
for such a fake coupling we also have the true stable holonomies $H^s_{\infty}$ and these satisfy $(H^s_{\infty})_*P^1_{\infty}=P^2_{\infty}$.
\end{lem}

\begin{proof}[Proof of Lemma \ref{lem:local_coupling_lemma}]
We divide the proof into several steps. In Step 0, we introduce the constants that will be used later in the proof; naturally we will also make use of many constants from Proposition~\ref{prop:set_up_scale_prop}, which is essentially the setup for this lemma. Then in the following steps we give an iterative procedure showing how one may construct a new fake coupled pair out of an old one. By iterating that procedure, we then obtain the result.

\noindent\textbf{Step 0: Introduction of constants.} At this step we
introduce some of constants that will be used in the proof.
Most of these constants will be chosen when they appear in the proof.

\begin{enumerate}[leftmargin=*]
    \item\label{item:proof_stretching_intervals}
First, we let $\lambda,\lambda',\epsilon',D_1,$ etc., be the constants from the statement of Proposition \ref{prop:set_up_scale_prop}. For the given $\hat{\gamma}_1$ and $\hat{\gamma}_2$ we let $B_{\delta_0}(x)$ be a neighborhood so that the conclusions of part~\ref{item:there_exists_well_cfd_nbd} of that proposition apply. We will simply write $\mc{C}_{\theta}$ rather than $\mc{C}^i_{\theta}$ below for the cones defined on $B_{\delta_0}(x)$ such that $\hat{\gamma}_1$ and $\hat{\gamma}_2$ have segments that are both uniformly $\theta_0$-transverse to $\mc{C}_{3\theta}^i$ on $B_{\delta_0}(x)$. We let $\Lambda_{\max}$ be sufficiently large so that $\|D_xf_i\|\le e^{\Lambda_{\max}}$ for all $x\in M$ and $1\le i\le m$.
\item\label{item:choice_of_delta_lcl}
Further, in the application of Proposition \ref{prop:set_up_scale_prop} we will insist that $\delta$ is so small that for any $C_0$-good curve with density $\rho$ on a ball of size $\delta$, 
the $\log$-H\"older condition on $\rho$ implies that $1/2<\rho(y)/\rho(x)<2$ on this ball. 
\item
Below, we have certain estimates that will only hold as long as $n$ is sufficiently large. We will have some cutoffs 
$N_1,N_2$ that we define in the course of the proof  at the ends of steps 2 and 6, respectively.
 The cutoffs $N_1$ and $N_2$ only depend on the fixed constants from \eqref{item:proof_stretching_intervals} and \eqref{item:choice_of_delta_lcl} above.
 We then set
$\displaystyle 
N_0=\max\{N,N_1,N_2\}
$ in the conclusion of the theorem
where $N$ is the cutoff for Proposition \ref{prop:set_up_scale_prop} to hold.
\end{enumerate}

\noindent\textbf{Step 1: Definition of $\mc{I}_n^1$.}
Let $\Gamma_1$ be a connected component of $B_{\delta_0/2}(x)\cap \gamma_1$ 
 within distance $\upsilon$ of $\gamma_2$.
Let $G^n_{\omega}$ be the $(C',\lambda',\epsilon',\mc{C}_{2\theta})$-tempered points at time $n$ lying in $\Gamma_1$ (See Definition \ref{defn:cone_tempered}).
Note that $G^{n}_{\omega}\subseteq G^{n-1}_{\omega}$. 
We set
\begin{equation}\label{eqn:defn_eta_n}
\eta_n(x)=\frac{1}{4(\displaystyle\max_{1\le m\le n}\{\|D_xf^m_{\omega}\|e^{(n-m)\lambda'/2}\})},
\end{equation}
and 
\begin{equation}
\label{ChooseDeltaN}
\delta_n(x)=\eta_n^{(1+\sigma)}(x).
\end{equation}

We now construct $\mc{I}_n^1$.
For each $x\in G^n_{\omega}$, we say that $x$ is \emph{padded} if $B_{\delta_n(x)}^{\gamma_1}(x)\subseteq G^n_{\omega}$, where the $B^{\gamma_1}_{\delta_n}(x)$ denotes a ball of radius $\delta_n$ about $x$ in $\gamma_1$ with respect to the arclength on $\gamma_1$. We let $H^n_{\omega}$ denote the set of all such padded
 points. 
 Let $\hat{\mc{I}}_n^1\subset \gamma_1$ be the set
\begin{equation}\label{defn:hat_I_n}
\hat{\mc{I}}_n^1=\bigcup_{x\in H^{n}_{\omega}} B_{\delta_n(x)}^{\gamma_1}(x).
\end{equation}
Note that $\hat{\mc{I}}_n^1$ is a finite union of intervals. 
 Delete intervals of length $K_1e^{-4\Lambda_{\max} n}$ from the edges of each  component
where $K_1>0$ is a fixed small constant that we choose below.  Call this trimmed collection of intervals $\mc{I}^1_n$.

We next check that $\mc{I}_{n}^1\subseteq \mc{I}_{n-1}^1$.
By the definition of $\delta_n$, $\delta_n(x)\le e^{-\lambda'/2}\delta_{n-1}(x)$, thus as long as $K_1$ is sufficiently small,
\begin{equation}\label{eqn:choice_of_K_1_nbd}
\delta_n\le e^{-\lambda'/4}\delta_{n-1}-K_1e^{-4(n-1)\Lambda_{\max}}.
\end{equation}
Thus from the definition of $\mc{I}^1_n$, it is immediate that $\mc{I}^1_n\subseteq \mc{I}^1_{n-1}$.

\noindent\textbf{Step 2: Definition of $\mc{I}^2_n$.}
From the previous step, we know that any point in $\mc{I}^1_n$ satisfies the hypotheses of Proposition \ref{prop:set_up_scale_prop}. Since $\hat{\gamma}_1$ and $\hat{\gamma}_2$ are uniformly $\theta_0$-transverse to $\mc{C}_{3\theta}$, 
it follows from Proposition \ref{prop:set_up_scale_prop}\eqref{item:fake_stable_manifolds_are_transverse_to_pairs}
that 
the fake stable manifold $W^s_{n,\delta_1}(y)$ 
of each point $y\in \mc{I}_n^1$ intersects $\gamma_2$. 
Hence, there is a well defined holonomy $H^s_n\colon \mc{I}^1_n\to \gamma_2$ which satisfies all the conclusions of Proposition \ref{prop:set_up_scale_prop}.
We define
\begin{equation}
\mc{I}^2_n=H^s_{n}(\mc{I}_{n}^1). 
\end{equation}

Next we check that $\mc{I}^2_n\subseteq \mc{I}^2_{n-1}$. For this we will use the control on the fluctuations in the size of the holonomies from Claim \ref{claim:control_on_fluctuations} below.
As we vary $n$, the fluctuations in $H^s_n(y)$
are smaller than the width of the neighborhoods $\delta_n$  in  \eqref{ChooseDeltaN}, and the result will follow.

Suppose that $x\in \mc{I}^1_n$. We must show that $H^s_{n}(x)\in \mc{I}^2_{n-1}$. Note that while $x$ might not be in $H^n_{\omega}$ it is in $G^n_{\omega}$. So, there exists some point $y$ such that $x\in B^{\gamma_1}_{\delta_n(y)}(y)$ and hence also in $B^{\gamma_1}_{\delta_{n-1}(y)}(y)$.

 To show that $H^s_{n}(x)\in \mc{I}_{n-1}^2$, we estimate how far $H^s_{n}(x)$ is from $H^s_{n-1}(x)$ and then estimate how far $H^s_{n}(x)$ is from the boundary of $H^s_{n-1}(B^{\gamma_1}_{\delta_{n-1}(y)}(y))$.
For the former, we use the following claim.

\begin{claim}\label{claim:control_on_fluctuations}
There exists $C_1>0$ such that if $x\in B^{\gamma_1}_{\delta_n(y)}(y)$  for some $y\in \mc{I}^1_{n}$
then
\[
d_{\gamma_2}(H^s_{n}(x),H^s_{n-1}(x))\le C_1\eta_n^{1.99(1-\sigma)^2}(y).
\]
\end{claim}
\begin{proof}
First we show that there exists $C_a$ such that
\begin{equation}\label{eqn:lower_bound_D_y}
\|D_yf^n_{\omega}\|\ge C_a\eta_n^{-(1-\sigma)}(y).
\end{equation}
Let $N\le m\le n$, be the number achieving the maximum in the definition of $\eta_n$, \eqref{eqn:defn_eta_n}.
From $(C',\lambda',\epsilon')$ temperedness, 
\begin{equation}\label{eqn:first_est_f_n_eta_n}
\|D_yf^n_{\omega}\|=\|D_yf^{m+(n-m)}_{\omega}\|\ge e^{-C'}e^{\lambda'(n-m)}e^{-\epsilon' m}\|D_yf^m_{\omega}\|.
\end{equation}
By the definition of $\eta_n$,
\begin{equation}
\eta_n^{-(1-\sigma)}\le 4^{(1-\sigma)}e^{(n-m)(1-\sigma)\lambda'/2}\|D_yf^m_{\omega}\|^{(1-\sigma)}.
\end{equation}
But from $(C',\lambda',\epsilon')$-temperedness, $\|D_yf^m_{\omega}\|\ge e^{-C'}e^{\lambda' m}$. 
Hence as long as $\lambda'\sigma>2\epsilon'$, it follows that there exists $C_b$ such that for all $n\ge N$
we have $C_b\|D_yf^m_{\omega}\|^{(1-\sigma)}\le \|D_yf^m_{\omega}\|e^{-2\epsilon' m}$. Hence there is some $C_c$ such that
\[
C_c\eta_n^{-(1-\sigma)}\le e^{(n-m)\lambda'/2} \|D_yf^m_{\omega}\|e^{-2\epsilon' m}.
\]
Comparing the above equation with \eqref{eqn:first_est_f_n_eta_n} yields equation \eqref{eqn:lower_bound_D_y}.

Next, as explained in  Step 1 
above, all points in $\mc{I}^1_n$ satisfy the conclusions of 
Proposition~\ref{prop:set_up_scale_prop}\eqref{item:nearby_points_are_also_tempered}. Thus,
\[
\|D_xf^n_{\omega}\|\ge \|D_yf^n_{\omega}\|^{(1-\sigma)}.
\]
Combining this with \eqref{eqn:lower_bound_D_y} gives:
\[
\|D_xf^n_{\omega}\|\ge C_a^{(1-\sigma)}\eta_n^{-(1-\sigma)^2}(y).
\]
Then applying  
Proposition \ref{prop:set_up_scale_prop}\eqref{item:fluctuations_in_holonomies} gives the conclusion.
\end{proof}

We now continue with the proof that $\mc{I}^2_{n}\subseteq \mc{I}^2_{n-1}$. 
First, note that by the triangle inequality,
\[
d_{\gamma_1}(x,\partial B_{\delta_{n-1}(y)-K_1e^{-\Lambda_{\max}(n-1)}}^{\gamma_1}(y))\ge \delta_{n-1}(y)-K_1e^{-4(n-1)\Lambda_{\max}}-\delta_{n}(y).
\]
We then apply $H^s_{n-1}$. By Proposition \ref{prop:set_up_scale_prop}\eqref{item:lower_bound_on_derivative} it follows that 
\begin{align}\label{eqn:dist_from_boundary_B_n}
&d_{\gamma_2}(H^s_{n-1}(x),\partial H^s_{n-1}(B_{\delta_{n-1}(y)-K_1e^{-4\Lambda_{\max}(n-1)}}^{\gamma_1}(y)))\\
\ge & D_1(\delta_{n-1}(y)-K_1e^{-4(n-1)\Lambda_{\max}}-\delta_{n}(y)).\notag
\end{align}
But by Claim \ref{claim:control_on_fluctuations}, 
$\displaystyle
d_{\gamma_2}(H^s_{n-1}(x),H^s_{n}(x))\le C_2\eta_n^{1.99(1-\sigma)^2}(y).
$
Hence by the triangle inequality
$$
d_{\gamma_2}(H^s_{n}(x),\partial H^s_{n-1}(B_{\delta_{n-1}(y)-K_1e^{-4\Lambda_{\max}(n-1)}}^{\gamma_1}(y)))
$$$$
\ge D_1(\delta_{n-1}(y)-K_1e^{-4(n-1)\Lambda_{\max}}-\delta_{n}(y))-C_2\eta_n^{-1.99(1-\sigma)^2}(y).
$$
By  
\eqref{eqn:choice_of_K_1_nbd}  $\displaystyle\delta_{n-1}-K_1e^{-(n-1)\Lambda_{\max}}-\delta_n
\ge \left(1-e^{-\lambda/4}\right)\delta_{n-1}.$
Hence as $\eta_n^{1.99(1-\sigma)^2}$ is  of a higher order
than $\delta_n$, there exists some $N_1$ such that for $n\ge N_1$, 
\begin{equation}
d_{\gamma_2}(H^s_{n}(x),\partial \mc{I}^2_{n-1})\ge 2^{-1}D_1
\left(1-e^{-\lambda/4}\right)\delta_{n-1}(y)>0.
\end{equation}
This shows that $H^s_{n}(\mc{I}^1_n)\subset \mc{I}^2_{n-1}$ as desired.\\

\noindent\textbf{Step 3. Lengths of curves in $\mc{I}_n^i\setminus \mc{I}_{n-1}^i$.}
This is needed to estimate the regularity of $P_n^i\setminus P_{n+1}^i$. 

First we consider the size of the trimmed segments when we pass from $\hat{\mc{I}}^1_n$ to $\mc{I}^1_n$. Any connected component of $\hat{\mc{I}}^1_n$ has length at least $\delta_n(x)$ for some $x$. Note that this is bounded below by an exponential $e^{-\Lambda_{\max} n}$. Then as we trim a remaining $K_1e^{-4\Lambda_{\max} n}$ length off these intervals when we pass from $\hat{\mc{I}}^1_n$ to $\mc{I}^1_n$, we see that each interval we trim has length at least $K_1e^{-4\Lambda_{\max} n}$.

There are two ways that $x\in \mc{I}^1_{n-1}$ may fail to be in $\mc{I}^1_{n}$. Write $\mc{I}^1_{n-1}(x)$ for the connected component of $\mc{I}^1_{n-1}$ containing $x$.  Then either $\mc{I}_{n-1}(x)$ contains a point $y$ that is in $\mc{I}^1_n$ or 
the entire component containing $x$ is deleted.
In the first case the connected component of $\mc{I}^1_{n-1}\setminus \mc{I}^1_n$ containing $x$ has length at least $k_1e^{-4\Lambda_{\max} n}$ by the previous paragraph. In the second case, the removed segment is at least $e^{-\Lambda_{\max}n}$ long. Thus we have obtained an exponential lower bound on the lengths of curves in $\mc{I}^1_{n-1}\setminus \mc{I}^1_n$. 

As $H^s_{n}(\mc{I}^1_n)=\mc{I}^2_n$ we can use the size of the gaps in $\mc{I}^1_n$ to estimates the size of those in $\mc{I}^2_n$.
Note that from estimate \eqref{eqn:dist_from_boundary_B_n},  each segment in $\mc{I}^2_{n-1}\setminus \mc{I}^2_n$ has width at least $D_1K_1e^{-4\Lambda_{\max}n}$.

\noindent\textbf{Step 4. Definition of the densities.}
So far we have defined the underlying curves $\mc{I}^1_n,\mc{I}^2_n$ that the standard families $P^1_n,P^2_n$ will be defined on. We now define the densities on $\mc{I}^1_n$ and $\mc{I}^2_n$.
To begin, we will define $\rho^1_N$ and $\rho^2_N$ where $N$ is the first time we attempt to fake-couple.
From Proposition \ref{prop:set_up_scale_prop}\eqref{item:lower_bound_on_derivative}
, there exists $a_0>0$ such that for $B\subseteq \mc{I}^1_n$,
\begin{equation}\label{eqn:comparison_of_densities}
a_0(H^s_{n})_*\rho^1\vert_{B}\le \rho^2\vert_{H^s_{n-1}(B)}.
\end{equation}
We then take as our initial definition:
\begin{equation}\label{defn:rho1_N_and_rho2_N}
\rho^1_N=a_0\rho^1\vert_{\mc{I}^1_N}\text{ and }\rho^2_N=(H^s_{N})_*\rho^1_N.
\end{equation}
This gives us $\rho^1_N$ and $\rho^2_N$.

We now define $\rho^1_n$ and $\rho^2_n$ for $n\ge N$. We set:
\begin{equation}\label{eqn:defn_rho_1_n}
\rho^1_n=\rho^1_{n-1}(1-e^{-n\hat{\eta}})\vert_{\mc{I}^1_n}
\end{equation} where $\hat{\eta}$ is chosen in equation \eqref{eqn:choice_of_hat_eta} below. We then define
\begin{equation}
\rho^2_n=(H^s_{n})_*(\rho^1_{n}).
\end{equation}
As we push forward $\rho^1_{n}$ by the holonomy $H^s_{n}$, which carries $\mc{I}^1_n$ to $\mc{I}^2_n$, $\rho^2_n$ is a measure on $\mc{I}^2_n$. This defines completely $P^1_n$ and $P^2_n$.

The rest of the proof will be checking that the standard families $P^1_n$ and $P^2_n$ have the required properties to be a fake coupling. Some are evident from the definition above, but it remains to check:

(1) the regularity of $\rho^1_n$ and $\rho^2_n$, 

(2) that $\rho^2_n$ is a decreasing sequence of measures, and 

(3) the goodness of the standard families $P^i_n\setminus P^i_{n-1}$ for $i\in \{1,2\}$.\\

\noindent\textbf{Step 5: Regularity of $\rho^1_n$ and $\rho^2_n$.}
In this step we study the log-H\"older constants of $\rho^1_n$ and $\rho^2_n$ for $n\ge N$. 
Note that $\rho^1_n$ is $\rho^1_{N}$ scaled by a constant that it has the same log-H\"older constant as $\rho^1_{N}$.

Before proceeding to study the regularity of $\rho^2_n$, we introduce some notation related to the Jacobian of the holonomies. Typically the Jacobian of an invertible, absolutely continuous map $\phi\colon (X,\nu)\to (Y,\mu)$ is the Radon-Nikodym derivative $d\phi^*\mu/d\nu$. 
In our case, as we are pushing forward the density $\rho^1_n$ by $H^s_{n}$, the result is the same thing as pulling back $\rho^1_n$ by $(H^s_{n})^{-1}$. 
To simplify notation, we will simply write $J_{n}$ for the Jacobian of $(H^s_{n})^{-1}$, which is a function $J_{n}\colon \mc{I}^2_n\to \R_{>0}$. 
Returning to $\rho^2_n$, this function satisfies for $y\in \mc{I}^2_n$ that
\begin{equation}\label{eqn:rho_2_formula}
\rho^2_n(y)=J_{n}(y)\rho^1_n((H^s_{n})^{-1}(y)).
\end{equation}
As the assumptions on the holonomies are symmetric in $\gamma_1$ and $\gamma_2$, we know from 
Proposition~\ref{prop:set_up_scale_prop}\eqref{item:lower_bound_on_derivative} that $H^s_{n}$ is $D_1$-bilipschitz. 
 Thus by Proposition \ref{prop:set_up_scale_prop}\eqref{item:hold_cont_jac}, there exists $D_2$ such that $J_n$ is log-$\alpha$-H\"older with constant $D_2$ for all $n\ge N$. Next,
since $\rho^1_n$ is log-$\alpha$-H\"older with constant $C_0$, $\rho^1_n\circ (H^s_{n})^{-1}$ is log-$\alpha$-H\"older with constant $D_1^{\alpha}C_0$. 
As mentioned before, $J_{n}$ is log-$\alpha$-H\"older with constant $D_2$.
The product of log-$\alpha$-H\"older functions is log-$\alpha$-H\"older with constant equal to the sum of the constants. 
Thus by \eqref{eqn:rho_2_formula}, we see that $\rho^2_n$ is $D_1^{\alpha}C_0+D_2$ log-$\alpha$-H\"older. Thus we have obtained uniform log-$\alpha$-H\"older control for $\rho^1_n$ and $\rho^2_n$.

We need one more estimate before we continue: an actual H\"older, rather than log-H\"older, bound on $\rho^1_n$ and $\rho^2_n$; we need this as at a certain point we will compare the difference of these functions rather than their ratio. We obtain this bound by rescaling the functions by a constant; however we need to be sure the constant is not too big.

From \eqref{eqn:comparison_of_densities}, it follows from the $C_0$ log-$\alpha$-H\"older constant of the density that there exists $D\ge 1$ such for any $x\in \hat{\gamma}_1$ and $y\in \hat{\gamma}_2$,
\begin{equation}
D^{-1}\le \frac{\rho^1(x)}{\rho^2(y)}\le D.
\end{equation}
Note that for a log-$\alpha$-H\"older function $\rho\colon K\to (0,\infty)$ on a set $K$ of diameter at most $1$ that there exists $D$ depending only on the log-H\"older constant of $\rho$ such that 
\[
D^{-1}\le \rho/\max \rho\le 1.
\]
If we let $M$ denote the larger of the maximum of $\rho^1_n$ and the maximum of $\rho^2_n$, then we may define for $i\in \{1,2\}$, $\wt{\rho}^i_n=\rho^i/M$. Then as the maximums of $\rho^1$ and $\rho^2$ are uniformly comparable, note that there exists $D>0$ depending only on $C_0$ such that for $i\in \{1,2\}$,
\[
D^{-1}\le \wt{\rho}^i\le 1.
\]
In particular, as as $\exp$ is $1$-Lipschitz on $(-\infty,0]$, it follows that $\wt{\rho}_n^1,\wt{\rho}_n^2$ are both uniformly $\alpha$-H\"older with the same constant as their log-H\"older constant.  Below we will work with these rescaled functions that have maximum $1$ and just write $\rho^1_n$ instead of $\wt{\rho}^1_n$. Note that we have not gained any extra regularity for free: to get the lower bound $D$ depending only on the log-H\"older constant on both at the same time used substantial input from our setup.
\\

\noindent\textbf{Step 6. Sign and regularity of $\rho_{n-1}^2-\rho_{n}^2$.}
We now analyze $\rho^2_{n-1}-\rho^2_{n}$. In particular, we show that $\rho^2_n$ is a decreasing sequence of densities.
To begin, we will obtain a lower bound on $\rho^2_{n-1}-\rho^2_n$.
Then we will use the various lemmas relating H\"older and log-H\"older functions to conclude a bound on the regularity of $\rho^2_{n-1}-\rho^2_{n}$. 
By definition:
\begin{align*}
\rho^2_{n-1}-\rho^2_{n}=&\rho_{n-1}^1((H^s_{n-1})^{-1}y)J_{n-1}(y)-(1-e^{-n\hat{\eta}})\rho_{n-1}^1((H^s_{n})^{-1}y)J_{n}(y)\\
=&[J_{n-1}(y)(\rho^1_{n-1}((H^s_{n-1})^{-1}(y))-\rho^1_{n-1}((H^s_{n})^{-1}(y)))]+\\
&[\rho^1_{n-1}((H^s_{n})^{-1}(y))(J_{n-1}(y)-J_{n}(y))]+[e^{-n\hat{\eta}}\rho_{n-1}^1((H^s_{n})^{-1}y)J_{n}(y)]\\
=&A+B+C.
\end{align*}
We next estimate $A,B$, and $C$.

\textbf{Term $A$.} 
To estimate term $A$, we first pull the function back to $\gamma_1$ by composing with $H^s_{n}$. 
Let $Q_n=(H^s_{n-1})^{-1}\circ H^s_{n}$. 
For $y\in \mc{I}^1_n$, there exists $y'\in G^n_{\omega}$ satisfying the hypotheses of Claim \ref{claim:control_on_fluctuations} such that $d_{\gamma_2}(H^s_{n}(y),H^s_{n-1}(y))<C_1\eta_n(y')^{-1.99(1-\sigma)^2}$. By Lipschitzness of the holonomies from 
 Proposition \ref{prop:set_up_scale_prop}\eqref{item:lower_bound_on_derivative}, this implies that 
\begin{equation}\label{eqn:est_on_Q}
d_{\gamma_1}(Q(y),y)<D_1C_1\eta_n(y')^{1.99(1-\sigma)^2}.
\end{equation}
Precomposing again with $(H^s_{n})^{-1}$ gives that for $y\in \mc{I}^2_n$,
\begin{equation}\label{eqn:comparison_of_holonomies12}
d_{\gamma_1}((H^s_{n-1})^{-1}(y),(H^s_{n})^{-1}(y))\le D_1^2C_1\eta_n(y')^{1.99(1-\sigma)^2}.
\end{equation}
But this implies, using Lemma \ref{lem:log_holder_C_A_est} and \eqref{eqn:comparison_of_holonomies12} in the second line, that:
\begin{align}
\abs{A}&= \abs{J_{n-1}(y)(\rho^1_{n-1}((H^s_{n-1})^{-1}(y))-\rho^1_{n-1}((H^s_{n})^{-1}(y)))}\\
&\le \abs{J_{n-1}(y)}(D_1^2C_2\eta_n(y')^{1.99(1-\sigma)^2})^{\alpha}\rho^1_{n-1}((H^s_{n})^{-1}y)\\
&\le \abs{J_{n-1}(y)}C_3\eta_n^{1.99(1-\sigma)^2\alpha}\rho^1_{n-1}(H^s_{n}(y))\\
&\le C_Ae^{-1.99\lambda'(1-\sigma)^2\alpha n}\rho^1_{n-1}((H^s_{n})^{-1}y).
\end{align}
where we have used temperedness to pass to the last line.
We now turn to the next term.

\textbf{Term $B$}.  This term is simpler. We use 
\eqref{JacExp} in the third step below:
$$
\abs{B}\le \abs{\rho^1_{n-1}((H^s_{n})^{-1}(y))(J_{n-1}(y)-J_{n}(y))}
\le \abs{\rho^1_{n-1}}\abs{J_{n-1}(y)-J_{n}(y)}
$$$$
\le \abs{\rho^1_{n-1}}e^{-n\eta}
\le e^{-n\eta}\rho^1_{n-1}((H^s_{n})^{-1}(y)).
$$

\textbf{Term $C$}. 
The final term is straightforward 
\[
C=e^{-\hat{\eta}n}\rho_{n-1}^1((H^s_{n})^{-1}y)J_{n}(y)\le D_2 e^{-\hat{\eta} n}\rho_{n-1}^1((H^s_{n})^{-1}(y)).
\]

We can now conclude. Combining the estimates on $A,B,C$, we see that
\[
\rho^2_{n-1}(y)-\rho^2_n(y)\ge [D_2e^{-\hat{\eta} n}-e^{-\eta n}-C_Ae^{-1.99\lambda'(1-\sigma)^2\alpha n}]\rho^1_{n-1}((H^s_{n})^{-1}y).
\]
In particular, as long as 
\begin{equation}\label{eqn:choice_of_hat_eta}
0<\hat{\eta}<\min\{\eta/2,-1.99\lambda'(1-\sigma)^2\alpha/2\},
\end{equation}
it follows that there exists $N_{2}$ such that for $n\ge N_2$, 
\[
\rho^2_{n-1}-\rho^2_{n}\ge e^{-2\hat{\eta} n}\rho^2_{n-1}.
\]
Also because $J_n, \rho^2_n,\rho^2_{n-1}$ are uniformly bounded, there exists $D_3$ such that
\[
D_3\ge \rho^2_{n}-\rho^2_{n-1}\ge e^{-2\hat{\eta} n}D_3^{-1}.
\]
Thus we can apply Claim \ref{claim:bounded_below_plus_holder_gives_log_holder} to the function $(\rho^2_n-\rho^2_{n-1})\ge D_3^{-1}e^{-2\hat{\eta}n}$. As $\rho^2_n$ and $\rho^2_{n-1}$ are uniformly $\alpha$-H\"older from Step 5, we obtain that there exists $D_4$ such that $\rho^2_{n-1}-\rho^2_n$ is uniformly $D_4e^{2\hat{\eta} n}$ log-$\alpha$-H\"older. This concludes the analysis of the H\"older regularity of $\rho^2_{n-1}-\rho^2_n$.
\\

\noindent
\textbf{Step 7: Bookkeeping.}
In this step we verify that for each point $y\in \mc{I}_n^1$ that a positive proportion of the mass over $y$ is retained during the fake coupling procedure. This is straightforward to see because at each step, we discard $e^{-n\hat{\eta}/2}$ proportion of the remaining mass in $\rho^1_n(y)$. Thus from the definition \ref{defn:rho1_N_and_rho2_N} of $\rho^1_N$ the amount of mass is bounded below by
\[
\rho^1_{n}(y)\ge a_0\rho^1(y)\prod_{n\ge N}(1-e^{-n\hat{\eta}})>0.
\]
Thus we keep a positive proportion of the mass above each $y\in \mc{I}^n_{\omega}$ for all $n\ge N$. 
\smallskip

\textbf{Step 8: $n=\infty$ behavior} As the sequences $\rho^1_n$ and $\rho^2_n$ are decreasing they converge to some limiting measures $\rho^1_{\infty}$ and $\rho^2_{\infty}$. Further, by Proposition \ref{prop:holonomies_converge_exponentially_fast}, 
the true stable holonomies $H^s_{\infty}$ satisfy $(H^s_{\infty})_*\rho^1_n=\rho^2_n$ as required.
\smallskip

\textbf{Step 9: $(C,\lambda,\epsilon,\mc{C}_{\theta})$-tempered points are never dropped.} Finally, we must show that we actually keep the $(C,\lambda,\epsilon,\mc{C}_{\theta})$ tempered points throughout the entire procedure, so that part (3) of the requirements for a fake coupling are satisfied. Suppose that $(\omega,x)$ is such a $(C,\lambda,\epsilon,\mc{C}_{\theta})$-tempered trajectory. It suffices to show that for each $n$ that all points in $B_{\delta_n(x)}(x)$ are $(C',\lambda',\epsilon',\mc{C}_{2\theta})$-tempered, as from the procedure above this ensures that $x\in \mc{I}^1_n$ for all $n$.
By Part \eqref{item:nearby_points_are_also_tempered} of Proposition \ref{prop:set_up_scale_prop}, this follows as long as $\delta_n(x)\le \|D_xf^n_{\omega}\|^{-(1+\sigma)}$. This inequality holds because by the definition of $\eta_n$, \eqref{eqn:defn_eta_n}, 
$\eta_n(x)\le \|D_xf^n_{\omega}\|^{-1}$, and $\delta_n(x)=\eta_n^{(1+\sigma)}$.
\smallskip

Thus we have verified all of the required claims in the definition of fake coupling as well as the additional required claim about the goodness of the families $P^i_{n-1}\setminus P^i_{n}$, we conclude the proof.
\end{proof}

We now have everything ready to prove the local coupling lemma, Lemma \ref{ref:small_scale_coupling_lemma}.

\begin{proof}[Proof of Lemma \ref{ref:small_scale_coupling_lemma}]
Almost everything in the statement of Lemma \ref{ref:small_scale_coupling_lemma} is contained in the statement of Lemma \ref{lem:local_coupling_lemma}. We explain them in order.

Item \ref{item:equal_mass_stops_lcl} follows because the points we stop trying to couple at time $n$ are precisely the points in $\hat{\gamma}_i$ that are in $P^i_n\setminus P^i_{n-1}$. 
As the standard family $P^i_n\setminus P^i_{n-1}$ is $n\Lambda$-good, the claim follows with $L=\Lambda$.

Item \ref{item:tail_points_intertwined} is the statement in the final paragraph of Lemma \ref{lem:local_coupling_lemma}.

Item \ref{item:good_tail} is more complicated. There are two ways that a point $x\in \mc{I}^1_{n-1}$ fails to appear in $\mc{I}^1_n$. The first is that $x$ is not in any interval $B_{\delta_n(y)}^{\gamma_1}(y)$ for any $y\in H^n_{\omega}$. The second is if $x$ is in an interval that gets trimmed off of $\hat{\mc{I}}^1_n$.

First we consider the former case. This means that some $y$ such that $x\in B_{\delta_{n-1}(y)}^{\gamma_1}(y)$ failed to be tempered at time $n$. In $\Sigma\times \gamma_1$, we consider the union of these intervals:
\[
U_n=\bigcup \{\{\omega\}\times B_{\delta_{n-1}(y)}^{\gamma_1}(y): y\in \mc{I}_{n-1}^1\setminus \mc{I}_{n}^1 \text{ for the word }\omega\}.
\]
Note that as each of these sets $B_{\delta_{n-1}(y)}^{\gamma_1}(y)$ contains a point $z$ that fails to be tempered at time $n$ that $(\omega,z)$ has cushion that is within $\Lambda_{\max}$ of $C'$, the cutoff for tempering to fail. By Proposition \ref{prop:nearby_points_close_cushion}, as all the points in $B^{\gamma_1}_{\delta_{n-1}(y)}(y)$ satisfy the hypotheses of that proposition due to the size of $\delta_{n-1}(y)\le \|D_yf^n_{\omega}\|^{-(1+\sigma)}$ and the tempering, this implies that all points in $B_{\delta_{n-1}(y)}^{\gamma_1}(y)$ have cushion at most $C'+\Lambda_{\max}+D$. But by Proposition \ref{prop:large_deviations_for_cushion}, the number of points having cushion of this size is exponentially small. Thus $\mu\otimes \rho(U_n)\le D_1e^{-n\eta}$ for some $D_1,\eta>0$, and we have an exponential tail for points experiencing the first type of failure.

In the case that a point fails to be included because it was trimmed off, it was observed in Step 3 of the coupling construction, that every curve being trimmed has length at least $e^{-(1+\sigma)\Lambda_{\max}n}$ and the amount we cut off has length $2K_1e^{-4\lambda_{\max}n}$. Thus as $1/2\le \rho(x)/\rho(y)\le 2$ for two points $x,y$ along the curve we are coupling, the amount we trim has mass at most $4e^{-2\Lambda_{\max} n}$ times the mass of the curve. Thus summing over all curves we stop on at most $4e^{-2\Lambda_{\max} n}$ mass, which is exponentially small.

The last way that mass is lost during the local coupling procedure is when we rescale the density by $(1-e^{-n\hat{\eta}})$ in Step 4, which also gives at most an exponentially small amount of mass is stopping at time $n$. This concludes the proof of the tail bound.

Item \ref{item:positive_prob_of_coupling} follows from  Proposition 
\ref{prop:set_up_scale_prop}\eqref{item:positive_probability_of_tangency}.
\end{proof}

\section{Mixing theorems}\label{sec:claims_for_use_in_limi_theorem}
\subsection{Overview of the section}
 In this section we prove our main result, Theorem \ref{ThQEM}. The proof will rely on coupling and expansion following the standard argument, see e.g.~\cite{chernov2006chaotic}.

First, we show that coupling implies equidistribution of standard families by coupling a given family to a family 
representing volume and using that volume is invariant by the dynamics. 
See Proposition 
\ref{prop:quenched_exp_equidistribution_on_subfamilies} for details.

Next, we use the expansion and exponential equidistribution to obtain exponential mixing using the following reasoning.
Consider an $R$-good standard family $\hat\gamma$ and let $f^n_{\omega}(\hat{\gamma})$ be its image after $n$
iterations. We shall show that for almost all $\omega$ that $f^n_{\omega}(\hat{\gamma})$
contains a subfamily $P_n$ with the following properties:
\begin{enumerate}
    \item $P_n$ consists of $\epsilon n$-good standard pairs
    \item standard pairs in $P_n$ contract backwards in time
    \item the forward image of pairs from $P_n$ equidistribute at an exponential rate
    \item the complement of $P_n$ has exponentially small measure.
\end{enumerate}

Now given H\"older functions $\phi$ and $\psi$ we obtain exponential decorrelation between $\phi\circ f^N_{\omega}$
with $N=cn$
and $\psi$ using that $\psi$ is constant on the elements of $P_n$ (up to exponentially small error),
$\phi\circ f^N_{\omega}$ is equidistributed on the elements of $P_n$ (up to exponentially small error),
and the complement of $P_n$ is exponentially small.

The purpose of this section is to execute this argument precisely, using the results of Sections~\ref{sec:temperedness}, 
\ref{sec:finite_time_smoothing_estimates},
and the appendices.

\subsection{Preparatory lemmas}
Below we will use Definition \ref{defn:forward_tempered_relative_to_a_curve} from \S \ref{SSLoss} in the appendix. 
Briefly, this definition concerns a $(C,\lambda,\epsilon,\theta)$\emph{-forward tempered point} at time $n$ for a vector $v\in T_xM$, which is a $(C,\lambda,\epsilon)$-forward tempered time $n$ such that $E^s_n$ makes angle at least $\theta$ with $v$.

\begin{prop}\label{prop:most_points_forward_subtempered}
Suppose that $M$ is a closed surface and 
$(f_1,\ldots,f_m)$ is an expanding on average tuple of diffeomorphisms in $\Diff^2_{\vol}(M)$. There exists $\lambda>0$ such that for all sufficiently small 
$\epsilon>0$ there exist $C_0$, $N\in \N$, and $\alpha>0$ such that for all $n\ge N$, and any direction $v\in T_x^1M$, 
\begin{equation*}
\mu(\omega: (\omega,x) \text{ is \emph{not} } (\epsilon n+C_0,\lambda,\epsilon,
{ C_0 e^{-\epsilon n}})\text{-forward tempered at time $n$ relative to $v$})
\le e^{-n\alpha}.
\end{equation*}
\end{prop}

\begin{proof}
Proposition \ref{prop:splitting_with_high_probability} says that there exist $\lambda>0$ such that for arbitrarily small $\epsilon>0$, there exists $\alpha>0$ such that the measure of the words $\omega$ that are not $(C,\lambda,\epsilon)$-subtempered for all $n\ge 0$ is at most $e^{-\alpha C}$. 
From Proposition \ref{prop:finite_time_distribution_of_stable} there exists some $C_2,c,\theta>0$ such that for all sufficiently small $\epsilon'$ as long as $n\ge c\abs{\ln (\epsilon')}=N_0$ , then for all $n\ge N_0$, the probability that $E^s_n\in B_{\epsilon'}(v)$ is at most $C_2(\epsilon')^{\theta}$. Taking $\epsilon'=e^{-\epsilon n}$, this gives that the probability that  
 
 $E^s_n\in B_{e^{-\epsilon n}}(v)$ for $n\geq N_0$ is at most $C_2e^{-\theta \epsilon n}$
 as long as $\epsilon$ is sufficiently small relative to $c$, $n\ge c\epsilon n$. Combining these two estimates, we obtain the result.
\end{proof}

 Below, we will typically assume that the standard family or standard pair we are considering has unit mass. The statements below can be adapted to any amount of mass by multiplying the right hand side of the bound by the mass of the family.

\begin{defn}
Given a standard pair $\hat{\gamma}=(\gamma,\rho)$, for $x\in \gamma$ we say that $(\omega,x)$ is $(n,\lambda,\epsilon)$\emph{-backwards good} if 
\begin{enumerate}
    \item 
$f^n_{\omega}(x)$ is contained in a standard pair $B(\omega,x)\subseteq f^n_{\omega}(\hat{\gamma})$ that is $\epsilon n$-good, and  
    \item 
    $Df^{-n}_{\omega}B_{e^{-14\epsilon n}}^{\gamma}(x)$ has diameter at most $e^{-(\lambda/2) n}$. 
\end{enumerate}
\noindent We define analogously the same notion for a standard family.
\end{defn}

\begin{prop}\label{prop:average_curves_contract_backwards}
(Annealed goodness) 
Suppose that $M$ is a closed surface and $(f_1,\ldots,f_m)$ is an expanding on average tuple in $\Diff_{\vol}^2(M)$. 
Then there exists $\lambda>0$ such that for all sufficiently small $\epsilon>0$, if we fix $R>0$ there exists $\alpha,C>0$ such that for any 
$R$-good, unit mass standard family $\hat{\gamma}$ with associated measure $\rho$:
\begin{equation}\label{eqn:exponentially_many_backwards_good}
(\mu\otimes \rho)(\{(x,\omega):(x,\omega) \text{ is not }(n,\lambda,\epsilon)\text{-backwards good}\})\le C
e^{-\alpha n}.
\end{equation}
\end{prop}
\begin{proof}
This  is immediate from Propositions \ref{prop:fwd_up_to_epsilon_smoothing} and \ref{prop:most_points_forward_subtempered}.
\end{proof}

From Proposition \ref{prop:average_curves_contract_backwards}, we can deduce a related quenched statement for almost every $\omega$. 
\begin{lem}\label{lem:quenched_epsilon_lambda_n_good}
(Quenched goodness) Under the hypotheses of Proposition \ref{prop:average_curves_contract_backwards}, 
there exist $\lambda,\alpha,D>0$ such that for all sufficiently small $\epsilon>0$ and a unit mass $R$-good standard family $\hat{\gamma}$, then for almost every $\omega$, there exists $C_{\omega}$ such that $1-C_{\omega}e^{-\alpha n}$ proportion of points in $\hat{\gamma}$ are $(n,\lambda,\epsilon)$-backwards good for $\omega$.
Further,
\[
\mu(\omega: C_{\omega}>C)\le DC^{-1}.
\]
\end{lem}
\begin{proof}
 Let $A_n^{\omega}$ be the set of points in $\hat{\gamma}$ that are not $(n,\lambda, \epsilon)$-backwards good for $\omega$. 
 Then

$$
\mu(\omega:\exists n\text{ }\rho(A^\omega_n)>Ce^{-(\alpha/2)n})\le \sum_{n\ge 0}\mu(\omega: \rho(A^{\omega}_n)>Ce^{-(\alpha/2)n})
\le \sum_{n\ge 0} C^{-1}C_1e^{-(\alpha/2) n}
\le C^{-1}D,
$$
where the second inequality follows from \eqref{eqn:exponentially_many_backwards_good} and the Markov 
inequality. The result follows.
\end{proof}

We also need another proposition, that says that on the $\epsilon n$-good neighborhoods at time $n$ that we have rapid coupling, which will then imply that these neighborhoods rapidly equidistribute. The following estimate is immediate from Proposition \ref{prop:main_coupling_proposition}.

\begin{prop}\label{prop:coupling_on_pushed_curves_average}
Suppose $(f_1,\ldots,f_m)$ is as in Proposition \ref{prop:average_curves_contract_backwards}. 
Then there exists $\lambda>0$ such that for any sufficiently small $\epsilon>0$ there exist $C,\alpha>0$ such that the following holds.
For any $n\in \N$, suppose $P^1$ and $P^2$ are two unit mass standard families of $\epsilon n$-good curves. 
Then there exists a coupling function $\Upsilon$ and stopping times $\hat{T}^1,\hat{T}^2$ as in Proposition \ref{prop:main_coupling_proposition} such that
for $i\in \{1,2\}$:
\[
(\mu\otimes \rho^i)(\{(x,\omega): \hat{T}^i(x,\omega)>j\})\le Ce^{\epsilon n}e^{-\alpha j}.
\]
\end{prop}

\begin{rem}\label{rem:defn_of_T_i_hat}
In the applications of Proposition  \ref{prop:coupling_on_pushed_curves_average} 
below we will assume unless it is explicitly stated otherwise that $P^2$ is the family representing
the volume from Proposition~\ref{PrVolStandard}. We couple with a family representing volume because it implies that the statistics of an arbitrary standard family $P_1$ approach those of volume.

In what follows for a word $\omega$ at time $i$, we have subfamilies $P^1_{i,\omega}$ and $P^2_{i,\omega}$ of $f^i_{\omega}(P^1)$. We then apply Proposition \ref{prop:coupling_on_pushed_curves_average} above, to find a pair of stopping times $\hat{T}^1_i$ and $\hat{T}^2_i$  defined on $f^i_{\omega}(P^1_{i,\omega})$ and $f^i_{\omega}(P^2_{i,\omega})$ respectively. Note that the the $\hat{T}^i$ are not defined 
on all of $f^i_{\omega}(\hat{\gamma})$ because not all points in this pair need be $\epsilon n$-good. 
\end{rem}

Then from Proposition \ref{prop:coupling_on_pushed_curves_average} we obtain the following.

\begin{prop}\label{prop:tail_P_i_omega}
Let $(f_1,\ldots,f_m)$, $\hat{\gamma}$, $\rho$, and $\lambda,\epsilon,\alpha>0$ as be as in Proposition \ref{prop:average_curves_contract_backwards}, then there exists $C$ such that if we let the $\hat{T}_n^1$ be the stopping time defined as in Remark \ref{rem:defn_of_T_i_hat}, for all $i,n\ge 0$ we have the bound:
\begin{equation}\label{eqn:subfamily_coupling_time_bound}
(\mu\otimes \rho)((x,\omega): x\in P^1_{i,\omega}\text{ and }
\hat{T}_i^1(x,\omega)>i+n)\le Ce^{\epsilon i}e^{-n\alpha}. 
\end{equation}
\end{prop}

From this, we easily deduce a statement about each $\omega$.

\begin{prop}\label{prop:quenched_coupling_lemma}
Let $(f_1,\ldots,f_m)$, $\lambda,\epsilon>0$ and $\hat{\gamma},\rho$ be as in the setting of Proposition \ref{prop:average_curves_contract_backwards} and Remark \ref{rem:defn_of_T_i_hat}, then there exists $\alpha,D_1>0$ such that
\begin{align}
\label{BackGood}
\mu(&\omega:{\text{there exists } i\text{ such that }\rho(x: (x, \omega) \text{ is }(i,\lambda,\epsilon)\text{-backwards good})<1-Ce^{-i\alpha}} 
\text{ or }
\\
\label{LongWait}
&\text{there exist } (i,n)\text{ such that }  \rho(x\in P_{i,\omega}^1:\hat{T}_i^1(x,\omega)\ge i+n)\ge C^2e^{\epsilon i}e^{-n\alpha})\le D_1C^{-1}.
\end{align}
\end{prop}

\begin{proof}
To control the event in  \eqref{LongWait}
 let $B_{i,n}^{\omega}=\{x\in P^1_{i,\omega}: \hat{T}_i^1(x,\omega)>i+n\}$.
By \eqref{eqn:subfamily_coupling_time_bound} and the Markov inequality, there is $C_1>0$ such that
\begin{equation}\label{eqn:upper_bound_on_i_n}
\mu(\{\omega: \rho(B^{\omega}_{i,n})>Ce^{2\epsilon i}e^{-(\alpha/2)n}\})\le C_1C^{-1}e^{-\epsilon i}e^{-(\alpha/2)n}.
\end{equation}
Then using \eqref{eqn:upper_bound_on_i_n}, we find that
$$
\mu(\omega:\text{for some $i,n$ } \rho(\{x:\hat{T}_i^1(x,\omega)\ge i+n\})\ge C^2e^{\epsilon i}e^{-n\alpha/2})
$$$$
\le \sum_{i\ge 0}\sum_{n\ge 0}\mu(\{\omega:\rho(B^{\omega}_{i,n})\ge C^2e^{\epsilon i}e^{-(\alpha/2)n}\})
\le \sum_{i\ge 0}\sum_{n\ge 0} C_1C^{-1}e^{-\epsilon i}e^{-(\alpha/2)n}
\le C^{-1}C_2
$$
for some $C_2$  provided that $\epsilon$ is small enough.
Combining this estimate with
Proposition \ref{prop:average_curves_contract_backwards} to control the event in \eqref{BackGood} allows us to conclude.
\end{proof}

\subsection{Quenched equidistribution}
 Using the quenched coupling lemmas above, it is straightforward to deduce quenched equidistribution and correlation decay theorems.
The ideas in the proofs below are essentially standard, compare with \cite[Ch.~7]{chernov2006chaotic}, however some modifications are necessary because the quenched random dynamics is not stationary. 

We start with quenched equidistribution.

\begin{prop}\label{prop:quenched_exp_equidistribution_on_subfamilies}
(Quenched exponential equidistribution on subfamilies) Let $(f_1,\ldots,f_m)$ be an expanding on average tuple in $\Diff^2_{\vol}(M)$, where $M$ is a closed surface. 
There exists $\lambda>0$ such that for all sufficiently small $\epsilon>0$, fixed $\beta\in (0,1)$ and $R$, there exists $D_1$ such that for any $R$-good, unit mass standard family $\hat{\gamma}$, there exists $\alpha,\nu>0$ such that
for almost every $\omega$, there exists $C_{\omega}\ge 1$ such that, such that coupling as in Remark \ref{rem:defn_of_T_i_hat}:
\begin{enumerate}[leftmargin=*]
    \item
    There exists a subfamily $P_{i,\omega}$ of $(f^i_{\omega})_*\hat{\gamma}$ of $e^{\epsilon i}$-good standard pairs having total  $\rho$--measure $(1-C_{\omega}e^{-\alpha i})$
    \item 
    The atoms of $(f^i_{\omega})^{-1}(P_{i,\omega})$ have diameter at most $e^{-\lambda/2 i}$.
    \item
    The atoms $A_{i,\omega}\in P_{i,\omega}$ exponentially equidistribute, i.e., letting $\overline{A}_{i,\omega}$ be the normalized measure on $A_{i,\omega}$,
    \begin{equation}
    \abs{\int \phi\circ f^n_{\sigma^i\omega}\,d\overline{A}_{i,\omega}-\int \phi\,d\vol}\le C_{\omega}e^{\epsilon i}e^{-\alpha n}\|\phi\|_{C^{\beta}}.
    \end{equation}
    \item
	We have a tail bound 
	$\displaystyle
	\mu(\{\omega:C_{\omega}>C\})\le D_1C^{-1}.
	$
\end{enumerate}

\end{prop}

\begin{proof}
From Lemma \ref{lem:quenched_epsilon_lambda_n_good}, the only thing that remains to be checked is that the individual atoms of $A_{i,\omega}$ are exponentially equidistributing.

Let $P_{i,\omega}$ be the subfamily of $f^i_{\omega}(\hat{\gamma})$ of curves that are $i\epsilon$-good.
Let $P^2$ be a standard family representing volume as in Remark \ref{rem:defn_of_T_i_hat}.
Then coupling with $P^2$, we have the stopping time $\hat{T}_i$ on $P_{i,\omega}$ as discussed in Proposition \ref{prop:tail_P_i_omega}
and uniform $\alpha,C_{\omega}>0$ such that for all $i,n\in \N$,
\begin{equation}\label{eqn:quenched_upper_bound}
\rho(x\in P_{i,\omega}:\hat{T}_i(x,\omega)>n+i)\le C_{\omega}e^{\epsilon i}e^{-n\alpha}. 
\end{equation}
We would like to know that most of the curves in $P_{i,\omega}$ have all but an exponentially small amount of their points coupling quickly.

We claim that for a.e.~$\omega$ there exists a subfamily $G_{i,\omega}$ of $\epsilon i$-good curves in $P_{i,\omega}$ of measure at least $1-{C_{\omega}} e^{-\alpha i/3}$ such that for each $A\in G_{i,\omega}$ all but $e^{i\epsilon n}e^{-\alpha/3n}$ of the mass of the subfamily has coupled to volume by time $i+n$, i.e.~$\hat{T}_i(x,\omega)\le i+n$.
Suppose that $\omega$ satisfies \eqref{eqn:quenched_upper_bound} and for the sake of contradiction, suppose that there is a subfamily $B_i$ (of bad pairs) of $\hat{P}_{i,\omega}$ having measure more than than $e^{-\alpha i/3}$ so that for some $n$ all pairs in $B_i$ have more than $e^{i\epsilon}e^{-n\alpha  /3}$ proportion of points not coupled at time $n+i$, i.e.~$\hat{T}_i>i+n$. This implies that 
$\displaystyle
\rho(x:\hat{T}_i(x,\omega)>n+i)\ge C_{\omega}e^{-2\alpha n/3}e^{i\epsilon },
$
contradicting \eqref{eqn:quenched_upper_bound}. Thus the claim about $G_{i,\omega}$ holds.

Suppose now that $A_{i,\omega}\in G_{i,\omega}\subseteq P_{i,\omega}$ is such a good atom where at time $n+i$ all but at most $e^{i\epsilon }e^{-(\alpha/3)n}$ proportion of the mass of $A_{i,\omega}$ has coupled to volume. 
Let $A_{i,\omega}^n\subseteq A_{i,\omega}$ be the set of points that have coupled by time $i+n$. Let $\Upsilon$ be the measure preserving coupling function and let $V^n=\Upsilon(A_{i,\omega}^n)$ be the corresponding set of points in the standard family representing volume that have $\hat{T}_i(x,\omega)\le i+n$. 
Then we may write the integral in question as 
$$
    \abs{\int \phi\circ f^n_{\sigma^i\omega}\,d\overline{A}_{i,\omega}-\int \phi\,dP_{\vol}}
    $$$$
    \le \abs{\int_{A^{n/2}_{i,\omega}} \phi\circ f^n_{\sigma^i\omega}\,d\overline{A}_{i,\omega}-\int_{V^{n/2}}\phi\,dP_{\vol}}
  +\abs{\int_{A_{i,\omega}\setminus A_{i,\omega}^{n/2}}\phi\circ f^n_{\sigma^i\omega}\,d\overline{A}_{i,\omega}}+\abs{\int_{(V^{n/2})^c}\phi\,dP_{\vol}}
  $$$$
    \le \abs{\int_{\Upsilon^{-1}(V^{n/2})} \phi\circ f^n_{\sigma^i\omega}(\Upsilon(x))-\phi(x)\,dP_{\vol}}+2C_{\omega}e^{i\epsilon}e^{-n\alpha/6}\|\phi\|_{C^{\beta}}.
$$
As the points $\Upsilon(x)$ and $x$ both lie in a common $(C_0,\lambda,\epsilon)$-tempered local stable leaf of uniformly bounded length at time $i+n/2$, then we see that at time $i+n$, that 
$$d(f^n_{\sigma^i\omega}\Upsilon(x), f^n_{\sigma^i\omega}(x))\le C^{-1}_0e^{-\lambda/2n}.$$ 
Now the H\"older regularity of $\phi$  implies that 
\begin{equation}
   \abs{\int \phi\circ f^n_{\sigma^i\omega}\,d\overline{A}_{i,\omega}-\int \phi\,dP_{\vol}} \le C^{-\beta}_0e^{-\lambda\beta/2n}\|\phi\|_{C^\beta}   +2C_{\omega}e^{i\epsilon}e^{-n\alpha/6}\|\phi\|_{C^{\beta}},
\end{equation}
which is what what we wanted for the pair $A_{i,\omega}$. The required tail bound  on $C_{\omega}$ follows from Proposition \ref{prop:quenched_coupling_lemma} and \eqref{eqn:quenched_upper_bound} 
 by taking $D_1$ sufficiently large because the first term involving $C_0^{\beta}$  is uniformly bounded independent of $C_{\omega}\ge 1$.
\end{proof}

\begin{thm}\label{prop:quenched_equidistribution}
(Quenched, tempered equidistribution) Suppose that $M$ is a closed surface, $(f_1,\ldots,f_m)$ is an expanding on average tuple in $\Diff_{\vol}^2(M)$, and  $\beta\in (0,1)$ is a H\"older regularity. For any $\epsilon>0$ there exists $\eta>0$ such that for any $R$-good standard family $\hat{\gamma}$ with associated measure $\rho$,  this family satisfies quenched, tempered equidistribution. Namely, for a.e.~$\omega\in\Sigma$, there exists $C_{\omega}$ such that for any $\phi\in C^{\beta}(M)$,  for all natural numbers $k$ and $n$,
\[
\abs{\int \phi\circ f^n_{\sigma^k(\omega)}\,d\rho-\int \phi\,d\vol}\le C_{\omega}e^{k\epsilon}e^{-\eta n}\|\phi\|_{C^{\beta}}.
\]
\end{thm}
The above theorem is an immediate consequence of 
 Proposition \ref{prop:quenched_exp_equidistribution_on_subfamilies}, 
so we do not write a separate proof of it.  Next we turn to exponential mixing.

\subsection{Exponential mixing}\label{sec:limit_theorems}

We are now ready to prove  exponential mixing. In a subsequent paper we plan to 
 show that several classical statistical limit theorems are valid in our setting.

\begin{proof}[Proof of Theorem \ref{ThQEM}.]
As before, let $P_{\vol}$ be an $R$-good standard family representing volume. 
We then apply Proposition \ref{prop:quenched_exp_equidistribution_on_subfamilies} with $\hat{\gamma}=P_{\vol}$, and obtain $\lambda,\epsilon,\alpha>0$ such that the conclusions of that proposition hold for these constants. 
Pick some $\omega\in \Sigma$ such that the conclusion of Proposition \ref{prop:quenched_exp_equidistribution_on_subfamilies} holds for $\omega$, and let $C_{\omega}$ be the associated constant. 
We will now show that $f^n_{\omega}$ is exponentially mixing. Let $\delta\in (0,1)$ be some fixed number small enough that $\epsilon\delta-(1-\delta)\alpha<0$.

Below, we will be implicitly rounding to nearest integers so that everything makes sense. 
In particular, we will denote by
$P_{\delta n}$ the standard family $P_{\lfloor \delta n\rfloor,\omega}$ from Proposition \ref{prop:quenched_exp_equidistribution_on_subfamilies}; as $\omega$ is fixed we will omit it below. 

We now record some useful properties of $P_{\delta n}$. First, $P_{\delta n}$ comprises all but $C_{\omega}e^{-\delta \alpha n}$ of the mass of $f^{\delta n}_{\omega}(P_{\vol})$.
Thus, by volume preservation:
\begin{align}
\int \phi \cdot \psi\circ f^n_{\omega}\,dP_{\vol}&=\int \phi\circ (f^{\delta n}_\omega)^{-1}\cdot\psi\circ f^{(1-\delta)n}_{\sigma^{\delta n}(\omega)}\,d(f^{\delta n}_{\omega})_*(P_{\vol})\\
&=\sum_{A\in P_{\delta n}} \int \phi\circ (f^{\delta n}_{\omega})^{-1}\cdot \psi\circ f^{(1-\delta) n}_{\sigma^{\delta n}(\omega)}\,dA\pm C_{\omega}e^{-\delta\alpha n}\|\phi\|_{C^\beta}\|\psi\|_{C^{\beta}}\label{eqn:first_expanded_sum11}.
\end{align}

Now, by Proposition \ref{prop:quenched_exp_equidistribution_on_subfamilies}, the preimage of each curve $A\in P_{\delta n}$ has length at most $e^{-\delta \lambda n/2}$. 
By H\"older continuity of $\phi$
\begin{equation}
\abs{\max \phi\circ (f^{\delta n}_{\omega})^{-1}\vert_A-\min \phi\circ (f^{\delta n}_{\omega})^{-1}\vert_A}<e^{-\beta\delta \lambda n/2}\|\phi\|_{C^{\beta}}.
\end{equation}
In particular, applying this observation to each summand in \eqref{eqn:first_expanded_sum11}, we see that
\begin{align}\label{eqn:phi_nearly_constant1}
\sum_{A\in P_{\delta n}} \int \phi\circ (f^{\delta n}_{\omega})^{-1}\cdot \psi\circ f^{(1-\delta)n}_{\sigma^{\delta n}(\omega)}\,dA
=&\sum_{A\in P_{\delta n}}\int \phi\circ (f^{\delta n}_{\omega})^{-1}\,d\overline{A}\int \psi\circ f^{(1-\delta)n}_{\sigma^{\delta n}(\omega)}\,dA \notag \\
&\,\,\,\,\pm e^{-n\beta \delta \lambda/2}\|\phi\|_{C^\beta}\|\psi\|_{C^{\beta}},
\end{align}
where $\overline{A}$ denotes the unit mass version of $A$.
By the exponential equidistribution estimate from Proposition \ref{prop:quenched_exp_equidistribution_on_subfamilies},
\begin{equation}\label{eqn:equidist_one_pair_}
\int \psi\circ f^{(1-\delta)n}_{\sigma^{\delta n}(\omega)}\,dA=\rho(A)
\left(\int \psi\,d\vol \pm C_{\omega}e^{-((1-\delta)\alpha-\delta \epsilon)n}\|\psi\|_{C^{\beta}})\right),
\end{equation}
where $\rho(A)$ is the mass of the pair $A$.  Note by our choice of $\delta$ that the exponent appearing in the above equation is negative.

Combining  \eqref{eqn:first_expanded_sum11}, \eqref{eqn:equidist_one_pair_}, and \eqref{eqn:phi_nearly_constant1}, we find that
\begin{align*}
\int \phi\cdot \psi\circ f^n_{\omega}\,dP_{\vol}&=\sum_{A\in P_{\delta n}} \left(\int \phi\circ (f^{\delta n}_{\omega})^{-1}\,dA\right)\left(\int \psi\,d\vol\right)\\ 
&\,\,\,\,\,\,\,\,\pm C_{\omega}(e^{-\delta \alpha n}+e^{-\beta\delta\lambda n/2}+e^{-((1-\delta)\alpha-\delta\epsilon)n})\|\phi\|_{C^{\beta}}\|\psi\|_{C^{\beta}}
\end{align*}
But as $P_{\delta n}$ comprises all but at most $C_{\omega}e^{-\delta \alpha n}$ of the mass of $f^n_{\omega}(P_{\vol})$, it follows that:
\begin{align*}
\int \phi\cdot \psi\circ f^n_{\omega}\,dP_{\vol}
&=\left(\int \phi\,dP_{\vol}\pm C_{\omega}\|\phi\|_{C^{\beta}}e^{-\delta \alpha n}\right)\left(\int \psi\,d\vol\right)\\
&\hspace{2em} \pm C_{\omega}(e^{-\delta \alpha n}+e^{-\beta\delta\lambda n/2}+e^{-((1-\delta)\alpha-\delta\epsilon)n})\|\phi\|_{C^{\beta}}\|\psi\|_{C^{\beta}}\\
&=\int \phi\,d\vol\int\psi\,d\vol\pm 4C_{\omega}(e^{-\eta n}\|\phi\|_{C^{\beta}}\|\psi\|_{C^{\beta}}),
\end{align*}
where 
$\displaystyle
\eta=\min\{\delta \alpha,\beta\delta\lambda/2,(1-\delta)\alpha-\delta\epsilon \}.
$
Since the tail bound on $C_{\omega}$ is part of  Proposition~\ref{prop:quenched_exp_equidistribution_on_subfamilies}, the proof is complete.
\end{proof}

We now give the proof of annealed exponential mixing, i.e.~exponential mixing of the skew product.
\begin{proof}[Proof of Corollary~\ref{CrAnEM}.]
   Let $\bar\Phi(\omega)=\int_M \Phi(\omega, x)\,d \vol$, $\bar\Psi(\omega)=\int_M \Psi(\omega, x)\, d\vol$.
   Note that 
   $$ \iint \Phi (\Psi\circ F^n) \,d\mu d\vol=\mathbb{E}_\omega\left(\Phi(\omega, x) \Psi(\sigma^n \omega, f_\omega^n x) \,d\vol\right).
   $$
   Splitting the right hand side into the regions where 
  $C_\omega\!\leq\!\! e^{\eta n/2}$ and $C_\omega\!>\!\! e^{\eta n/2}$ and using 
  \eqref{EqQEM} in the first region and \eqref{EqEMTail} in the second region we obtain
  $$   \iint \Phi (\Psi\circ F^n) \,d\mu \,d\vol=\int \bar\Phi (\bar\Psi\circ \sigma^n) d\mu
  +O\left(e^{-\eta n/2} \|\Phi\|_{C^\beta} \|\Psi\|_{C^\beta} \right).
  $$
  Now the result follows from the exponential mixing for the shift, see 
\cite[Chapter~2]{ParryPollicott}.
\end{proof}

\appendix

\section{Finite time smoothing estimates}
\label{AppSmoothing}
In the following two appendices we present finite time estimates for nonuniformly hyperbolic systems. While such estimates should be familiar to experts in Pesin theory,
it is difficult to find precise references in the literature since most works concentrate on 
infinite orbits. The finite time estimates play an important role in the paper because 
in the main coupling algorithm we want to use the independence of the dynamics,
hence we decide to stop at time $n$ based only on the 
dynamics on the time interval from zero to $n.$

\subsection{Finite time Lyapunov metrics}
\label{SSFTLyap}

Typically one defines Lyapunov metrics for an infinite sequence of diffeomorphisms. In our case have only a finite sequence, so we show that these also have Lyapunov metrics. The most important point in Lemma \ref{lem:lyapunov_metric} below is item (3), which tells us that at a reverse tempered point the Lyapunov metric will not be distorted. 

The appearance of $\lambda'$ in Lemma \ref{lem:lyapunov_metric} reflects 
 that we need to make a small sacrifice in the rate of growth to obtain the uniform estimates.
If we consider sequences that are $(C,\lambda,\epsilon)$-tempered, and construct the Lyapunov metrics that guarantee a growth rate of exactly $e^{\lambda}$ up to a factor of $\epsilon$, then as we let $\epsilon$ go to zero, the Lyapunov metrics get very distorted with respect to the reference metrics. With the lemma below, as $\epsilon$ goes to zero the metrics do not get any more distorted, however, they guarantee only expansion at some rate $\lambda'\le \lambda$.

\begin{lem}\label{lem:lyapunov_metric}
(Lyapunov Metric Estimates) Fix $(C,\lambda)$. Then for any $0<\lambda'\le \lambda$, and any sequence of linear maps $A_1,\ldots,A_n\in \SL(2,\R)$ 
that have a $(C,\lambda,\epsilon)$-subtempered splitting, 
$E^s_i\oplus E^u_i$ with respect to a sequence of uniformly bounded reference metrics $\|\cdot \|_i$,  there exists a sequence of metrics $\|\cdot\|_i'$ such that
\begin{enumerate}
    \item 
    $\|A_i\vert_{E^s}\|_{i}'\le e^{-\lambda'}$
    \item 
    $\|A_i\vert_{E^u}\|_{i}'\ge e^{\lambda'}$
    \item 
    $\frac{1}{\sqrt{2}}\|\xi\|_i\le \|\xi\|_i'\le  4e^{2C+2\epsilon i}\pez{1-e^{2(\lambda'-\lambda)}}^{-1/2}\|\xi\|_i$, for $\xi\in \R^2$.
\end{enumerate}
The same holds for reverse tempered sequences of maps, \emph{mutatis mutandis}. 
\end{lem}

The estimates below are similar to \cite[Lem.~III.1.3]{Liu1995smooth}. The reverse version follows by just taking inverses.  This result holds because dropping terms from the definition of the Lyapunov metric doesn't stop them from satisfying the required estimates.

\begin{proof}
We begin by defining the new Lyapunov metric. Then we check the desired properties.

\noindent
For 
$\xi\in E^s_i,$ let $\displaystyle \|\xi\|_i'\!=\!\!\left(\sum_{l=0}^{n-i}\|A_i^l\xi\|_i^2e^{2\lambda'l}\right)^{\!\!1/2}$
and 
for  $\xi\in E^u_i,$ let 
$\displaystyle \|\xi\|_i'\!=\!\!\left(\sum_{l=0}^i e^{2\lambda'l}\|[A_{i-l}^l]^{-1}\xi\|_{i-l}^2\right)^{\!\!1/2}\!\!.
$

We then define $\|\cdot\|_i'$ on all 
of $\R^2$ by declaring $E^s_i$ and $E^u_i$ to be orthogonal.

We now check the required estimate for the stable norm. Let $\xi\in E_i^s$, then
\begin{align*}
(\|A_i\xi\|'_{i+1})^2&=\sum_{l=0}^{n-i-1}\|A_{i+1}^lA_i\xi\|^2e^{2\lambda'l}
=\sum_{l=0}^{n-i-1} \|A^{l+1}_i\xi\|^2e^{2\lambda'l}\\
&=e^{-2\lambda'}\sum_{l=0}^{n-i-1}\|A^{l+1}_i\xi\|^2e^{2\lambda'(l+1)}
\le e^{-2\lambda'}(\|\xi\|'_i)^2.
\end{align*}
Note that the last inequality follows because the  
penultimate expression
is missing the first term in the sum that defines $\|\xi\|_i'
$.

We now check the estimate on $E^u_i$. Suppose $\xi\in E^u_i$, $i<n$, then
\begin{align*}
(\|A_i\xi\|'_{i+1})^2&=\sum_{l=0}^{i+1} e^{2\lambda'l}\|[A_{i+1-l}^l]^{-1}A_i\xi\|_{i+1-l}^2\\
&=\|A_i\xi\|^2_{i+1}+e^{2\lambda'}\sum_{l=1}^{i+1} e^{2\lambda'(l-1)}\|[A_{i-(l-1)}^{l-1}]^{-1}\xi\|_{i-(l-1)}^2\\
&=\|A_i\xi\|^2_{i+1}+e^{2\lambda'}\sum_{l=0}^{i} e^{2\lambda'l}\|[A_{i-l}^{l}]^{-1}\xi\|^2_{i-l}\ge e^{2\lambda'}(\|\xi\|_i')^2.
\end{align*}
This verifies the first two estimates in the lemma. Note that neither of the above required any control on the angle between $E^s$ and $E^u$. 

We now compare the two norms on $E^s_i$ and $E^u_i$. For $\xi\in E^s_i$,
$$
\|\xi\|_i'^2=\sum_{l=0}^{n-i}\|A^l_i\xi\|_i^2e^{2\lambda'}
\le \sum_{l=0}^{n-i} e^{2C}e^{-2\lambda l}e^{2\epsilon i}\|\xi\|_i^2 e^{2\lambda'l}
\le \frac{e^{2C}e^{2\epsilon i}}{1-e^{2(\lambda'-\lambda)}}\|\xi\|_i^2.
$$
Next for $\xi\in E^u_i$, we estimate
\begin{align*}
(\|\xi\|_i')^2&=\sum_{l=0}^i e^{2\lambda'l}\|[A^l_{i-l}]^{-1}\xi\|^2_{i-l}
=\sum_{l=0}^ie^{2\lambda'l}e^{2C}e^{2(i-l)\epsilon}e^{-2\lambda l}\|\xi\|_i^2\\
&\le e^{2C}e^{2i\epsilon}\sum_{l=0}^i e^{2(\lambda'-\lambda)l}e^{-2\epsilon l}\|\xi\|_i^2
\le \frac{e^{2C}e^{i2\epsilon}}{1-e^{2(\lambda'-\lambda)}}\|\xi\|_i^2. 
\end{align*}

We now check final estimate in the theorem. For the lower bound, note that by definition $\|\xi^s\|_i'\ge \|\xi^s\|_i$ and $\|\xi^u\|_{i}'\ge \|\xi^u\|_i$, thus
\begin{equation}
	\|\xi\|_i^2\le (\|\xi^s\|_i+\|\xi^u\|_i)^2\le 2[(\|\xi^s\|_i')^2+(\|\xi^u\|_i')^2]=2(\|\xi\|_i')^2.
\end{equation}
For the upper bound, we have that 
\begin{align}\label{eqn:last_eqn_without_prime}
\|\xi\|_i'\le \|\xi^s\|_i'+\|\xi^u\|_i'&\le \frac{e^{C+\epsilon i}}{\sqrt{1-e^{2(\lambda'-\lambda)}}}(\|\xi^s\|_i+\|\xi^u\|_i).
\end{align}
But we know from subtemperedness that the angle $\theta$ between $E^s_i$ and $E^u_i$ is at least $e^{-C}e^{-i\epsilon}$. So 
 by the Law of Sines we have that for $*\in \{u,s\}$ that $\|\xi^*\|_i\leq \|\xi\|_i/\sin\theta\leq 2\|\xi\|_i/\theta$
because for $0\le \theta\le \pi/2$, $\theta/2\le\sin(\theta)$. 
Thus \eqref{eqn:last_eqn_without_prime} gives
$\displaystyle
\|\xi\|_i'\le \frac{4e^{2C+2\epsilon i}}{\sqrt{1-e^{2(\lambda'-\lambda)}}}\|\xi\|_i,
$
which completes the final estimate in the proof.
\end{proof}

\subsection{Basic calculus facts}
We now record some facts from calculus that will be needed when we study estimates for the graph transform. In the following statements, as elsewhere, we use $\|\phi\|_i$ to denote the supremum of norm of the $i$th partial derivatives of $\phi$.
\begin{lem}\label{lem:C_2_norm_twisted_charts}
(Norms of functions in twisted charts) Suppose that $\phi\colon \R^2\to \R^2$ is a $C^2$ function. Then if we apply a linear change of coordinates $L_1,L_2$ to $\phi$, then we see that 
\[
\|L_2\circ \phi\circ L_1\|_{1}\le \|L_1\|\|L_2\|\|\phi\|_1.
\]
Further, for the second derivatives of $\phi$:
\[
\|L_2\circ \phi \circ L_1\|_2\le \|L_2\|\|\phi\|_2\|L_1\|^2.
\]
\end{lem}

The next lemma studies how the $C^2$ norm of a curve changes when we apply a linear map.
\begin{lem}\label{lem:linear_map_C_2_norm_curve_est}

Suppose that $\gamma$ is a $C^2$ curve in $\R^2$ and that $L\colon \R^2\to \R^2$ is an invertible linear map. Then 
$\displaystyle
\|L\circ \gamma\|_{C^2}\le \frac{\|L\|}{(m(L))^2}\|\gamma\|_{C^2}
$. Here $\|\gamma\|_{C^2}$ refers to the $C^2$ norm of $\gamma$ as a curve in $\R^2$ and $m(L)$ is the conorm of the matrix, 
$\displaystyle m(L)=\min_{v\neq 0} \|Lv\|/\|v\|$.
\end{lem}
\begin{proof}
 By definition, the $C^2$ norm of a curve is the supremum of the second derivative of its graph over each of its tangent spaces.
So, without loss of generality suppose that $\gamma$ passes through the origin and that at this point $\gamma$ is the curve $t\mapsto (t,\lambda t^2)$
($O(t^3)$ terms do not change the computation below).
Then we apply 
$\displaystyle
L=\begin{bmatrix}
a & b \\
c & d
\end{bmatrix}
$
to $(t,\lambda t^2)^T$ to get the curve
$\displaystyle
t\begin{pmatrix}
a \\
c
\end{pmatrix}+\lambda 
t^2\begin{pmatrix}
b\\
d
\end{pmatrix}.
$

To study the $C^2$ norm of $L\circ \gamma$ at $0$, we must write it as a graph over its tangent space, i.e. in the form 
$tu+t^2\hat\lambda u^{\perp}$, where $u$ is a unit vector and $\hat\lambda$ is to be determined.
Let $v=(a,c)^{T}$, $u=v/\|v\|$ and  $w=(b,d)^{T}$. Then we may reparametrize $vt+\lambda wt^2$ in the form 
$u t+\lambda (w/\|v\|^2) t^2$. 
Decomposing $w=p u+q u^\perp$ we obtain the parametrization 
$us+(\lambda q/\|v\|^2) s^2 u^\perp+O(s^3)$ where $s=t+p \lambda/\|v\|^2 t^2.$
Thus $\hat\lambda=q\lambda/\|v\|^2.$ Since $|q|\leq \|w\|,$
$\|w\|\le \|L\|$, and $1/\|v\|\le 1/m(L)$, the result follows.
\end{proof}

We now estimate the $C^2$ norm of a function in terms of its inverse.

\begin{lem}\label{lem:C_2_expanding_est1}
Suppose that $\phi\colon \R\to \R$ (or from one interval to another) is a $C^2$ diffeomorphism. If $\abs{D\phi}>\lambda$, then $\abs{D\phi^{-1}}\le \lambda^{-1}$ and $\|\phi^{-1}\|_2\le \lambda^{-3}\|\phi\|_2$.
\end{lem}

\begin{proof}
At each point, we  express the Taylor polynomial of $\psi^{-1}$ in terms of the Taylor polynomial of $\psi$. Suppose that $\psi$ has Taylor polynomial $\nu x+Ax^2$ at some point, with $|\nu|\leq \lambda.$ Then the Taylor polynomial of $\psi^{-1}$ at the corresponding point is 
$\nu^{-1}x+Cx^2$, where
$\displaystyle
C=-\nu^{-3}A.
$
The conclusion follows.
\end{proof}

 For the future reference, we record a bound on compositions. An overview of estimates like these is contained in \cite[App.~A]{hormander1976boundary}.
\begin{lem}
Suppose we are composing three functions $f,g,h\colon \R^n\to \R^n$, then
\[
\|f\circ g\|_2\le \|f\|_2\|g\|_1^2+\|f\|_1\|g\|_2.
\]
and 
\[
\|f\circ g\circ h \|_2 \le \|f\|_2\|g\|_1^2\|h\|_1^2+\|f\|_1\|g\|_2\|h\|_1^2+\|f\|_1\|g\|_1\|h\|_2.
\]
\end{lem}

When we study how fast the dynamics smooths curves, we will represent the curve as a graph and then apply the graph transform to it. The following relates the $C^2$ norm of an embedded curve with the $C^2$ norm of the curve represented as a graph. Recall that the $C^2$ norm of an embedded curve is the same thing as the norm of the curve as a graph over its tangent space at each point in an exponential chart. 

\begin{lem}\label{lem:C_2_norm_as_graph_vs_curve_R_2}
Suppose $\gamma$ is a $C^2$ curve in $\R^2$ that is $\theta$-transverse to the $y$-axis. Then if we represent $\gamma$ as the graph over the $x$-axis of a function $\hat{\gamma}$, then
\[
 \|\hat\gamma\|_1\le \cot \theta\text{, and }\|\hat{\gamma}\|_2\le (\sin \theta)^{-3}\|\gamma\|_{C^2}.
\]
\end{lem}

\begin{proof}
The first estimate is essentially the definition of tangent, so we will show the second.

Locally we may represent $\gamma$ as a graph:
\[
p+(\sin \theta_p,\cos\theta_p) t+\phi_p(t)(\!-\!\!\cos\theta_p,\sin\theta_p)\!\!=\!\!
p+(t\sin\theta_p-\phi_p(t)\cos\theta_p,0)+(0,t\cos\theta_p+\phi_p(t)\sin\theta_p)
\]
where $\phi'_p(t)=0$. By definition of $\|\gamma\|_{C^2}$, $\abs{\phi_p''(0)}\le \|\gamma\|_{C^2}$.

In order to estimate $\hat{\gamma}''(0)$, we must write the graph in the form $p+(t,\psi(t))$ for some $\psi$ and estimate $\psi''(0)$.
Accordingly, 
we make a change of variables $s=t/\sin\theta_p$ getting
\begin{equation}\label{eqn:t_over_sin_theta}
 p+\left(s-\phi_p\left(\frac{s}{\sin\theta_p}\right)\cos\theta_p,0\right)+\left(0,\frac{\cos\theta_p}{\sin\theta_p}s+\phi_p\left(\frac{s}{\sin\theta_p}\right)\sin\theta_p\right).
\end{equation}
To estimate the second derivative of the graph at $0$, we need a representation of the form $(u+O(u^3),\psi(u)+O(u^3))$,
so we make a further change of variables $\displaystyle u=s-\phi_p\left(\frac{s}{\sin\theta_p}\right)$.
Then $\displaystyle s=u+\frac{\phi_p''(0) u^2 \cos\theta_p}{2\sin^2 \theta_p}+o(u^2).$
Plugging this into \eqref{eqn:t_over_sin_theta} and using that $\cos^2 \theta_p+\sin^2 \theta_p=1$
we obtain the parametrization
$$p+\left(u,\frac{u \cos\theta_p}{\sin\theta_p}+\frac{\phi''(0) u^2}{2\sin^3 \theta_p}
+o(u^2)\right)$$
and the result follows.
\end{proof}

The next lemma estimates how the density is distorted by diffeomorphisms.

\begin{lem}\label{lem:pushforward_density_est_diffeo}
Suppose that $M$ is a closed Riemannian manifold. There exists $C>0$ such that if $f\colon M\to M$ is a $C^2$ diffeomorphism, $\gamma$ is a $C^2$ curve in $M$ and $\rho$ is a log-$\alpha$-H\"older density along $\gamma$, then the density $f_*\rho$ along $f(\gamma)$ satisfies
\begin{equation}
\label{FRho}
\|\ln(f_*\rho)\|_{C^{\alpha}}\le (1/m(Df))^{1+\alpha}\left(\|\ln \rho\|_{C^{\alpha}}+C\|f\|_{C^2}(1+\|\gamma\|_{C^2})\right).
\end{equation}
The same estimate holds for local diffeomorphisms, mutatis mutandis.
\end{lem}

We leave the proof of the lemma to the readers, since we provide a similar estimate below (see 
\eqref{NewDensity}). 

Next we record an estimate comparing two inner products.

\begin{lem}\label{lem:change_of_metric_change_of_angle}
Suppose that we have two inner products $\|\cdot\|_1$ and $\|\cdot\|_2$ on a vector space $V$ and that 
\[
A\|\cdot\|_1\le \|\cdot\|_2\le B \|\cdot\|_1.
\]
Then for $v,w\in V\setminus \{0\}$
\[
AB^{-1}\angle_1(v,w)\le \angle_2(v,w)\le A^{-1}B\angle_1(v,w),
\]
where $\angle_i$ denotes the angle with respect to the metric $\|\cdot\|_i$.
\end{lem}

\begin{proof}
We show the upper bound; the lower bound is a straightforward consequence.  Let $S_i^1$ denote the unit sphere with respect to the inner product $i$ and $v$ and $w$ be two unit vectors with respect to $\|\cdot\|_1$. Let $I$ be a curve between $v$ and $w$ such that $\len_1(I)=\angle_1(v,w)$. Then $\len_2(I)\le B\len_1(I)$. Let $\pi_2\colon V\setminus\{0\}\to S^1_2$ denote the radial projection onto $S^1_2$. Then $\angle_2(v,w)\le \len_2(\pi_2(I))$.  Note that the norm of $D\pi_2\vert_{I}$ is bounded above by $1/d_2(0,I)$. Since $d_{2}(0,I)\ge A$, we see that $\len_2(\pi_2(I))\le A^{-1}B\len_1(I)$, so we are done.
\end{proof}

We now record an estimate on how fast a $C^2$ curve can get worse under the dynamics. Note that one iteration
can instantaneously make a line into an $O(1)$ bad curve, hence  the estimate has the form below.

\begin{lem}\label{lem:growth_of_C_2_norms}
Fix $D>0$ then there exists $\Lambda>0$, such that if for $1\le i\le n$, $f_i\in \Diff^2(M)$ is a sequence of diffeomorphisms of a closed Riemannian manifold $M$ with $\|f\|_{C^2}<D$,  $\gamma$ is a $C^2$ curve in $M$, and $\gamma_n=f^n_1(\gamma)$, then
\[
\|\gamma_n\|_{C^2}\le \max\{e^{\Lambda n}\| \gamma\|_{C^2},e^{\Lambda n}\}.
\] 
\end{lem}
\begin{proof}
Recall that the $C^2$ norm of $\gamma$ is bounded by the maximum over all $t\in \gamma$ of the second derivative of $\gamma$ in an exponential chart at $t$ where $\gamma$ is viewed as a graph over its tangent plane. The result then follows because the second derivative of a sequence of maps with uniformly bounded $C^2$ norm grows at most exponentially fast.
\end{proof}

\subsection{Properties of H\"older functions}
In this subsection, we record some additional claims about H\"older and log-H\"older functions that will be used in the proof of the coupling lemma. 
\begin{claim}\label{claim:bounded_below_plus_holder_gives_log_holder}
Suppose that $\rho\colon M\to \R$ is a $(C,\alpha)$-H\"older function on a metric space $M$
such that $\rho\ge A^{-1}$, for some $A>0$. Then $\ln \rho$ is $AC$-log-$\alpha$-H\"older.
\end{claim}
\begin{proof}
First, observe that on $[A^{-1},\infty)$, that $\ln$ is $A$-Lipschitz because its derivative $1/x$ is at most $A$. Thus 
$\displaystyle
\abs{\ln(\rho(x))-\ln(\rho(y))}\le A\abs{\rho(x)-\rho(y)}\le AC\abs{x-y}^{\alpha},
$
as desired.
\end{proof}

The next lemma relates two different ways of dealing with log-H\"older functions.

\begin{lem}\label{lem:log_holder_C_A_est}
Suppose that $\rho$ is an $(A,\alpha)$-log H\"older function on a metric space of diameter at most $D$. Then there exists $C_{A,D}$ such that 
\begin{equation}
\abs{\rho(x)-\rho(y)}\le \rho(x)C_{A,D}\abs{x-y}^{\alpha}.
\end{equation}
\end{lem}
\begin{proof}
Suppose that $\rho(y)\ge \rho(x)$. 
Then log-$\alpha$-H\"older gives that 
\[
\ln(\rho(y)/\rho(x))=\abs{\ln(\rho(y)/\rho(x))}\le A\abs{x-y}^{\alpha}.
\]
Thus taking $e^x$, by boundedness of the metric space and the constant $A$, there exists $C_{A,D}$ such that
\[
\frac{\rho(y)}{\rho(x)}\le e^{A\abs{x-y}^{\alpha}}\le 1+C_{A,D}\abs{x-y}^{\alpha}.
\]
Thus
\[
\rho(y)-\rho(x)\le \rho(x)C_{A,D}\abs{x-y}^{\alpha}.
\]
The case when $\rho(y)< \rho(x)$ is similar, so we are done.
\end{proof}

\subsection{Graph transform with estimates on the second derivative}\label{subsec:graph_transform}

We now study the graph transform and record how $C^2$ norms of curves are affected by it. 
If one constructs the stable manifolds by using the graph transform, then after one has checked that the stable manifold is $C^1$, one can check that the manifolds are $C^r$ inductively by studying the action of the graph transform on the jet of the stable manifold which is $C^1$. See for instance the construction in \cite{shub1987global}, which proceeds along these lines.

\begin{prop}\label{prop:1-step_smoothing_estimate}
($C^2$ estimates for the graph transform) Suppose $\lambda>1$ and $F\colon \R^2\to \R^2$ is a $C^2$ diffeomorphism  of the form 
\begin{equation}\label{eqn:F_in_charts}
F=(\sigma_1 x +f_1(x,y),\sigma_2 y+f_2(x,y)),
\end{equation}
with $\min\{\sigma_1,\sigma_2^{-1}\}\ge \lambda$. 
 Suppose that $\gamma$ is a $C^2$ curve given as the graph of a function $\phi\colon I_1\to \R$.
Assume that $F(0,0)=(0,0)$ and that we have the following estimates:
\begin{align}
\|f_1\|_{C^1}&=\epsilon_1,\\
\|f_2\|_{C^1}=\epsilon_2&<\lambda^{-1},\\
\lambda-\epsilon_1-\epsilon_1\|\phi\|_1&>0\label{eqn:main_stretching_assumption_C_1}.
\end{align}
Then the following hold.
\begin{enumerate}
    \item \label{item:horizontal_stretching_1_step}
The curve $F\circ \gamma$ is given as the graph of a function $\wt{\phi}\colon I_2\to \R$ and
\begin{equation}
    \len(I_2)\ge (\lambda-\epsilon_1-\epsilon_1\|\phi\|_{1})\len(I_1).
\end{equation}

\item\label{item:1-step_smoothing_C_1}
We have an estimate on how much $F$ smooths $\phi$,
    \begin{align}\label{eqn:c_1_norm_decay_one_iterate}
    \|\wt{\phi}\|_{C^0}&\le \lambda^{-1}\|\phi\|_{C^0}+\epsilon_2,\\
\|\wt{\phi}\|_{1}&\le (\lambda^{-1}\|\phi\|_1+\epsilon_2+\epsilon_2\|\phi\|_1)(\lambda-\epsilon_1-\epsilon_1\|\phi\|_{1})^{-1}.
\end{align}
\item\label{item:one_step_C_2_smoothing}
There is $\epsilon_0\!>\!0$ such that under the additional assumption that 
$\epsilon_1,\epsilon_2,\|\phi\|_1\!<\!\epsilon_0$
\begin{equation}\label{eqn:one_step_smoothing_C_2}
\|\wt{\phi}\|_2\le \lambda^{-1.99}\|f\|_2+\lambda^{-2.99}\|\phi\|_2.
\end{equation}

\item \label{item:one_step_density_smoothing}
The graph transform smooths densities along curves. If $\rho(x,\phi(x))$ is a 
 log $\alpha$-H\"older
density along $\gamma$
 with respect to the arclength, 
write $\wt{\rho}(x,\wt{\phi}(x))$ for the density of the pushforward of $\rho$ along $F(\gamma)$.
For $\epsilon_0>0$ as in \eqref{item:one_step_C_2_smoothing}, 
if $\epsilon_1,\epsilon_2,\|\phi\|_1<\epsilon_0$, 
then
\begin{equation}\label{eqn:smoothing_rho_1_step_chart_est}
\|\ln \wt{\rho}\|_{C^{\alpha}}\le \lambda^{-.9\alpha}(\|\ln \rho\|_{C^{\alpha}}+\|f\|_2+\|\phi\|_2).
\end{equation}
\end{enumerate}
\end{prop}
\vskip2mm

Note that part (1) of the proposition implies that if $I_1$ contains a neighborhood of $0$ of size $\delta$, then then  $I_2$ contains a neighborhood of size $(\lambda-\epsilon_1-\epsilon_1\|\phi\|_{1})\delta$.

\begin{proof}
We write down explicitly a formula for $\wt{\phi}$ and then estimate each term that appears in the formula. It is tedious but straightforward. Throughout we will use $\pi_1$ and $\pi_2$ for the projections onto the two factors in $\R^2$.

We estimate the $C^1$ norm of $\phi$ as a graph over $\R\times\{0\}$. To this end we first study how much the graph of $\phi$ is stretched horizontally, which will verify \eqref{item:horizontal_stretching_1_step} above. To do this we consider a natural map $\psi^{-1}\colon I_1\to \R$: 
\begin{equation}\label{eqn:defn_of_psi_1}
\psi^{-1}\colon x\mapsto (x,\phi(x))\mapsto \pi_1(F(x,\phi(x)))=\lambda x+f_1(x,\phi(x)).
\end{equation}
From the definition of $\psi^{-1}$,
\begin{equation}\label{eqn:psi_-1_stretching}
\|D\psi^{-1}\|\ge \lambda-\epsilon_1-\epsilon_1\|\phi\|_{1},
\end{equation}
thus by \eqref{eqn:main_stretching_assumption_C_1}, $\|D\psi^{-1}\|$ is positive, so $\psi^{-1}$ is monotone. Hence
$F(\gamma)$ is the graph of a function $\wt{\phi}$, and we may write $\psi^{-1}\colon I_1\to I_2$. By  \eqref{eqn:psi_-1_stretching},
$\displaystyle
(\lambda-\epsilon_1-\epsilon_1\|\phi\|_{1})\len(I_1)\le \len(I_2).
$
This completes the proof of item \eqref{item:horizontal_stretching_1_step}.

We now prove item \eqref{item:1-step_smoothing_C_1}. First we give the $C^0$ estimate and then the estimate on the first derivative.
By the assumption on $f_2$, we see that the image of $\phi$  is at most
$\lambda^{-1}\|\phi\|_{C^0}+ \epsilon_2$ from the $x$-axis. Thus 
\begin{equation}
\|\wt{\phi}\|_{C^0}\le \lambda^{-1}\|\phi\|_{C^0}+\epsilon_2.
\end{equation}

Now we estimate $\|\wt{\phi}\|_1$. From equation \eqref{eqn:psi_-1_stretching}, we obtain that:
\begin{equation}\label{eqn:upper_bound_on_psi}
\|D\psi\|\le (\lambda-\epsilon_1-\epsilon_1\|\phi\|_{1})^{-1}.
\end{equation}
This allows us to estimate the $C^1$ norm of $F\circ \gamma$ as a graph over $I_2\times \{0\}$. The curve $F(\gamma)$ is given by the graph of
\begin{equation}
x\mapsto \pi_2F(\psi(x),\phi(\psi(x)))=\lambda^{-1}\phi(\psi(x))+f_2(\psi(x),\phi(\psi(x)))=\wt{\phi}.
\end{equation}
Thus by the chain rule 
$$
\|\wt{\phi}\|_1\le \lambda^{-1}\|\phi\|_1(\lambda-\epsilon_1-\epsilon_1\|\phi\|_{1})^{-1}+\epsilon_2(\lambda-\epsilon_1-\epsilon_1\|\phi\|_{1})^{-1}+\epsilon_2\|\phi\|_{1}(\lambda-\epsilon_1-\epsilon_1\|\phi\|_{1})^{-1}.
$$
Hence, 
\begin{equation}
\|\wt{\phi}\|_{1}\le (\lambda^{-1}\|\phi\|_1+\epsilon_2+\epsilon_2\|\phi\|_1)(\lambda-\epsilon_1-\epsilon_1\|\phi\|_{C^1})^{-1},
\end{equation}
 which finishes the proof of item \eqref{item:1-step_smoothing_C_1}.

We now turn to the $C^2$ estimates and check item \eqref{item:one_step_C_2_smoothing}. To begin we need to obtain a $C^2$ estimate on the function $\psi$ used above. 
 By  \eqref{eqn:defn_of_psi_1} and the chain rule,
\[
\|\psi^{-1}\|_2\le\|f\|_2+2\|f\|_2\|\phi\|_1+\|f\|_2\|\phi\|_1^2+\|f\|_1\|\phi\|_2.
 \]
Thus by Lemma \ref{lem:C_2_expanding_est1}, 
\[
\|\psi\|_{2}\le (\lambda-\epsilon_1-\epsilon_1\|\phi\|_{1})^{-3}(\|f\|_2+2\|f\|_2\|\phi\|_1+\|f\|_2\|\phi\|_1^2+\|f\|_1\|\phi\|_2).
\]

We can now plug everything in to estimate the $C^2$ norm of the image of $\phi$.
By definition $\wt{\phi}$ is equal to  
$\displaystyle \lambda^{-1}\phi(\psi(x))+f_2(\psi(x),\phi(\psi(x))).
$
For the first term, we have the estimate
$$
\|\lambda^{-1}\phi\circ \psi\|_2\le \lambda^{-1}(\|\phi\|_2\|\psi\|_1^2+\|\phi\|_1\|\psi\|_2).
$$
By the chain rule
\begin{align*}
\|f_2(\psi(x),\phi(\psi(x)))\|_2 &\le \|f_2\|_2\|\psi\|_1^2+\|f_2\|_2
\|\phi\|_1\|\psi\|_1^2+\|f_2\|_1 \|\phi\|_1
\|\psi\|_2\\
&+(\|f_2\|_1\|\phi\|_1\|\psi\|_2\!+\!\|f_2\|_1\|\phi\|_2\|\psi\|_1^2\!+\!\|f_2\|_2\|\phi\|_1^2\|\psi\|_1^2
\!+\!\|f_2\|_2\|\phi\|_1\|\psi\|_1^2).
\end{align*}

Hence if 
$\epsilon_1,\epsilon_2,\|\phi\|_1<\epsilon_0$ and  $\epsilon_0$ sufficiently small, then
\begin{align}
\|\psi\|_1&\le \lambda^{-.9999}, \\
\|\psi\|_2&\le \lambda^{-2.999}(\epsilon_0\|\phi\|_2+\|f\|_2).
\end{align}
In particular, as long as $\epsilon_0>0$ is sufficiently small, under the assumptions just listed applying the estimates on $\|\phi\|_1$ and $\|\phi\|_2$ gives 
\begin{align}
\|\lambda^{-1}\phi\circ \psi\|_2&\le \lambda^{-2.999}(\epsilon_0\|f\|_2+\|\phi\|_2), \\
\|f_2(\psi(x),\phi(\psi(x)))\|_2&\le \lambda^{-1.999}\|f\|_2+\epsilon_0\lambda^{-1.8}\|\phi\|_2.
\end{align}
 Combining these estimates, we see that as long as $\epsilon_0$ is sufficiently small,
\[
\|\wt{\phi}\|_2\le \lambda^{-1.99}\|f\|_2+\lambda^{-2.99}\|\phi\|_2.
\]

We next study 
how the H\"older norm of the log of the density $\rho$ along $\gamma$ changes when we iterate the dynamics
and prove item \eqref{item:one_step_density_smoothing}.
From the change of variables formula, we must estimate the following:
\begin{equation}
    \label{NewDensity}
\ln [\rho(\psi(x),\phi(\psi(x)))\|{ DF\vert_{(1,d\phi/dx)}(\psi(x),\phi(\psi(x)))}\|^{-1}]=
\end{equation}
$$
\ln \rho(\psi(x),\phi(\psi(x)))+\ln\|{ DF\vert_{(1,d\phi/dx)}(\psi(x))}\|^{-1}=I+II.
$$

\textbf{Term I}. The estimate of the term $I$ is straightforward: 
$$
\|\ln \rho(\psi(x),\phi(\psi(x)))\|_{C^{\alpha}}\le \|\psi\|_1^{\alpha}\|\ln\rho\|_{C^{\alpha}}
\le \lambda^{-.9\alpha}\|\ln\rho\|_{C^{\alpha}}
$$
by equation \eqref{eqn:upper_bound_on_psi} as we are assuming $\|f_1\|_1,\|f_2\|_1,\|\phi\|_1$ are all small.

\textbf{Term II.} The second term is more complicated to estimate. Note that this term does not actually involve $\rho$ as it is just the Jacobian of the map between two curves. So, to control the log-$\alpha$-H\"older norm of this function we can estimate the derivative of the logarithm, which is an upper bound on the log-$\alpha$-H\"older constant for all $\alpha\le 1$. To begin, we write
\begin{align*}
D\ln \|DF\vert_{(1,d\phi/dx)}(\psi(x),\phi(\psi(x)))\|^{-1}&=2^{-1}D\ln\|DF\vert_{(1,d\phi/dx)}\|^2-2^{-1}D\ln\|(1,d\phi/dx)\|^2\\
&=III+IV,
\end{align*}
where $III$ and $IV$ are evaluated at the point $(\psi(x),\phi(\psi(x))$.

\textbf{Term III.} We now bound term $III$. Because $D\ln f\circ\psi=\psi'D\ln f$, we see that the required estimate will hold assuming that it holds without precomposing with $\psi$ because $\abs{\psi'}\le 1$ under these assumptions.  Thus we suppress the $\psi$ below.  From before we have an expression for $DF$ in terms of $\lambda, f_1,f_2$:
\[
DF=\begin{bmatrix}
\lambda+\frac{df_1}{dx}& \frac{df_1}{dy}\\
\frac{df_2}{dx}& \lambda^{-1}+\frac{df_2}{dy}
\end{bmatrix}.
\]
Thus we are reduced to evaluating
\begin{equation}\label{eqn:derivative_of_top}
D\ln\left[\pez{\lambda+\frac{df_1}{dx}+\frac{df_1}{dy}\frac{d\phi}{dx}}^2 +\pez{\frac{df_2}{dx}+\lambda^{-1}\frac{d\phi}{dx}+\frac{df_2}{dx}\frac{d\phi}{dx}}^2\right],
\end{equation}
where the $df_i$ terms are evaluated at $(x,\phi(x))$. Then taking derivatives gives:
\begin{equation}\label{eqn:A+B_over}
\frac{A+B}{\pez{\lambda+\frac{df_1}{dx}+\frac{df_1}{dy}\frac{d\phi}{dx}}^2 +\pez{\frac{df_2}{dx}+\lambda^{-1}\frac{d\phi}{dx}+\frac{df_2}{dx}\frac{d\phi}{dx}}^2}=\frac{A}{Q}+\frac{B}{Q}.
\end{equation}
where $A$ and $B$ are the derivatives of the two parenthetical terms in equation \eqref{eqn:derivative_of_top} and $Q$ is the denominator of the left hand side of equation \eqref{eqn:A+B_over}. Note that $Q$ can be made arbitrarily close to $\lambda^2$ 
as long as $df_1/dx,df_1/dy$ and $d\phi/dx$ are sufficiently small. 

Keeping in mind that the $f_i$ terms are evaluated at $(x,\phi(x))$, we find that:
$$
A=2\pez{\lambda+\frac{df_1}{dx}+\frac{df_1}{dy}\frac{d\phi}{dx}}\pez{\frac{d^2f_1}{dx^2}+\frac{d^2f_1}{dxdy}\frac{d\phi}{dx}+\pez{\frac{d^2f_1}{dxdy}+\frac{d^2f_1}{dy^2}\frac{d\phi}{dx}}\frac{d\phi}{dx}+\frac{df_1}{dy}\frac{d^2\phi}{dx^2}}
$$
and 
$$
B\!\!=\!\!2\pez{\frac{df_2}{dx}\!+\!\lambda^{-1}\frac{d\phi}{dx}\!+\!\frac{df_2}{dx}\frac{d\phi}{dx}}\pez{\frac{d^2f_2}{dx^2}+\lambda^{-1}\frac{d^2\phi}{dx^2}
+
\pez{\frac{d^2f_2}{dxdy}+\frac{d^2f_2}{dx^2}\!+\!\frac{d^2f_2}{dxdy}\frac{d\phi}{dx}}\frac{d\phi}{dx}\!+\!\frac{df_2}{dx}\frac{d^2\phi}{dx^2}
}.
$$

Pick a small number $\bar\epsilon$.
Then as $\|f_1\|_1,\|f_2\|_1$ and $\|\phi\|_1$ are sufficiently small it is easy to see from the above expressions for $A$ and $B$, that
\begin{equation}\label{eqn:term_III_bound}
\abs{III}\le \frac{\lambda+\bar\epsilon}{\lambda^2-\bar\epsilon}(\|f\|_2+\|\phi\|_2).
\end{equation}

\textbf{Term IV}. We now bound Term IV. For this term we have
$$
2^{-1} D\ln \|(1,d\phi/dx)\|^2=2^{-1}D\ln \left(1+\pez{\frac{d\phi}{dx}}^2\right)\\
=\frac{\frac{d\phi}{dx}\frac{d^2\phi}{dx^2}}{1+\pez{\frac{d\phi}{dx}}^2}.
$$
Since we are assuming $\|\phi\|_1$ is small, we see that 
$\displaystyle
\abs{IV}\le \bar\epsilon \|\phi\|_2.
$

\textbf{Conclusion of estimates on $\abs{D\ln \wt{\rho}}$.} From the above discussion,
\begin{align*}
\|\ln \wt{\rho}(x,\wt{\phi}(x))\|_{C^{\alpha}}&\le \abs{I}+\abs{III}+\abs{IV}\\
\le \lambda^{-.9\alpha}\|\ln \rho\|_{C^{\alpha}}+\left[ \frac{\lambda+\bar\epsilon}{\lambda^2-\bar\epsilon}
+\bar\epsilon\right]
(\|f\|_2+\|\phi\|_2)
&\le \lambda^{-.9\alpha}(\|\ln \rho\|_{C^{\alpha}}+\|f\|_2+\|\phi\|_2).
\end{align*}
where the last inequality holds since the expression in square brackets is less than 1 provided that $\bar\epsilon$
is sufficiently small.
This concludes  the proof of the proposition.
\end{proof}

\subsection{Finite time smoothing estimate}
Now that 
we control the amount of smoothing due to a single iteration of the graph transform, we 
study
a reverse subtempered point for a sequence of diffeomorphisms. 
An important feature of the estimate below is that it covers
curves that are extremely close to the contracting direction. This complicates the estimates 
compared to the case that one only considers curves lying in a cone near the expanding direction.

\begin{prop}\label{prop:finite_time_smoothing_estimate}
Fix constants $C,\lambda,\epsilon,D_1>0$ with $\epsilon<\lambda/30$. Suppose that $f_i\colon \R^2\to \R^2$, $1\le i\le n$, is a sequence of diffeomorphisms  such that $f_i(0)=0$, the sequence $D_0f_i$ has a $(C,\lambda,\epsilon)$-reverse tempered splitting $E^s_i,E^u_i$ in the sense of Definition \ref{defn:tempered_splitting}, $\|f_i\|_{C^2}<D_1$, and $\|Df^{-1}_i\|\ge D_1^{-1}$. Then there exist  constants $ \epsilon_0,\ell_{\max},D_2,D_3,D_4,D_5,D_6, D_7,D_8,C_0$  depending only on $(C,\lambda,\epsilon)$ and $D_1$ such that the following holds. Let $\gamma$ be a $C^2$ curve in $\R^2$ passing through $0$ not tangent to $E^s_0$ at $0$, containing an $R$-good neighborhood of $0$. Let  $f^n=f_n\circ \cdots\circ f_1$. 
Let $\theta=\angle(\dot{\gamma}(0),E^s_0)$, and $\gamma_0$ be a segment of $\gamma$ containing $0$ of length at least
\begin{equation}\label{eqn:max_length_curve_finite_smoothing_estimate}
\len(\gamma_0)=D_2\min\{e^{-R}\theta,e^{-.9\lambda n}\}.
\end{equation}
There is an associated auxiliary quantity
\begin{equation}\label{eqn:deterministic_l_0}
l_0=D_3\theta e^{-2\epsilon n}\min\{e^{-R}\theta,e^{-.9\lambda n}\},
\end{equation}
and a subcurve $\gamma_n$ of $f^n(\gamma_0)$  containing $0$ such that the following hold: 

\begin{enumerate}[leftmargin=*]
    \item \label{item:length_of_curve_seq_1}
The curve $\gamma_n$ has length at least
\[
\ell_{out}=\min\{l_0e^{.9\lambda n}, \ell_{\max}\}.
\]

\item\label{item:preimages_of_recovered_curves_shrink}

If the minimum in item \eqref{item:length_of_curve_seq_1} is realized by $\ell_{\max}$, then
the preimage of $\gamma_n$ in $\gamma$ has length at most
$\displaystyle
D_4e^{-.9\lambda n},
$
and this occurs as long as 
\[
n\ge D_{5}+\frac{\max\{R,0\}-2\ln(\theta)}{.99\lambda}.
\]
Further, in this case, the preimage of $\gamma_n$ in $f^i(\gamma)$ has length at most $D_4e^{-.9\lambda (n-i)}$. 
In fact if $I\subseteq f^i(\gamma)$ is a curve of length at least $D_4e^{-.85\lambda(n-i)}$ containing a point $f^i(x)$, then $f^{n-i}(I)$ contains a $C_0$-good neighborhood of $f^n(x)$.

\item\label{item:C_2_estimate_smoothing_seq}
On $\gamma_n$, we have the estimate:
\begin{equation}\label{eqn:C_2_norm_finite_deterministic_estimate}
\|\gamma_n\|_{C^2}<D_6e^{-2.9\lambda n}e^{D_7\ln\theta}\max\{\|\gamma\|_{C^2},1\}+D_8.
\end{equation}
    
    \item\label{item:regularity_of_density_sequence_lem}
    Finally, for any arbitrarily large $D_9>0$ and fixed $\alpha$, there exist $D_{10},D_{11}$ such that the following holds. Suppose that $\rho$ is a density along $\gamma$ that is log-$\alpha$-H\"older. Then for the same collection of $n$, the density of $\rho_n=(f^n)_*(\rho)$ along $\gamma_n$ with respect to arclength parametrization of $\gamma_n$ satisfies the following estimate, as long as $\|\gamma_n\|_2<D_9$, 
\begin{equation}
    \|\ln \rho_n\vert_{\gamma_n}\|_{C^{\alpha}}\le D_{10}e^{-.9\alpha \lambda n}e^{D_7\ln \theta}(1+\|\ln \rho\|_{C^{\alpha}}+\|\gamma\|_{C^2})+D_{11}.
\end{equation}
\end{enumerate}

\noindent The analogous statement holds for sequences of local diffeomorphisms $f_i$ defined on a sequence of neighborhoods of $0$ in $\R^2$ or of a closed manifold.
\end{prop}

\begin{proof}
We begin by fixing some notation and constants that we will use throughout the argument. 
Let $\lambda'=.999\lambda$. Then from  Lemma \ref{lem:lyapunov_metric} we obtain finite time Lyapunov metrics $\|\cdot\|'_i$, $0\le i\le n$, associated to this splitting that satisfy for all $\xi\in \R^2$:
\begin{equation}\label{eqn:estimate_on_metric_smoothing_lemma}
    \frac{1}{\sqrt{2}}\|\xi\|_i\le \|\xi\|_i'\le  4e^{2C+2\epsilon (n-i)}\pez{1-e^{2(\lambda'-\lambda)}}^{-1/2}\|\xi\|_i.
\end{equation}
 Note that because the sequence is reverse tempered $\|\cdot\|_n'$ is uniformly comparable
to the original metric independent of $n$.
As is standard, the metrics $\|\cdot\|'_i$ give new linear coordinates $L_i\colon \R^2\to \R^2$ that satisfy that
$(L_i)^*\|\cdot\|_i=\|\cdot\|_i'$. We let $\hat{f}_i=L_{i+1}\circ f_i\circ L_i^{-1}$. Thus from properties of the Lyapunov metric, $D_0\hat{f}_i$ is a uniformly hyperbolic sequence satisfying 
\begin{equation}\label{eqn:strength_lambda_prime}
D_0\hat{f}_i\vert_{E^u_i}\ge e^{.999\lambda}, \quad
D_0\hat{f}_i\vert_{E^s_i}\le e^{-.999\lambda}.
\end{equation}
We write:
\begin{equation}
\hat{f}_i(x,y)=(\sigma_{1,i}x+\hat{f}_{i,1}(x,y),\sigma_{2,i}y+\hat{f}_{i,2}(x,y)),
\end{equation}
 where $D_0\hat{f}_i=\diag(\sigma_{1,i},\sigma_{2,i})$ and $\sigma_{i,1},\sigma_{i,2}^{-1}\ge e^{.999\lambda}$.

 We now record estimates on $C^2$ norms in these charts. By \eqref{eqn:estimate_on_metric_smoothing_lemma}, there is $C_1$ such that:
\begin{equation}\label{eqn:norm_lyapunov_chart_1}
\max\{\|L_i\|,\|L_i^{-1}\|\}\le C_1e^{2C+2\epsilon (n-i)}.
\end{equation}
 Thus by Lemma \ref{lem:C_2_norm_twisted_charts}, for $1\le i\le n$, 
 \begin{equation}\label{eqn:C_2_Lyapunov_chart_est1}
 \|\hat{f}_i\|_{C^2}\le D_1e^{6C} e^{6(n-i)\epsilon}.
\end{equation} 

 For $0\le i\le n$, let
\begin{equation}\label{eqn:r_i_defn}
r_i=C_2\min \{\theta \|\gamma\|_2^{-1}e^{.9\lambda i}, e^{-.9\lambda(n-i)}\},
\end{equation}
where $0<C_2<1$ is a small number that we will choose later. We then restrict to studying the segment of $\gamma$ inside the cube $B_i$ centered at $0\in \R^2$ of side length $r_i$ with respect to the $\|\cdot \|_i'$ metric. Let $\gamma_i$ be the connected component of $0$ in $f^i(\gamma)\cap B_i$. We write $\hat{\gamma}_i$ for the function giving $\gamma_i$ as a graph over the $x$-axis and let $l_i$ be the length of the projection of $\hat{\gamma}_i$ to the $x$-axis in $\R^2$ measured with respect to $\|\cdot\|'_i$.

We begin working with the ambient metric. By the mean value theorem, there exists $C_3$ such that for a $C^2$ curve $\gamma$ in $\R^2$ in an arclength parametrization, 
\[
\angle (\dot{\gamma}(t),\dot\gamma(s))\le C_3\|\gamma\|_{C^2}\abs{t-s},
\]
because 
 $\dot{\gamma}(t)$ is  orthogonal to $\ddot{\gamma}$.
In particular, as our curve $\gamma$ satisfies $\angle(\dot\gamma(0),E^s_0)>\theta$ restricted to a segment of $\gamma$ of length $C_3^{-1}\|\gamma\|_{C^2}^{-1}\theta/2$ around $\gamma(0)$, that on this segment $\angle (E^s_0,\dot{\gamma}(t))>\theta/2$. 
Then from Lemma \ref{lem:change_of_metric_change_of_angle} in the Lyapunov chart we have that, letting $\angle'$ denote angle with respect to the Lyapunov metric, there exists $C_4$ such that:
\begin{equation}\label{eqn:lyapunov_coordinates_angle1}
 C_4^{-1}\theta e^{-2\epsilon n}
 \le \angle'(E^s_0, \dot{\gamma}(t))
 \le C_4\theta
 e^{2\epsilon n}.
 \end{equation}

From the construction of the Lyapunov metric, $\frac{1}{\sqrt{2}}\|\cdot\|_i\le \|\cdot\|'_i$, thus the length of $\gamma$ in the Lyapunov chart is at least $e^{-R}/2$. We now restrict to a segment of $\hat\gamma$, which we call $\hat{\gamma}_0$, with length with respect to the Lyapunov metric:
\begin{equation}\label{eqn:length_of_initial_segment_smoothing_lem}
\len'(\hat{\gamma}_0)=\min\{C_3^{-1}e^{-R}\theta/2,r_0\}.
\end{equation}
From \eqref{eqn:estimate_on_metric_smoothing_lemma}, as the ratio of $\|\cdot\|_n$ to $\|\cdot\|_n'$ does not depend on $n$, we obtain
the restriction \eqref{eqn:length_of_initial_segment_smoothing_lem}  on the length of the initial segment $\hat{\gamma}_0$ gives the condition \eqref{eqn:max_length_curve_finite_smoothing_estimate} appearing in the theorem.

Note that \eqref{eqn:length_of_initial_segment_smoothing_lem} implies that:
$\angle'(E^s_0,\dot\gamma)\ge C_4^{-1}\theta
 e^{-2\epsilon n}/2$. 
 So the length of the projection of $\hat{\gamma}_0$ 
 to the $E^u_0$ axis, which we call $l_0$, has length (with respect to the Lyapunov metric) of at least 
 $\displaystyle
\len'(\gamma_0) \sin(C_4^{-1}\theta
 e^{-2\epsilon n}/2).
$
 Thus 
\begin{equation}\label{eqn:l_0_lower_bnd}
l_0\ge \len'(\gamma_0) \sin(C_4^{-1}\theta
 e^{-2\epsilon n}/2) \ge C_5\theta e^{-2\epsilon n}\min\{e^{-R}\theta,e^{-.9\lambda n}\}\end{equation}

 Also by Lemma \ref{lem:C_2_norm_as_graph_vs_curve_R_2}
$$ 
    \|\hat{\gamma}_0\|_1\le \cot(C_4^{-1}\theta
 e^{-2\epsilon n})\le 2C_4 \theta^{-1}e^{2\epsilon n}.
$$

We apply Proposition \ref{prop:1-step_smoothing_estimate}\eqref{item:one_step_C_2_smoothing}, and get an $\epsilon_0<1/\sqrt{3}$,
 which is the cutoff for the one step $C^2$ smoothing estimate \eqref{eqn:one_step_smoothing_C_2} to hold.

In keeping with the previous proposition, denote 
$$\epsilon_{1,i}=\|\hat{f}_{i,1}\vert_{B_i}\|_1\quad\text{and}\quad \epsilon_{2,i}=\|\hat{f}_{2,i}\vert_{B_i}\|_1.$$ 
Because $\hat{f}_i=D_0\hat{f}_i+(\hat{f}_1,\hat{f}_2)$, we see from the $C^2$ bound on $\hat{f}_i$ that on $B_i$, 
\begin{equation}\label{eqn:estimate_on_f_i_lyapunov_charts}
\max\{\epsilon_{1,i},\epsilon_{2,i} \}\le r_i\|\hat{f}_i\|_2\le C_2e^{-.9\lambda(n-i)}e^{6(n-i)\epsilon}=C_2e^{-(.9\lambda-\epsilon)(n-i)}.
\end{equation}
We now proceed to the main part of the proof.

\noindent\textbf{Step 1.} We begin by checking  that if we inductively define: $\hat{f}_i\hat{\gamma}_i\vert_{B_i}=\hat{\gamma}_{i+1}$, and, as before, $l_i$ is the length of the projection of $\hat{\gamma}_i$ to $E^u_i$ measured with respect to $\|\cdot\|_i'
$, then the sequence $\hat{\gamma}_i$ satisfies the following estimates:
\begin{align}
l_i& \ge \min\{r_i,e^{.99\lambda i} l_0\}\label{eqn:inductive_length_estimate},\\
\|\hat{\gamma}_i\|_1&\le \max\{2\theta^{-1} e^{-i\lambda},\epsilon_0\}\label{eqn:inductive_C_1_est}.
\end{align}

(1) ($l_i$): By Proposition~\ref{prop:1-step_smoothing_estimate}(1)
\[
l_{i+1}\ge \min\{(e^{.999\lambda}-\epsilon_{1,i}-\epsilon_{1,i}\|\hat{\gamma}_i\|_{1})l_i,r_i\}.
\]
Hence to verify \eqref{eqn:inductive_length_estimate}, it suffices to show that
\[
(e^{.999\lambda}-\epsilon_{1,i}-\epsilon_{1,i}\|\hat{\gamma}_i\|_{1})l_i\ge e^{.99\lambda}l_i.
\]
which follows  by \eqref{eqn:estimate_on_f_i_lyapunov_charts} and the inductive hypothesis \eqref{eqn:inductive_C_1_est}
if $C_2$ is chosen sufficiently small.

(2) We now check the estimate on $\|\hat{\gamma}_{i+1}\|_1$ assuming it holds for $i$.

To begin, from Proposition~\ref{prop:1-step_smoothing_estimate}\eqref{item:1-step_smoothing_C_1},
\begin{equation}
\|\hat{\gamma}_{i+1}\|_1\le (e^{-.999\lambda}\|\hat{
\gamma}_i\|_1+\epsilon_{2,i}+\epsilon_{2,i}\|\hat{\gamma}_i\|_1)(e^{.999\lambda}-\epsilon_{1,i}-\epsilon_{1,i}\|\hat{\gamma}_i\|_1)^{-1}.
\end{equation}
There are two cases depending on whether $\|\hat{\gamma}_i\|_1\ge \epsilon_0$ or not. If $\|\hat{\gamma}_i\|_1\ge \epsilon_0$, then as long as $C_2$ is chosen sufficiently small, then the second parenthetical term in the above equation is at most $e^{-.9\lambda}$
 by \eqref{eqn:estimate_on_f_i_lyapunov_charts}.
Hence
\[
\|\hat{\gamma}_{i+1}\|_1\le e^{-.9\lambda}(e^{-.999\lambda}\|\hat{\gamma}_i\|_1+\epsilon_{2,i}+\epsilon_{2,i}\|\hat{\gamma}_i\|_1)\le 
 e^{-.9\lambda}
\|\hat{\gamma}_{i}\|_1(e^{-.999\lambda}+\epsilon_{2,i}\epsilon_0^{-1}+\epsilon_{2,i}).
\]
Because $\epsilon_0>0$ is independent of $C_2$, if $C_2$ is sufficiently small then \eqref{eqn:estimate_on_f_i_lyapunov_charts} gives
$
\|\hat{\gamma}_{i+1}\|_1\le e^{-3\lambda/2}\|\hat{\gamma}_i\|_1,
$
 which concludes the proof since $\|\hat{\gamma}_i\|_1\le 2 \theta^{-1} e^{-i\lambda}$. 

We now consider the case $\|\hat{\gamma}_i\|_1\le \epsilon_0$. In this case it suffices to show that $\|\hat{\gamma}_{i+1}\|_1\le \epsilon_0$. The argument in this case is similar and follows because, as in the previous case, we may ensure that $\epsilon_{1,i}, \epsilon_{2,i}$ are small relative to $\epsilon_0$ through our initial choice of $C_2$.

Thus we have shown that both estimates hold inductively
 proving \eqref{eqn:inductive_length_estimate} and \eqref{eqn:inductive_C_1_est}.

We now conclude item \eqref{item:length_of_curve_seq_1}. 
 Since the Lyapunov metric $\|\cdot\|_n'$ is uniformly comparable to the ambient metric $\|\cdot\|_n$ due to
\eqref{eqn:estimate_on_metric_smoothing_lemma}, it is enough to prove the lower bound on $\len'(\gamma_n).$
Thus the length of $\gamma_0$ is at least $\min\{r_n,e^{.9\lambda n}l_0\}$. Note that 
\[
e^{.9\lambda n}l_0\ge C_5e^{-R}\theta^2 e^{-2\epsilon n}.
\]
Hence if the minimum of $\min\{r_n,e^{.9\lambda n}l_0\}$ is realized by $r_n$, then $r_n=C_2$ because the first term in the definition of $r_n$ (see \eqref{eqn:r_i_defn}) is bigger than $e^{.9\lambda n}\ell_0$. 
 This
shows that $\len'(\gamma_n)\geq \min(\ell_0 e^{.9\lambda n}, C_2)$, completing the proof of part (1).

We now check the claim about the length of the preimage of $\hat{f}^n\hat{\gamma}_0$ in part \eqref{item:preimages_of_recovered_curves_shrink}. This is immediate from our choice of $\len'(\hat{\gamma}_0)$ in \eqref{eqn:length_of_initial_segment_smoothing_lem}. Because the preimage of $\gamma_n$ is contained in a segment of length at most $e^{-.9\lambda n}$ with respect to the Lyapunov metric, and because  $\|\cdot\|_0\le \sqrt{2}\|\cdot\|_0'$, this implies that the length of the initial segment we consider with respect to the ambient metric is at most $\sqrt{2}e^{-.9\lambda n}$. Similar considerations give the claim about the length of the preimage of $\gamma_n$ in $f^i(\gamma)$ at the end of item \eqref{item:preimages_of_recovered_curves_shrink}. Note that the final curve $\gamma_n$ promised by the lemma is not unique: for instance, it need not be centered at $f^n(x)$. The final claim in item \eqref{item:preimages_of_recovered_curves_shrink} follows because any such curve is long enough that it fills the entire segment of $f^i(\gamma)$ we are considering by our choice of $r_i$.

To finish the proof of item \eqref{item:preimages_of_recovered_curves_shrink}, we must see how large $n$ must be in order too ensure that $r_n=\ell_{\max}$. For this to occur $n$ must satisfy
$\displaystyle
e^{.99\lambda n}l_0\ge \ell_{\max}.
$
That is,
$\displaystyle
n\ge \frac{\ln(\ell_{\max})-\ln(l_0)}{.99\lambda}
.$
Now the definition of $l_0$ (see \eqref{eqn:l_0_lower_bnd}) gives
\[
n\ge \frac{\ln(\ell_{\max})-\ln( C_5\theta e^{-2\epsilon n}\min\{e^{-R}\theta,e^{-.9\lambda n}\})}{.99\lambda}.
\]
Now the needed conclusion in item \eqref{item:preimages_of_recovered_curves_shrink}  follows  by considering the two cases depending on which term realizes the minimum and using that $\epsilon<\lambda/30$.

\noindent\textbf{Step 2.} We now obtain item \eqref{item:C_2_estimate_smoothing_seq}, the $C^2$ estimate on $\hat{\gamma}_i$. Should it happen that there is an index $i$ such that $\|\hat{\gamma}_i\|_1\le \epsilon_0$, we call this index $N_0$. 
We proceed under the assumption that there is some such $N_0$. After concluding in this case, we explain how the same estimate holds otherwise.
Observe that if $\|\hat{\gamma}_i\|_1\le \epsilon_0$, then for all $j\ge i$, $\|\hat{\gamma}_j\|_1\le \epsilon_0$ as well. Keeping in mind the strength of hyperbolicity from \eqref{eqn:strength_lambda_prime},
 for all indices $i\ge N_0$, we have from \eqref{eqn:one_step_smoothing_C_2}, that 
\begin{equation}\label{eqn:one_step_C_2_hat_gamma_i}
\|\hat{\gamma}_{i+1}\|_2\le e^{-1.99\lambda \cdot .999}\|\hat{f}_i\|_2+e^{-2.99\lambda\cdot.999}\|\hat{\gamma}_i\|_2.
\end{equation}
By applying the above equation iteratively, we can obtain an estimate on $\|\hat{\gamma}_n\|_2$ in terms of $\|\hat{\gamma}_{N_0}\|_2$.  This gives the required estimate because the homogeneous part of 
\eqref{eqn:one_step_C_2_hat_gamma_i} has multipliers smaller than 1.

By \eqref{eqn:C_2_Lyapunov_chart_est1}, $\|\hat{f}_i\|_2\le D_1e^{6C} e^{6(n-i)\epsilon}$. Let $M=n-N_0$. 
Applying iteratively  \eqref{eqn:one_step_C_2_hat_gamma_i}, we get
\begin{equation}\label{eqn:upperbound_1_C_2_gamma}
\|\hat{\gamma}_n\|_2\le \|\hat{\gamma}_{N_0}\|_2e^{-2.99\lambda\cdot .999 M}+\sum_{i=1}^M D_1e^{6C}e^{6(n-i)\epsilon} e^{-1.99\lambda \cdot .999}e^{-2.99\lambda\cdot.999(M-i-1)}.
\end{equation}
Note that the second term is bounded by a constant $C_6$ depending only on $C,\lambda$ and $\epsilon$. 

To conclude, we also need a bound for $\|\hat{\gamma}_{N_0}\|_2$. By  Lemma \ref{lem:growth_of_C_2_norms}, there exists $\Lambda$ depending only on the $C^2$ norm of the maps $f_i$, which is uniformly bounded by $D_1$, such that 
\begin{equation}\label{eqn:C_2_bound_gamma_i}
\|f^i\gamma\|_{C^2}\le e^{\Lambda i}\max\{ \|\gamma\|_{C^2},1\}.
\end{equation}
Hence $\|\gamma_{N_0}\|\le e^{\Lambda N_0}\max\{\|\gamma\|_{C^2},1\}$. We then need an estimate on $\hat{\gamma}_{N_0}$. Note that in the Lyapunov coordinates  that $\hat{\gamma}_{N_0}$, which as a graph over $E^u_0$ has slope at most $\epsilon_0<1/\sqrt{3}$. Thus by Lemma \ref{lem:C_2_norm_as_graph_vs_curve_R_2} ,
\[
\|\hat{\gamma}_{N_0}\|_2\le \sin(\text{arccot}(\epsilon_0))^{-3}e^{N_0\Lambda}\max\{\|\gamma\|_{C^2},1\}
\le
 2e^{\Lambda N_0}\max\{\|\gamma\|_{C^2},1\},
\]
because $\epsilon_0<1/\sqrt{3}=\tan(\pi/6)=\cot(\pi/3)$. Combining this with \eqref{eqn:upperbound_1_C_2_gamma},
\begin{equation}
\|\hat{\gamma}_n\|_2\le 2e^{\Lambda N_0}e^{-2.99\lambda\cdot .999 M}\max\{\|\gamma\|_{C^2},1\}+C_6.
\end{equation}
But we also have a straightforward estimate for the cutoff $N_0$. From equation \eqref{eqn:inductive_C_1_est}, we know that $N_0\le(\ln(2)-\ln(\theta))/\lambda$. Hence because $N_0\approx -\ln(\theta)$, it is straightforward to see that there exist $C_7,C_8$ such that 
\begin{equation}
\|\hat{\gamma}_n\|_2<C_6+C_7e^{-2.9\lambda n}e^{C_8\ln\theta}\max\{\|\gamma\|_{C^2},1\}.
\end{equation}

In the case that there is no index $i$ such that $\|\hat{\gamma}_i\|_1\le \epsilon_0$, we may conclude similarly as equation \eqref{eqn:inductive_C_1_est} implies that $n\le (\ln(2)-\ln(\theta))/\lambda$. Thus we have finished with Step 2 and conclude item \eqref{item:C_2_estimate_smoothing_seq}.

Before going to Step 3, we record an additional more precise estimate on the rate that $\|\hat{\gamma}_i\|_{2}$ improves. Similar to above, we find:
\begin{align}
\|\hat{\gamma}_{N_0+i}\|_{2}&\le \|\hat{\gamma}_{N_0}\|_2e^{-2.99\lambda\cdot .999 i}+\sum_{j=1}^i D_1e^{6C}e^{6(n-N_0-j)\epsilon} e^{-1.99\lambda \cdot .999}e^{-2.99\lambda\cdot.999(i-j-1)},\notag \\
&\le \|\hat{\gamma}_{N_0}\|_2e^{-2.99\lambda\cdot .999 i}+e^{6(n-N_0)\epsilon}e^{-1.99\lambda \cdot .999}e^{-6i\epsilon}D_1e^{6C} \sum_{k=1}^ie^{-(2.99\lambda\cdot.999-6\epsilon)(k-1)},\notag \\
&\le 2e^{\Lambda N_0}e^{-2.99\lambda\cdot .999 i}\max\{\|\gamma\|_{C^2},1\}+C_9e^{6\epsilon (n-N_0-i)},\label{eqn:finite_time_hat_gamma_est}
\end{align}
for some $C_9>0$.

\noindent\textbf{Step 3.}
We now show item \eqref{item:regularity_of_density_sequence_lem}, i.e.~we obtain estimates for smoothing a density along $\gamma$. 
We let $\hat{\rho}_i$ be the function giving the density $\rho$ on $\hat{\gamma}_i$ in the Lyapunov coordinates. 

We now apply the smoothing estimate. 
As in Step 2, supposing it exists, let $N_0$ be the first index such that $\|\hat{\gamma}_i\|_1\le \epsilon_0$. If such an index $N_0$ does not exist, then we may conclude similarly to in Step 2.
Then for any $i\ge N_0$, by \eqref
{eqn:smoothing_rho_1_step_chart_est},
\begin{equation}
\|\ln \hat{\rho}_{i+1}\|_{C^{\alpha}}\le \lambda^{-.9\alpha}(\|\ln \hat{\rho}_i\|_{C^{\alpha}}+\|\hat{f}_i\|_2+\|\hat{\gamma}_i\|_2).
\end{equation}
As before, let $M=n-N_0$.
By a bookkeeping similar to  Step $2$, we find that 
\begin{equation*}
\|\ln\hat{\rho}_n\|_1\le e^{-.9\cdot.999\lambda \alpha M}\|\ln \hat{\rho}_{N_0}\|_{C^{\alpha}}+\sum_{i=1}^M e^{ -.9\cdot.999\lambda \alpha (M-i)}(\|\hat{f}_{N_0+i}\|_{C^2}+\|\hat{\gamma}_{N_0+i}\|_{C^2}). 
\end{equation*}
By \eqref{eqn:finite_time_hat_gamma_est} and \eqref{eqn:C_2_Lyapunov_chart_est1}, we see that there exists $C_{11}$ such that 
\begin{align}\label{eqn:rho_n_final_est}
\|\ln \hat{\rho}_n\|_{C^{\alpha}}&\le e^{-.9\cdot.999\lambda \alpha M}\|\ln \hat{\rho}_{N_0}\|_{C^{\alpha}}+2e^{\Lambda N_0}\|\gamma\|_{C^2}e^{-.9\lambda\cdot .999 M}+C_{11}.
\end{align}

We now estimate $\|\ln \hat{\rho}_{N_0}\|$. We first obtain an estimate without the use of the Lyapunov charts.  By Lemma \ref{lem:growth_of_C_2_norms} because of the uniform $C^2$ bound $D_1$, there exist $C_{12},\Lambda>0$ such that 
\[
\|\ln(f^{N_0})_*\rho\|_{C^{\alpha}}\le C_{12}(e^{\Lambda N_0}+e^{\Lambda N_0}\|\ln \rho\|_{C^{\alpha}}).
\]
Next, we push forward $\gamma_{N_0}$ and $\rho_{N_0}$ by $L_{N_0}$ to obtain a density in the Lyapunov coordinates. Because $\max\{\|L_{N_0}\|,\|L_{N_0}\|^{-1}\}\le C_1e^{2C}e^{2\epsilon n}$, 
Lemma \ref{lem:pushforward_density_est_diffeo}  gives that there exists $C_{13}$ such that 
\[
\|\ln (L_{N_0})*(f^{N_0}_1)_*\rho\|_{C^{\alpha}}\le C_{13}e^{(2+2\alpha)\epsilon n}\left( e^{\Lambda N_0}+e^{\Lambda N_0}\|\ln \rho\|_{C^{\alpha}} +e^{2\epsilon n}(1+\|\gamma_{N_0}\|_{C^2}) \right).
\]
For the application we are then interested in the regularity of $(Lf^{N_0}_1)_*\rho$ as a function parametrized by $E^u_0$. As at time $N_0$, $\gamma_{N_0}$ is uniformly transverse to $E^s_0$, this projection has uniformly bounded norm. From before, we have the $C^2$ bound on $\gamma_{N_0}$ following \eqref{eqn:C_2_bound_gamma_i}, which gives that there exists $C_{14}$ such that:
\begin{equation}\label{eqn:estimate_on_rho_N_0}
\|\ln\hat{\rho}_{N_0}\|_{C^{\alpha}} \le C_{14}e^{7\epsilon n}e^{\Lambda N_0}(1+\|\ln \rho\|_{C^{\alpha}}+\|\gamma\|_{C^2}). 
\end{equation}

Combining this with \eqref{eqn:rho_n_final_est}, we find 
\begin{equation*}
\|\ln \hat{\rho}_n\| \le e^{-.9\cdot.999\lambda \alpha M}(C_{14}e^{7\epsilon n}e^{\Lambda N_0}(1+\|\ln \rho\|_{C^{\alpha}}+\|\gamma\|_{C^2}))+2e^{\Lambda N_0}\|\gamma\|_{C^2}e^{-.9\lambda\cdot .999 M}+C_{11}.
\end{equation*}
Then as before, because $N_0$ is order $\ln(\theta)$ and $M=n-N_0$,
\begin{equation}\label{eqn:est_on_rho_n}
\|\ln \hat{\rho}_n\|\le C_{15}e^{-.9\lambda\alpha n}e^{C_{16}\ln(\theta)}(1+\|\ln\rho\|_{C^{\alpha}}+\|\gamma\|_{C^2}). 
\end{equation}
As $\|\gamma_n\|_{C^2}<D_9$ for some fixed $D_9$ by assumption, then \eqref{eqn:est_on_rho_n} gives the corresponding estimate on $\rho$ with respect to the arclength parameters on $\gamma_n$, and we conclude item \eqref{item:regularity_of_density_sequence_lem}. 
\end{proof}

\subsection{Loss of regularity}
\label{SSLoss}
In this subsection, we prove some additional estimates that will be used later in the proof of mixing but not the proof of the coupling lemma. 
These estimates say that for all but an exponentially small amount of the curve $\gamma$, typically the images of points in $f^n_{\omega}(\gamma)$ are in a neighborhood that is at least $n\epsilon$-good. 
First we introduce  in Definition \ref{defn:forward_tempered_relative_to_a_curve}, 
a notion of a forward tempered point relative to a curve. 
Then, in Proposition \ref{prop:fwd_up_to_epsilon_smoothing}, we show that the image of a curve at a forward tempered time will be $18\epsilon n$ good.

We begin by stating the main definition of this section. Note that it is similar to definitions we also considered for backwards good points (Definition \ref{defn:backwards_good_time}).

\begin{defn}\label{defn:forward_tempered_relative_to_a_curve}
For a standard pair $\hat{\gamma}=(\gamma,\rho)$ and a word $\omega\in \Sigma$, we say that $n$ is a $(C,\lambda,\epsilon,\theta)$-forward tempered time for $x\in \gamma$ if the sequence of maps $(D_xf^i_{\omega})_{1\le i\le n}$ is $(C,\lambda,\epsilon$)-subtempered and the most contracted direction of $D_xf^n_{\omega}$ exists and is at least $\theta$-transverse to $\gamma$. Similarly, we speak of a trajectory being forward tempered relative to a vector $v\in T_xM$.
\end{defn}

The following lemma gives a quantitative estimate on the length of an image of a curve experiencing a forward tempered time.

\begin{prop}\label{prop:fwd_up_to_epsilon_smoothing}
Suppose that $M$ is a closed surface and that $(f_1,\ldots,f_m)$ is a tuple in $\Diff^2_{\vol}(M)$. 
Then for any $\lambda>0$ and $C_1>0$ there exist $D_0,D_1>0$ and $N\in \N$, such that for all $\theta>0$ and $\lambda/30>\epsilon>0$,
if $\hat{\gamma}=(\gamma,\rho)$ is a $C_1$-good standard pair, $\omega\in \Sigma$ and $x\in \gamma$ has a $(C,\lambda,\epsilon,\theta)$ forward tempered time at time 
\begin{equation}
 \label{N-Theta}   
n\ge N+D_0\ln\theta
\end{equation}
then

(1) The pushforward $f^n_{\omega}(\gamma)$ contains 
    a neighborhood of $f^n_{\omega}(x)$, $B(x)$, such that denoting by $\hat{B}(x)$ the restriction of the standard pair $f^n_{\omega}(\hat{\gamma})$ to $B(x)$, then $\hat{B}(x)$ is an $(18\epsilon n+18\max\{C,0\}+D_1)$-good standard pair.
    
  (2) The preimage of $\hat{B}(x)$, $(f^{n}_{\omega})^{-1}(\hat{B}(x))$, has length at most $e^{-(\lambda/2)n}$.
\end{prop}

\begin{proof}
As before, we will use the deterministic smoothing lemmas. 
We begin by first picking a choice of Lyapunov metrics
to use.
Applying Lemma \ref{lem:lyapunov_metric} with $\lambda'=.999\lambda$ we get,
since the trajectory is forward tempered, that
\begin{equation}\label{eqn:fwd_smoothing_lyapunov_metric}
    \frac{1}{\sqrt{2}}\|\xi\|_i\le \|\xi\|_i'\le  4e^{2C+2\epsilon
i}\pez{1-e^{2(\lambda'-\lambda)}}^{-1/2}\|\xi\|_i.
\end{equation}
As in the proof of Proposition \ref{prop:finite_time_smoothing_estimate}, using the Lyapunov metric, we obtain new dynamics $\hat{f}_i$ in the Lyapunov coordinates, which are given by composing with a sequence of maps $L_i$.
Crucially, these dynamics satisfy that
$\displaystyle
D_0\hat{f}_i\vert_{E^s_i}\le e^{-.999\lambda}$ and $\displaystyle D_0\hat{f}_i\vert_{E^u_i}\ge e^{.999\lambda}.
$
Moreover, we can write:
\[
\hat{f}_i(x,y)=(\sigma_{1,i}x+\hat{f}_{i,1}(x,y),\sigma_{2,i}y+\hat{f}_{i,2}(x,y)).
\]
Further there exists $C_2$ such that 
\begin{equation}\label{eqn:norm_of_Lyapunov_chart_2}
\max\{\|L_i\|,\|L_i\|^{-1}\}\le C_2e^{2C+2\epsilon i}.
\end{equation}
Proceeding as in \eqref{eqn:C_2_Lyapunov_chart_est1}, there exists $D$ such that:
\begin{equation}\label{eqn:C_2_est_Lyapunov_charts_2}
\|\hat{f}_i\|_{C^2}\le De^{6C+6i\epsilon}.
\end{equation}
From here, we set up the constants in a manner similar to before. Things are slightly simpler because by assumption the standard pair is $C_1$-good and hence uniformly long and good. We will take some small $C_3>0$ that we will choose later. Set
\[
r_i=C_3e^{-10\epsilon n-6C}\min\{\theta^{-1},e^{-.9\lambda(n-i)}\}.
\]
As before, we let $B_i$ be the square of side length $r_i$ centered at $0$ with 
respect to the $\|\cdot\|'_i$ metric. As in the previous argument, we let 
$\hat{\gamma}_i$ denote the portion of $f^{i-1}(\hat\gamma)$ lying in $B_{i}$ 
and we let $\hat{\rho}_n$ denote the density along $\hat{\gamma}_n$. Let 
$\epsilon_0>0$ be the cutoff so that \eqref{eqn:one_step_smoothing_C_2} holds 
in Proposition \ref{prop:1-step_smoothing_estimate}

As above, we denote $\epsilon_{1,i}=\|\hat{f}_{i,1}\vert_{B_i}\|_1$ and 
$\epsilon_{2,i}=\|\hat{f}_{2,i}\vert_{B_i}\|_1$. Because 
$\hat{f}_i=D_0\hat{f}_i+(\hat{f}_1,\hat{f}_2)$, we see that from the $C^2$ 
bound on $\hat{f}_i$ that on $B_i$, 
\begin{equation}\label{eqn:estimate_on_f_i_lyapunov_charts_fwd_tempered}
\max\{\epsilon_{1,i},\epsilon_{2,i} \}\le r_i\|\hat{f}_i\|_2
\le C_3e^{-10\epsilon n-6C}e^{-.9\lambda(n-i)}e^{6C}e^{6(n-i)\epsilon}\le
C_3e^{-.9\lambda(n-i)}.
\end{equation}
In particular, note that by choosing $C_3$ sufficiently small in a manner that only depends on
$\lambda$, we may ensure that for all $i$ that $\max\{\epsilon_{1,i},
\epsilon_{2,i}\}<\epsilon_0$.

We then carry out an inductive argument to determine the regularity of $\hat{\gamma}_n$.
In order, we obtain estimates on the length, the $C^2$ norm, and then $\|\ln f^n(\rho)\|_{C^{\alpha}}$.

\noindent\textbf{Step 1. }(Length of the curve)
As in the proof
of Proposition \ref{prop:finite_time_smoothing_estimate}, we see that from the 
choice of constants $\hat{\gamma}_n$ is uniformly transverse to $E^s_n$ and the 
projection of its graph to the $E^u_{n}$ axis fills $E^u_n\cap B_{n}$. Thus 
there exists $C_4$, depending only on $C_3$ such that 
$\hat{\gamma}_n$ has length at least 
$C_4e^{-\epsilon n}$ in the Lyapunov charts. 
By equation 
\eqref{eqn:norm_of_Lyapunov_chart_2}, this implies that,
in the ambient metric,
$f^n_{\omega}(x)$ lies in a
neighborhood of length at least 
\begin{equation}\label{eqn:length_in_final_chart_1}
C_2^{-1}C_4e^{-2C-3\epsilon n}.
\end{equation}

\noindent\textbf{Step 2. }($C^2$ norm of the curve)
We now turn to an estimate on the $C^2$ norm of $\hat{\gamma}_n$.
This is perhaps the most complicated part of the argument along with the estimate on 
smoothing the density.
We apply the estimate 
\eqref{eqn:one_step_smoothing_C_2} from Proposition 
\ref{prop:1-step_smoothing_estimate}. Let $N_0$ be the first iterate such that 
$\|\hat{\gamma}_{N_0}\|_1<\epsilon_0$. From our choice of the size of the 
neighborhood and the comment on the size of $C_3$ immediately after 
\eqref{eqn:estimate_on_f_i_lyapunov_charts_fwd_tempered}, we have that for all 
$i\ge N_0$, the estimate \eqref{eqn:one_step_smoothing_C_2} holds, i.e.~the $C^2$ smoothing estimate is valid. 
Thus we find that:
\begin{equation}
\|\hat\gamma_{i+1}\|_2\le \lambda^{-1.99}\|\hat{f}_i\|_2
+\lambda^{-2.99}\|\hat{\gamma}_{i}\|_2.
\end{equation}
From  \eqref{eqn:C_2_est_Lyapunov_charts_2}, it follows inductively 
that:
\begin{equation}\label{eqn:prelimin_C_2_n_estimate1}
\|\hat\gamma_n\|_2\le e^{-2.99\lambda (n-N_0)} \|\hat{\gamma}_{N_0}\|_2+De^{6C}e^{6\epsilon n}\sum_{j=0}^{n-N_0-1} e^{-1.99\lambda j}e^{6\epsilon j}.
\end{equation}
We then need to estimate $N_0$.
As in the proof of Proposition \ref{prop:finite_time_smoothing_estimate}, we get 
$N_0\!\!=\!\!O_{\lambda}(-\ln(\theta))$.
Thus there exists $C_4,C_5>0$, such that 
\[
\|\hat\gamma_n\|_2\le e^{-2.99 n\lambda}e^{-C_4\ln \theta}\|\hat\gamma_{N_0}\|_2+C_5e^{6C}e^{6\epsilon n}.
\]

As in the proof of Proposition \ref{prop:finite_time_smoothing_estimate}, after \eqref{eqn:C_2_bound_gamma_i}, we see that there exists $\Lambda>0$ such that $\|f^i\gamma\|_{C^2}$, with respect to the ambient metric is at most $e^{\Lambda i}$. 
Using \eqref{N-Theta} and the fact that the angle between $\gamma$ and 
$E^s_{N_0}$ is uniformly large, 
we see that there exists $C_6$ such that with respect to the Lyapunov metric, 
\[
\|\hat{\gamma}_{N_0}\|_2\le e^{-C_6\ln\theta}.
\]
Thus for some $C_7$,
\begin{equation}\label{eqn:hat_gamma_n_eqn_22}
\|\hat\gamma_n\|_2\le e^{-2.99 n\lambda}e^{-C_7\ln \theta}+C_5e^{6C}e^{6\epsilon n}.
\end{equation}
We now record an intermediate estimate that will be useful later.
By possibly increasing the constants, for each $N_0\le i\le n$, we find:
\begin{equation}\label{eqn:intermediate_est_C_2_norm_i}
\|\hat\gamma_i\|_2\le e^{-2.99 i\lambda}e^{-C_7\ln \theta}+C_5e^{6C}e^{6\epsilon i}.
\end{equation}

Equation \eqref{eqn:hat_gamma_n_eqn_22} is an estimate in the Lyapunov chart,
but we need the estimate with respect to the original metric. 
The $C^2$ norm of $\hat{\gamma}_n$ as a curve is uniformly comparable to $\|\hat{\gamma}_n\|_2$ because $\hat{\gamma}_n$ is uniformly transverse to $E^s_n$. 
By Lemma \ref{lem:linear_map_C_2_norm_curve_est} there exists
$C_7$, such that, letting $\gamma_n$ be the segment of $\gamma$ lying in $B_n$,
 we get the following bound in the ambient metric
\[
\|\gamma_n\|_{C^2}\le  (e^{-2.99 n\lambda}e^{-C_7\ln \theta}+C_5e^{6C}e^{6\epsilon n})C_2^3e^{6C+6\epsilon n}.
\]
This is the bound required by the proposition. Indeed for $D_0$ sufficiently large we have:
\begin{equation}\label{eqn:final_est_gamma_n_good_curves}
\|\gamma_n\|_{C^2}\le C_{8}e^{12\max\{C,0\}+12\epsilon n}.
\end{equation}

\textbf{Step 3.} (Regularity of the density) Finally, we turn to estimating the H\"older norm of the pushed density. 
At the same iterate $N_0$ from Step 2, we have that $\epsilon_{N_0,1},\epsilon_{N_0,2},\|\hat{\gamma}_{N_0}\|_1\le \epsilon_0$ and that these estimates hold for all future iterates. 
Consequently, estimate \eqref{eqn:smoothing_rho_1_step_chart_est} applies, hence for
$N_0\le i\le n-1$,
\[
\|\ln \wt\rho_{i+1}\|_{C^\alpha}\le e^{-.9\alpha \lambda}(\|\ln \wt{\rho}_i\|_{C^{\alpha}}+\|\hat{f}_i\|_{2}+\|\hat{\gamma}_i\|_2). 
\]
This leads inductively to the estimate that
\begin{equation}\label{eqn:rho_n_1st_inductive_est}
\|\ln \wt{\rho}_n\|_{C^{\alpha}}\le e^{-.9\alpha \lambda (n-N_0)}\|\ln\wt{\rho}_{N_0}\|_{C^{\alpha}}+\sum_{i=N_0}^{n-1}e^{-.9\lambda\alpha(n-i)}(\|\hat{f}_i\|_2+\|\hat{\gamma}_i\|_2).
\end{equation}
We then need some further estimates in order to simplify this. 

We start with an estimate on $\|\ln \wt{\rho}_{N_0}\|_{C^{\alpha}}$. A similar argument to that giving \eqref{eqn:estimate_on_rho_N_0} yields that $\|\ln \wt{\rho}_{N_0}\|_{C^{\alpha}}\le e^{\Lambda N_0}$, where $\Lambda>0$ only depends on the $C^2$ norm of the diffeomorphisms and the initial regularity of $\gamma$. Hence as long as $D_0$ is large enough, it follows that the first term is uniformly bounded.

For the other terms, we already have estimates for $\|\hat{f}_i\|_2$ and $\|\hat{\gamma}_i\|_2$, \eqref{eqn:C_2_est_Lyapunov_charts_2} and \eqref{eqn:intermediate_est_C_2_norm_i}. 
These yield a bound on the sum in 
 \eqref{eqn:rho_n_1st_inductive_est}:
$$
\sum_{i=N_0}^{n-1}\!\!e^{-.9\lambda\alpha(n-i)}(\|\hat{f}_i\|_2+\|\hat{\gamma}_i\|_2) 
$$$$
\le \sum_{j=0}^{n-N_0-1} \!\! e^{-.9\lambda \alpha j}\left(De^{6C}e^{6\epsilon(n-j)}\!\!+\!e^{-2.99\lambda(n-j)}e^{-C_7\ln(\theta)}\!\!+\!C_5e^{6C}e^{6\epsilon(n-j)}\right)\!\!.
$$
The sum of the first and third terms inside the parentheses is straightforward to evaluate. There is a constant $C_{9}$ such that each is bounded by $C_9e^{6C}e^{6\epsilon n}$. 
The terms involving the $\ln(\theta)$ are only slightly more complicated as either $j$ or $n-j$ is large, hence the terms involving $\lambda$  dominate the $e^{-C_7\ln\theta}$ term as long as $D_0$ is large enough.
Thus by the above estimates, it follows that as long as $D_0$ is sufficiently large that there exists $C_{10}$ such that
\[
\|\ln\wt{\rho}_n\|_{C^\alpha}\le C_{10}e^{6\max\{C,0\}}e^{6\epsilon n}.
\]
This is the form of the estimate in the Lyapunov charts. We then need to pass back to the original metric. 
Applying Lemma \ref{lem:pushforward_density_est_diffeo} 
we see that 
letting $C'$ denote the constant from that lemma 
and using \eqref{eqn:norm_of_Lyapunov_chart_2} and \eqref{eqn:final_est_gamma_n_good_curves})
 we get,
\begin{align*}
\|\ln \rho_n\|_{C^{\alpha}}&\le e^{(1+\alpha)(2C+2\epsilon n)}\left(C_{10}e^{6\max\{C,0\}+6\epsilon n} +C'e^{2C+\epsilon n}(1+C_8e^{12\max\{C,0\}+12\epsilon n})\right)\\
&\le C_{11} e^{18\max\{C,0\}}e^{18\epsilon n}.
\end{align*}
This is the needed conclusion, so we are done.
\end{proof}

\section{Finite time Pesin theory and fake stable manifolds}\label{sec:finite_time_pesin_theory}
\subsection{Fake stable manifolds}\label{subsec:fake_stable_manifolds}
In the proof of the coupling lemma, we will use the holonomies of some ``fake" stable manifolds $W^s_n(\omega,x)$. These manifolds  behave for  finite a time like a true stable manifold insofar as they contract. 
We then prove some lemmas about fake stable curves. 
Some of the results below are variants on standard facts in Pesin theory, however, some of the proofs are a little different due to us only using a finite portion of an orbit. For other facts that look  standard we needed to supply our 
own proofs because we could not find a similar enough statement in the literature.

For a given word $\omega$ and $n\in \N$ the fake stable manifolds are curves that have analogous properties to the stable manifolds up until time $n$. So, unlike true stable manifolds, they are not canonically defined.

Before we begin we recall some notation. Throughout this section we will write 
$\Lambda^{\omega}_n(C,\lambda,\epsilon)$ for the set of   $(C,\lambda,\epsilon)$-tempered 
points $x\in M$
at time $n$ for the word $\omega\in \Sigma$. This is essentially the finite time version of a Pesin block. For many of the results there is a lower bound  on $n$, which is required to ensure that the orbit is actual experiencing hyperbolicity.

Below we will make a number of arguments concerning these fake stable manifolds. The main properties we need concern the holonomies between two transversals to the $W^s_n$ lamination. We need to know that the $W^s_n$ holonomies have a uniformly H\"older continuous Jacobian independent of $n$. In addition we would like to know that as $n\to \infty$ that the holonomies are converging exponentially quickly to the true stable holonomy. 

Before proceeding to the proof, we remark that there are other approaches to fake stable manifolds that are adapted to different sorts of dynamical problems and may differ from each other substantially. For example,  Burns and Wilkinson \cite{burns2010ergodicity}, which originated the term fake manifold, use fake center and stable manifolds where a potentially different fake foliation is defined at every point in the manifold. A different approach in Dolgopyat, Kanigowski, Rodriguez-Hertz \cite{dolgopyat2024exponential} uses a fake foliation that is globally defined but does not cover the entire Pesin regular set. 
 Note that, in contrast with our setting, \cite{burns2010ergodicity} and \cite{dolgopyat2024exponential}
allow systems with some zero exponents, and so the invariant manifolds need not be unique in their settings.
One benefit of the construction described below is that it applies to \emph{every} point in a Pesin block and further gives a single fake stable lamination defined on the manifold rather than a collection of different overlapping laminations. While this makes the fake stable lamination simple to think about, it requires more work to show that it exists.

\subsection{Preliminaries}

Here we present some background that will be used in the next subsection to study the regularity of $E_n^s$.

We start with a useful fact for showing that the limit of a sequence of functions is H\"older continuous. This fact is completely standard.
Note that the statement is false if the diameter of $M_2$ is unbounded. Also, recall that in our setup, the H\"older constant only applies to estimates on the distance between $g(x)$ and $g(y)$ for points with $d(x,y)\le 1$.

\begin{lem}\label{lem:holder_seq_convergence_lemma}
Suppose that $M_2$ is a metric space with bounded diameter. Fix $\eta,\lambda,\delta,\beta>0$. Then there exists $0<\alpha<\beta$ and $D(\eta,\lambda,\delta,\beta,\alpha)$ such that  for any metric space $M_1$
the following holds. Let $g_n\colon M_1\to M_2$, $1\le n\le N$ be a finite or infinite sequence of $\beta$-H\"older continuous functions such that:

\noindent (1)
For $1\le n< N$, $d_{C^0}(g_n,g_{n+1})\le C_1e^{-\delta n}$.

\noindent (2)
The function $g_n$ is $C_3e^{\eta n}$ $\beta$-H\"older continuous
at scale $e^{-C_2}e^{-\lambda n}$,
i.e., if $d(x,y)\le e^{-C_2}e^{-\lambda n}$ then
$ d(g_n(x),g_n(y))\le C_3e^{\eta n}d(x,y)^{\beta}$.

\noindent Then the functions $g_1,\ldots,g_N$ in the sequence, as well as the possible limiting value of the sequence are all uniformly $\alpha$-H\"older with constant at most 
\[
\max\left\{De^{C_2\alpha},2C_1(1-e^{-\delta})^{-1}e^{\beta C_2}D^{\beta}+C_3e^{-C_2(\beta-\alpha)}\right\}.
\]
\end{lem}
\begin{proof}
We will assume throughout the proof that $g_N$ is fixed and obtain an estimate for $g_N$ that is independent of $N$. 
As the resulting estimate is independent of $N$, the conclusion holds for infinite sequences as well.

To begin we pick some constants.
First, for fixed $\eta>0$ and any $0<\alpha_1<\beta$ there exists $\gamma\ge \lambda$ such that
\begin{equation}\label{eqn:choice_of_alpha_1}
\eta-\gamma\beta \le -\alpha_1\gamma\quad\text{and}\quad \eta<\gamma\alpha_1. 
\end{equation}
Note that $\gamma$ only depends on $\eta,\alpha_1,\lambda,\beta$, but not on $C_1,C_2,C_3$.

Next given $\delta$, let $0<\alpha_2<\beta$ be sufficiently small that we have
\begin{equation}\label{eqn:choice_of_alpha_2}
\delta\ge \alpha_2\gamma.
\end{equation}

Due the first assumption, we have a uniform estimate independent of $N$:
\[
\abs{g_N-g_n}\le \sum_{i=n}^{N-1} C_1e^{-i\delta}\le \frac{C_1e^{-n\delta}}{1-e^{-\delta}}.
\]

Having picked those constants, now consider a pair of points $x,y\in M_1$. We consider three cases 
depending on how far apart $x$ and $y$ are. We proceed from closest to furthest away. 

(1) First suppose that $d(x,y)<\min\{e^{-C_2}e^{-\gamma N},1\}$. Then 
$$
d(g_N(x),g_N(y))\le C_3e^{\eta N}d(x,y)^{\beta}
\le C_3e^{\eta N}  d(x,y)^{\alpha_1} d(x,y)^{\beta-\alpha_1}
$$$$
\le C_3e^{\eta N}e^{-C_2\alpha_1}e^{-\gamma \alpha_1 N
}d(x,y)^{\beta-\alpha_1}
\le C_3e^{-C_2\alpha_1}d(x,y)^{\beta-\alpha_1},
$$
where we have used \eqref{eqn:choice_of_alpha_1}.

(2) Next, we consider the case where $e^{-C_2}e^{-\gamma}e^{-\gamma n}\le d(x,y)\le\min\{1, e^{-C_2}e^{-\gamma n}\}$ for some $1\le n< N$. By the choice of constants $\alpha_1,\alpha_2$ and $\gamma$ in the first part of the proof we find:
\begin{align*}
d(g_N(x),g_N(y))&\le  d(g_N(x),g_n(x))+d(g_n(x),g_n(y))+d(g_n(y),g_N(y))\\
&\le C_1e^{-n\delta}(1-e^{-\delta})^{-1}+C_1e^{\eta n}d(x,y)^{\beta}+C_1e^{-n\delta}(1-e^{-\delta})^{-1}\\
&\le 2C_1e^{-n\gamma\alpha_2}(1-e^{-\delta})^{-1}+C_3e^{\eta n}e^{- n\gamma \alpha_1}e^{-C_2\alpha_1}d(x,y)^{\beta-\alpha_1}.
\end{align*}
Then due to the lower bound on $d(x,y)$ and $\eta<\gamma\alpha_1$ from \eqref{eqn:choice_of_alpha_1}:
\begin{align*}
d(g_N(x),g_N(y))&\le 2C_1(1-e^{-\delta})^{-1}e^{\alpha_2C_2}e^{\alpha_2\gamma}d(x,y)^{\alpha_2}+C_3e^{-C_2\alpha_1}d(x,y)^{\beta-\alpha_1}\\
&\le \left(2C_1(1-e^{-\delta})^{-1}e^{\alpha_2C_2}e^{\alpha_2\gamma}+C_3e^{-C_2\alpha_1}\right)d(x,y)^{\min\{\alpha_2,\beta-\alpha_1\}}. 
\end{align*}
 
(3) Finally we consider the case where $e^{-C_2}e^{-\gamma}<d(x,y)$.  Then we use a trivial bound
\[
d(g_N(x),g_N(y))\le \diam { (M_2)}\le \left(\frac{\diam { (M_2)}}{(e^{-C_2}e^{-\gamma})^{\beta-\alpha_1}}\right)d(x,y)^{\beta-\alpha_1}.
\]

Now using all three cases above, we may conclude. Note that the $(\beta-\alpha_1)$-H\"older constant obtained in the second item above is at least as big as the constant obtained in the first item in the list. Thus the function $g_N(x)$ is uniformly $(\beta-\alpha_1)$-H\"older with constant at most 
\[
\max\left\{(\diam M_2)e^{(C_2+\gamma)(\beta-\alpha_1)},2C_1(1-e^{-\delta})^{-1}e^{\alpha_2C_2}e^{\alpha_2\gamma}+C_3e^{-C_2\alpha_1}\right\}.
\]
As the choice of constants $\alpha_1,\alpha_2,\gamma$ depend only on $\delta,\eta$ we obtain the needed conclusion. 
\end{proof}

We will apply Lemma \ref{lem:holder_seq_convergence_lemma} to obtain regularity of $E^s_n$ after we 
 obtain small scale H\"older continuity of $E^s_n$.

 Next we present a perturbation result on the singular subspaces of linear transformations called Wedin's theorem.
This
theorem gives a bound on the change in the angle between the singular directions. We state a specialized version of this theorem adapted from the presentation in \cite[Thm.~4]{stewart1991perturbation}. First we describe the theorem in some generality, but below we give a precise statement for $\SL(2,\R)$ independent of the discussion and definitions mentioned below.
If $A$ and $\wt{A}$ are two $n\times n$ matrices then we may list their singular values as $\sigma_1\ge \cdots\ge \sigma_n$ and $\wt{\sigma}_1\ge \cdots \wt{\sigma}_n$. Write $\|E\|_F$ for the Frobenius norm of the matrix $E$, 
 i.e.~the $L^2$ norm of its entries viewed as a vector. Fix some index $k$ such that $\sigma_k\ge \wt{\sigma}_{k+1}$. If $\abs{\sigma_k-\wt{\sigma}_{k+1}}\ge \delta$, and 
$\wt{\sigma}_k\ge \delta$, then Wedin's theorem implies that:
\[
\|\sin\Phi\|_F\le \frac{\sqrt{2}\|E\|_F}{\delta},
\] where $\|\sin\Phi\|_F$ denotes the Frobenius norm of the matrix that defines the canonical angles between the right singular subspace associated to $\sigma_1,\ldots,\sigma_k$ and $\wt{\sigma}_1,\ldots,\wt{\sigma}_k$. (The matrix $\sin \Phi$ is defined by taking the inner products between an orthonormal basis of the right singular subspaces of $A$ and $\wt{A}$.) Note that the statement in \cite[Thm.~4]{stewart1991perturbation} is in terms of certain residuals, but by the comment before the theorem, these are bounded by $\|E\|_F$. Below we will use that the Frobenius norm of a $2$ by $2$ matrix satisfies the bound $\|E\|_F\le \sqrt{2}\|E\|$, where $\|E\|$ is the usual operator norm of the matrix \cite[5.6.P23]{horn2013matrix}. 

Although the statement from the above paragraph is somewhat technical, when both the matrix $A$ and its perturbation $A+E$ are in $\SL(2,\R)$, as is the case for us, the statement simplifies considerably. This is because for such a matrix $\sigma_1=\sigma_2^{-1}$ and the top singular value of a matrix in $\SL(2,\R)$ can change by at most $\|E\|$ when we perturb by $E$. If $\|A\|\ge 2$ and $E$ is a perturbation with $\|E\|\le \|A\|/2$, then 
\begin{align*}
\|A+E\|\ge \frac{1}{2}\|A\|,\\
\|A\|-(\|A+E\|)^{-1}\ge \|A\|- \frac{2}{\|A\|}\ge \frac{1}{2}\|A\|,
\end{align*}
as long as $\|A\|\ge 2$. So, we may apply Wedin's theorem with $\delta=\|A\|/2$. In this case, the matrix of canonical angles described above consists of a single number: the angle between the original most expanded singular direction and the new one. Thus we obtain the following proposition.

\begin{prop}\label{prop:SL2_perturbation_directions_lemma}
Suppose that $A$ is a matrix in $\SL(2,\R)$ with $\|A\|\ge 2$. Consider a perturbation $A+E\in \SL(2,\R)$
with $\|E\|\le \|A\|/2$. Denote by $v_A$ and $v_{A+E}$  the most expanded singular vectors of $A$ and $A+E$. Then
\[
\abs{\sin\angle(v_A,v_{A+E})}\le \frac{2\sqrt{2}\|E\|}{\|A\|}.
\]
\end{prop}

\subsection{Regularity of the most contracting direction}
We now estimate the regularity of $E^s_n(x)$, the most contracted direction of $D_xf^n_{\omega}$, on the set of $(C,\lambda,\epsilon)$-tempered points at time $n$ in terms of $C$. The approach to studying H\"older regularity here may be contrasted with the approach in Shub \cite[Thm.~5.18(c)]{shub1987global}. That approach establishes H\"older regularity for an invariant section of a bundle automorphism under an appropriate bunching condition by comparing the contraction in the fiber with the strength of hyperbolicity in the base. In some sense the approach is similar: it uses the dynamics to study the H\"older regularity at different scales. One can compare equation (***) there with our Lemma \ref{lem:holder_seq_convergence_lemma}.

\begin{prop}\label{prop:contracted_subspace_holder_est}
Suppose that $(f_1,\ldots,f_m)$ is a tuple of diffeomorphisms in $\Diff^2_{\vol}(M)$ of a closed surface $M$. Fix $\lambda>0$ then there exists $\epsilon_0,\beta>0$ such that for any $0\le \epsilon\le \epsilon_0$ there exists $D_1$ such that if for $C\ge 0$, $\Lambda^n_{\omega}(C)$ denotes the $(C,\lambda,\epsilon)$ tempered points at time $n$ for $\omega\in \Sigma$, and $n\ge N_0(C)=\lceil (C+\ln(2))/\lambda\rceil$, then restricted to $\Lambda^n_{\omega}(C)$, $E^s_n$ is $\beta$-H\"older with constant $e^{D_1C}$.
\end{prop}

\begin{proof}
We may always study the dynamics in an atlas of uniformly smooth volume preserving charts on $M$. So, in what follows we will implicitly be working with such charts. 

The first claim is an immediate analog of \cite[Lem.~5.3.4]{barreira2007nonuniform}. There exists $\Lambda>0$ such that for $n\in \N$, if $x,y\in M$ with $d(x,y)\le e^{-\Lambda n}$, then (as viewed in charts),
\begin{equation}\label{eqn:tempered_norm_implies_splitting}
\|D_xf^n_{\omega}-D_yf^n_{\omega}\|\le e^{\Lambda n}d(x,y).
\end{equation}

Our plan is to apply Lemma \ref{lem:holder_seq_convergence_lemma}, so we need to estimate the regularity of $E^s_n$. The first thing we need is a lower bound on $n$ for the subspace $E^s_n$ to necessarily exist. From the definition of $(C,\lambda,\epsilon)$ tempered, we see that as long as 
\begin{equation}
n\ge \left\lceil \frac{C+\ln 2}{\lambda}\right\rceil=N_0(C),
\end{equation}
then $\|D_xf^n_{\omega}\|\ge 2$ and hence there is a well defined most contracted subspace.

Next we estimate the H\"older regularity of $E^s_n$ on $\Lambda^N_{\omega}$ for $N\ge n\ge N_0$. If $x\in \Lambda^n_{\omega}$ and $d(x,y)\le e^{-\Lambda n}/2$, then it follows from \eqref{eqn:tempered_norm_implies_splitting} that 
\[
\|D_xf^n_{\omega}-D_yf^n_{\omega}\|\le e^{\Lambda n}d(x,y)\le 1/2. 
\]
Thus, from Proposition \ref{prop:SL2_perturbation_directions_lemma}, as $\|D_xf^n\|\ge 2$, it follows that for $d(x,y)\le e^{-\Lambda n}/2$ that
\begin{equation}\label{eqn:holder_regularity_E_s_n}
d(E^s_n(x),E^s_n(y))<\sqrt{2} e^{\Lambda n}d(x,y),
\end{equation}
which is the small scale H\"older estimate we were seeking.

Next, we study how fast $E^s_n$ fluctuates as we increase $n$. By assumption the sequence of points is $(C,\lambda,\epsilon)$-tempered. Hence by Proposition \ref{prop:tempered_norm_implies_splitting}, there exists $D_8$ depending only on $\lambda,\epsilon$ such that for $n$ greater than or equal to our same $N_0$ it follows that
 on $\Lambda^{\omega}_n(C,\lambda,\epsilon)$

\begin{equation}\label{eqn:fluctuations_inE_s_n_pf}
\angle(E^s_n(x),E^s_{n+1}(x))\le e^{4C+D_8}e^{-2(\lambda-\epsilon)n}.
\end{equation}

We can now apply Lemma \ref{lem:holder_seq_convergence_lemma} to the sequence of distributions $E^s_n$, for $N_0\le n\le N$ by combining estimates \eqref{eqn:holder_regularity_E_s_n} and \eqref{eqn:fluctuations_inE_s_n_pf}. Thus there exists $0<\beta<1$ and $C_3$ such that the $E^s_n$ are $\beta$-H\"older with constant 
\[
\max\{C_3e^{\Lambda N_0},2e^{4C+D_8}(1-e^{-2(\lambda-\epsilon)})^{-1}e^{\Lambda N_0}C_3+C_3e^{-\Lambda N_0}\}
\]
But by our choice of $N_0\approx C/\lambda$ and absorbing some constants into each other, we find that there is some $C_4$ such that the $\beta$-H\"older constant of $E^s_n$ is at most
$\displaystyle C_4e^{((\Lambda/\lambda)+4)C},
$
which gives the needed conclusion.
\end{proof}

The above lemma will give us a H\"older estimate on the regularity of $Df^n(E^s_n)$ as well, which will allow us to define the fake stable manifolds. Before proceeding, we use the above results to record another useful fact about the continuity of the distribution of the stable directions.

\begin{prop}\label{prop:stationary_measure_is_continuous}
Suppose that $(f_1,\ldots,f_m)$ is an expanding on average tuple of diffeomorphisms on a surface $M$ in $\Diff^2_{\vol}(M)$. Let $\nu^s_x$ be the distribution of stable subspaces through $x$, which is a probability measure on $\mathbb{P}T_xM$, the projectivization of $T_xM$. Then if we identify nearby fibres by parallel transport, the map $x\mapsto \nu^s_x$ is continuous in the weak* topology.
\end{prop}

\begin{proof}
Let $\nu^s_x(C,\lambda,\epsilon,n)$ denote the distribution of $E^s_n(\omega)$ for words $\omega$ that are $(c,\lambda,\epsilon)$-tempered for some $c$ in $[C,C+1)$.
Then by Proposition \ref{prop:contracted_subspace_holder_est}, the distribution $E^s_n$ for such words $\omega$  is uniformly H\"older continuous in $n$ for fixed $C$.
So, if $\nu^s_x(C,\lambda,\epsilon)$ denotes the distribution of $E^s_{\omega}(x)$ for $(C,\lambda,\epsilon)$-tempered $\omega$, we see that the measures $\nu^s_x(C,\lambda,\epsilon)$ vary weak* continuously. 
Almost every word $\omega$ is $(C,\lambda,\epsilon)$-tempered for some $C$. 
Thus we see that 
\[
\nu^s_x=\sum_{C=0}^{\infty} \nu^s_x(C,\lambda,\epsilon).
\]
Note that each partial sum of this series varies weak* continuously and that the mass is uniformly absolutely summable pointwise.
Thus the limiting family $\nu^s_x$ is seen to vary weak* continuously.
\end{proof}

\subsection{Construction of fake stable manifolds}\label{subsec:construction_of_fake_stable}

As mentioned above, we will define the fake stable manifolds by taking curves tangent to a smooth approximation to the distribution $V_n$, which is defined to equal $Df^n_{\omega}(E^s_n)$ as above. 
First, we note that Lemma \ref{prop:contracted_subspace_holder_est} above will be applicable to studying the regularity of $V_n$ due to the following.

\begin{lem}\label{lem:temperedness_for_reversed_sequence}
Suppose that $A_1,\ldots,A_n$ is a sequence of linear transformations that are $(C,\lambda,\epsilon)$-tempered. Then the sequence $A_n^{-1},\ldots,A_1^{-1}$ is $(C+\epsilon n,\lambda,\epsilon)$-tempered, and the corresponding splitting is the splitting with the stable and unstable subspaces from the original splitting swapped.
\end{lem}

Using Lemmata \ref{prop:contracted_subspace_holder_est} and \ref{lem:temperedness_for_reversed_sequence}
we can  estimate the regularity of $V_n$. 

\begin{lem}\label{lem:pushforward_is_tempered_norm}
Suppose that $(f_1,\ldots,f_m)$ is a tuple in $\Diff^2_{\vol}(M)$ where $M$ is a compact surface. 
Fix $C,\lambda>0$, then there exist $\beta,\eta>0$ such that for any sufficiently small $\epsilon>0$ there exists $D_1,N\in \N$, such that if $\Lambda^n_{\omega}$ is the set of points that are $(C,\lambda,\epsilon)$-tempered at some time $n \ge N$, 
then the distribution $V_n$ defined on $f^n_{\omega}(\Lambda^n_{\omega})$ by $D_xf^n_{\omega}(E^s_n(x))$, is $\beta$-H\"older with constant $D_1e^{\eta\epsilon n}$.
\end{lem}
\begin{proof}
Apply Proposition \ref{prop:contracted_subspace_holder_est} with $\lambda$ as above to the diffeomorphisms $(f_1^{-1},\ldots,f_m^{-1})$. 
Then there exist $\beta$ and $\epsilon_0$ such that restricted to the set of $(C,\lambda,\epsilon)$-tempered points at time $n\ge O_{\lambda}(C)$, $E^s_n$ is $\beta$-H\"older with constant at most $e^{D_1C}$.
From Lemma \ref{lem:temperedness_for_reversed_sequence}, we see that for the backwards dynamics $(f_{\sigma^{n-i}(\omega)})^{-1}$, the points in $f^n_{\omega}(\Lambda^n_{\omega})$ are $(C+\epsilon n, \lambda, \epsilon)$-tempered. 
Note that $V_n$ is equal to the distribution of the most expanded direction for $(f^n_{\omega})^{-1}$ and that $V_n^{\perp}$ is the most contracted direction of $(f^n_{\omega})^{-1}$. As the set $f^n_{\omega}(\Lambda^n_{\omega})$ is $(C+\epsilon n,\lambda,\epsilon)$-tempered for the backwards dynamics, it follows that as long as $\epsilon$ is sufficiently small and $N_0$ is sufficiently large, for all $n\ge N_0$, $V_n^{\perp}$ is $e^{D_1(C+\epsilon n)}$ $\beta$-H\"older. The statement of the lemma now follows. 
\end{proof}

Next we take a smooth approximation $\wt{V}_n$ to the distribution $V_n$ that will be defined in an open neighborhood of $f^n_{\omega}(\Lambda^n_{\omega})$.
First we extend the domain of $V_n$, and then we smooth the extension. If we do not extend the domain, then we won't be able to integrate the distribution.
If we do not do this smoothing, then we will have little control over the norm of the integral curves to $V_n$ rather than tempered growth in $n$.

\begin{lem}\label{lem:extension_of_holder_dist}
Suppose that $M$ is a smooth closed surface. There exist $D_1,D_2$ such that if $K\subseteq M$ is a subset and $E$ is a distribution defined over $K$ that is $(C,\alpha)$-H\"older then $E$ admits a $(D_1C,\alpha)$-H\"older extension to a neighborhood of $K$ of size $\delta=D_2\min\{1,C^{-1/\alpha}\}$. 
\end{lem}
\begin{proof}
We first prove the result with vector fields instead of distributions.
Cover $M$ by finitely many charts. 
In each chart the vector field $X$ is represented as a map $\phi_0\colon K\to S^1\subset\R^2$. 
The McShane extension theorem \cite[Cor.~1]{mcshane1934extension} says that a $(C,\alpha)$-H\"older function defined from a subset $X$ of an arbitrary metric space to $\R$ admits a $(C,\alpha)$-extension to all of $X$. 
Then we  glue the maps from different charts using a partition of unity. 
This proves the result for vector fields. Note that the resulting vector field is defined on the whole manifold. To obtain the result for distributions, we take a unit vector field on $K$ in the direction of $E$, extend it to
a vectorfield $\tilde{X}$ as above and note that the resulting extension is nonzero inside the $\delta$
neighborhood of $K$, so we can take $\tilde E$ to be the direction of $\tilde X$.
\end{proof}

The content of the following lemma is item (2), the $C^2$ estimate on $\wt{V}_n$. While $V_n$ could be seen to be $C^2$, we have little ability to control its norm; thus we need to produce a more regular approximation to this distribution.

\begin{lem}\label{lem:smoothed_version_of_V_n}
Let $(f_1,\ldots,f_m)$ be a tuple of diffeomorphisms in $\Diff^2_{\vol}(M)$, for $M$ a closed surface. Fix $\lambda>0$. Then there exists $\epsilon_1>0$, $\nu_1,\nu_2>0$ and $N\in \N$, $D_1,D_2,D_3$, such that if for $\epsilon<\epsilon_1$ $\Lambda^{\omega}_n$ denotes the set of $(C,\lambda,\epsilon)$-tempered points, then there exists a distribution $\wt{V}_n$ such that 
 \begin{enumerate}
    \item The domain of $\wt{V}_n$ contains all points within distance $D_1e^{-\epsilon\nu_1 n}$ of the domain of $V_n$.
    \item $\wt{V}_n$ is $C^2$ with $\|\wt{V}_n\|_{C^2}\le D_2e^{\epsilon\nu_1 n}$.
    \item At each $x$ in the domain of $V_n$, $d(\wt{V}_n(x),V_n(x))<D_3e^{-\epsilon \nu_2 n}$. 
\end{enumerate}
\end{lem}

\begin{proof}
First, from Lemma \ref{lem:pushforward_is_tempered_norm}, given $\epsilon>0$ we may choose $\epsilon_1$ sufficiently small that $V_n$ is $\beta$-H\"older with constant $D_1e^{\eta\epsilon n}$. 
Let $\hat{V}_n$ be an extension of $V_n$ obtained from Lemma \ref{lem:extension_of_holder_dist}, then from the H\"older estimate on $V_n$,\; $\hat{V}_n$ is defined in a neighborhood of $Df^n_{\omega}(\Lambda^n_{\omega})$ of size at least $D_1^{-1/\beta}e^{-\eta\epsilon n/ \beta}$. 

We now take a smooth approximation to $\hat{V}_n$. 
For this we can represent $\hat{V}_n$ in charts as a function $\phi\colon U\to S^1\subset \R^2$, then mollify  $\phi$. From \cite[Eq.~(11)]{fisher2013global}, we have estimates for convolution $f_{\epsilon}=f *\psi_{\epsilon}$ of a standard mollifier $\psi_{\epsilon}$ with a compactly supported function $f\colon \R^2\to \R$:
\begin{equation}\label{eqn:mollifier_estimates}
\|f_{\epsilon}\|_{2}\le \epsilon^{\alpha-4}\|f\|_{\alpha}\hspace{1em}\text{ and }\hspace{1em} \|f-f_{\epsilon}\|_0 \le \epsilon^{\alpha}\|f\|_{\alpha}.
\end{equation}
 As  domain of $\hat{V}_n$ has size at least 
 $D_1^{-1/\beta}e^{-\eta\epsilon n/\beta}$, we can mollify with any $\epsilon'<D_1^{-1/\beta}e^{-\eta\epsilon n/\beta}/100$ and obtain a function that is well defined at all points at least distance $D_1^{-1/\beta}e^{-\eta\epsilon n/\beta}/100$ from the boundary of the domain of $\hat{V}_n$. Let $\wt{V}_n$ denote the mollified function restricted to the points in the domain of $\hat{V}_n$ of distance at most $D_1^{-1/\beta}e^{-\eta\epsilon n/\beta}/100$ from the domain of $V_n$. Then taking $\epsilon'=e^{-\nu\epsilon}$ for some large $\nu$, mollifying with $\psi_{\epsilon'}$, and applying the estimates in \eqref{eqn:mollifier_estimates} gives that there exist constants $D_2,D_3,D_4,D_5$ such that 
\[
\|\wt{V}_n\|_2\le D_2e^{-D_3\epsilon n}\text{ and } d(\wt{V}_n,V_n)<D_4e^{-D_5\epsilon n}.
\]
This gives the needed conclusion.
\end{proof}

The use of the distributions $\wt{V}_n$ is that they are integrable and their $C^2$ norm is well controlled. This implies that if we take a holonomy along the distribution, then we will have good control of the norm of the Jacobian. 

\begin{defn}\label{defn:fake_stable_manifolds}
Fix $\lambda>0$ and sufficiently small $\epsilon>0$. Then take $\epsilon_1<\epsilon/\max\{\nu_1,\nu_2\}$ where $\nu_1,\nu_2$ are as in Proposition \ref{lem:smoothed_version_of_V_n}. We consider a collection of $(C,\lambda,\epsilon_1)$-tempered points. Let $\wt{W}_n$ be the foliation defined by the integral curves to $\wt{V}_n$. The fake stable leaf through $x\in \Lambda^n_{\omega}$ is then defined to be 
$\displaystyle
W^s_{n}(\omega,x)=(f^n_{\omega})^{-1}(\wt{W}_n(f^n_{\omega}(x))). 
$
\end{defn}

We will now state basic facts about the fake stable manifolds.
In particular, we show that
the fake stable manifolds of sufficiently small size enjoy uniform contraction.

\begin{prop}\label{prop:fake_stable_C_2_control}
Suppose that $(f_1,\ldots,f_m)$ is an expanding on average tuple of diffeomorphisms in $\Diff^2(M)$, where $M$ is a closed surface. Fix $\lambda>0$. Then there exists $\lambda',\epsilon_0>0$ such that for any $0\le \epsilon\le \epsilon_0$ and any $C$, there exist $N_0,\delta_0,C_0,\alpha>0$ such that if 
$\Lambda_n^{\omega}\subset M$ is any collection of $(C,\lambda,\epsilon)$-tempered points at time $n\ge N_0$ lying in some ball $B_{\delta_0}\subset M$. Then
\begin{enumerate}
\item
\!\!For $N_0\!\!\le\!\! i\!\!\le\!\! n$ the fake stable manifolds $W^s_{i,\delta_0}(\omega,x)$\! exist and have $C^2$ norm at most $C_0$. 
\item
$d(T_xW^s_i,E^s_i(x))\le e^{-\lambda i/2}$.
\item
The fake stable direction $E^s_i$ is $(C_0,\alpha)$-H\"older continuous on $\Lambda_n^{\omega}$.
    \item 
    The fake stable leaves $W^s_{i,\delta_0}(\omega,x)$ vary H\"older continuously in the $C^1$ topology, and the H\"older constants are independent of $N_0\le i\le n$.   
\item
The fake stable leaves $W^s_{i,\delta_0}(\omega,x)$ are contracting,
i.e.~for $y,z\in W^s_{i,\delta_0}(\omega,x)$, for each $0\le k\le i$,\;
$\displaystyle
d_{W^s_{i,\delta_0}(x)}(f^k_{\omega}(y),f^k_{\omega}(z))\le C_0e^{-\lambda'k}.
$
\end{enumerate}
\end{prop}

\begin{proof}[Proof Sketch.]
The claim about the existence and regularity of the fake stable manifolds in (1) essentially follows from the construction of the stable manifolds described in Section \ref{ScStMan} or Proposition \ref{prop:finite_time_smoothing_estimate}, depending on taste. An integral curve to the $\wt{V}_n$ distribution has $C^2$ norm that is order $e^{O(\epsilon)}$, and is almost tangent to the most expanded direction of $(Df^n_{\omega})^{-1}$ allowing us to apply those lemmas.  Similarly, the final item in the lemma says that the dynamics on the fake stable manifolds is contracting. This also follows from the graph transform argument. 
Specifically one can produce this statement by a generalization of Step 1 in the proof of Proposition \ref{prop:finite_time_smoothing_estimate}, which studies the growth in length of curves in the Lyapunov charts.

The statement (2) saying that $T_xW^s_i$ is near to $E^s_i$ is immediate because $\|Df^n_{\omega}\|\!\!\ge\!\! Ce^{\lambda n}$ by assumption. Since $Df^n_{\omega}E^s_i$ and $\wt{V}_n$ are exponentially close, they will attract further under $(Df^n_{\omega})^{-1}$.

The statements about H\"older-ness are standard facts; it follows from the same argument as in \cite[Sec.~5.3]{barreira2007nonuniform} applied for only finitely many iterations. Alternatively,
Lemma \ref{lem:linearized_perturbation_RP1} contains an explicit computation showing that nearby points inherit a nearby splitting.
The proof of that lemma does not rely on any of the claims from this section.
We will not use (4) as everything we need for the main result of this paper follows from (1), (2), and (3). So will will omit detailed proof. The claim essentially follows H\"older continuity of the stable distribution, H\"older continuity of the holonomies, which will be obtained in Proposition \ref{prop:holonomies_converge_exponentially_fast}, and Lemma \ref{lem:holder_seq_convergence_lemma}. Compare for example, with
\cite[Sec.~8.1.5]{barreira2007nonuniform}, which describes a similar argument.
\end{proof}

\subsection{Rate of convergence of fake stable manifolds}

Proposition \ref{prop:fluctuations_in_fake_stable_leaves}, proven in this section, is one of the key estimates in this paper playing an important role in the local coupling procedure. 

The main crucial feature that the fake stable leaves exhibit is that the fluctuations in $W^s_i$ as we increase $i$ decay exponentially fast. In fact, we have a
quantitative estimate that directly relates the speed of convergence of $W^s_i(\omega,x)$ with the hyperbolicty of $D_xf^i_{\omega}$.

In the following proposition, we will use an additional refinement of $(C,\lambda,\epsilon)$-tempered points that also requires that the stable direction points in a particular direction.  
The definition below is structured so that it is hopefully straightforward to think about. 
When a point is $(C,\lambda,\epsilon)$-tempered, there is a definite rate at which $E^s_n$ converges to $E^s$. Thus if $E^s_n$ happens to lie sufficiently far from the boundary of a cone $\mc{C}$ at a sufficiently large time $n_1$, then $E^s_i\in \mc{C}$ for all $i\ge n_1$.

\begin{defn}\label{defn:cone_tempered}
Suppose $x\in M$ and $\mc{C}\subset T_xM$ is a cone. We say that a word $\omega$ is $(C,\lambda,\epsilon,\mc{C},n_1,n_2)$-tempered if for all $n_1\le i\le n_2$, $E^s_{i}$ is defined and lies in $\mc{C}$. We may also speak of being $(C,\lambda,\epsilon, \mc{C})$-tempered at a time $n$, in which case we mean $n_1=n_2=n$ in the previous sentence.
\end{defn}

 We now estimate how much the fake stable leaves fluctuate. The requirements on the cone are, strictly speaking, not necessary in order to state the theorem below: as long as $N$ is chosen sufficiently large, one can use $E^s_N(x)$ to define the cone $\mc{C}$ in the following proposition and obtain the same result.

\begin{prop}\label{prop:fluctuations_in_fake_stable_leaves}
(Fluctuations in fake-stable leaves) 
 Let $(f_1,\!\ldots\!,f_m)$ be a tuple in $\Diff^2_{\vol}(M)$ for a closed surface $M$.
Fix $\lambda,C_1,\theta_0>0$, then there exists $\epsilon_0>0$ such that for all $0\le \epsilon<\epsilon_0$ and $C>0$ there exist $D_1,N,\delta_0>0$ 
 such that for any $\delta\le \delta_0$ the following holds. Given $x\in M$ and a cone $\mc{C}\subset T_xM$, extend $\mc{C}$ by parallel transport to a conefield $\mc{C}$ defined over $B_{2\delta}(x)$. Suppose that $\gamma$ is a $C$-good curve with distance $d(x,\gamma)<\delta$ and $\gamma$ is $\theta_0$ transverse to $\mc{C}$. If $\omega$ is a 
 $(C_1,\lambda,\epsilon,\mc{C}, n, n+1)$-tempered with $n\ge N$, then 
\begin{equation}
d_{\gamma}(W^s_{n}(x)\cap \gamma,W^s_{n+1}(x)\cap \gamma)\le e^{-1.99 \ln \|D_xf^n_{\omega}\|},
\end{equation}
where $W^s_{n}(x), W^s_{n+1}(x)$ are the fake stable manifolds from Definition \ref{defn:fake_stable_manifolds}.
\end{prop}

The reason the proposition follows is evident in the case of a linear map. 
Consider the action of the map $L=\diag(\sigma,\sigma^{-1})$
on $\RP^1$ where $\sigma>1$. 
Note that the map $L$ has an attracting fixed point of multiplier $\sigma^{-2}$, which suggests the asymptotic in the theorem. 
Consider what happens if we apply $L$ to two curves tangent at $(0,0)$ to the expanded direction of $L$: the distance between them will contract by a factor of $\sigma^{-2}$.
The result for a sequence of maps will follow because the temperedness assures a uniform $E^s,E^u$ splitting. 
When we work with this splitting, the full strength of the hyperbolicity will be available allowing us to recover almost $e^{-2\ln \|D_xf^n_{\omega}\|}$ contraction as in the theorem.

The formal proof will rely on the study of the graph transform. 
The argument for this proposition is simpler than the argument in the recovery lemma since the curves we consider in this lemma are (by assumption) well positioned with respect to the stable and unstable splitting. 

There are three steps in the proof. 
We have two curves at $f^n_{\omega}(x)$, one corresponding to the time $n$ fake stable manifolds and one corresponding to the time $n+1$ fake stable manifolds.
In the first step, we iterate the graph transform until these curves look uniformly Lipschitz in the Lyapunov charts. 
In the second step, we iterate the graph transform to see that these two curves approach each other at the appropriate exponential rate.
In the third step, we do some bookkeeping to 
conclude.

\begin{proof}
 Recall that,
 by definition, the fake stable manifold $W^s_{n}(x)$ is given by taking a curve $\gamma_n$ tangent to the distribution $\wt{V}_n$ from Lemma \ref{lem:smoothed_version_of_V_n} and letting $W^s_{n}(x)$ equal $(f^{n}_{\omega})^{-1}(\gamma_n)$ restricted to a segment of length $\delta_0$ about $x$ where $\delta_0$ is chosen as in Proposition \ref{prop:fake_stable_C_2_control}.
Note that we need not take the $\delta_0$ in this proposition to be the same as the one in Proposition \ref{prop:fake_stable_C_2_control}. Indeed, at certain points in the analysis below it may be convenient to decrease $\delta_0$ in a way that depends only on the parameters of the proposition.

The proposition is comparing $(f^n_{\omega})^{-1}(\gamma_n)$ and $(f^{n+1}_{\omega})^{-1}(\gamma_{n+1})$. 
As in previous sections, we will view both of these curves as graphs of functions from $E^u$ to $E^s$ in the Lyapunov charts. 
In this proof we will work with the splitting into stable and unstable subspaces for the subspaces defined by the associated splitting for $Df^n_{\omega}$ rather than $Df^{n+1}_{\omega}$. Recall that $E^s_n$ denotes the most contracted subspace for $Df^n_{\omega}$ and $E^s_{n+1}$ denotes the most contracted subspace for $Df^{n+1}_{\omega}$.

In the Lyapunov charts at $f^j_{\omega}(x)$, we write $(f^{n-j}_{\sigma^{j}(\omega)})^{-1}(\gamma_n)$ as the graph of the function $\phi^1_j$ and we write $(f^{n-j+1}_{\sigma^{j}(\omega)})^{-1}(\gamma_{n+1})$ as the graph of $\phi^2_j(x)$. Let $e^{\Lambda}$ be an upper bound on $\|Df_i\|$, $1\le i\le m$, with $\Lambda>100$.

With respect to the Lyapunov metrics, we use the similar choices as in previous arguments, specifically Proposition \ref{prop:fwd_up_to_epsilon_smoothing}, and thereby obtain essentially identical intermediate estimates. 
View the sequence of maps $f_{\omega}^j$ as being reversed tempered starting at $f^{n+1}_{\omega}(x)$ and ending at $x$.
So, set $\lambda'=.9999\lambda$ and take the finite time Lyapunov metrics as in Lemma \ref{lem:lyapunov_metric} for this sequence.
In particular, note that from the construction of the Lyapunov metrics, the $e^{O(\epsilon n)}$ bound on the $C^2$ norm of the curves $\gamma_n$ from Lemma \ref{lem:smoothed_version_of_V_n} and the angle $\wt{V}_n$ makes with $V_n$ of 
$\displaystyle O\left(e^{O(-\epsilon n)}\right)$ combine to show that there exist $C_2,\nu>0$ such that $\|\phi^1_n\|_1,\|\phi^2_n\|_1\le C_2e^{\nu \epsilon n}$. We now proceed with the proof.

\noindent\textbf{Step 1.} (Lipschitzness) In this step, we will identify $N_l\approx (1-O(\epsilon))n$ such that for $j\le N_l$, $\phi^1_j$ and $\phi^2_j$ are $C^0$ close.

To begin we estimate how far apart $Df^n_{\omega}(E^s_n)$ and $Df^{n}_{\omega}(E^s_{n+1})$ are. 
We claim that there exists $N_0$ such that for $n\ge N_0$, then $\angle(Df^n_{\omega}(E^s_n),Df^{n}_{\omega}(E^s_{n+1}))\le 1/4$.
Note that if $N_0$ is sufficiently large that both $\|Df^{n}_{\omega}\|$ and $\|Df^{n+1}_{\omega}\|$ are at least $e^{10\Lambda}$ and $\angle (E^s_n,E^s_{n+1})<1/100$ both of which follow from  the $(C_1,\lambda,\epsilon)$-temperedness (The latter claim is part of Proposition~\ref{prop:tempered_norm_implies_splitting}).
As in previous computations, it follows that if $\angle(Df^n_\omega E^s_n,Df^{n}_{\omega}E^s_{n+1})\!\!>\!\!1/4$, then $\|Df^n_{\omega}(E^s_{n+1})\|>2$ because $Df^n_{\omega}$ expands $E^u_n$ and contracts $E^s_n$. 
Consequently, $\|Df^{n+1}_{\omega}(E^s_{n+1})\|>2e^{-\Lambda}$. 
But this is not less than $e^{-10\Lambda}$, so it is impossible that
$\angle(Df^n_{\omega}(E^s_n),Df^{n}_{\omega}(E^s_{n+1}))>1/4$. 

Note that in Proposition \ref{prop:finite_time_smoothing_estimate}, we considered smoothing estimates for a reverse tempered point. In the case of this theorem, we may consider $x$ as a reverse tempered point for the sequence of maps $(f^{j}_{\sigma^{n-j}(\omega)})^{-1}$ beginning at $f^{n}_{\omega}(x)$. Consequently, we may read off the intermediate estimates from the proof of that theorem. In particular, as in equation \eqref{eqn:inductive_C_1_est} by possibly restricting the domain of $\phi^1_j$ and $\phi^2_j$ as in that proposition, it follows that there exists $C_3$ such that for $i\in \{1,2\}$ that
\[
\|\phi^i_{n-j}\|_1\le C_3e^{\nu\epsilon n}e^{-j\lambda}.
\]
In particular this shows that if we let $N_l=\lfloor n-\nu\epsilon/\lambda n\rfloor$, then because both curves pass through $0$ and our choice of $N_l$, we see that there exists $C_4$ such that for $i\in \{1,2\}$,
$\|\phi^i_j\|_1\le C_4$. 
Because both pass through $0$, the following estimate holds for all $N_0\le j\le N_l$:
\begin{equation}\label{eqn:basic_lipschitz_separation_estimate}
\abs{\phi^1_j(x)-\phi^2_j(x)}\le 2C_4\abs{x},
\end{equation}
which is the desired estimate for this step in the proof. 

\noindent\textbf{Step 2.} (Contraction) In this step, we study how fast the curves $\phi^1_j$ and $\phi^2_j$ attract as we apply the dynamics $(f_{\sigma^j(\omega)})^{-1}$. Our goal is to show that the $C^0$ distance between these functions is rapidly decreasing, which is the content of  \eqref{eqn:phi_1_2_j_vert_dist}.

First, in the Lyapunov chart we have
\begin{equation}
\hat{f}_{\sigma^j(\omega)}^{-1}=(e^{\sigma_j^1}x +\hat{f}_{j,1}(x,y), e^{\sigma_j^2}y+\hat{f}_{j,2}(x,y)),
\end{equation}
where $\min\{\sigma_j^1,-\sigma_j^2\}\ge .999\lambda$. Then in the Lyapunov charts, the differential is
\begin{equation}\label{eqn:differential_at_sigma_j}
D\hat{f}^{-1}_{\sigma^j(\omega)}=\begin{bmatrix}
e^{\sigma_{j,1}} +\partial_x \hat{f}_{j,1}& \partial_y\hat{f}_{j,1}\\
\partial_x \hat{f}_{j,2}& e^{\sigma_{j,2}}+\partial_y\hat{f}_{j,2}
\end{bmatrix}.
\end{equation}
In addition, write 
\begin{equation}
\Lambda_j=\sum_{i=j}^{N_l} \sigma_{j,1}-\sigma_{j,2}.
\end{equation}
As in Proposition \ref{prop:finite_time_smoothing_estimate}, we have a $C^2$ estimate in the Lyapunov charts. There exists $C_5>0$ such that 
 \begin{equation}
 \|(\hat{f}_{\sigma^i(\omega)})^{-1}\|_{C^2}\le C_5e^{6C_1} e^{6 i \epsilon}.
\end{equation} 

We will now verify inductively that a strengthening of \eqref{eqn:basic_lipschitz_separation_estimate} holds for $N_0<j<N_l$. We now show that by possibly increasing $N_0$, which is fixed and does not depend on $n$, that 
for all $\abs{x}<e^{-{ (\lambda/2)}j}$, and $N_0\le j<N_l$,
\begin{equation}\label{eqn:phi_1_2_j_vert_dist}
    \abs{\phi^1_j(x)-\phi^2_j(x)}\le C_4e^{-1.999\Lambda_j}\abs{x}.
\end{equation}

To show \eqref{eqn:phi_1_2_j_vert_dist}, we measure the distance between $\phi^1_j$ and $\phi^2_j$ using a piece of the 
vertical curve $V(t)$ parallel to $E^s$ between $\phi^1_{j+1}(x)$ and $\phi^2_{j+1}(x)$. We then apply $(f_{\sigma^{j}(\omega)})^{-1}$ to the curve and estimate its length. We then use the Lipschitzness of $\phi^1_{j}$ and $\phi^2_{j}$ to obtain 
\eqref{eqn:phi_1_2_j_vert_dist}.
Let $V(t)$  be a vertical curve (parallel to $E^s$) defined on $[-1,1]$ taking values in the Lyapunov charts such that $V(-1)\in\phi_{j+1}^1$ and $V(1)\in \phi_{j+1}^2$ passing through the point $(x,0)$. 
Then from the inductive hypothesis, we see that $\len(V)\le C_4e^{-1.999\Lambda_{j}}\abs{x}$. 

By applying the differential to $V$, we see by \eqref{eqn:differential_at_sigma_j}, $(\hat{f}_{\sigma^j(\omega)})^{-1}(V)$ is tangent to a vector of the form 
\begin{equation}\label{eqn:action_differential_on_V_t}
\partial_t((\hat{f}_{\sigma^j(\omega)})^{-1}V(t))=\begin{bmatrix}
    \partial_y\hat{f}_{j,1}\\
e^{\sigma_{j,2}}+\partial_{y}\hat{f}_{j,2}
\end{bmatrix}.
\end{equation}
In particular, for $C_5$ as before if we are restricted to a ball of radius $C_5^{-1}e^{-(\lambda/2) j}$, then as the $C^2$ norm of $(\hat{f}_{\sigma^j(\omega)})^{-1}$ is $O(e^{6j\epsilon})$, it follows that 
\begin{equation}\label{eqn:estimate_on_f_j_i_hat}
\abs{\partial_y\hat{f}_{j,i}}<e^{ -(\lambda/4) j}
\end{equation}
for $i\in \{1,2\}$. 
Let $\pi_u$ be the projection onto the $E^u$ direction in the Lyapunov coordinates and let $\pi_s$ be the projection onto the $E^s$ direction in the Lyapunov coordinates.
We see that there exists $C_6$ such that:
\begin{equation}\label{eqn:vertical_dist_est_1}
\abs{\pi_s((\hat{f}_{\sigma^j(\omega)})^{-1}V(-1))-\pi_s((\hat{f}_{\sigma^j(\omega)})^{-1}V(1)))}\le 
C_4e^{-1.999\Lambda_j}e^{(1-\epsilon_j)\sigma_{j,2}}\abs{x} 
\end{equation}
where $\abs{\epsilon_j}\le C_6e^{-\lambda/4 j}$.

We now use \eqref{eqn:vertical_dist_est_1} to estimate the $C^0$  norm of $\phi^1_j$ and $\phi^2_j$, rather than just the distance between two points along these curves.
The endpoints of $(\hat{f}_{\sigma^j(\omega)})^{-1}V(t)$ lie in $\phi^1_j$ and $\phi^2_j$.
Note that when $(\hat{f}_{\sigma^j(\omega)})^{-1}V$ is viewed as a graph over the vertical line parallel to $E^s$ through $\pi_u(\hat{f}_{\sigma^j\omega})^{-1}(x,0)$, that $(\hat{f}_{\sigma^j(\omega)})^{-1}V$ is distance at most $e^{-\lambda/4 j}\len(V)$ from a vertical line by \eqref{eqn:action_differential_on_V_t} and \eqref{eqn:estimate_on_f_j_i_hat}.
Thus as $\phi^1_j$ and $\phi^2_j$ are both $C_4$ Lipschitz for $N_0\le j\le N_l$, we see that
\begin{align}\notag
\abs{\phi^1_j(\pi_1(\hat{f}_{\sigma^j\omega})^{-1}(x,0))-\phi^2_j(\pi_1(\hat{f}_{\sigma^j\omega})^{-1}(x,0))}&<C_4e^{-1.999\Lambda_j}e^{(1-\epsilon_j)\sigma_{j,2}}\abs{x}+C_4e^{-\lambda/4 j}\len(V)\\
&\le (e^{(1-\epsilon_j)\sigma_{j,2}}+C_4e^{-\lambda/4 j})e^{-1.999\Lambda_{j}}\abs{x}.\label{eqn:concluding_line_epsilon_j}
\end{align}
As long as $N_0$ is sufficiently large, for $j\ge N_0$,
\begin{equation}\label{eqn:bound_on_x_epsil1}
\abs{x}\le e^{-(1-\epsilon_j)\sigma_{j,1}}\abs{\pi_1(\hat{f}_{\sigma^j(\omega)})^{-1}(x,0))}.
\end{equation}
Note that if $j$ is larger than some fixed $N_0$ and $\epsilon_j$ is sufficiently small relative to $\lambda$, then
\begin{equation}\label{eqn:smallness_epsilon_j_bound}
(e^{(1-\epsilon_j)\sigma_{j,2}}+C_4e^{-\lambda/4 j}) e^{-(1-\epsilon_j)\sigma_{j,1}}\le e^{1.999(\sigma_{j,2}-\sigma_{j,1})}.
\end{equation}
Combining \eqref{eqn:concluding_line_epsilon_j}, \eqref{eqn:bound_on_x_epsil1}, and \eqref{eqn:smallness_epsilon_j_bound}, we get
$\displaystyle
\abs{\phi^1_{j}(x)-\phi^2_{j}(x)}\le C_4e^{-{ 1.999}\Lambda_{j-1}},
$
as required.

\noindent\textbf{Step 3.} (Bookkeeping and Conclusion) 
So far,  we have obtained that for some $N_0$ and $C_4$ depending only on the constants in the theorem
\[
\abs{\phi^1_{N_0}(x)-\phi^2_{N_0}(x)}\le C_4e^{-1.999\Lambda_{N_0}}
\]
Thus as $\phi^1_0$ and $\phi^2_0$ are related to $\phi^1_{N_0}$ and $\phi^2_{N_0}$ by applying only the fixed number $N_0$ more maps, we see that there exists $C_7$ and $\delta_2>0$ such that on a ball of radius $\delta_2$ in the Lyapunov charts at $x$:
\[
\abs{\phi^1_{0}(x)-\phi^2_{0}(x)}\le C_7e^{-1.999\Lambda_{N_0}}.
\]

Consider a nearby $C$-good curve $\gamma$ that is $\theta_0$-transverse to $\mc{C}$ and hence to $E^s$, $\phi^1_0$, and $\phi^2_0$. It then follows easily from transversality, that as $\phi^1_0$ is nearly tangent to $E^s$ by Proposition \ref{prop:fake_stable_C_2_control}(2) and $\phi^1_0,\phi^2_0$ are uniformly Lipschitz, there exists $C_8$ such that 
\[
d_{\gamma}(\phi^1_0\cap \gamma,\phi^2_0\cap \gamma)\le C_8e^{-1.999\Lambda_{N_0}}.
\]

The only remaining thing we need is to know that $\Lambda_{N_0}$ is within a factor of $.001\Lambda$ of $\ln\|Df^n_{\omega}\|$. 
This will follow as long as we take $\epsilon$ sufficiently small relative to $\lambda,\nu_1,\nu_2$ and the maximum of the norm of the differentials of $f_1,\ldots,f_m$. We omit the computation of exactly how small $\epsilon$ must be. 
Such sufficiently small $\epsilon$ exists
because when we look in the Lyapunov charts, we obtain the straightforward bound that there exists $C_9$ such that
\[
 \ln \|Df^n_{\omega}\|\le C_9+ 4\epsilon n+\sum \sigma_{j,1}.
\]
But $\Lambda_{N_0}$ includes only the hyperbolicity for the iterates $N_0\le j\le N_l$. From volume preservation of the $f_i$, it similarly follows that $\ln \|Df^n\|\le C_{10}+ 4\epsilon n -\sum\sigma_{j,2}$ for some $C_{10}$.
As $N_l=(1-O(\epsilon))n$ and $N_0$ is a fixed independent of $n$, it follows that for sufficiently small $\epsilon$ and sufficiently large $n$ that
$
e^{1.99\ln\|Df^n\|}\le e^{1.999\Lambda_{N_0}},
$
which is the needed conclusion.
\end{proof}

\subsection{Jacobian of the fake stable holonomies}

Now that we have defined the fake stable manifolds and have an estimate for the rate at which their holonomies converge,
we study the Jacobian of their holonomies, whose properties are crucial in the coupling argument. 
The next quantity of interest is the fluctuations in the Jacobian of the holonomies for the fake stable manifolds. 

\begin{prop}\label{prop:holonomies_converge_exponentially_fast}
Suppose that $(f_1,\ldots,f_m)$ is a tuple of diffeomorphisms in $\Diff^2_{\vol}(M)$ for a closed surface $M$.
For $\lambda>0$ there exists $\epsilon_0>0$ such that for all $0\le \epsilon\le \epsilon_0$ and $C>0$, there exists $N\in \N$ and $\delta, \eta,\alpha>0$ such that for any $n\ge N$, 
and any $\omega\in \Sigma$, if $\Lambda_n^{\omega}$ is the set of $(C,\lambda,\epsilon)$-tempered points up to time $n$ then for any ball $B_{\delta}\subseteq M$ of radius $\delta$, the following holds for $x\in \Lambda^{\omega}_n\cap B_{\delta}$.

For any two uniform transversals $T_1$ and $T_2$ to the $W_N^s$ laminations of $B_{\delta}(x)$, $T_1$ and $T_2$ will be uniform transversals to the $W_i^s$ lamination for $N\le i\le n$.
Where defined, consider the holonomies $H^s_i$ between $T_1$ and $T_2$ and moreover the Jacobian $\Jac H^s_i$, which is defined on a subset of $T_1$. 
Then we have the following for all $N\le i\le n$:

(1)
The Jacobians of the holonomies between uniform transversals are uniformly $\alpha$-H\"older and bounded away from zero. 
In particular, this implies that these Jacobians are uniformly log-$\alpha$-H\"older between uniform transversals. 
Specifically,
for fixed $(C_1,\delta_1)$, there exist $D_1,D_2,D_3$ such that if $\gamma_1$ and $\gamma_2$ are a $(C_1,\delta_1)$-configuration in the sense of Definition \ref{defn:C_0_delta_configuration}
with $\gamma_1$ and $\gamma_2$ uniformly transverse to the $E^s_N(x)$ extended by parallel transport in a small neighborhood, and $I\subseteq \Lambda^{\omega}_n$ is a subset of $\gamma_1$
then, for $x,y\in I$,
\begin{equation}
\abs{\log \Jac H^s_n(x)-\log \Jac H^s_n(y)}\le D_1d_{\gamma_1}(x,y)^{\alpha}.
\end{equation}

(2)  The Jacobians from item (1) converge exponentially quickly, i.e.
\begin{equation}
\abs{\Jac H^s_{i-1}-\Jac H^s_{i}}\le D_2e^{-\eta i},
\end{equation}
and
\begin{equation}
\abs{\frac{\Jac H^s_i}{\Jac H^s_{i-1}}-1}\le D_3e^{-\eta i}.
\end{equation}

 (3) The true stable holonomy restricted to $\Lambda^{\omega}_\infty\cap T_1$ is absolutely continuous.
The Jacobian of the fake stable holonomies converges to the Jacobian of the true stable holonomies restricted to the set $\Lambda^{\omega}_\infty\cap T_1$. Namely,
for almost every point of this intersection, $\Jac H^s_n\to \Jac H^s$, this convergence is uniform, and the limit is uniformly H\"older and bounded away from zero.

\end{prop}
\begin{proof}
\noindent\textbf{Part 1.} (Formula for Jacobian) 
We begin  by exhibiting a formula for the Jacobian of the stable holonomies. This may be compared with \cite[Sec.~8.6.4]{barreira2007nonuniform}, which uses a similar formula though analyzes it differently.
Suppose that $T_1$ and $T_2$ are the two transversals we are considering as in the statement of the proposition. 
Then write $\Pi_i^s$ for the holonomy along $f^i_{\omega}(W^s_i)=\wt{W}_i^s$, the smooth integral curves to $\wt{V}_i$ we used when defining the fake stable foliation.
Then we have the following formula for the Jacobian of $H^s_i$:
\begin{equation}\label{eqn:finite_time_holonomy_formula}
\Jac(H^s_i)(y)=\prod_{k=0}^{i-1} \frac{\Jac(D(f_{\sigma^{k}\omega})^{-1}\vert T_{f^k_{\omega}{H^s_i(y)}}f^k_{\omega}(T^2))}{\Jac(D(f_{\sigma^k\omega})^{-1}\vert T_{f^k_{\omega}(y)}f^k_{\omega}(T^1))}\Jac(\Pi^s_i(y)).
\end{equation}
For finite time this formula is evident because all of the foliations we are considering are smooth: it is just the change of variables formula.

\noindent\textbf{Part 2.} (Exponential convergence)
Applying Lemma \ref{lem:holder_seq_convergence_lemma} we will obtain H\"older continuity for the Jacobians once we know that $\Jac(H^s_i)$ is converging exponentially fast. 

To see that \eqref{eqn:finite_time_holonomy_formula} converges exponentially quickly, two estimates are needed. 

(1) The first is showing that for some $\eta>0$
\begin{equation}\label{eqn:exp_convergence_of_Jacobian1}
\abs{\Jac(\Pi^s_{n})-1}\le C_1e^{-n\delta_1}.
\end{equation}
This is the Jacobian of the foliation holonomy of $\wt{W}^s_n$.  The foliation holonomy is between two transversals that are distance $e^{-(\lambda/2)n}$ apart. By working in Lyapunov charts, it is straightforward to see that the $f^n_{\omega}(T_1)$ and $f^n_{\omega}(T_2)$ make angle at least $Ce^{-\epsilon n}$ with $\wt{W}^s_n$. As $\wt{W}^s_n$ itself has $C^2$ norm at most $e^{O(\epsilon n)}$ from Lemma \ref{lem:smoothed_version_of_V_n}, it is easy to see that there exists some $C_1,\delta_1>0$ such that \eqref{eqn:exp_convergence_of_Jacobian1} holds.

(2) Next we estimate the rate of convergence of:
\begin{equation}\label{eqn:jacobian_to_est}
\prod_{k=0}^{i-1} \frac{\Jac(D(f_{\sigma^{k}\omega})^{-1}\vert T_{f^k_{\omega}{H^s_i(y)}}f^k_{\omega}(T^2))}{\Jac(D(f_{\sigma^k\omega})^{-1}\vert T_{f^k_{\omega}(y)}f^k_{\omega}(T^1))}=\exp\left(\sum_{k=0}^{i-1} P(k,i)\right),
\end{equation}
where  $P(k,i)$ is the logarithm of the $k$th term of the product. 

We claim that there exist $C_2,\delta_2,N_2$, such that
for $i\ge N_2$ and $k\ge 0$,
 \begin{equation}\label{eqn:rate_of_convergence_relative_jacobians}
\abs{\frac{\Jac(D(f_{\sigma^{k}\omega})^{-1}\vert T_{f^k_{\omega}{H^s_i(y)}}f^k_{\omega}(T^2))}{\Jac(D(f_{\sigma^k\omega})^{-1}\vert T_{f^k_{\omega}(y)}f^k_{\omega}(T^1))}-1}\le C_2e^{-\delta_2k}.
\end{equation}
We will not give a detailed proof of this estimate because it standard. The key claim is that if $V_1$ and $V_2$ are the tangent vectors to $\gamma_1$ and $\gamma_2$ at $y$ and $H^s_i(y)$, respectively, then there exists a uniform constant $C_2'$ and $\varpi>0$ such that when we identify $Df^k_{\omega}V_1$ and $Df^k_{\omega}V_2$ by parallel transport along the distance minimizing geodesic between their basepoints, then
\begin{equation}\label{eqn:tangent_vectors_attract}
d(Df^k_{\omega}V_1,Df^k_{\omega}V_2)\le C_2'e^{-k\varpi}.
\end{equation}
One can deduce this in a very similar way to the argument for \cite[Lem.~III.3.7]{mane1987ergodic}, which inductively checks that as one applies more iterates of the dynamics that these two vectors attract exponentially quickly by using that the basepoints of the vectors do as well; this argument is similar to the proof of our Proposition \ref{prop:nearby_points_inherit_temperedness}. 
Once \eqref{eqn:tangent_vectors_attract} is known, then it is straightforward to conclude \eqref{eqn:rate_of_convergence_relative_jacobians} because the Jacobian of a diffeomorphism $f\colon M\!\!\to\!\! M$ restricted to a curve $\gamma\!\subset\! M$ depends H\"older continuously on the direction of  $\dot\gamma$.

 \eqref{eqn:rate_of_convergence_relative_jacobians} shows that the product \eqref{eqn:jacobian_to_est} is uniformly bounded. It then suffices to estimate:
\[
\sum_{k=0}^{i-1} P(k,i)-\sum_{k=0}^{i} P(k,i+1).
\]
We will pick some $0<\theta<1$, and split this sum as follows:
\[
\sum_{k=0}^{\theta i}\left( P(k,i)-P(k,i+1)\right)+\left[\sum_{k\ge \theta i}^i P(k,i)-\sum_{k\ge \theta i}^i P(k,i+1)\right]= I+II.
\]
For any such $\theta$, it follows from \eqref{eqn:rate_of_convergence_relative_jacobians} that there exists $C_3,\delta_3>0$ such that  $\abs{II}\le C_3e^{-\delta_3 i}$. Thus to conclude we need only bound term $I$.
From Proposition \ref{prop:fluctuations_in_fake_stable_leaves} and the temperedness, we know that there exists $C_4,\delta_4$ such that 
\begin{equation}\label{eqn:fluctuations_holonomies_2}
    d_{T_2}(H^s_i(y),H^s_{i+1}(y))\le C_4e^{-\delta_4 i}.
\end{equation}

It is straightforward to see that there exists $\beta,\beta_1>0$ such that the function 
\[
\frac{\Jac(D(f_{\sigma^{k}\omega})^{-1}\vert T_{f^k_{\omega}{H^s_i(y)}}f^k_{\omega}(T^2))}{\Jac(D(f_{\sigma^k\omega})^{-1}\vert T_{f^k_{\omega}(y)}f^k_{\omega}(T^1))}
\]
viewed as a function of $H^s_i(y)$
is $\beta$-H\"older with the H\"older constant at most $e^{\beta_1 k}$ for all $k\le i$. Thus by combining \eqref{eqn:fluctuations_holonomies_2} with the H\"older continuity, we see that $\abs{P(k,i)-P(k,i+1)}\le e^{\beta k}e^{-\delta i}$. Thus as long as $\theta>\beta/\delta$, we see that there exists $C_5,\delta_5$, such that 
\[
\abs{I}\le C_5e^{-\delta_5 i}. 
\]
Combining the estimates on  $I$ and $II$ implies that there exists $C_6,\delta_6$ so that \eqref{eqn:jacobian_to_est} is converging exponentially fast, as desired.

Thus we see that the Jacobian of the holonomies converges exponentially fast pointwise and is uniformly positive. Thus we have concluded (2) of the statement of the proposition.

\noindent\textbf{Part 3.} (Uniform H\"olderness) 
We now apply Lemma \ref{lem:holder_seq_convergence_lemma}. We have just shown that the Jacobian of the holonomies is converging exponentially fast, and certainly the H\"older norm of the terms is growing at most exponentially fast as well as it is the composition of diffeomorphisms along with a holonomy, whose H\"older norm is also growing at most exponentially fast. Thus we conclude (1) above.

\noindent\textbf{Part 4.} The final claim (3) about the holonomies is fairly standard. The following lemma implies the conclusion:
\begin{lem}
Let $\gamma_1$ and $\gamma_2$ be two curves with finite Lebesgue measure and for $n\in \N$ let $\Omega_n\subseteq \gamma_1$ be a decreasing sequence of subsets, each of which is a union of intervals. Suppose that $\displaystyle K:=\bigcap_{n\ge \N}\Omega_n$ is compact.
 Let 
 $\phi_n\colon \Omega_n\to \gamma_2$ be a sequence of absolutely continuous maps with uniformly continuous, equicontinuous Jacobians $J_n$. If $(\phi_n)$ converges uniformly to an injective map $\phi\colon K\to \gamma_2$, and $J_n\vert_k$ converges uniformly to an integrable function $J\colon K\to \R$, then $\phi$ is absolutely continuous with Jacobian $J$.
\end{lem}
We will not include a proof of the above lemma since it is a variant of a lemma in Ma\~n\'e \cite[Thm.~3.3]{mane1987ergodic}
and the proof of \cite{mane1987ergodic} can be modified to obtain a proof of this lemma.
\end{proof}

\bibliographystyle{amsalpha}
\bibliography{biblio.bib}

\end{document}